\documentclass{amsart}
\usepackage[all]{xy}
\usepackage{latexsym}
\usepackage{amssymb}
\usepackage{amsmath}
\usepackage{amsthm}
\usepackage{amscd}
\usepackage{fancyhdr}
\usepackage[dvips]{graphicx}
\usepackage{fullpage}
\usepackage{mathrsfs}
\input cyracc.def
\xyoption{arc}

 \newfam\cyrfam

  \font\tencyr=wncyr10

  \font\sevencyr=wncyr7

  \font\fivecyr=wncyr5

  \textfont\cyrfam=\tencyr \scriptfont\cyrfam=\sevencyr

    \scriptscriptfont\cyrfam=\fivecyr

  \newfam\cyifam

  \font\tencyi=wncyi10

  \font\sevencyi=wncyi7

  \font\fivecyi=wncyi5

  \textfont\cyifam=\tencyi \scriptfont\cyifam=\sevencyi

    \scriptscriptfont\cyifam=\fivecyi

\def\id{{\mbox{1 \hskip -7pt 1}}}

 \newcommand{\lon}{\longrightarrow}
 \newcommand{\bu}{\bullet}
 
 \newcommand{\rar}{\rightarrow}
 \newcommand{\hook}{\hookrightarrow}

\newcommand{\p}{{\partial}}
\newcommand{\Id}{{\mathrm I\mathrm d}}

\newcommand{\oC}{{\overline{C}}}

 \newcommand{\Z}{{\mathbb Z}}
 \newcommand{\bS}{{\mathbb S}}
 \renewcommand{\P}{{\mathbb P}}
 \newcommand{\C}{{\mathbb C}}
 \newcommand{\R}{{\mathbb R}}
 \newcommand{\N}{{\mathbb N}}
 \newcommand{\K}{{\mathbb K}}

\newcommand{\V}{{\mathbb V}}

 \newcommand{\bbH}{{\mathbb H}}
\newcommand{\Conf}{{\mathit C\mathit o \mathit n\mathit f}}
 \newcommand{\ot}{\otimes}

\newcommand{\sG}{{\mathsf G}}

 \newcommand{\Beq}{\begin{equation}}
 \newcommand{\Eeq}{\end{equation}}
 \newcommand{\Beqr}{\begin{eqnarray}}
 \newcommand{\Eeqr}{\end{eqnarray}}
 \newcommand{\Beqrn}{\begin{eqnarray*}}
 \newcommand{\Eeqrn}{\end{eqnarray*}}
 \newcommand{\Ba}{\begin{array}}
 \newcommand{\Ea}{\end{array}}
 \newcommand{\Bi}{\begin{itemize}}
 \newcommand{\Ei}{\end{itemize}}
 \newcommand{\Bc}{\begin{center}}
 \newcommand{\Ec}{\end{center}}
 \newcommand{\fg}{{\mathfrak g}}

\newcommand{\fp}{{\mathfrak p}}
\newcommand{\ii}{{\mathfrak i}}

\newcommand{\fC}{{\mathfrak C}}
 \newcommand{\fG}{{\mathfrak G}}

 \newcommand{\f}{{\mathcal O}}
 \newcommand{\cA}{{\mathcal A}}
 
 \newcommand{\cC}{{\mathcal C}}
 
 \newcommand{\cE}{{\mathcal E}}
 \newcommand{\cF}{{\mathcal F}}
 \newcommand{\cG}{{\mathcal G}}

 \newcommand{\cJ}{{\mathcal J}}
 
 \newcommand{\caL}{{\mathcal L}}
 \newcommand{\cM}{{\mathcal M}}
 
 \newcommand{\cP}{{\mathcal P}}
 
 \newcommand{\cS}{{\mathcal S}}
 \newcommand{\cT}{{\mathcal T}}

 \newcommand{\cU}{{\mathcal U}}

 \newcommand{\ccc}{{\mathit c}}
 
\newcommand{\cce}{{\mathit e}}

\newcommand{\cci}{{\mathit i}}

\newcommand{\ccm}{{\mathit m}}
\newcommand{\ccn}{{\mathit n}}
\newcommand{\cco}{{\mathit o}}

\newcommand{\ccr}{{\mathit r}}

\newcommand{\cct}{{\mathit t}}
\newcommand{\ccu}{{\mathit u}}

 \newcommand{\al}{\alpha}
 \newcommand{\be}{\beta}
 \newcommand{\ga}{\gamma}
 
 \newcommand{\Ga}{\Gamma}

 \newcommand{\var}{\varepsilon}
 \newcommand{\la}{\lambda}
 \newcommand{\om}{\omega}

 \newcommand{\Img}{{\mathsf I\mathsf m}\, }
 \newcommand{\Hom}{{\mathrm H\mathrm o\mathrm m}}
 
 \newcommand{\sip}{\smallskip}
 \newcommand{\bip}{\bigskip}
 \newcommand{\mip}{\vspace{2.5mm}}

\theoremstyle{plain}
\swapnumbers
\newtheorem{theorem}{Theorem}[subsection]
\newtheorem{corollary}[theorem]{Corollary}

\newtheorem{prop-def}[theorem]{Proposition-definition}

\newtheorem{f-theorem}{Formality Theorem}[section]
\newtheorem{main-theorem}{Main~Theorem}[section]
\newtheorem{section-theorem}{Theorem}[section]

\newtheorem{subtheorem}{Theorem}[subsubsection]
\newtheorem{subcorollary}{Corollary}[subsubsection]

\theoremstyle{definition}

\renewcommand{\thesubsection}{\bf\arabic{section}.\arabic{subsection}}
\renewcommand{\thesubsubsection}{\bf\arabic{section}.\arabic{subsection}.\arabic{subsubsection}}

\begin{document}

 \sloppy

 \newenvironment{proo}{\begin{trivlist} \item{\sc {Proof.}}}
  {\hfill $\square$ \end{trivlist}}

\long\def\symbolfootnote[#1]#2{\begingroup%
\def\thefootnote{\fnsymbol{footnote}}\footnote[#1]{#2}\endgroup}

 \title{Operads, configuration spaces and quantization}
 \author{ S.A.\ Merkulov}
\address{Sergei~A.~Merkulov: Department of Mathematics, Stockholm University, 10691 Stockholm, Sweden}
\email{sm@math.su.se}

 \begin{abstract}
 We review several well-known operads of compactified configuration spaces and construct several new such operads, $\overline{C}$,
 in the category of smooth manifolds with corners
 whose  complexes of fundamental chains give us
(i) the 2-coloured operad of $A_\infty$-algebras and their homotopy morphisms, (ii)
 the 2-coloured operad of $L_\infty$-algebras and their homotopy morphisms, and (iii)
 the 4-coloured operad of open-closed homotopy algebras and their homotopy morphisms.

\sip

Two gadgets --- a (coloured) operad of Feynman graphs and a de Rham field theory
on $\overline{C}$ --- are introduced and used to construct  quantized representations
of the (fundamental)
 chain
 operad of
 $\overline{C}$ which are given by Feynman type sums over graphs and depend on choices of
propagators.


\sip

\end{abstract}
 \maketitle
\markboth{S.A.\ Merkulov}{Exotic automorphisms of open-closed homotopy algebras}

{\small
{\small
\tableofcontents
}
}
{\large
\section{\bf Introduction}
}

{\bf 1.1. Configuration spaces}. This paper is inspired by Kontsevich's proof \cite{Ko} of his celebrated
formality theorem.
A central role in that proof
is played by a 2-coloured operad of compactified configuration spaces, $\overline{C}(\bbH)=\overline{C}_\bu(\C)\bigsqcup
\overline{C}_{\bu,\bu}(\bbH)$, whose
associated operad of fundamental chains, $\cF \cC hains(\overline{C}(\bbH))$,
 was termed in \cite{KS}  an operad, $\f\cC_\infty$,  of {\em open-closed homotopy algebras}.

\sip

We review in this paper the operad $\overline{C}(\bbH)$, its lower and higher dimensional versions, and also construct
several new operads, $\overline{C}$, of compactified configuration spaces
 in the category of smooth manifolds with corners (or in the category of semialgebraic manifolds)
 whose  complexes of fundamental chains,  $\cF \cC hains(\overline{C})$, give us
\Bi
\item[(i)] the 2-coloured operad of $\cA_\infty$-algebras and their homotopy morphisms, $\cM or(\cA_\infty)$,
\item[(ii)] the 2-coloured operad of $\caL_\infty$-algebras and their homotopy morphisms, $\cM or(\caL_\infty)$, and
\item[(iii)]
 the 4-coloured operad of open-closed homotopy algebras and their homotopy morphisms, $\cM or(\f\cC_\infty)$.
\Ei
An upper-half space model for $\cM or(\caL_\infty)$ was studied earlier in \cite{Me-Auto}; in this paper we
introduce several other
configuration space models for this important 2-coloured operad including the ones which use configurations of points
in the complex plane $\C$.

\sip

{\bf 1.2. Operads of Feynman graphs}. Kontsevich formality map $F$ is given by a sum \cite{Ko},
$$
F=\sum_{\Ga\in \fG} c_\Ga \Phi_\Ga
$$
where the summation runs over a family of graphs $\fG$ and, for each graph $\Ga\in \fG$ $c_\Ga$ is
 a complex number
 given by an integral over a  fundamental chain in $\overline{C}_{\bu,\bu}(\bbH)$ of a differential
 form $\Omega_\Ga$,
 and $\Phi_\Ga$ is a certain polydifferential operator.  We show that the family $\fG$ can be
 equipped with a natural
 structure of a 2-coloured {\em operad of  Feynman graphs}\, which admits a canonical representation
$$
\Ba{rccc}
\rho: & \fG &\lon & \cE nd_{\{\cT_{poly}(V), \f_V\}}\\
& \Ga &\lon & \Phi_\Ga
\Ea
$$
into the two-coloured endomorphism operad generated by the vector space of smooth (formal) polyvector
 fields $\cT_{poly}(V)$ and the vector space, $\f_V$, of smooth (formal) functions on an affine space $V$. This representation
 is given precisely by the aforementioned polydifferential operators $\Phi_\Ga$.
 \sip

One can construct natural analogues of $\fG$ for any (coloured) operad of compactified configuration spaces, $\overline{C}$, studied in this
paper. To distinguish these (coloured) operads of Feynman graphs from  each other we use an appropriate subscript,
$\fG_{\overline{C}}$, to indicate which  geometric operad $\overline{C}$ an operad of Feynman diagrams $\fG$
is associated to (or,  speaking plainly,  which space the graphs from $\fG$ are drawn on).

\mip

{\bf 1.3. De Rham field theories on $\overline{C}$}. The numbers $c_\Ga=\int_{\overline{C}_{\bu,\bu}(\bbH)} \Omega_\Ga$
in the Kontsevich formula also have a clear operadic
meaning. To explain it we have to articulate a new concept (cf.\ \cite{Ko2}).

\sip

For any operad, $\overline{C}=\{\overline{C}(n)\}$, in the category of smooth manifolds with corners,
the associated $\bS$-module of de Rham algebras,
$\Omega_{\overline{C}}=\{\Omega_{\overline{C}(n)},d_{DR}\}$,
is a dg cooperad (if equipped with a completed tensor product, see \S {\ref{8 Section: De Rham}} for details).
Let ${\fG}^*_{\overline{C}}$
be the dual cooperad of Feynman graphs, and let
$\check{\fG}_{\overline{C}}\subset {\fG}^*_{\overline{C}}$ be its sub-cooperad spanned by {\em finite}\, linear combinations
of graphs. A {\em de Rham field theory}\, on $\overline{C}$ is, by definition, a morphism of dg cooperads,
$$
\Ba{rccc}
\Omega: & (\check{\fG}_{\overline{C}}, 0) & \lon & (\Omega_{\overline{C}},d_{DR})\\
        &  \Ga                            & \lon & \Omega_\Ga
\Ea
$$
where $\check{\fG}_{\overline{C}}$ is equipped with the trivial differential (there exist variants of this definition
in which $\check{\fG}_{\overline{C}}$ has a non-trivial differential but we do not need such variants in this paper). Any such a theory
defines an associated morphism of dg operads,
$$
\Ba{rccc}
\Omega^*: & \cC hains(\overline{C}) & \lon & \fG_{\overline{C}}\vspace{3mm}\\
          &    X                    & \lon &  \sum_{\Ga\in \fG_{\overline{C}}} (\int_X \Omega_\Ga) \Ga.
\Ea
$$
Therefore, any representation,
$$
\rho: \fG_{\overline{C}} \lon \cE nd_{W}
$$
of the (coloured) operad of Feynman graphs in a (collection of) vector space(s) $W$ can be
{\em quantized}\, as follows
$$
\rho^{quant}: \cF \cC hains(\overline{C}) \hookrightarrow \cC hains(\overline{C}) \stackrel{\Omega^*}{\lon}
 \fG_{\overline{C}} \stackrel{\rho}{\lon}  \cE nd_{W}.
$$
When one applies this general construction to Kontsevich's configuration spaces, $\overline{C}=\overline{C}(\bbH)$,
and uses his formulae for $\Omega_\Ga$ in terms of a propagator, then one obtains precisely his formality map
as the quantization of the aforementioned standard representation
$\fG \rar \cE nd_{\{\cT_{poly}(V), \f_V\}}$.
Note that Kontsevich formulae admit a natural extension from the suboperad of {\em fundamental chains},
  $\cF \cC hains(\overline{C})$ to the {\em full}\, operad
of chains in $\overline{C}(\bbH)$; this extension plays no role in our paper but we refer to a
 beautiful work of
 Johan Alm
\cite{Alm} who employed this  observation to construct another less obvious sub-operad
of  $\cC hains(\overline{C}(\bbH))$ and then used this new suboperad to  extend explicitly Duflo-Kontsevich algebra
isomorphism \cite{CaRo, Du,Ko, MT, PT}
$$
H^\bu(\fg, \odot^\bu \fg) \lon H^\bu(\fg, U(\fg)),
$$
at the level of cohomologies to an $\cA_\infty$ quasi-isomorphism between the associated Chevalley-Eilenberg complexes equipped with certain $A_\infty$-structures. Here $\fg$ stands for an
arbitrary finite-dimensional graded Lie algebra, and $U(\fg)$ for its universal enveloping algebra.

\sip

Another useful output of this interpretation of Kontsevich's deformation quantization is that one can apply this
 technique to any operad of compactified configuration spaces and
to any representation of the operad of Feynman graphs, not necessarily to the standard representation in
 $\cE nd_{\{\cT_{poly}(V), \f_V\}}$. We show several new explicit examples below.

\mip

{\bf 1.4. Content of the paper}. Section~{\ref{2: section}} reminds a well-known interpretation of Stasheff's
 associahedra (or, in essence, of the operad of $\cA_\infty$-algebras) as compactified configuration spaces of points
 on the real line $\R$.
In \S {\ref{2'': section}} we give a similar description of  Stasheff's multiplihedra
(or the 2-coloured operad, $\cM or(\cA_\infty)$, of $\cA_\infty$-morphisms of
$\cA_\infty$-algebras). The main novelty here is a new compactification of configuration spaces
of points on $\R$ whose boundary strata involve not only collapsing points but also points going far away from each other
in the standard Euclidean metric on $\R$; this construction is a 1-dimensional version of the 2-dimensional
geometric model  \cite{Me-Auto} for the 2-coloured operad  $\cM or(\caL_\infty)$.
In fact we give in \S {\ref{2'': section}} two inequivalent configuration space models for $\cM or(\cA_\infty)$
and discuss at length their similarities and differences as the same idea will be repeated several times later
in higher dimensions.

\mip

In \S {\ref{3: Section}} we remind Kontsevich's compactification  \cite{Ko} of configuration spaces  of points
on the closed
upper half-plane \cite{Ko} and the associated notion of {\em open-closed homotopy algebra}\, \cite{KS}.
In \S {\ref{4': Section on Mor(L_infty)}} we discuss several configuration space models for the 2-coloured
operad, $\cM or(\caL_\infty)$, of $\caL_\infty$ morphisms; one of them was studied earlier in \cite{Me-Auto}.
In \S {\ref{6: Section}} we construct two configuration space models, $\widehat{\fC}_{\bu,\bu}(\bbH)$,
 for the operad, $\cM or(\f\cC_\infty)$,
of morphisms of open-closed homotopy algebras.

\sip

Operads of Feynman graphs and their representations are studied  in \S {\ref{7: section}}.
De Rham field theories on operads of configuration spaces are introduced in \S {\ref{8 Section: De Rham}}; a de Rham field theory on the Fulton-MacPherson
compactification, $\overline{C}(\R^d)$, of points in $\R^d$ --- one of the simplest in the class ---
is studied there in full details.

\sip

In \S {\ref{9: section}} we consider several concrete quantized  representations of operads of Feynman diagrams
including the one which gives a strange non-flat $\cA_\infty$-algebra structure on $\cT_{poly}(V)$
induced from the standard homogeneous volume form on the circle $S^1$. We also consider a version of the Kontsevich
construction in the 3-dimensional hyperbolic space and use it to give explicit formulae for a 1-parameter
(homotopy trivial) deformation of the standard Gerstenhaber algebra structure in $\cT_{poly}(V)$
which involves an infinite sequence of Bernoulli numbers.

\sip

In Sect.\ {\ref{10: section}} we discuss de Rham field theories on configuration space models
for the 2-coloured operad $\cM or(\caL_\infty)$ and on
 the 4-coloured operad, $\widehat{\fC}_{\bu,\bu}(\bbH)$. This machinery  is expected to produce morphisms of
 open-closed homotopy algebras
out of a propagator, $\om$, on the following 3-dimensional version of the Kontsevich eye,
$$
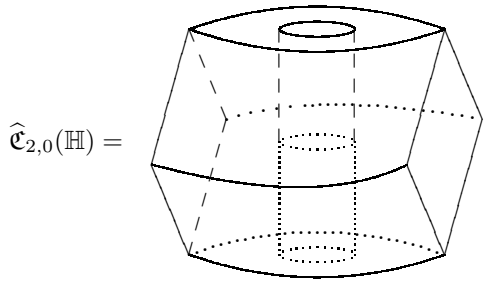
\widehat{\fC}_{2,0}(\bbH)=\Ba{c}
\xy
 <5mm,15mm>*{};<5mm,0mm>*{}**@{.},
<-5mm,15mm>*{};<-5mm,0mm>*{}**@{.},
<5mm,30mm>*{};<5mm,15mm>*{}**@{--},
<-5mm,30mm>*{};<-5mm,15mm>*{}**@{--},
%
<-17mm,0mm>*{};<-12mm,18mm>*{}**@{--},
<-17mm,0mm>*{};<-22mm,12mm>*{}**@{-},
<-17mm,30mm>*{};<-22mm,12mm>*{}**@{-},
<-17mm,30mm>*{};<-12mm,18mm>*{}**@{--},
<17mm,0mm>*{};<12mm,12mm>*{}**@{-},
<17mm,0mm>*{};<22mm,18mm>*{}**@{-},
<17mm,30mm>*{};<22mm,18mm>*{}**@{-},
<17mm,30mm>*{};<12mm,12mm>*{}**@{-},
(0,15)*\ellipse(5,1){-};
(0,7.5)*\ellipse(5,1){.};
(0,0)*\ellipse(5,1){.};
(-17,0)*+{};(17.0,-0)*-{}
**\crv{~*=<4pt>{.}(0,6)}
**\crv{(0,-6)};
(-17,30)*+{};(17.0,30)*-{}
**\crv{(0,36)}
**\crv{(0,24)};
(-22,12)*+{};(12.0,12)*-{}
**\crv{(0,6)};
(-12,18)*-{};(22.0,18)*+{}
**\crv{~*=<4pt>{.}(0,22)}
\endxy
\Ea
$$
and give us {\em explicit}\, formulae for the homotopy action of the
Grothendieck-Teichmueller group on deformation quantizations,
an open problem which we hope to address elsewhere and which was the main motivation for
writing this paper.

\sip

We assume that the reader knows the language of (coloured) operads. However in the Appendix we collected all the
information about this concept which is necessary to read our text.
We paid special attention to the presence and absence of units in operads as several
operads of configuration spaces have no units so that some classical definitions of operads become
inequivalent to each other.

\sip

We tried not to be sketchy and attempted to show  every important detail of all the constructions and
illustrate with examples every non-evident definition. Hence the size of this text.

\mip

{\bf 1.5. Some notation}.
 The set $\{1,2, \ldots, n\}$ is abbreviated to $[n]$;  its group of automorphisms is
denoted by $\bS_n$. The
cardinality of a finite set
$A$ is denoted by $\# A$. If $V=\oplus_{i\in \Z} V^i$ is a graded vector space, then
$V[k]$ stands for the graded vector space with $V[k]^i:=V^{i+k}$ and
and $s^k$ for the associated isomorphism $V\rar V[k]$; for $v\in V^i$ we set $|v|:=i$.
For a pair of graded vector spaces $V_1$ and $V_2$, the symbol $\Hom_i(V_1,V_2)$ stands for the space
of homogeneous linear maps of degree $i$, and $\Hom(V_1,V_2):=\bigoplus_{i\in \Z}\Hom_i(V_1,V_2)$; for example,
$s^k\in \Hom_{-k}(V,V[k])$.

\sip
If $\om_1$ and $\om_2$
are differential forms on manifolds $M_1$ and, respectively, $M_2$, then the form $p_1^*(\om_1)\wedge p_2^*(\om_2)$
on $M_1\times M_2$,
where $p_1: M_1\times M_2\rar M_1$ and  $p_2: M_1\times M_2\rar M_2$ are natural projections,
is often abbreviated to $\om_1\wedge \om_2$.
\sip

\sip

We work throughout in the category of smooth manifolds with corners.
 However, all the main theorems of this paper hold true in the category of semialgebraic
 manifolds introduced in \cite{KoSo} and further developed in \cite{HLTV} so that in applications one can employ not only ordinary smooth differential forms but also  $PA$-forms, where $PA$ stands
for ``piecewise semi-algebraic" as defined in the above mentioned papers. We use this freedom to change the category
of geometric species we work in throughout the text.


{\large
\section{\bf Associahedra as compactified configuration spaces of points on the real line}
\label{2: section}
}

\subsection{Stasheff's associahedra and configuration spaces}\label{2.1: associahedra}
Here we remind a well-known construction \cite{St,Ko} identifying
 the operad of $\cA_\infty$-algebras with the fundamental chain  complex\footnote{All operads $\cC=\{\cC_n\}_{n\geq 1}$
 of compactified
configuration spaces considered in this paper are {\em free}\, as operads in the category of sets; the topological closures of its generators
are called  {\em faces}\, or {\em fundamental chains}\, of $\cC$; moreover,  the subspace of the chain operad of the topological operad $\cC$ generated
 by the fundamental chains is always a dg {\em suboperad}\,  called  the {\em fundamental chain operad}\,
   of $\cC$, or its {\em face complex}.}
of the topological operad, $\overline{C}(\R)=\{\overline{C}_n(\R)\}_{n\geq 2}$,
of compactified configuration spaces
of (equivalence classes of) points on the real line $\R$. Let
$$
\Conf_n(\R):=\{[n]\hook \R\},
$$
be the space of all possible injections of the set $[n]$ into the real line $\R$.
This space is a disjoint union of $n!$ connected components each of which is isomorphic
 to the space $$
 \Conf_n^{o}(\R)=\{x_{1}< x_{2} <\ldots < x_{n}\}.
 $$
  The set $\Conf_n(\R)$ has a natural structure of an oriented $n$-dimensional manifold
 with orientation on $\Conf_n^{0}(\R)$ given by the volume form $dx_1\wedge dx_2\wedge\ldots\wedge dx_n$;
  orientations of all other connected components are then fixed once we assume that the natural smooth action of $\bS_n$  on $\Conf_n(\R)$ is orientation preserving.
 In fact, we can (and often do) label points by an arbitrary finite set $I$, that is, consider the space
 of injections of sets,
$$
\Conf_I(\R):=\{I\hook \R\}.
$$
 A $2$-dimensional Lie group $G_{(2)}=\R^+ \ltimes \R$ acts freely on $\Conf_n(\R)$ by the law,
 $$
\Ba{ccccc}
\Conf_n(\R) & \times & \R^+ \ltimes \R & \lon & \Conf_n(\R)\\
p=\{x_1,\ldots,x_n\}&& (\la,\nu) &\lon & \la p+\nu:= \{\la x_1+\nu, \ldots, \la x_n+\nu\}.
\Ea
$$
The action is free
so that the quotient space,
$$
C_n(\R):= \Conf_n(\R)/G_{(2)},\ \ \  n\geq 2,
$$
is naturally an $(n-2)$-dimensional real oriented manifold equipped with a smooth orientation preserving
action of the group $\bS_n$. In fact,
$$
C_n(\R)=C_n^o(\R)\times \bS_n
$$ with orientation, $\Omega_n$, defined
on $C_n^o(\R):=\Conf^o_n(\R)/G_{(2)}$ as follows: identify $C_n^o(\R)$ with the subspace of
$\Conf^o_n(\R)$ consisting of points $\{0=x_{1}< x_{2} <\ldots < x_{n}=1\}$ and then set
$\Omega_n:= dx_2\wedge\ldots\wedge dx_{n-1}$.

The space $C_2(\R)$ is closed as it is the disjoint union, $C_2(\R)\simeq \bS_2$,  of two points.
The topological compactification,  $\overline{C}_n(\R)$, of  $C_n(\R)$ for higher $n$
 can be defined as $\overline{C}_n^o(\R)\times \bS_n$ where $\overline{C}_n^o(\R)$ is, by definition, the closure of an embedding,
$$
\Ba{ccc}
C_n^o(\R) & \lon & (\R\P^2)^{n(n-1)(n-2)}\\
(x_{i_1}, \ldots, x_{i_n}) & \lon &  \prod_{i_p\neq i_q\neq i_r\neq i_p}\left[|x_{i_p}-x_{i_q}| :
|x_{i_q}-x_{i_r}|: |x_{i_p}-x_{i_r}|\right].
\Ea
$$
Its codimension one strata are given by
$$
\p \overline{C}_n^o(\R) = \bigsqcup_{A} \overline{C}^o_{n - \# A + 1}(\R)\times
 \overline{C}_{\# A}^o(\R),
$$
where the union runs over {\em connected}\, proper  subsets, $A$, of the set $[1,2,\ldots,n]$
 with $\# A\geq 2$. The fundamental chain operad of $\overline{C}(\R)$ is a
dg free operad (in the category of linear spaces)
generated by the $\bS$-module,
$$
\K[\bS_n] =\left\langle \xy
(1,-5)*{\ldots},
(-13,-7)*{_{\sigma(1)}},
(-7,-7)*{_{\sigma(2)}},
(14,-7)*{_{\sigma(n)}},
 (0,0)*{\circ}="a",
(0,5)*{}="0",
(-12,-5)*{}="b_1",
(-8,-5)*{}="b_2",
(-3,-5)*{}="b_3",
(8,-5)*{}="b_4",
(12,-5)*{}="b_5",
\ar @{-} "a";"0" <0pt>
\ar @{-} "a";"b_2" <0pt>
\ar @{-} "a";"b_3" <0pt>
\ar @{-} "a";"b_1" <0pt>
\ar @{-} "a";"b_4" <0pt>
\ar @{-} "a";"b_5" <0pt>
\endxy   \right\rangle_{\sigma\in \bS_n}, \ \ \ n\geq 2,
$$
with the differential given by\footnote{This formula follows immediately from the above formula for
the above formula for $\p \overline{C}_n(\R)$ except for
the sign factor
which compares the  induced orientation on the boundary with the  product
orientation on the right hand side.
We shall prove this sign factor in \S {\ref{2: subsection on induced orientation}}
 below.}
\begin{equation}\label{2: delta_ass(n)}
\p
\xy
(1,-5)*{\ldots},
(-13,-7)*{_{i_1}},
(-8,-7)*{_{i_2}},
(-3,-7)*{_{i_3}},
(8,-7)*{_{i_{n-1}}},
(14,-7)*{_{i_n}},
 (0,0)*{\circ}="a",
(0,5)*{}="0",
(-12,-5)*{}="b_1",
(-8,-5)*{}="b_2",
(-3,-5)*{}="b_3",
(8,-5)*{}="b_4",
(12,-5)*{}="b_5",
\ar @{-} "a";"0" <0pt>
\ar @{-} "a";"b_2" <0pt>
\ar @{-} "a";"b_3" <0pt>
\ar @{-} "a";"b_1" <0pt>
\ar @{-} "a";"b_4" <0pt>
\ar @{-} "a";"b_5" <0pt>
\endxy
=\sum_{k=0}^{n-2}\sum_{l=2}^{n-k}
(-1)^{k+l(n-k-l)+1}
\begin{xy}
<0mm,0mm>*{\circ},
<0mm,0.8mm>*{};<0mm,5mm>*{}**@{-},
<-9mm,-5mm>*{\ldots},
<14mm,-5mm>*{\ldots},
<-0.7mm,-0.3mm>*{};<-13mm,-5mm>*{}**@{-},
<-0.6mm,-0.5mm>*{};<-6mm,-5mm>*{}**@{-},
<0.6mm,-0.3mm>*{};<20mm,-5mm>*{}**@{-},
<0.3mm,-0.5mm>*{};<8mm,-5mm>*{}**@{-},
<0mm,-0.5mm>*{};<0mm,-4.3mm>*{}**@{-},
<0mm,-5mm>*{\circ};
<-5mm,-10mm>*{}**@{-},
<-2.7mm,-10mm>*{}**@{-},
<2.7mm,-10mm>*{}**@{-},
<5mm,-10mm>*{}**@{-},
<4mm,-7mm>*{^{i_1\ \  \dots\ \   i_k\ \ \qquad \ \   i_{k+l+1}\dots  \ \ i_n}},
<2mm,-12mm>*{_{i_{k+1} \ \dots\ \  i_{k+l}}},
\end{xy}
\end{equation}
Therefore, the operad of fundamental chains of $\overline{C}(\R)$
is nothing but the minimal resolution, $\cA ss_\infty$, of the operad of associative algebras.

\subsubsection{\bf Example} $C_3^0(\R)$ is an open interval,
$$
{C}_3^0(\R)=(0,1)\simeq \xy
<-3mm,0mm>*{};<12mm,0mm>*{}**@{-},
<0mm,0mm>*{\bu};
<4mm,0mm>*{\bu};
<10mm,0mm>*{\bu};
<0mm,2mm>*{^{x_1=0}};
<10mm,2mm>*{^{x_3=1}};
<4mm,-3mm>*{^{x_2}};
\endxy
$$
Its compactification $\overline{C}_3^0(\R)$ is, by definition,  the closure of the following embedding,
$$
\Ba{rccc}
i: & C_3^0(\R) & \lon & \R\P^2\\
& (x_1<x_2<x_3) & \lon & \left[|x_1-x_2|: |x_2-x_3|: |x_1-x_3] \right]
\Ea
$$
so that
$$
\overline{C}_3^0(\R)=\overline{i(C_3^0(\R))}=i(C_3^0(\R))\sqcup [0:1:1]\sqcup [1:0:1]=(0,1)\sqcup(0)\sqcup (1)=[0,1]
$$
Therefore,
$$
\p
\begin{xy}
 <0mm,-4mm>*{};<0mm,4mm>*{}**@{-},
 <0.39mm,-0.39mm>*{};<3.2mm,-4mm>*{}**@{-},
 <-0.35mm,-0.35mm>*{};<-3.2mm,-4mm>*{}**@{-},
<0mm,0mm>*{\circ};
   <0.39mm,-0.39mm>*{};<4mm,-6.9mm>*{^3}**@{},
   <-0.35mm,-0.35mm>*{};<-4mm,-6.9mm>*{^1}**@{},
    <-0.35mm,-0.35mm>*{};<0mm,-6.9mm>*{^2}**@{},
\end{xy}=-
\begin{xy}
 <0mm,0mm>*{\circ};<0mm,0mm>*{}**@{},
 <0mm,0.69mm>*{};<0mm,3.0mm>*{}**@{-},
 <0.39mm,-0.39mm>*{};<2.4mm,-2.4mm>*{}**@{-},
 <-0.35mm,-0.35mm>*{};<-1.9mm,-1.9mm>*{}**@{-},
 <-2.4mm,-2.4mm>*{\circ};<-2.4mm,-2.4mm>*{}**@{},
 <-2.0mm,-2.8mm>*{};<0mm,-4.9mm>*{}**@{-},
 <-2.8mm,-2.9mm>*{};<-4.7mm,-4.9mm>*{}**@{-},
    <0.39mm,-0.39mm>*{};<3.3mm,-4.0mm>*{^3}**@{},
    <-2.0mm,-2.8mm>*{};<0.5mm,-6.7mm>*{^2}**@{},
    <-2.8mm,-2.9mm>*{};<-5.2mm,-6.7mm>*{^1}**@{},
 \end{xy}
\ + \
 \begin{xy}
 <0mm,0mm>*{\circ};<0mm,0mm>*{}**@{},
 <0mm,0.69mm>*{};<0mm,3.0mm>*{}**@{-},
 <0.39mm,-0.39mm>*{};<2.4mm,-2.4mm>*{}**@{-},
 <-0.35mm,-0.35mm>*{};<-1.9mm,-1.9mm>*{}**@{-},
 <2.4mm,-2.4mm>*{\circ};<-2.4mm,-2.4mm>*{}**@{},
 <2.0mm,-2.8mm>*{};<0mm,-4.9mm>*{}**@{-},
 <2.8mm,-2.9mm>*{};<4.7mm,-4.9mm>*{}**@{-},
    <0.39mm,-0.39mm>*{};<-3mm,-4.0mm>*{^1}**@{},
    <-2.0mm,-2.8mm>*{};<0mm,-6.7mm>*{^2}**@{},
    <-2.8mm,-2.9mm>*{};<5.2mm,-6.7mm>*{^3}**@{},
 \end{xy}.
$$
where
$$
\begin{xy}
 <0mm,0mm>*{\circ};<0mm,0mm>*{}**@{},
 <0mm,0.69mm>*{};<0mm,3.0mm>*{}**@{-},
 <0.39mm,-0.39mm>*{};<2.4mm,-2.4mm>*{}**@{-},
 <-0.35mm,-0.35mm>*{};<-1.9mm,-1.9mm>*{}**@{-},
 <2.4mm,-2.4mm>*{\circ};<-2.4mm,-2.4mm>*{}**@{},
 <2.0mm,-2.8mm>*{};<0mm,-4.9mm>*{}**@{-},
 <2.8mm,-2.9mm>*{};<4.7mm,-4.9mm>*{}**@{-},
    <0.39mm,-0.39mm>*{};<-3mm,-4.0mm>*{^1}**@{},
    <-2.0mm,-2.8mm>*{};<0mm,-6.7mm>*{^2}**@{},
    <-2.8mm,-2.9mm>*{};<5.2mm,-6.7mm>*{^3}**@{},
 \end{xy}
 \ \simeq \overline{C}_2^0\times \overline{C}_2^0 \ \ \simeq \ \
 \xy
<-3mm,0mm>*{};<12mm,0mm>*{}**@{-},
<0mm,0mm>*{\bu};
<10mm,0mm>*{\bu};
<0mm,2mm>*{^{x_1=0}};
<10mm,2mm>*{^1};
\endxy \
\times \
 \xy
<-3mm,0mm>*{};<12mm,0mm>*{}**@{-},
<0mm,0mm>*{\bu};
<10mm,0mm>*{\bu};
<0mm,2mm>*{^{x_2=0}};
<10mm,2mm>*{^{x_3=1}};
\endxy
$$
$$
\begin{xy}
 <0mm,0mm>*{\circ};<0mm,0mm>*{}**@{},
 <0mm,0.69mm>*{};<0mm,3.0mm>*{}**@{-},
 <0.39mm,-0.39mm>*{};<2.4mm,-2.4mm>*{}**@{-},
 <-0.35mm,-0.35mm>*{};<-1.9mm,-1.9mm>*{}**@{-},
 <-2.4mm,-2.4mm>*{\circ};<-2.4mm,-2.4mm>*{}**@{},
 <-2.0mm,-2.8mm>*{};<0mm,-4.9mm>*{}**@{-},
 <-2.8mm,-2.9mm>*{};<-4.7mm,-4.9mm>*{}**@{-},
    <0.39mm,-0.39mm>*{};<3.3mm,-4.0mm>*{^3}**@{},
    <-2.0mm,-2.8mm>*{};<0.5mm,-6.7mm>*{^2}**@{},
    <-2.8mm,-2.9mm>*{};<-5.2mm,-6.7mm>*{^1}**@{},
 \end{xy}
 \ \simeq \overline{C}_2^0\times \overline{C}_2^0 \ \ \simeq \ \
 \xy
<-3mm,0mm>*{};<12mm,0mm>*{}**@{-},
<0mm,0mm>*{\bu};
<10mm,0mm>*{\bu};
<0mm,2mm>*{^{x_1=0}};
<10mm,2mm>*{^{x_2=1}};
\endxy \
\times \
 \xy
<-3mm,0mm>*{};<12mm,0mm>*{}**@{-},
<0mm,0mm>*{\bu};
<10mm,0mm>*{\bu};
<0mm,2mm>*{^{0}};
<10mm,2mm>*{^{x_3=1}};
\endxy
$$

\subsection{Smooth structure on $\overline{C}_n(\R)$}\label{2: subsubs smooth atlas on associh}
The codimension $l$  boundary strata of $\overline{C}_n(\R)$
is a disjoint union,
$$
\coprod_{T\in \cT_{n,l}}\  \underbrace{\prod_{v\in V(T)} {C}_{\# In(v)}(\R)}_{C_T(\R)}
$$
running over the set, $\cT_{n,l}$, of all possible trees (built from the above corollas) with $l+1$ vertices and $n$
input legs which are labeled  by elements of $[n]$. Here $V(T)$ stands for the set of vertices of a tree $T$ and
  $In(v)$  for the set of input legs of a vertex $v$ (we also use below the symbol $E(T)$ to denote the set of internal edges of $T$). The resulting stratification,
$$
 \overline{C}_n(\R)=\coprod_{l\geq 0}\coprod_{T\in \cT_{n,l}}\ C_T(\R),
$$
can be used to make the compactified configuration space $\overline{C}_n(\R)$ into a smooth manifold with corners.
 For that purpose we need to construct a coordinate chart $\cU_T$ near the boundary stratum
$C_T(\R)\subset \overline{C}_n(\R)$ corresponding to an arbitrary  tree $T$, say to this one
\Beq\label{2: tree T}
T=
\Ba{c}
\xy
(-10.5,-2)*{_1},
(-11,-17)*{_3},
(-2,-17)*{_5},
(3,-10)*{_6},
(8,-10)*{_2},
(14,-10)*{_4},
(21,-10)*{_7},
(0,14)*{}="0",
 (0,8)*{\circ}="a",
(-10,0)*{}="b_1",
(-2,0)*{\circ}="b_2",
(12,0)*{\circ}="b_3",
(2,-8)*{}="c_1",
(-7,-8)*{\circ}="c_2",
(8,-8)*{}="c_3",
(14,-8)*{}="c_4",
(20,-8)*{}="c_5",
(-11,-15)*{}="d_1",
(-3,-15)*{}="d_2",
\ar @{-} "a";"0" <0pt>
\ar @{-} "a";"b_1" <0pt>
\ar @{-} "a";"b_2" <0pt>
\ar @{-} "a";"b_3" <0pt>
\ar @{-} "b_2";"c_1" <0pt>
\ar @{-} "b_2";"c_2" <0pt>
\ar @{-} "b_3";"c_3" <0pt>
\ar @{-} "b_3";"c_4" <0pt>
\ar @{-} "b_3";"c_5" <0pt>
\ar @{-} "c_2";"d_1" <0pt>
\ar @{-} "c_2";"d_2" <0pt>
\endxy
\Ea
\Eeq
and then check that the gluing mappings at the intersections, $\cU_T\cap \cU_{T'}$, of such charts are smooth.
The construction of $\cU_T$ goes in four steps
 (cf.\ Sect.\ 5.2 in \cite{Ko}) which we discuss in some detail as it applies to all configuration spaces we
 study in this paper:
\Bi
\item[(i)] Associate to $T$ a metric graph, $T_{\ccm \cce \cct\ccr \cci \ccc}$, by assigning a small
non-negative parameter $\var$ to each internal edge of $T$, e.g.
\Beq\label{2: metric tree T}
T_{\ccm \cce \cct\ccr \cci \ccc}=
\xy
(2.0,3.0)*{_{\var_{1}}},
(11.2,3.7)*{_{\var_{2}}},
(-6.8,-4)*{_{\var_{3}}},
(-10.5,-2)*{_1},
(-11,-17)*{_3},
(-2,-17)*{_5},
(3,-10)*{_6},
(8,-10)*{_2},
(14,-10)*{_4},
(21,-10)*{_7},
(0,14)*{}="0",
 (0,8)*{\circ}="a",
(-10,0)*{}="b_1",
(-2,0)*{\circ}="b_2",
(12,0)*{\circ}="b_3",
(2,-8)*{}="c_1",
(-7,-8)*{\circ}="c_2",
(8,-8)*{}="c_3",
(14,-8)*{}="c_4",
(20,-8)*{}="c_5",
(-11,-15)*{}="d_1",
(-3,-15)*{}="d_2",
\ar @{-} "a";"0" <0pt>
\ar @{-} "a";"b_1" <0pt>
\ar @{-} "a";"b_2" <0pt>
\ar @{-} "a";"b_3" <0pt>
\ar @{-} "b_2";"c_1" <0pt>
\ar @{-} "b_2";"c_2" <0pt>
\ar @{-} "b_3";"c_3" <0pt>
\ar @{-} "b_3";"c_4" <0pt>
\ar @{-} "b_3";"c_5" <0pt>
\ar @{-} "c_2";"d_1" <0pt>
\ar @{-} "c_2";"d_2" <0pt>
\endxy \ \ \ \ \ \ \var_{1}, \var_{2}, \var_{3}\in [0,\var) \ \mbox{for some}\ 0\leq \var \ll +\infty;
\Eeq

\item[(ii)] choose an $\bS_n$-equivariant section, $\tau: C_n(\R)\rar \Conf_n(\R)$, of the natural projection
$\Conf_n(\R)\rar  C_n(\R)$ as well as an arbitrary smooth structure on its image,
$C_n^{st}(\R):=\tau(C_n(\R))$,
which is often called the space of configurations in the {\em standard position}; for example,
$C_n^{st}(\R)$ can be chosen to be a subspace of $\Conf_n(\R)$ satisfying  either the
conditions
$\sum_{i=1}^n x_i=0$ and $\sum_{i}|x_i|^2=1$, or the conditions that the leftmost point in
the configuration is at $0$ and the rightmost point is at $1$;
\item[(iii)] the required coordinate chart $\cU_T\subset \overline{C}_n(\R)$ is, by definition,
isomorphic to the manifold with corners
    $[0,\var)^{\# E(T)}\times \prod_{v\in V(T)} {C}_{\# In(v)}(\R)$ and the isomorphism is given
    by a map,
    $$
    \al_T: [0,\var)^{\# E(T)}\times \prod_{v\in V(T)} {C}^{st}_{\# In(v)}(\R)\lon \cU_T
    $$
which one reads from the graph $T_{\ccm \cce \cct\ccr \cci \ccc}$ by interpreting it
as a substitution scheme of $\var$-magnified standard configurations; for example, the map $\al_T$ corresponding
to the tree
 (\ref{2: metric tree T})  is given by  a continuous map,
$$
\Ba{ccccccccccc}
 (0,\var)^3 & \hspace{-2mm} \times \hspace{-2mm}& C^{st}_3(\R) & \hspace{-2mm} \times\hspace{-2mm} & C^{st}_2(\R) & \hspace{-2mm} \times \hspace{-2mm} &  C^{st}_3(\R)  &\hspace{-2mm} \times\hspace{-2mm} & C^{st}_2(\R)
& {\lon} & C_7(\R) \\
 (\var_1,\var_2,\var_3) & \hspace{-2mm} \times \hspace{-2mm} & (x_1, x',x'') & \hspace{-2mm}\times
 \hspace{-2mm}& (x''', x_6) &\hspace{-2mm} \times \hspace{-2mm}&
(x_2,x_4,x_7) &\hspace{-2mm} \times\hspace{-2mm} & (x_3,x_7) &\lon& (y_1, y_3, y_5,y_6, y_2,y_4, y_7)
\Ea
$$
$$
\Ba{lllrrrr}
y_1 &=& x_1 & &   y_2 &=&   \var_2 x_2 + x''\\
y_3 &=&   \var_1(\var_3 x_3 + x''')+ x'  &&  y_4 &=&   \var_2 x_4 + x''\\
y_5 &=&   \var_1(\var_3 x_5 + x''')+ x' &&  y_7 &=&   \var_2 x_7 + x''\\
y_6 &=&   \var_1 x_6 + x', &&   &&
\Ea
$$
whose domain is formally extended
to $[0,\var)^3\times C_3^{st}(\R) \times C_2^{st}(\R) \times C_3^{st}(\R) \times C_2^{st}(\R)$; the boundary stratum
$C_T(\R)$ is given in $\cU_T$ by the equations
$\var_1=\var_2=\var_3=0$.
\Ei
It is easy to check that the gluing mappings at every non-empty intersection $\cU_T\cap \cU_{T'}$
are smooth (cf.\ \cite{Ga}).

\sip

We have $\overline{C}_n(\R)= \cA_n\times \bS_n$, where $\cA_n:=\overline{C}_n^o(\R)$ is the $n$-th Stasheff's associahedron \cite{St},
for example
$$
\cA_2=\bu, \ \ \ \ \ \ \cA_3=
\xy
(0,0)*{\bu}="a",
(6,0)*{\bu}="b",
\ar @{-} "a";"b" <0pt>
\endxy,
\ \ \ \ \ \
\cA_4=\Ba{c}
\xy
(-3,0)*{\bu}="a",
(3,0)*{\bu}="b",
(6,6)*{\bu}="c",
(0,12)*{\bu}="d",
(-6,6)*{\bu}="e"
\ar @{-} "a";"b" <0pt>
\ar @{-} "b";"c" <0pt>
\ar @{-} "c";"d" <0pt>
\ar @{-} "d";"e" <0pt>
\ar @{-} "e";"a" <0pt>
\endxy
\Ea
,
\ \ \ \ \ \ \mbox{etc}.
$$

\subsection{Induced orientation on the boundary strata}\label{2: subsection on induced orientation}
Let us prove  the formula for signs in  (\ref{2: delta_ass(n)}), that is, let us compare the
orientation, $\Omega_T$, induced on a generic codimension 1 boundary stratum $C_T(\R) \simeq
\overline{C}_{n-l+1}(\R)\times \overline{C}_l(\R)\subset \overline{C}_n(\R)$
corresponding to a tree,
$$
T=\Ba{c}
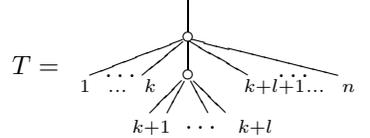
\begin{xy}
<0mm,0mm>*{\circ},
<0mm,0.8mm>*{};<0mm,5mm>*{}**@{-},
<-9mm,-5mm>*{\ldots},
<14mm,-5mm>*{\ldots},
<-0.7mm,-0.3mm>*{};<-13mm,-5mm>*{}**@{-},
<-0.6mm,-0.5mm>*{};<-6mm,-5mm>*{}**@{-},
<0.6mm,-0.3mm>*{};<20mm,-5mm>*{}**@{-},
<0.3mm,-0.5mm>*{};<8mm,-5mm>*{}**@{-},
<0mm,-0.5mm>*{};<0mm,-4.3mm>*{}**@{-},
<0mm,-5mm>*{\circ};
<-5mm,-10mm>*{}**@{-},
<-2.7mm,-10mm>*{}**@{-},
<2.7mm,-10mm>*{}**@{-},
<5mm,-10mm>*{}**@{-},
<4mm,-7mm>*{^{1\ \  \dots\ \   k\ \ \qquad \ \   {k+l+1}\dots  \ \ n}},
<2mm,-12mm>*{_{k+1} \ \dots\ \  _{k+l}},
\end{xy}
\Ea ,
$$
from the standard
orientation, $\Omega$, on $\overline{C}_n(\R)$ with the product, $\Omega_1\times \Omega_2$, of the standard
orientations on
$\overline{C}_{n-l+1}(\R)$ and $\overline{C}_{l}(\R)$.
The upper corolla in $T$ corresponds to the configuration space $C_{n-l+1}(\R)$ with the volume form
in the standard coordinates (in which the left most point is at 0 at the right most point at $1$)
given by
$$
\Omega_1=dx_2\wedge \ldots \wedge dx_k \wedge dx_\bu\wedge dx_{k+l+1}\wedge\ldots \wedge dx_{n-1}.
$$
The lower corolla corresponds to  $C_{l}(\R)$ with the volume form given
in the standard coordinates  by
$$
\Omega_2=d\bar{x}_{k+2}\wedge\ldots \wedge d\bar{x}_{k+l-1}.
$$
The inclusion $C_T(\R)\hook \overline{C}_n(\R)$ is best described in the coordinate chart $\cU_T$
corresponding to the metric tree,
$$
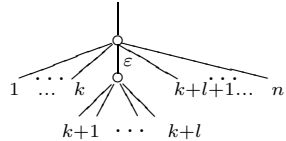
\Ba{c}
\begin{xy}
<1.5mm,-3.5mm>*{^\var},
<0mm,0mm>*{\circ},
<0mm,0.8mm>*{};<0mm,5mm>*{}**@{-},
<-9mm,-5mm>*{\ldots},
<14mm,-5mm>*{\ldots},
<-0.7mm,-0.3mm>*{};<-13mm,-5mm>*{}**@{-},
<-0.6mm,-0.5mm>*{};<-6mm,-5mm>*{}**@{-},
<0.6mm,-0.3mm>*{};<20mm,-5mm>*{}**@{-},
<0.3mm,-0.5mm>*{};<8mm,-5mm>*{}**@{-},
<0mm,-0.5mm>*{};<0mm,-4.3mm>*{}**@{-},
<0mm,-5mm>*{\circ};
<-5mm,-10mm>*{}**@{-},
<-2.7mm,-10mm>*{}**@{-},
<2.7mm,-10mm>*{}**@{-},
<5mm,-10mm>*{}**@{-},
<4mm,-7mm>*{^{1\ \  \dots\ \   k\ \ \qquad \ \   {k+l+1}\dots  \ \ n}},
<2mm,-12mm>*{_{k+1} \ \dots\ \  _{k+l}},
\end{xy}
\Ea
$$
i.e.\ in the coordinates
$$
\left(\var, x_2, \ldots, x_k, x_\bu, x_{k+l+1},\ldots, x_{n-1}, \bar{x}_2,\bar{x}_3 \dots,
\bar{x}_{k+l-1}\right).
$$
These coordinates are related to the standard coordinates $0<y_2<\ldots< y_{n-1}<1$ on
${C}_n(\R)$ as follows,
$$
y_2=x_2,\ \ \ y_3=x_3, \ \ \ldots, \ \ \ y_k= x_k,
$$
$$
y_{k+1}= x_\bu, \ \ y_{k+2}= \var \bar{x}_{k+2} + x_\bu, \ \ \ldots, \ \
y_{k+l-1}= \var \bar{x}_{k+l-1} + x_\bu, \ \  y_{k+l}= \var + x_\bu,
$$
$$
y_{k+l+1}=x_{k+l+1}, \ \ \ldots, \ \ \ y_{n-1}= x_{n-1},
$$
so that the orientation form on $C_n(\R)$ is given in the ``metric tree" coordinates as follows,
\Beqrn
\Omega&=& dy_2\wedge dy_3\wedge\ldots\wedge dy_{n-1}\\
        &=& \var{^{l-2}}dx_1\wedge\ldots\wedge dx_{k}\wedge dx_\bu \wedge d\bar{x}_{k+2}
        \wedge\ldots \wedge d\bar{x}_{k+l-1} \wedge d\var\wedge  dx_{k+l+1}\wedge\ldots \wedge dx_{n-1}\\
    &=&  (-1)^{k+l+1+ (l-2)(n-k-l-1)}  \var^{l-2}d\var \Omega_1\wedge\Omega_2\\
    &=&  (-1)^{k+1+ l(n-k-l)}  \var^{l-2}d\var \Omega_1\wedge\Omega_2.
\Eeqrn
As the boundary $C_T(\R)\hook \overline{C}_n$ is given by the equation $\var=0$ and $\var\geq 0$,
the induced orientation on $C_T(\R)$ is given by the form
$$
 \Omega_T= -(-1)^{k+1+ l(n-k-l)}  \Omega_1\wedge\Omega_2=  (-1)^{k+ l(n-k-l)} \Omega_1\wedge\Omega_2
$$
proving thereby the sign formula in (\ref{2: delta_ass(n)}).

\subsection{An equivalent definition of $\overline{C}_\bu(\R)$}
 Let
$$
\widetilde{\Conf_n}(\R):=\{[n]\rar \R\}
$$
be the space of all possible (not necessarily injective) maps   of the set $[n]$ into the real line $\R$.
For a configuration $p=(x_{i_1},\ldots x_{i_n})$ in $\Conf_n(\R)$ or in $\widetilde{\Conf}_n(\R)$  we set
$$
  x_c(p):=\frac{1}{n}\sum_{k=1}^n x_{i_k}, \ \ \ \  ||p||:=\sqrt{\sum_{k=1}^n |x_{i_k}-x_c(p)|^2}.
$$

Recall that each space $C_n(\R)$ can be  identified with a subspace,
$$
C_n^{st}(\R):= \{ p\in \Conf_n(\R)\, |\   x_c(p)=0, \ \ ||p||=1\}.
$$
Define next a {\em compact}\, space,
$$
\widetilde{C}^{st}_n(\R):= \{p\in \widetilde{\Conf}_n(\R)\ |\ x_c(p)=0,\ ||p||=1\}
$$
which contains $C_n^{st}(\R)$ as a subspace. For any subset $A\subseteq [n]$
there is a natural map
$$
\Ba{rccc}
\pi_A : & C_n(\R) & \lon & C_A(\R)\\
        & p=\{x_i\}_{i\in [n]} & \lon & p_A:=\{x_i\}_{i\in A}
\Ea
$$
which forgets all the points labeled by elements of the complement $[n]\setminus A$.

\sip

The compactification, $\overline{C}_n(\R)$, can be defined  as the closure  of the embedding \cite{AT},
$$
C_n(\R) \stackrel{\prod \pi_A}{\lon} \prod_{A\subseteq [n]\atop \# A\geq 2} C_A(\R)
 \stackrel{\simeq }{\lon} \prod_{A\subseteq [n]\atop \# A\geq 2} C_A^{st}(\R)
\hook  \prod_{A\subseteq [n]\atop \# A\geq 2} \widetilde{C}^{st}_A(\R).
$$
For example, consider the case $n=3$,
$$
\Ba{rccccccc}
i: & C_3(\R) & {\hook} & \widetilde{C}^{st}_{123}(\R) &\times & \widetilde{C}^{st}_{12}(\R) &\times &
\widetilde{C}^{st}_{23}(\R)\\
& p=(x_1,x_2,x_3) &\lon& \frac{p- x_{c}(p)}{||p||} && \left(\frac{x_1-x_2}{2|x_2-x_1|},
\frac{x_2-x_1}{2|x_2-x_1|}\right)=\left(-\frac{1}{2},\frac{1}{2}\right)  &&
\left(\frac{x_2-x_3}{2|x_2-x_1|},
\frac{x_3-x_2}{2|x_2-x_3|}\right)=\left(-\frac{1}{2},\frac{1}{2}\right)
\Ea
$$
It is clear that $\overline{i(C_3(\R))}$ is the union of $C_3(\R)=(0,1)$ with two limiting points,
$0$ corresponding to $x_1-x_2\rar 0$ and $1$ corresponding to  $x_2-x_3\rar 0$, i.e.\ again
$\overline{C}_{3}(\R)=[0,1]$.

\bip

{\large
\section{\bf Multiplihedra as compactified configuration spaces of points on the real line}
\label{2'': section}
}
\mip

 Here we construct two different compactifications
of configuration spaces of distinct points on the real line both of which come equipped with the structure
of a two coloured operad
in the category of smooth manifolds with corners and share one and the same property: the associated face complex
is precisely the 2-coloured operad, $\cM or(\cA_\infty)$,
whose representations are the same as a pair of $\cA_\infty$-algebras together with
an $\cA_\infty$ morphism between them.
 These compactifications involve not only strata of collapsing points but also strata of points going far away from
 each other with respect to
the standard Euclidean metric on $\R$.

\mip

\subsection{The first configuration space model}\label{2'': the first model for multiplihedra}
Let a $1$-dimensional Lie group $G_{(1)}= \R$ act on the configuration space  $\Conf_n(\R)$ by translations,
$$
\Ba{ccc}
\Conf_n(\R) \times \R& \lon & \Conf_n(\R)\\
(p=\{x_1,\ldots,x_n\},\nu) &\lon & p+\nu:= \{x_1+\nu, \ldots, x_n+\nu\}
\Ea
$$

This action is free so that the quotient space,
\Beq\label{2': def of fC(R)}
{\mathfrak C}_n(\R):= \Conf_n(\R)/G_{(1)},
\Eeq
is a naturally oriented $(n-1)$-dimensional real manifold equipped with a smooth action of the group $\bS_n$
(one defines analogously  ${\mathfrak C}_A(\R):= \Conf_A(\R)/G_{(1)}$ for any non-empty set $A$).
The space $\fC_1(\R)$ is a point and hence closed. For $n\geq 2$ we have an $\bS_n$-equivariant isomorphism,
$$
\Ba{rccl}
\Psi_n: &\fC_{n}(\R) & \lon & C^{st}_{n}(\R) \times \R^+\simeq  C_{n}(\R) \times (0,+\infty)   \\
& p=(x_{i_1},\ldots, x_{i_n}) & \lon & (\frac{p-x_c(p)}{||p||}, ||p||)
\Ea
$$
Note that the configuration $\frac{p-x_c(p)}{||p||}$ is $\R^+ \ltimes \R$-invariant and hence represents a uniquely
defined point in $C_n(\R)$.

\sip

For any  subset $A\subseteq [n]$ there is a natural map
$$
\Ba{rccc}
\pi_A : & \fC_n(\R) & \lon & \fC_A(\R)\\
        & p=\{x_i\}_{i\in [n]} & \lon & p_A:=\{x_i\}_{i\in A}
\Ea
$$
which forgets all points labeled by elements of the complement $[n]\setminus A$ (which can be empty). We have  $\fC_n(\R)=  \fC_n^0(\R)\times \bS_n$ where  $\fC_n^0(\R):= \Conf_n^o(\R)/G_{(1)}$.

\sip

A {\em topological compactification}, $\widehat{\fC}_n(\R)$, of $\fC_{n}(\R)$
can be defined as $\widehat{\fC}_n^o(\R)\times \bS_n$ where $\widehat{\fC}_n^o(\R)$ is the closure of a composition (cf.\ \cite{Me-Auto}),
\Beq\label{2: first compactifn of fC(R)}
\fC_{n}^0(\R)\stackrel{\prod \pi_A}{\lon} \prod_{A\subseteq [n]} \fC_{A}(\R)
\stackrel{\prod \Psi_A}{\lon} \prod_{A\subseteq [n]}
 C_{A}^{st}(\R)\times (0, +\infty) \hook \prod_{A\subseteq [n]}
 \widetilde{C}^{st}_{A}(\R)\times [0, +\infty].
\Eeq
here the product runs over {\em connected}\, nonempty subsets, $A$, of the set $[1,2,\ldots,n]$.
Thus all the limiting points in this compactification
come  from configurations when a group or groups of points move too {\em close}\,  to each other
 within each group (as in the case of $\overline{C}_n(\R)$) and/or a group or
groups of points are moving too {\em far}\, (with respect to the relative Euclidean distances inside each group) away
from each other.

\mip

It is not hard to see from the above definition (we refer to
\S {\ref{2: Examples of fC(R) compactif}} below for a detailed discussion
of several explicit examples)
 that
the codimension 1 boundary strata in $\widehat{\fC}_{n}(\R)$ are given by
\Beq\label{Ch3:codimension 1 boundary}
\displaystyle
\p \widehat{\fC}_{n}(\R) = \bigcup_{n=p+q+r\atop p,r\geq 0, q\geq 2} \widehat{\fC}_{n - r + 1}(\R)\times
 \overline{C}_{r}(\R)\
 \bigcup_{n=n_1 +\ldots+ n_k\atop{
 2\leq k\leq n \atop n_1,\ldots, n_k\geq 1}}\overline{ C}_{k}(\R)\times \widehat{\fC}_{n_1}(\R)\times \ldots\times
 \widehat{\fC}_{ n_k}(\R)
\Eeq
where
\Bi
\item the first summation runs over all possible partitions of the form
$$
\{x_{i_1},\ldots, x_{i_n}\}= \{x_{i_1},\ldots, x_{i_p}\}\sqcup\{x_{i_{p+1}},\ldots, x_{i_{p+q}}\}
\sqcup \{x_{i_{p+q+1}},\ldots, x_{i_n}\}
$$
with the corresponding stratum $\fC_{n - r + 1}(\R)\times C_{r}(\R)$ describing
limit configurations in which the points $\{x_{i_{p+1}},\ldots, x_{i_{p+q}}\}$ collapse into a
single point in the real line $\R$, and
\item the second summation runs over all possible partitions of the form
$$
\{x_{i_1},\ldots, x_{i_n}\}= \underbrace{\{x_{i_1},\ldots, x_{i_{n_1}}\}}_{I_1}\sqcup
\underbrace{\{x_{i_{n_1+1}},\ldots, x_{i_{n_1+n_2}}\}}_{I_2}
\sqcup\ldots \sqcup \underbrace{\{x_{i_{n_1+\ldots+n_{k-1}+1}},\ldots, x_{i_n}\}}_{I_k}
$$
with the corresponding stratum  $C_{k}(\R)\times \fC_{n_1}(\R)\times \ldots\times
 \fC_{ n_k}(\R)$ describing
limit configurations in which the distances  between different configurations $I_i$ and $I_j$
tend to $+\infty$  while
the diameter of each such a configuration stands finite.

\Ei

By analogy to \S {\ref{2.1: associahedra}} the collection of spaces
$\{\overline{C}_\bu(\R)\sqcup \widehat{\fC}_n(\R)\sqcup \overline{C}_\bu(\R)\}$ is
 naturally  a dg topological operad, but this time
a {\em two}\, coloured dg operad. Note that the faces of the type $C_k(\R)$ appear in the natural stratification
of $\widehat{\fC}_n(\R)$ in two ways --- as the strata of collapsing points and as the strata controlling groups of points at ``infinity" --- and they never intersect in $\widehat{\fC}_n(\R)$. For that reason we have to assign to these two groups
of faces  different colours and represent collapsing $\overline{C}_\bu(\R)$-strata  by, say, solid white corollas
$$
\overline{C}_q(\R) \simeq
\xy
(1,-5)*{\ldots},
(-13,-7)*{_{i_1}},
(-8,-7)*{_{i_2}},
(-3,-7)*{_{i_3}},
(7,-7)*{_{i_{q-1}}},
(13,-7)*{_{i_q}},
 (0,0)*{\circ}="a",
(0,5)*{}="0",
(-12,-5)*{}="b_1",
(-8,-5)*{}="b_2",
(-3,-5)*{}="b_3",
(8,-5)*{}="b_4",
(12,-5)*{}="b_5",
\ar @{-} "a";"0" <0pt>
\ar @{-} "a";"b_2" <0pt>
\ar @{-} "a";"b_3" <0pt>
\ar @{-} "a";"b_1" <0pt>
\ar @{-} "a";"b_4" <0pt>
\ar @{-} "a";"b_5" <0pt>
\endxy.
$$
and $\overline{C}_\bu(\R)$-stratum at ``infinity" by, say, dashed white corollas
$$
\overline{C}_p(\R) \simeq
\xy
(1,-5)*{\ldots},
(-13,-7)*{_{i_1}},
(-8,-7)*{_{i_2}},
(-3,-7)*{_{i_3}},
(7,-7)*{_{i_{p-1}}},
(13,-7)*{_{i_p}},
 (0,0)*{\circ}="a",
(0,5)*{}="0",
(-12,-5)*{}="b_1",
(-8,-5)*{}="b_2",
(-3,-5)*{}="b_3",
(8,-5)*{}="b_4",
(12,-5)*{}="b_5",
\ar @{.} "a";"0" <0pt>
\ar @{.} "a";"b_2" <0pt>
\ar @{.} "a";"b_3" <0pt>
\ar @{.} "a";"b_1" <0pt>
\ar @{.} "a";"b_4" <0pt>
\ar @{.} "a";"b_5" <0pt>
\endxy,
$$

Finally, representing the faces $\widehat{\fC}_\bu(\R)$ by the black corollas
$$
\widehat{\fC}_{n}(\R)\simeq
\xy
(1,-5)*{\ldots},
(-13,-7)*{_{i_1}},
(-8,-7)*{_{i_2}},
(-3,-7)*{_{i_3}},
(7,-7)*{_{i_{n-1}}},
(13,-7)*{_{i_n}},
 (0,0)*{\bullet}="a",
(0,5)*{}="0",
(-12,-5)*{}="b_1",
(-8,-5)*{}="b_2",
(-3,-5)*{}="b_3",
(8,-5)*{}="b_4",
(12,-5)*{}="b_5",
\ar @{.} "a";"0" <0pt>
\ar @{-} "a";"b_2" <0pt>
\ar @{-} "a";"b_3" <0pt>
\ar @{-} "a";"b_1" <0pt>
\ar @{-} "a";"b_4" <0pt>
\ar @{-} "a";"b_5" <0pt>
\endxy
$$
we can rewrite the boundary differential (\ref{2: delta_ass(n)}) in the associated face complex
in a more informative way
\Beqr\label{2: differential A_infty morphism}
\p
\xy
(1,-5)*{\ldots},
(-13,-7)*{_{i_1}},
(-8,-7)*{_{i_2}},
(-3,-7)*{_{i_3}},
(8,-7)*{_{i_{n-1}}},
(14,-7)*{_{i_n}},
 (0,0)*{\bullet}="a",
(0,5)*{}="0",
(-12,-5)*{}="b_1",
(-8,-5)*{}="b_2",
(-3,-5)*{}="b_3",
(8,-5)*{}="b_4",
(12,-5)*{}="b_5",
\ar @{.} "a";"0" <0pt>
\ar @{-} "a";"b_2" <0pt>
\ar @{-} "a";"b_3" <0pt>
\ar @{-} "a";"b_1" <0pt>
\ar @{-} "a";"b_4" <0pt>
\ar @{-} "a";"b_5" <0pt>
\endxy
&=& -\sum_{l=2}^{n-1}\sum_{k=1}^{n-l} (-1)^{k+l+l(n-k)}
\Ba{c}
\begin{xy}
<0mm,0mm>*{\bullet},
<0mm,0.8mm>*{};<0mm,5mm>*{}**@{.},
<-9mm,-5mm>*{\ldots},
<-9mm,-7mm>*{_{i_1\ \,  \ldots\  \, i_k}},
<13mm,-7mm>*{_{i_{k+l+1}\ \,  \ldots\  \, i_n}},
<0mm,-10mm>*{...},
<14mm,-5mm>*{\ldots},
<-0.7mm,-0.3mm>*{};<-13mm,-5mm>*{}**@{-},
<-0.6mm,-0.5mm>*{};<-6mm,-5mm>*{}**@{-},
<0.6mm,-0.3mm>*{};<20mm,-5mm>*{}**@{-},
<0.3mm,-0.5mm>*{};<8mm,-5mm>*{}**@{-},
<0mm,-0.5mm>*{};<0mm,-4.3mm>*{}**@{-},
<0mm,-5mm>*{\circ};
<-5mm,-10mm>*{}**@{-},
<-2.7mm,-10mm>*{}**@{-},
<2.7mm,-10mm>*{}**@{-},
<5mm,-10mm>*{}**@{-},
<2mm,-12mm>*{_{i_{k+1}}\ \ldots \ _{i_{k+l}}},
\end{xy}
\Ea
\nonumber\\
& +& \sum_{k=2}^n \sum_{[n]=n_1+\ldots+ n_k
\atop n_1\geq 1, \ldots, n_k\geq 1}
(-1)^{\sum_{i=1}^k(k-i)(n_i-1)}\hspace{-6mm}
\Ba{c}
\xy
(-15.5,-7)*{...},
(19,-7)*{...},
(7.5,0)*{\ldots},
(-16,-9)*{_{i_1\ \ \ldots\ \ i_{n_1}}},
(1,-9)*{_{i_{n_1+1}\ldots \ i_{n_1+n_2}}},
(20,-9)*{\ldots\ \ _{i_{n}}},
(-1.8,-7)*{...},
%
%
 (0,7)*{\circ}="a",
(-14,0)*{\bullet}="b_0",
(-4.5,0)*{\bullet}="b_2",
(15,0)*{\bullet}="b_3",
(0,13)*{}="0",
(1,-7)*{}="c_1",
(-8,-7)*{}="c_2",
(-5,-7)*{}="c_3",
(-22,-7)*{}="d_1",
(-19,-7)*{}="d_2",
(-13,-7)*{}="d_3",
(12,-7)*{}="e_1",
(15,-7)*{}="e_2",
(22,-7)*{}="e_3",
\ar @{.} "a";"0" <0pt>
\ar @{.} "a";"b_0" <0pt>
\ar @{.} "a";"b_2" <0pt>
\ar @{.} "a";"b_3" <0pt>
\ar @{-} "b_2";"c_1" <0pt>
\ar @{-} "b_2";"c_2" <0pt>
\ar @{-} "b_2";"c_3" <0pt>
\ar @{-} "b_0";"d_1" <0pt>
\ar @{-} "b_0";"d_2" <0pt>
\ar @{-} "b_0";"d_3" <0pt>
\ar @{-} "b_3";"e_1" <0pt>
\ar @{-} "b_3";"e_2" <0pt>
\ar @{-} "b_3";"e_3" <0pt>
\endxy
\Ea.
\Eeqr
which takes into account signs coming from natural orientations of the faces. The next statement is now
obvious.

\subsubsection{\bf Proposition.}\label{2: Propos on face complex of Mor(A_infty)}
 {\em The face complex of the disjoint union $\overline{C}_\bu(\R)\sqcup \widehat{\fC}_{\bu}(\R)\sqcup
\overline{C}_\bu(\R)$ has naturally a structure of a  dg free non-unital 2-coloured operad of transformation type,
$$
\cM or(A_\infty):= \cF ree
\left\langle
\xy
(1,-5)*{\ldots},
(-13,-7)*{_{i_1}},
(-8,-7)*{_{i_2}},
(-3,-7)*{_{i_3}},
(7,-7)*{_{i_{p-1}}},
(13,-7)*{_{i_p}},
 (0,0)*{\circ}="a",
(0,5)*{}="0",
(-12,-5)*{}="b_1",
(-8,-5)*{}="b_2",
(-3,-5)*{}="b_3",
(8,-5)*{}="b_4",
(12,-5)*{}="b_5",
\ar @{-} "a";"0" <0pt>
\ar @{-} "a";"b_2" <0pt>
\ar @{-} "a";"b_3" <0pt>
\ar @{-} "a";"b_1" <0pt>
\ar @{-} "a";"b_4" <0pt>
\ar @{-} "a";"b_5" <0pt>
\endxy,
\xy
(1,-5)*{\ldots},
(-13,-7)*{_{i_1}},
(-8,-7)*{_{i_2}},
(-3,-7)*{_{i_3}},
(7,-7)*{_{i_{n-1}}},
(13,-7)*{_{i_n}},
 (0,0)*{\bullet}="a",
(0,5)*{}="0",
(-12,-5)*{}="b_1",
(-8,-5)*{}="b_2",
(-3,-5)*{}="b_3",
(8,-5)*{}="b_4",
(12,-5)*{}="b_5",
\ar @{.} "a";"0" <0pt>
\ar @{-} "a";"b_2" <0pt>
\ar @{-} "a";"b_3" <0pt>
\ar @{-} "a";"b_1" <0pt>
\ar @{-} "a";"b_4" <0pt>
\ar @{-} "a";"b_5" <0pt>
\endxy
,
\xy
(1,-5)*{\ldots},
(-13,-7)*{_{i_1}},
(-8,-7)*{_{i_2}},
(-3,-7)*{_{i_3}},
(7,-7)*{_{i_{q-1}}},
(13,-7)*{_{i_q}},
 (0,0)*{\circ}="a",
(0,5)*{}="0",
(-12,-5)*{}="b_1",
(-8,-5)*{}="b_2",
(-3,-5)*{}="b_3",
(8,-5)*{}="b_4",
(12,-5)*{}="b_5",
\ar @{.} "a";"0" <0pt>
\ar @{.} "a";"b_2" <0pt>
\ar @{.} "a";"b_3" <0pt>
\ar @{.} "a";"b_1" <0pt>
\ar @{.} "a";"b_4" <0pt>
\ar @{.} "a";"b_5" <0pt>
\endxy
\right\rangle_{p,q\geq 2, n\geq 1}
$$
whose representations in a pair of vector spaces $V^{\cci\ccn}$ and $V^{\cco\ccu\cct}$ are in 1-1 correspondence with  with the triples,
$(\mu^{\cci\ccn}_\bu, \mu_{\bu}^{\cco\ccu\cct}, f_\bu)$, consisting of an $\cA_\infty$-structure
$\mu^{\cci\ccn}_\bu$ on the space $V^{\cci\ccn}$, an $\cA_\infty$-structure
$\mu^{\cco\ccu\cct}_\bu$ on the space $V^{\cco\ccu\cct}$, and of an $\cA_\infty$-morphisms,
$f_\bu:(V^{\cci\ccn},\mu^{\cci\ccn}_\bu)\rar (V^{\cco\ccu\cct}, \mu^{\cco\ccu\cct}_\bu)$, of $\cA_\infty$-algebras.
}

\subsubsection{\bf Examples}\label{2: Examples of fC(R) compactif}  (i)
 $\fC_1(\R)$ is already compact so that $\widehat{\fC}_1(\R)=\fC_1(\R)$.
\sip

(ii) ${\fC}_2(\R)$ is isomorphic to $(0,+\infty)$. Its compactification,
$\overline{\fC}_2(\R)$  is given as the closure
of the embedding,
$$
\Ba{ccccc}
{C}_2(\R) & \lon & \widetilde{C}^{st}_2(\R) & \times & [0,+\infty]\\
\{x_1< x_2\} & \lon & \left(-\frac{1}{\sqrt{2}},\frac{1}{\sqrt{2}}\right) && \frac{1}{\sqrt{2}}|x_1-x_2|,
\Ea
$$
and hence is isomorphic to the closed interval $[0, +\infty]$. In terms of fundamental chains
we get,
$$
\p
\xy
(-3,-6)*{_{1}},
(3,-6)*{_{2}},
 (0,0)*{\bullet}="a",
(0,4)*{}="0",
(-3,-4)*{}="b_1",
(3,-4)*{}="b_2",
\ar @{.} "a";"0" <0pt>
\ar @{-} "a";"b_2" <0pt>
\ar @{-} "a";"b_1" <0pt>
\endxy
= -
\Ba{c}
\begin{xy}
<0mm,0mm>*{\bullet},
<0mm,0mm>*{};<0mm,5mm>*{}**@{.},
<0mm,-4mm>*{\circ};
<0mm,0mm>*{}**@{-},
<-3mm,-8mm>*{}**@{-},
<3mm,-8mm>*{}**@{-},
<0mm,-10mm>*{_{1}\ \ \ \ _{2}},
\end{xy}
\Ea
 +
 \Ba{c}
 \xy
(-3,-10)*{_{1}},
(3,-10)*{_{2}},
 (0,0)*{\circ}="a",
(0,4)*{}="0",
(-3,-4)*{\bu}="b_1",
(3,-4)*{\bu}="b_2",
(-3,-8)*{}="c_1",
(3,-8)*{}="c_2",
\ar @{.} "a";"0" <0pt>
\ar @{.} "a";"b_2" <0pt>
\ar @{.} "a";"b_1" <0pt>
\ar @{-} "b_1";"c_1" <0pt>
\ar @{-} "b_2";"c_2" <0pt>
\endxy
\Ea
$$
where the first term in the r.h.s.\ represents the limit configuration $|x_1-x_2|\rar 0$ and
the second one the limit configuration $|x_1-x_2|\rar +\infty$.

\sip

(iii) the compactification $\widehat{\fC}_3(\R)$  is defined  as the closure
of an embedding\footnote{It should be clear from the context that the subscript $12$ in, say,
$\widetilde{C}^{st}_{12}(\R)$ refers to a subset of $[3]$ rather than to a natural number.},
$$
\Ba{ccccccc}
{\fC}_3(\R) & \lon & \widetilde{C}^{st}_3(\R)\times [0,+\infty] & \times &
\widetilde{C}^{st}_{12}(\R)\times [0,+\infty]&  \times &
\widetilde{C}^{st}_{23}(\R)\times [0,+\infty]\\
p=\{x_1< x_2< x_3\} & \lon & \left(\frac{p-x_c(p)}{||p||}, ||p||\right)
 && ||p_{12}||=\frac{1}{\sqrt{2}}|x_1-x_2| && ||p_{23}||=\frac{1}{\sqrt{2}}|x_2-x_3|.
\Ea
$$
The codimension 1 boundary strata of  $\widehat{\fC}_3(\R)$ decomposes into strata determined by
various  limit values of the parameters
$||p||$, $||p_{12}||$ and $||p_{23}||$ as follows:
\Bi
\item[(a)] the stratum, $\Ba{c}\begin{xy}
<0mm,0mm>*{\bullet},
<0mm,0mm>*{};<0mm,5mm>*{}**@{.},
<0mm,0mm>*{};<0mm,-3.6mm>*{}**@{-},
<0mm,-8mm>*{};<0mm,-4.4mm>*{}**@{-},
<0mm,-4mm>*{\circ};
<-3mm,-8mm>*{}**@{-},
<3mm,-8mm>*{}**@{-},
<0mm,-9.5mm>*{_{1}\ \ _2 \ _{3}},
\end{xy}
\Ea \simeq \fC_1(\R)\times C_3^{st}(\R)$, is given by $||p||=||p_{12}||=||p_{23}||=0$; it represents the limit
configurations in which all three points collapse into a single point in $\fC_1$ in such a way that the ratio
$\frac{p-x_c(p)}{||p||}$ gives in the limit a well-defined point in $C_3^{st}(\R)$;  any point in this
boundary stratum can be obtained as the $\la\rar 0$ limit of a configuration $(\la x_1, \la x_2, \la x_3)
\in \fC_3(\R)$ for some fixed $(x_1,x_2,x_3)\in C_3^{st}(\R)$;

\item[(b)] the stratum,
$ \Ba{c}
 \xy
(-3,-10)*{_{1}},
(0,-10)*{_{2}},
(3,-10)*{_{3}},
 (0,0)*{\circ}="a",
(0,4)*{}="0",
(-3,-4)*{\bu}="b_1",
(0,-4)*{\bu}="b_2",
(3,-4)*{\bu}="b_3",
(-3,-8)*{}="c_1",
(0,-8)*{}="c_2",
(3,-8)*{}="c_3",
\ar @{.} "a";"0" <0pt>
\ar @{.} "a";"b_2" <0pt>
\ar @{.} "a";"b_1" <0pt>
\ar @{.} "a";"b_3" <0pt>
\ar @{-} "b_1";"c_1" <0pt>
\ar @{-} "b_2";"c_2" <0pt>
\ar @{-} "b_3";"c_3" <0pt>
\endxy
\Ea
 \simeq  C_3^{st}(\R)\times \fC_1(\R)\times \fC_1(\R)\times \fC_1(\R)$, is given by $||p||=||p_{12}||=||p_{23}||=+\infty$; it represents
the limit configurations
in which all three points go infinitely far away from each other
 in such a way that the ratio
$\frac{p-x_c(p)}{||p||}$ gives in the limit a well-defined point in $C_3^{st}(\R)$;
any point in this
boundary stratum can be obtained as the $\la\rar +\infty$ limit of a configuration $(\la x_1, \la x_2, \la x_3)
\in \fC_3(\R)$ for some fixed $(x_1,x_2,x_3)\in C_3^{st}(\R)$;

\item[(c)] the stratum,
$ \Ba{c}
 \xy
(-3,-10)*{_{1}},
(0,-10)*{_{2}},
(6,-10)*{_{3}},
 (0,0)*{\circ}="a",
(0,4)*{}="0",
(-3,-4)*{\bu}="b_1",
(3,-4)*{\bu}="b_3",
(-3,-8)*{}="c_1",
(0,-8)*{}="c_2",
(6,-8)*{}="c_3",
\ar @{.} "a";"0" <0pt>
\ar @{.} "a";"b_3" <0pt>
\ar @{.} "a";"b_1" <0pt>
\ar @{-} "b_1";"c_1" <0pt>
\ar @{-} "b_3";"c_2" <0pt>
\ar @{-} "b_3";"c_3" <0pt>
\endxy
\Ea
 \simeq  C_2^{st}(\R)\times \fC_1(\R)\times \fC_2(\R)$, is given by $||p||=||p_{12}||=+\infty$,
  $0<||p_{23}||<+\infty$; it represents
the limit configurations
in which the point $1$ goes infinitely far away from the points $2$ and $3$
 in such a way that the ratio
$\frac{p-x_c(p)}{||p||}$  is well-defined;
using translation freedom we can fix $x_c=0$ so that
$$
\frac{p}{||p||}=\frac{(x_1=:-\la, \frac{1}{2}(\la+ (x_2-x_3)),  \frac{1}{2}(\la+ (x_3-x_2)))}
{\sqrt{ \la^2 + \frac{1}{4}(\la - (x_3-x_2))^2 +  \frac{1}{4}(\la + (x_3-x_2))^2 }}
\stackrel{\la\rar +\infty}{\lon} (-\frac{\sqrt{2}}{\sqrt{3}},  \frac{1}{\sqrt{6}},
\frac{1}{\sqrt{6}}).
$$
Thus in the limit the images of the points $x_2$ and $x_3$ in $\widetilde{C}_3(\R)$
collapse into a single point.  Any point in this
boundary stratum can be obtained as the $\la\rar +\infty $ limit of a configuration of the form
 $(-\la, \la  +  \Delta , \la  - \Delta)
\in \fC_3(\R)$ for some  $\Delta \in \R$;

\item[(d)] the stratum,
$ \Ba{c}
 \xy
(-6,-10)*{_{1}},
(0,-10)*{_{2}},
(3,-10)*{_{3}},
 (0,0)*{\circ}="a",
(0,4)*{}="0",
(-3,-4)*{\bu}="b_1",
(3,-4)*{\bu}="b_3",
(-6,-8)*{}="c_1",
(0,-8)*{}="c_2",
(3,-8)*{}="c_3",
\ar @{.} "a";"0" <0pt>
\ar @{.} "a";"b_3" <0pt>
\ar @{.} "a";"b_1" <0pt>
\ar @{-} "b_1";"c_1" <0pt>
\ar @{-} "b_1";"c_2" <0pt>
\ar @{-} "b_3";"c_3" <0pt>
\endxy
\Ea
 \simeq  C_2^{st}(\R)\times \fC_2(\R)\times \fC_1(\R)$, is given, by analogy to (c), by the
 following values of the parameters:
 $||p||=||p_{23}||=+\infty$,  $||p_{12}||$ is finite;  any point in this
boundary stratum can be obtained as the $\la\rar +\infty $ limit of a configuration of the form
 $(-\la  -  \Delta , -\la  + \Delta, \la )
\in \fC_3(\R)$ for some  $\Delta \in \R$;

\item[(e)] the stratum,
$ \Ba{c}
 \xy
(-3,-6)*{_{1}},
(0,-10)*{_{2}},
(6,-10)*{_{3}},
 (0,0)*{\bu}="a",
(0,4)*{}="0",
(-3,-4)*{}="b_1",
(3,-4)*{\circ}="b_3",
(-3,-8)*{}="c_1",
(0,-8)*{}="c_2",
(6,-8)*{}="c_3",
\ar @{.} "a";"0" <0pt>
\ar @{-} "a";"b_3" <0pt>
\ar @{-} "a";"b_1" <0pt>
\ar @{-} "b_3";"c_2" <0pt>
\ar @{-} "b_3";"c_3" <0pt>
\endxy
\Ea
 \simeq  C_2^{st}(\R)\times \fC_1(\R)\times \fC_2(\R)$, is given by $0<||p||,||p_{12}||<+\infty$, $||p_{23}||=0$; it represents
the limit configurations
in which all the points are at a finite distance from each other
and the points $2$ and $3$ collapse into a single point in $\fC_2(\R)$;
any point in this
boundary stratum can be obtained as the $\la\rar 0 $ limit of a configuration of the form
 $(-x,  x - \la, x+\la )
\in \fC_3(\R)$ for some $x\in \R$;

\item[(f)] the stratum,
$ \Ba{c}
 \xy
(-6,-10)*{_{1}},
(0,-10)*{_{2}},
(3,-6)*{_{3}},
 (0,0)*{\bu}="a",
(0,4)*{}="0",
(-3,-4)*{\circ}="b_1",
(3,-4)*{}="b_3",
(-6,-8)*{}="c_1",
(0,-8)*{}="c_2",
(0,-8)*{}="c_3",
\ar @{.} "a";"0" <0pt>
\ar @{-} "a";"b_3" <0pt>
\ar @{-} "a";"b_1" <0pt>
\ar @{-} "b_1";"c_1" <0pt>
\ar @{-} "b_1";"c_2" <0pt>
\endxy
\Ea
 \simeq  \fC_2(\R)\times C_2^{st}(\R)$, is given by $0<||p||,||p_{23}||<+\infty$, $||p_{12}||=0$;
any point in this
boundary stratum can be obtained as the $\la\rar 0 $ limit of a configuration of the form
 $(-x - \la, -x+\la,x )
\in \fC_3(\R)$ for some $x\in \R$;
\Ei
Taking into account natural orientations we can  finally summarize the above discussion in  a
single formula
$$
\p
\xy
(-3,-6)*{_{1}},
(0,-6)*{_{2}},
(3,-6)*{_{3}},
 (0,0)*{\bullet}="a",
(0,4)*{}="0",
(-3,-4)*{}="b_1",
(0,-4)*{}="b_2",
(3,-4)*{}="b_3",
\ar @{.} "a";"0" <0pt>
\ar @{-} "a";"b_2" <0pt>
\ar @{-} "a";"b_1" <0pt>
\ar @{-} "a";"b_3" <0pt>
\endxy
= - \Ba{c}\begin{xy}
<0mm,0mm>*{\bullet},
<0mm,0mm>*{};<0mm,5mm>*{}**@{.},
<0mm,0mm>*{};<0mm,-3.6mm>*{}**@{-},
<0mm,-8mm>*{};<0mm,-4.4mm>*{}**@{-},
<0mm,-4mm>*{\circ};
<-3mm,-8mm>*{}**@{-},
<3mm,-8mm>*{}**@{-},
<0mm,-9.5mm>*{_{1}\ \ _2 \ _{3}},
\end{xy}
\Ea
+
\Ba{c}
 \xy
(-6,-10)*{_{1}},
(0,-10)*{_{2}},
(3,-6)*{_{3}},
 (0,0)*{\bu}="a",
(0,4)*{}="0",
(-3,-4)*{\circ}="b_1",
(3,-4)*{}="b_3",
(-6,-8)*{}="c_1",
(0,-8)*{}="c_2",
(0,-8)*{}="c_3",
\ar @{.} "a";"0" <0pt>
\ar @{-} "a";"b_3" <0pt>
\ar @{-} "a";"b_1" <0pt>
\ar @{-} "b_1";"c_1" <0pt>
\ar @{-} "b_1";"c_2" <0pt>
\endxy
\Ea
-
\Ba{c}
 \xy
(-3,-6)*{_{1}},
(0,-10)*{_{2}},
(6,-10)*{_{3}},
 (0,0)*{\bu}="a",
(0,4)*{}="0",
(-3,-4)*{}="b_1",
(3,-4)*{\circ}="b_3",
(-3,-8)*{}="c_1",
(0,-8)*{}="c_2",
(6,-8)*{}="c_3",
\ar @{.} "a";"0" <0pt>
\ar @{-} "a";"b_3" <0pt>
\ar @{-} "a";"b_1" <0pt>
\ar @{-} "b_3";"c_2" <0pt>
\ar @{-} "b_3";"c_3" <0pt>
\endxy
\Ea
+
\Ba{c}
 \xy
(-3,-10)*{_{1}},
(0,-10)*{_{2}},
(3,-10)*{_{3}},
 (0,0)*{\circ}="a",
(0,4)*{}="0",
(-3,-4)*{\bu}="b_1",
(0,-4)*{\bu}="b_2",
(3,-4)*{\bu}="b_3",
(-3,-8)*{}="c_1",
(0,-8)*{}="c_2",
(3,-8)*{}="c_3",
\ar @{.} "a";"0" <0pt>
\ar @{.} "a";"b_2" <0pt>
\ar @{.} "a";"b_1" <0pt>
\ar @{.} "a";"b_3" <0pt>
\ar @{-} "b_1";"c_1" <0pt>
\ar @{-} "b_2";"c_2" <0pt>
\ar @{-} "b_3";"c_3" <0pt>
\endxy
\Ea
-
 \Ba{c}
 \xy
(-6,-10)*{_{1}},
(0,-10)*{_{2}},
(3,-10)*{_{3}},
 (0,0)*{\circ}="a",
(0,4)*{}="0",
(-3,-4)*{\bu}="b_1",
(3,-4)*{\bu}="b_3",
(-6,-8)*{}="c_1",
(0,-8)*{}="c_2",
(3,-8)*{}="c_3",
\ar @{.} "a";"0" <0pt>
\ar @{.} "a";"b_3" <0pt>
\ar @{.} "a";"b_1" <0pt>
\ar @{-} "b_1";"c_1" <0pt>
\ar @{-} "b_1";"c_2" <0pt>
\ar @{-} "b_3";"c_3" <0pt>
\endxy
\Ea
+
\Ba{c}
 \xy
(-3,-10)*{_{1}},
(0,-10)*{_{2}},
(6,-10)*{_{3}},
 (0,0)*{\circ}="a",
(0,4)*{}="0",
(-3,-4)*{\bu}="b_1",
(3,-4)*{\bu}="b_3",
(-3,-8)*{}="c_1",
(0,-8)*{}="c_2",
(6,-8)*{}="c_3",
\ar @{.} "a";"0" <0pt>
\ar @{.} "a";"b_3" <0pt>
\ar @{.} "a";"b_1" <0pt>
\ar @{-} "b_1";"c_1" <0pt>
\ar @{-} "b_3";"c_2" <0pt>
\ar @{-} "b_3";"c_3" <0pt>
\endxy
\Ea
$$
which is in agreement with (\ref{2: differential A_infty morphism}).

\sip

A choice of a total ordering on the set $[n]$ (say, the natural one, $1<2<\ldots< n$),
gives an equivariant smooth isomorphism $\overline{\fC}_n(\R)= \cJ_n\times \bS_n$, where $\cJ_n$ is the $n$-th Stasheff's multiplihedron \cite{St},
for example
$$
\cJ_1\ \mbox{is a point}\ \sim
\xy
(0,0)*{\bu}="a",
(0,4)*{}="b",
(0,-4)*{}="c",
\ar @{.} "a";"b" <0pt>
\ar @{-} "a";"c" <0pt>
\endxy \ \ \ \ \ \ \ , \ \ \ \ \ \ \
\cJ_2=
\xy
(0,5)*{
\xy
 (0,0)*{\bullet}="a",
(0,3)*{}="0",
(-2,-3)*{}="b_1",
(2,-3)*{}="b_2",
\ar @{.} "a";"0" <0pt>
\ar @{-} "a";"b_2" <0pt>
\ar @{-} "a";"b_1" <0pt>
\endxy
},
(-13,7)*{
\xy
<0mm,0mm>*{\bullet},
<0mm,0mm>*{};<0mm,3mm>*{}**@{.},
<0mm,0mm>*{};<0mm,-3mm>*{}**@{-},
<0mm,-3.7mm>*{\circ};
<-2mm,-6mm>*{}**@{-},
<2mm,-6mm>*{}**@{-},
\endxy
},
(13,7)*{\xy
 (0,0)*{\circ}="a",
(0,3)*{}="0",
(-2,-3)*{\bu}="b_1",
(2,-3)*{\bu}="b_2",
(-2,-6.5)*{}="c_1",
(2,-6.5)*{}="c_2",
\ar @{.} "a";"0" <0pt>
\ar @{.} "a";"b_2" <0pt>
\ar @{.} "a";"b_1" <0pt>
\ar @{-} "b_1";"c_1" <0pt>
\ar @{-} "b_2";"c_2" <0pt>
\endxy},
(-10,0.6)*{\centerdot},
(10,0.6)*{\centerdot},
(-10,0)*{}="a",
(10,0)*{}="b",
\ar @{-} "a";"b" <0pt>
\endxy
$$

\Beq\label{2: picture for J_3}
\cJ_3=\ \ \ \ \ \Ba{c}
\xy
(0,-35)*{\xy
 (0,0)*{\bullet}="a",
(0,4)*{}="0",
(-3,-4)*{}="b_1",
(0,-4)*{}="b_2",
(3,-4)*{}="b_3",
\ar @{.} "a";"0" <0pt>
\ar @{-} "a";"b_2" <0pt>
\ar @{-} "a";"b_1" <0pt>
\ar @{-} "a";"b_3" <0pt>
\endxy},
 (0,-8)*{ \xy
 (0,0)*{\circ}="a",
(0,4)*{}="0",
(-3,-4)*{\bu}="b_1",
(0,-4)*{\bu}="b_2",
(3,-4)*{\bu}="b_3",
(-3,-8)*{}="c_1",
(0,-8)*{}="c_2",
(3,-8)*{}="c_3",
\ar @{.} "a";"0" <0pt>
\ar @{.} "a";"b_2" <0pt>
\ar @{.} "a";"b_1" <0pt>
\ar @{.} "a";"b_3" <0pt>
\ar @{-} "b_1";"c_1" <0pt>
\ar @{-} "b_2";"c_2" <0pt>
\ar @{-} "b_3";"c_3" <0pt>
\endxy},
(-27,6)*{
 \xy
 (0,0)*{\circ}="a",
(0,3)*{}="0",
(-3,-3)*{\circ}="b_1",
(3,-3)*{\bu}="b_3",
(-6,-6)*{\bu}="c_1",
(0,-6)*{\bu}="c_2",
(3,-6)*{}="c_3",
(-6,-9)*{}="d_1",
(0,-9)*{}="d_2",
\ar @{.} "a";"0" <0pt>
\ar @{.} "a";"b_3" <0pt>
\ar @{.} "a";"b_1" <0pt>
\ar @{.} "b_1";"c_1" <0pt>
\ar @{.} "b_1";"c_2" <0pt>
\ar @{-} "b_3";"c_3" <0pt>
\ar @{-} "c_1";"d_1" <0pt>
\ar @{-} "c_2";"d_2" <0pt>
\endxy},
%
(27,6)*{
 \xy
 (0,0)*{\circ}="a",
(0,3)*{}="0",
(3,-3)*{\circ}="b_1",
(-3,-3)*{\bu}="b_3",
(6,-6)*{\bu}="c_1",
(0,-6)*{\bu}="c_2",
(-3,-6)*{}="c_3",
(6,-9)*{}="d_1",
(0,-9)*{}="d_2",
\ar @{.} "a";"0" <0pt>
\ar @{.} "a";"b_3" <0pt>
\ar @{.} "a";"b_1" <0pt>
\ar @{.} "b_1";"c_1" <0pt>
\ar @{.} "b_1";"c_2" <0pt>
\ar @{-} "b_3";"c_3" <0pt>
\ar @{-} "c_1";"d_1" <0pt>
\ar @{-} "c_2";"d_2" <0pt>
\endxy},
%
(-26,-18)*{
 \xy
 (0,0)*{\circ}="a",
(0,4)*{}="0",
(-3,-4)*{\bu}="b_1",
(3,-4)*{\bu}="b_3",
(-6,-8)*{}="c_1",
(0,-8)*{}="c_2",
(3,-8)*{}="c_3",
\ar @{.} "a";"0" <0pt>
\ar @{.} "a";"b_3" <0pt>
\ar @{.} "a";"b_1" <0pt>
\ar @{-} "b_1";"c_1" <0pt>
\ar @{-} "b_1";"c_2" <0pt>
\ar @{-} "b_3";"c_3" <0pt>
\endxy},
%
(26,-18)*{\xy
 (0,0)*{\circ}="a",
(0,4)*{}="0",
(-3,-4)*{\bu}="b_1",
(3,-4)*{\bu}="b_3",
(-3,-8)*{}="c_1",
(0,-8)*{}="c_2",
(6,-8)*{}="c_3",
\ar @{.} "a";"0" <0pt>
\ar @{.} "a";"b_3" <0pt>
\ar @{.} "a";"b_1" <0pt>
\ar @{-} "b_1";"c_1" <0pt>
\ar @{-} "b_3";"c_2" <0pt>
\ar @{-} "b_3";"c_3" <0pt>
\endxy},
%
%
(47,-35)*{\xy
 (0,0)*{\circ}="a",
(0,3)*{}="0",
(-3,-3)*{\bu}="b_1",
(3,-7)*{\circ}="e",
(3,-3)*{\bu}="b_3",
(-3,-6)*{}="c_1",
(0,-10)*{}="c_2",
(6,-10)*{}="c_3",
\ar @{.} "a";"0" <0pt>
\ar @{.} "a";"b_3" <0pt>
\ar @{.} "a";"b_1" <0pt>
\ar @{-} "b_1";"c_1" <0pt>
\ar @{-} "e";"c_2" <0pt>
\ar @{-} "b_3";"e" <0pt>
\ar @{-} "e";"c_3" <0pt>
\endxy},
%
(22,-50)*{\xy
 (0,0)*{\bu}="a",
(0,4)*{}="0",
(-3,-4)*{}="b_1",
(3,-4)*{\circ}="b_3",
(-3,-8)*{}="c_1",
(0,-8)*{}="c_2",
(6,-8)*{}="c_3",
\ar @{.} "a";"0" <0pt>
\ar @{-} "a";"b_3" <0pt>
\ar @{-} "a";"b_1" <0pt>
\ar @{-} "b_3";"c_2" <0pt>
\ar @{-} "b_3";"c_3" <0pt>
\endxy},
%
%
(27,-74)*{\xy
 (0,7)*{}="1",
 (0,0)*{\circ}="a",
(0,4)*{\bu}="0",
(-3,-3)*{}="b_1",
(3,-3)*{\circ}="b_3",
(-3,-7)*{}="c_1",
(0,-6)*{}="c_2",
(6,-6)*{}="c_3",
\ar @{.} "1";"0" <0pt>
\ar @{-} "a";"0" <0pt>
\ar @{-} "a";"b_3" <0pt>
\ar @{-} "a";"b_1" <0pt>
\ar @{-} "b_3";"c_2" <0pt>
\ar @{-} "b_3";"c_3" <0pt>
\endxy},
%
(-27,-74)*{\xy
 (0,7)*{}="1",
 (0,0)*{\circ}="a",
(0,4)*{\circ}="0",
(3,-3)*{}="b_1",
(-3,-3)*{\bu}="b_3",
(-3,-7)*{}="c_1",
(0,-6)*{}="c_2",
(-6,-6)*{}="c_3",
\ar @{.} "1";"0" <0pt>
\ar @{-} "a";"0" <0pt>
\ar @{-} "a";"b_3" <0pt>
\ar @{-} "a";"b_1" <0pt>
\ar @{-} "b_3";"c_2" <0pt>
\ar @{-} "b_3";"c_3" <0pt>
\endxy},
(0,-62)*{\begin{xy}
<0mm,0mm>*{\bullet},
<0mm,0mm>*{};<0mm,5mm>*{}**@{.},
<0mm,0mm>*{};<0mm,-3.6mm>*{}**@{-},
<0mm,-8mm>*{};<0mm,-4.4mm>*{}**@{-},
<0mm,-4mm>*{\circ};
<-3mm,-8mm>*{}**@{-},
<3mm,-8mm>*{}**@{-},
\end{xy}},
(-22,-50)*{ \xy
 (0,0)*{\bu}="a",
(0,4)*{}="0",
(-3,-4)*{\circ}="b_1",
(3,-4)*{}="b_3",
(-6,-8)*{}="c_1",
(0,-8)*{}="c_2",
(0,-8)*{}="c_3",
\ar @{.} "a";"0" <0pt>
\ar @{-} "a";"b_3" <0pt>
\ar @{-} "a";"b_1" <0pt>
\ar @{-} "b_1";"c_1" <0pt>
\ar @{-} "b_1";"c_2" <0pt>
\endxy},
%
(-47,-35)*{\xy
 (0,0)*{\circ}="a",
(0,3)*{}="0",
(3,-3)*{\bu}="b_1",
(-3,-7)*{\circ}="e",
(-3,-3)*{\bu}="b_3",
(3,-6)*{}="c_1",
(0,-10)*{}="c_2",
(-6,-10)*{}="c_3",
\ar @{.} "a";"0" <0pt>
\ar @{.} "a";"b_3" <0pt>
\ar @{.} "a";"b_1" <0pt>
\ar @{-} "b_1";"c_1" <0pt>
\ar @{-} "e";"c_2" <0pt>
\ar @{-} "b_3";"e" <0pt>
\ar @{-} "e";"c_3" <0pt>
\endxy},
%
%
(-21,0)*{\bu}="a",
(21,0)*{\bu}="b",
(40,-35)*{\bu}="c",
(21,-70)*{\bu}="d",
(-21,-70)*{\bu}="e",
(-40,-35)*{\bu}="f",
\ar @{-} "a";"b" <0pt>
\ar @{-} "b";"c" <0pt>
\ar @{-} "c";"d" <0pt>
\ar @{-} "d";"e" <0pt>
\ar @{-} "e";"f" <0pt>
\ar @{-} "f";"a" <0pt>
\endxy
\Ea
\Eeq

\subsubsection{\bf A smooth atlas on the compactification
$\widehat{\fC}_\bu(\R)$}\label{2: smooth atlas multipli}
Using metric graphs we can make $\widehat{\fC}_\bu(\R)$ into a smooth manifold with corners
 by constructing a coordinate chart, $\cU_T$, near the boundary stratum corresponding
to an arbitrary tree $T\in \cM or(A_\infty)$ containing at least one black vertex\footnote{If  $T$ does
not contain black vertices, then the associated boundary stratum lies in $\overline{C}_\bu(\R)$ and the construction of $\cU_T$ goes as in
Subsection~\ref{2.1: associahedra}.} as follows:

{\bf (i)} Associate to $T$ a metric graph, $T_{\ccm \cce \cct\ccr \cci \ccc}$, by
assigning
\Bi
\item[(a)]
to every internal edge of $T$ of the form
$
\xy (0,3)*{\bullet}="a", (0,-3)*{\circ}="b"
           \ar @{-} "b";"a" <0pt>\endxy$  a {\em small}\, positive real number $\var \ll +\infty$;

\item[(b)] to every (if any) white vertex of a dashed corolla in $T$ a {\em  large}\, positive real number $\tau \gg 0$,
\Beq\label{2: tau decorated corolla 1}
\xy
(3,0)*{^\tau},
(2,-6)*{\ldots},
 (0,0)*{\circ}="a",
(0,5)*{}="0",
(-12,-6)*{}="b_1",
(-5,-6)*{}="b_2",
(12,-6)*{}="b_3",
\ar @{.} "a";"0" <0pt>
\ar @{.} "a";"b_2" <0pt>
\ar @{.} "a";"b_1" <0pt>
\ar @{.} "a";"b_3" <0pt>
\endxy,
\Eeq

\item[(c)] to every (if any) two vertex subgraph of $T_{\ccm \cce \cct\ccr \cci \ccc}$ of the form
$\xy
(3,4)*{^{\tau_1}};
(3,-4)*{^{\tau_2}};
(0,4)*{\circ}="a",
(0,-4)*{\circ}="b",
\ar @{.} "b";"a" <0pt>
\endxy$
an inequality  $\tau_1\gg \tau_2\gg 0$. This can be understood as a relation $\tau_2=\var_{12} \tau_1$ for some {\em small}\,
parameter $\var_{12}$.
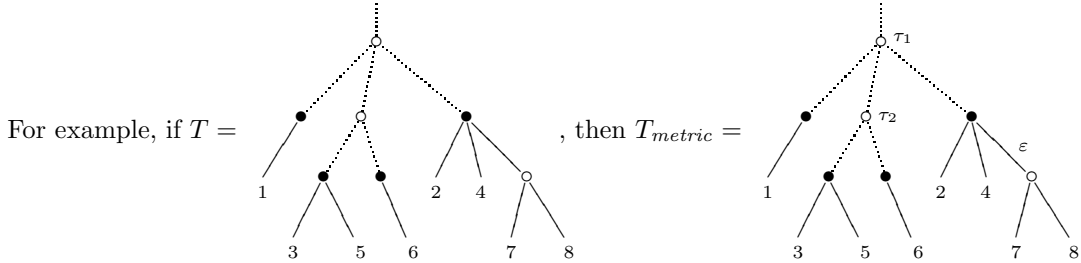
\Ei
For example, if
$
T=\Ba{c}
\xy
(-15,-10)*{_1},
(-11,-18)*{_3},
(-2,-18)*{_5},
(5,-18)*{_6},
(8,-10)*{_2},
(14,-10)*{_4},
(18,-18)*{_7},
(25.7,-18)*{_8},
(0,15)*{}="0",
 (0,10)*{\circ}="a",
(-10,0)*{\bullet}="b_1",
(-2,0)*{\circ}="b_2",
(12,0)*{\bullet}="b_3",
(-15,-8)*{}="c_0",
(0.6,-8)*{\bullet}="c_1",
(-7,-8)*{\bullet}="c_2",
(8,-8)*{}="c_3",
(14,-8)*{}="c_4",
(20,-8)*{\circ}="c_5",
(-11,-16)*{}="d_1",
(-3,-16)*{}="d_2",
(4,-16)*{}="d_3",
(18,-16)*{}="d_4",
(25,-16)*{}="d_5",
\ar @{.} "a";"0" <0pt>
\ar @{.} "a";"b_1" <0pt>
\ar @{.} "a";"b_2" <0pt>
\ar @{.} "a";"b_3" <0pt>
\ar @{-} "b_1";"c_0" <0pt>
\ar @{.} "b_2";"c_1" <0pt>
\ar @{.} "b_2";"c_2" <0pt>
\ar @{-} "b_3";"c_3" <0pt>
\ar @{-} "b_3";"c_4" <0pt>
\ar @{-} "b_3";"c_5" <0pt>
\ar @{-} "c_2";"d_1" <0pt>
\ar @{-} "c_2";"d_2" <0pt>
\ar @{-} "c_1";"d_3" <0pt>
\ar @{-} "c_5";"d_4" <0pt>
\ar @{-} "c_5";"d_5" <0pt>
\endxy
\Ea
$\hspace{-4mm},
then
$
T_{\ccm \cce \cct\ccr \cci \ccc}=
\Ba{c}
\xy
(19,-4)*{_{\var}},
(3,10)*{^{\tau_1}},
(1,0)*{_{\tau_2}},
(-15,-10)*{_1},
(-11,-18)*{_3},
(-2,-18)*{_5},
(5,-18)*{_6},
(8,-10)*{_2},
(14,-10)*{_4},
(18,-18)*{_7},
(25.7,-18)*{_8},
(0,15)*{}="0",
 (0,10)*{\circ}="a",
(-10,0)*{\bullet}="b_1",
(-2,0)*{\circ}="b_2",
(12,0)*{\bullet}="b_3",
(-15,-8)*{}="c_0",
(0.6,-8)*{\bullet}="c_1",
(-7,-8)*{\bullet}="c_2",
(8,-8)*{}="c_3",
(14,-8)*{}="c_4",
(20,-8)*{\circ}="c_5",
(-11,-16)*{}="d_1",
(-3,-16)*{}="d_2",
(4,-16)*{}="d_3",
(18,-16)*{}="d_4",
(25,-16)*{}="d_5",
\ar @{.} "a";"0" <0pt>
\ar @{.} "a";"b_1" <0pt>
\ar @{.} "a";"b_2" <0pt>
\ar @{.} "a";"b_3" <0pt>
\ar @{-} "b_1";"c_0" <0pt>
\ar @{.} "b_2";"c_1" <0pt>
\ar @{.} "b_2";"c_2" <0pt>
\ar @{-} "b_3";"c_3" <0pt>
\ar @{-} "b_3";"c_4" <0pt>
\ar @{-} "b_3";"c_5" <0pt>
\ar @{-} "c_2";"d_1" <0pt>
\ar @{-} "c_2";"d_2" <0pt>
\ar @{-} "c_1";"d_3" <0pt>
\ar @{-} "c_5";"d_4" <0pt>
\ar @{-} "c_5";"d_5" <0pt>
\endxy
\Ea
$

\sip

with $\tau_1,\tau_2=\var_{12}\tau_1 \in (l,+\infty)$ and $\var_{12},\var\in (0,s)$ for some real numbers $l\gg 0$ and  $s\ll +\infty$.
\mip

{\bf (ii)} Choose an equivariant section, $s: \fC_n(\R)\rar \Conf_n(\R)$, of the natural projection
$\Conf_n(\R)\rar  \fC_n(\R)$ as well as an arbitrary smooth structure on its image, $\fC_n^{st}(\R):=s(\fC_n(\R))$,
which is  called the space of configurations in the {\em standard position}; for example,
$\fC_n^{st}(\R)$ can be chosen to be the subspace of $\Conf_n(\R)$ satisfying the condition
$\sum_{i=1}^n x_i=0$; in particular, $\fC_1^{st}(\R)=0\in \R$.

\sip

{\bf (iii)} the required coordinate chart $\cU_T\subset \overline{\fC}_n(\R)$ is, by definition, isomorphic to the manifold with corners,
    $$(l,+\infty]^{\# V_\circ^{d}(T)}\times  [0,s)^{\# E_{\circ}^\bu(T)}\times
    \prod_{v\in V_\circ (T)} {C}^{st}_{\# In(v)}(\R)\times  \prod_{v\in V_\bu (T)} {\fC}^{st}_{\# In(v)}(\R)
    $$
where $V_\circ(T)$ is the set of white vertices of $T$, $V^d_\circ(T)\subset V_\circ(T)$ a subset
corresponding to dashed corollas, $V_\bu(T)$ the set of black vertices, and  $E_{\circ}^\bu(T)$ the set of internal edges of the type $
\xy (0,3)*{\bullet}="a", (0,-3)*{\circ}="b"
           \ar @{-} "b";"a" <0pt>\endxy$.
The isomorphism $\al_T$ between $\cU_T$ and the latter product
of manifolds with corners can be read from the metric graph via a simple procedure which we explain
on the particular example. For the tree $T$ shown above the map $\al_T$ is defined by
a formal extension
of the domain of the following continuous  map\footnote{We ordered factors $C_{\# In(v)}(\R)$ and $\fC_{\# In(v)}(\R)$
in the formula below in accordance with
a  natural ``from top to the bottom" and ``from left to the right" ordering of the vertices $v$ of the planar
tree $T$. 
 .}

$$
\Ba{ccccccccccccccccc}
 \hspace{-2mm}(l,+\infty)^2&\hspace{-2mm}\times\hspace{-3mm} & (0,s) & \hspace{-3mm} \times \hspace{-3mm}&
 C_3^{st}(\R) &\hspace{-3mm} \times\hspace{-3mm} &  \fC_1^{st}(\R) & \hspace{-3mm} \times\hspace{-3mm} & C_2^{st}(\R)
 &\hspace{-3mm} \times\hspace{-3mm}&
 \fC_3^{st}(\R)  & \hspace{-3mm} \times\hspace{-3mm} & \fC_2^{st}(\R) & \hspace{-3mm} \times \hspace{-3mm} & \fC_1^{st}(\R)
 & \hspace{-2mm} \times \hspace{-2mm} & C_2^{st}(\R)
\\
 \hspace{-2mm}(\tau_1,\tau_2)&\hspace{-3mm}\times \hspace{-3mm} & \var & \hspace{-3mm} \times \hspace{-3mm} & (x', x'',x''') & \hspace{-3mm}\times
 \hspace{-3mm}& (x_1\hspace{-1mm}=\hspace{-1mm}0) & \hspace{-3mm} \times\hspace{-3mm} &  (t' \hspace{-1mm}= \hspace{-1mm}\frac{-1}{\sqrt{2}},t''\hspace{-1mm}=\hspace{-1mm}\frac{1}{\sqrt{2}})
  & \hspace{-3mm} \times \hspace{-3mm}& (x_2,x_4,u) & \hspace{-3mm}\times
 \hspace{-3mm}& (x_3, x_5)& \hspace{-3mm} \times \hspace{-3mm}& (x_6\hspace{-1mm}=\hspace{-1mm} 0)  &\hspace{-3mm} \times \hspace{-3mm}&
(x_7\hspace{-1mm}=\hspace{-1mm}\frac{-1}{\sqrt{2}},x_8\hspace{-1mm}=\hspace{-1mm}\frac{1}{\sqrt{2}}) \\
\Ea
$$
$$
\hspace{120mm}
\lon
\Ba{c} \fC_8(\R)\\
(y_1,\ldots, y_8)
\Ea
$$

with
$$
\Ba{lllrrrl}
y_1 &=& \tau_1 x' + x_1 & &   y_2 &=&   \tau_1  x''' + x_2  \\
y_3 &=&  \tau_1 x'' + \tau_2 t' + x_3   &&  y_4 &=& \tau_1 x''' +  x_4\\
y_5 &=&  \tau_1 x'' + \tau_2 t' + x_5  &&  y_7 &=&\tau_1 x''' + u + \var x_7\\
y_6 &=&   \tau_1 x'' + \tau_2 t'' + x_6 , &&    y_8 &=&   \tau_1 x'''  +  u + \var x_8
\Ea
$$
The boundary strata in $\cU_T$ are given by allowing formally $\tau_1=+\infty$, $\tau_2=+\infty$ with $\tau_1/\tau_2=0$
and $\var=0$. Therefore, the main novelty comparing to the case of associahedra discussed in Sect.\ {\ref{2.1: associahedra}}
comes from the dashed corollas decorated by a large parameter $\tau$.


\subsection{Another configuration space model for $\cM or(\cA_\infty)$}
\label{2:  A different smooth structure on fC(R)}
For a pair of subspaces $B\subsetneq A\subseteq [n]$ we consider
$$
\Ba{rcccccc}
\pi_{A,B}: & \fC_n(\R) & \lon &  C_B^{st}(\R)&\times & (0,+\infty)\\
& p & \lon &  \frac{p_B-x_c(p_B)}{||p_B||}
&& ||p_{A,B}||:= ||p_A||\cdot ||p_B||
\Ea
$$
and then define a compactification\footnote{We use the same symbol $\widehat{\fC}_\bu(\R)$ to denote
this new compactification as it has the same combinatoric of the boundary strata as the model in
the previous subsection; we define essentially here a different manifold structure on the same set theoretic operad
 $\widehat{\fC}_\bu(\R)$.}
$\widehat{\fC}_\bu(\R)$ as the closure of
the following composition of embeddings,
\Beq\label{2: second compactifn of fC(R)}
\fC_{n}(\R)\stackrel{\Psi_n\prod \pi_{A,B}}{\lon} C_n^{st}(\R)\times (0,+\infty)
\hspace{-3mm}  \prod_{B\subsetneq A\subseteq [n]\atop \# B\geq 2} \hspace{-4mm} C_B^{st}(\R) \times (0,+\infty)
 \hook \widetilde{C}_n^{st}(\R)\times [0,+\infty]\hspace{-3mm}\prod_{B\subsetneq A\subseteq [n]\atop \# B\geq 2}\hspace{-4mm}  \widetilde{C}_B^{st}(\R) \times [0,+\infty].
\Eeq
This new compactification affects only configurations tending to infinity: if, for a pair  $B\subsetneq A\subseteq [n]$,
the parameter $||p_A||$ tends to $+\infty$, then, for the parameter $||p_{A,B}||$ used in the
above compactification
to take a finite value $a\in \R$, one must have $||p_B||=a/||p_A||\rar 0$. Put another way,
if in the previous
compactification (\ref{2: first compactifn of fC(R)}) the boundary stratum at ``infinity" was given by,
say, $k$ groups of points, $\{I_i\}_{i\in [k]}$, going far away from each other in such a way that each group $I_i$
has   a {\em finite}\, size $a_i\in \R$,
now that size should decrease as $a_i/||p_A||$ as points go to infinity, i.e.\ the points in each group $I_i$ gradually collapse to a single point
in the limit.

\sip

As an illustration of this ``renormalized" embedding formula, let us consider in detail the case $n=3$:
$$
\Ba{ccccccc}
{\fC}_3(\R) & \lon & \widetilde{C}^{st}_3(\R)\times [0,+\infty] & \times &
\widetilde{C}^{st}_{12}(\R)\times [0,+\infty]&  \times &
\widetilde{C}^{st}_{23}(\R)\times [0,+\infty]\\
p=\{x_1< x_2< x_3\} & \lon & \left(\frac{p-x_c(p)}{||p||}, ||p||\right)
 && ||p_{[3],12}||:= ||p||\cdot ||p_{12}|| && ||p_{[3],23}||:= ||p||\cdot ||p_{23}||.
\Ea
$$
The codimension 1 boundary strata of  $\widehat{\fC}_3(\R)$ decomposes into strata determined by
various  limit values of the parameters
$||p||$, $||p_{[3],12}||$ and $||p_{[3],12}||$ in a close analogy to the case
(\ref{2: first compactifn of fC(R)}):
\Bi
\item[(a)] the stratum, $\Ba{c}\begin{xy}
<0mm,0mm>*{\bullet},
<0mm,0mm>*{};<0mm,5mm>*{}**@{.},
<0mm,0mm>*{};<0mm,-3.6mm>*{}**@{-},
<0mm,-8mm>*{};<0mm,-4.4mm>*{}**@{-},
<0mm,-4mm>*{\circ};
<-3mm,-8mm>*{}**@{-},
<3mm,-8mm>*{}**@{-},
<0mm,-9.5mm>*{_{1}\ \ _2 \ _{3}},
\end{xy}
\Ea \simeq \fC_1(\R)\times C_3^{st}(\R)$, is given by $||p||=||p_{[3],12}||=||p_{[3],23}||=0$; it represents the limit
configurations in which all three points collapse into a single point in $\fC_1$ in such a way that the ratio
$\frac{p-x_c(p)}{||p||}$ gives in the limit a well-defined point in $C_3^{st}(\R)$;  any point in this
boundary stratum can be obtained as the $\la\rar 0$ limit of a configuration $(\la x_1, \la x_2, \la x_3)
\in \fC_3(\R)$  for some fixed $(x_1,x_2,x_3)\in C_3^{st}(\R)$;

\item[(b)] the stratum,
$ \Ba{c}
 \xy
(-3,-10)*{_{1}},
(0,-10)*{_{2}},
(3,-10)*{_{3}},
 (0,0)*{\circ}="a",
(0,4)*{}="0",
(-3,-4)*{\bu}="b_1",
(0,-4)*{\bu}="b_2",
(3,-4)*{\bu}="b_3",
(-3,-8)*{}="c_1",
(0,-8)*{}="c_2",
(3,-8)*{}="c_3",
\ar @{.} "a";"0" <0pt>
\ar @{.} "a";"b_2" <0pt>
\ar @{.} "a";"b_1" <0pt>
\ar @{.} "a";"b_3" <0pt>
\ar @{-} "b_1";"c_1" <0pt>
\ar @{-} "b_2";"c_2" <0pt>
\ar @{-} "b_3";"c_3" <0pt>
\endxy
\Ea
 \simeq  C_3^{st}(\R)\times \fC_1(\R)\times \fC_1(\R)\times \fC_1(\R)$, is
 given by $||p||=||p_{[3],12}||=||p_{[3],23}||=+\infty$; it represents
the limit configurations
in which all three points go infinitely far away from each other
 in such a way that the ratio
$\frac{p-x_c(p)}{||p||}$ is well-defined giving in the limit a point in $C_3^{st}(\R)$;
any point in this
boundary stratum can be obtained as the $\la\rar +\infty$ limit of a configuration $(\la x_1, \la x_2, \la x_3)
\in \fC_3(\R)$ for some fixed $(x_1,x_2,x_3)\in C_3^{st}(\R)$;

\item[(c)] the stratum,
$ \Ba{c}
 \xy
(-3,-10)*{_{1}},
(0,-10)*{_{2}},
(6,-10)*{_{3}},
 (0,0)*{\circ}="a",
(0,4)*{}="0",
(-3,-4)*{\bu}="b_1",
(3,-4)*{\bu}="b_3",
(-3,-8)*{}="c_1",
(0,-8)*{}="c_2",
(6,-8)*{}="c_3",
\ar @{.} "a";"0" <0pt>
\ar @{.} "a";"b_3" <0pt>
\ar @{.} "a";"b_1" <0pt>
\ar @{-} "b_1";"c_1" <0pt>
\ar @{-} "b_3";"c_2" <0pt>
\ar @{-} "b_3";"c_3" <0pt>
\endxy
\Ea
 \simeq  C_2^{st}(\R)\times \fC_1(\R)\times \fC_2(\R)$, is given by $||p||=||p_{[3],12}||=+\infty$,
  $0<||p_{[3],23}||<+\infty$; it  represents
the limit  configurations
in which the point $1$ goes infinitely far away from the points $2$ and $3$
 in such a way that the ratio
$\frac{p-x_c(p)}{||p||}$ and the product $||p||\cdot ||p_{23}||$ are well-defined
elements of $\widetilde{C}_3^{st}(\R)$ and $(0,+\infty)$ respectively;
any point in this
boundary stratum can be obtained as the $\la\rar +\infty$ limit of a configuration $(-\la, \la - \la^{-1}\Delta,
\la + \la^{-1}\Delta)\in
\fC_3(\R)$ for some $\Delta\in \R$ as
$$
\frac{p-x_c(p)}{||p||}=\frac{(-\frac{4}{3}\la, \frac{2}{3}\la - \la^{-1}\Delta,
\frac{2}{3}\la + \la^{-1}\Delta)}
{\sqrt{ \frac{16}{9}\la^2 + \left(\frac{2}{3}\la - \la^{-1}\Delta\right)^2 + \left(\frac{2}{3}\la + \la^{-1}\Delta\right)^2 }}
\stackrel{\la\rar +\infty}{\lon}  (-\frac{\sqrt{2}}{\sqrt{3}},  \frac{1}{\sqrt{6}},
\frac{1}{\sqrt{6}})
$$
\Beqrn
||p_{[3],23}||=||p||\cdot||p_{[3],23}|| & = &
\sqrt{ \frac{16}{9}\la^2 + (\frac{2}{3}\la - \la^{-1}\Delta)^2 + (\frac{2}{3}\la + \la^{-1}\Delta)^2 }
 \cdot \sqrt{2}{\la^{-1}}\Delta\\
 &\stackrel{\la\rar +\infty}{\lon}&  \frac{4}{\sqrt{3}}\Delta\in (0,+\infty).
\Eeqrn
Thus this stratum consists of limit configurations in which the point $x_1$ goes far
away from the points $x_2$ and $x_3$, and simultaneously, the points $x_2$ and $x_3$
approach each other with the speed given by $||p||^{-1}$ so that the product
$||p||\cdot ||p_{23}||$ is a well-defined {\em finite}\, number.

\item[(d)] the stratum,
$ \Ba{c}
 \xy
(-6,-10)*{_{1}},
(0,-10)*{_{2}},
(3,-10)*{_{3}},
 (0,0)*{\circ}="a",
(0,4)*{}="0",
(-3,-4)*{\bu}="b_1",
(3,-4)*{\bu}="b_3",
(-6,-8)*{}="c_1",
(0,-8)*{}="c_2",
(3,-8)*{}="c_3",
\ar @{.} "a";"0" <0pt>
\ar @{.} "a";"b_3" <0pt>
\ar @{.} "a";"b_1" <0pt>
\ar @{-} "b_1";"c_1" <0pt>
\ar @{-} "b_1";"c_2" <0pt>
\ar @{-} "b_3";"c_3" <0pt>
\endxy
\Ea
 \simeq  C_2^{st}(\R)\times \fC_2(\R)\times \fC_1(\R)$, is given by the following values
 of the parameters: $||p||=||p_{[3],23}||=+\infty$,  $||p_{[3],12}||$ is finite; any point in this
boundary stratum can be obtained as the $\la\rar +\infty$ limit of a configuration $(-\la - \la^{-1}\Delta,
-\la + \la^{-1}\Delta, \la)\in
\fC_3(\R)$ for some $\Delta\in \R$,

 \item[(e)] the stratum,
$ \Ba{c}
 \xy
(-3,-6)*{_{1}},
(0,-10)*{_{2}},
(6,-10)*{_{3}},
 (0,0)*{\bu}="a",
(0,4)*{}="0",
(-3,-4)*{}="b_1",
(3,-4)*{\circ}="b_3",
(-3,-8)*{}="c_1",
(0,-8)*{}="c_2",
(6,-8)*{}="c_3",
\ar @{.} "a";"0" <0pt>
\ar @{-} "a";"b_3" <0pt>
\ar @{-} "a";"b_1" <0pt>
\ar @{-} "b_3";"c_2" <0pt>
\ar @{-} "b_3";"c_3" <0pt>
\endxy
\Ea
 \simeq  C_2^{st}(\R)\times \fC_1(\R)\times \fC_2(\R)$, is given by $0<||p||,||p_{[3],12}||<+\infty$,
  $||p_{[3],23}||=0$; it represents
the limit configurations
in which all the points are at finite distance from each other
and the points $2$ and $3$ collapse into a single point in $\fC_2(\R)$;  any point in this
boundary stratum can be obtained as the $\la\rar 0$ limit of a configuration $(-1, 1 - \la \Delta,
1 + \la\Delta, 1)\in
\fC_3(\R)$ for some $\Delta\in \R$

\item[(f)] the stratum,
$ \Ba{c}
 \xy
(-6,-10)*{_{1}},
(0,-10)*{_{2}},
(3,-6)*{_{3}},
 (0,0)*{\bu}="a",
(0,4)*{}="0",
(-3,-4)*{\circ}="b_1",
(3,-4)*{}="b_3",
(-6,-8)*{}="c_1",
(0,-8)*{}="c_2",
(0,-8)*{}="c_3",
\ar @{.} "a";"0" <0pt>
\ar @{-} "a";"b_3" <0pt>
\ar @{-} "a";"b_1" <0pt>
\ar @{-} "b_1";"c_1" <0pt>
\ar @{-} "b_1";"c_2" <0pt>
\endxy
\Ea
 \simeq  \fC_2(\R)\times C_2^{st}(\R)$, is given by $0<||p||,||p_{23}||<+\infty$, $||p_{12}||=0$;
any point in this
boundary stratum can be obtained as the $\la\rar 0$ limit of a configuration $(-1 - \la \Delta,
-1 + \la\Delta, 1)\in
\fC_3(\R)$ for some $\Delta\in \R$.
\Ei

It is important to notice the following difference between
the compactification formulae (\ref{2: first compactifn of fC(R)}) and
(\ref{2: second compactifn of fC(R)}):
\Bi
\item[-] consider limit configurations in which the
points $x_2$ and $x_3$ go far away from the point $x_1$ while keeping a {\em finite}\,
distance, $||p_{23}||$, between themselves; in the first compactification formula such  limit configurations
fill  in the stratum $ \Ba{c}
 \xy
 (0,0)*{\circ}="a",
(0,4)*{}="0",
(-3,-4)*{\bu}="b_1",
(3,-4)*{\bu}="b_3",
(-3,-8)*{}="c_1",
(0,-8)*{}="c_2",
(6,-8)*{}="c_3",
\ar @{.} "a";"0" <0pt>
\ar @{.} "a";"b_3" <0pt>
\ar @{.} "a";"b_1" <0pt>
\ar @{-} "b_1";"c_1" <0pt>
\ar @{-} "b_3";"c_2" <0pt>
\ar @{-} "b_3";"c_3" <0pt>
\endxy
\Ea
$ as $||p||, ||p_{12}||\rar +\infty$ while $||p_{23}||$ stays finite;
in the second compactification formula all such configurations tend to one and the same
point in the boundary represented by the graph  $\xy
 (0,0)*{\circ}="a",
(0,3)*{}="0",
(3,-3)*{\circ}="b_1",
(-3,-3)*{\bu}="b_3",
(6,-6)*{\bu}="c_1",
(0,-6)*{\bu}="c_2",
(-3,-6)*{}="c_3",
(6,-9)*{}="d_1",
(0,-9)*{}="d_2",
\ar @{.} "a";"0" <0pt>
\ar @{.} "a";"b_3" <0pt>
\ar @{.} "a";"b_1" <0pt>
\ar @{.} "b_1";"c_1" <0pt>
\ar @{.} "b_1";"c_2" <0pt>
\ar @{-} "b_3";"c_3" <0pt>
\ar @{-} "c_1";"d_1" <0pt>
\ar @{-} "c_2";"d_2" <0pt>
\endxy
$  as in this case all three parameters, $||p||$,  $||p||\cdot ||p_{23}||$ and
$||p||\cdot ||p_{12}||$, tend to
$+\infty$ and the limit point $p/||p||$ in $\widetilde{C}_3^{st}$ consists of only two different
points (rather than of three ones as in the case of a generic point in the stratum
$ \Ba{c}
 \xy
 (0,0)*{\circ}="a",
(0,4)*{}="0",
(-3,-4)*{\bu}="b_1",
(0,-4)*{\bu}="b_2",
(3,-4)*{\bu}="b_3",
(-3,-7)*{}="c_1",
(0,-7)*{}="c_2",
(3,-7)*{}="c_3",
\ar @{.} "a";"0" <0pt>
\ar @{.} "a";"b_2" <0pt>
\ar @{.} "a";"b_1" <0pt>
\ar @{.} "a";"b_3" <0pt>
\ar @{-} "b_1";"c_1" <0pt>
\ar @{-} "b_2";"c_2" <0pt>
\ar @{-} "b_3";"c_3" <0pt>
\endxy
\Ea$);

\item[-] consider now  limit configurations in which the point $x_1$ goes far
away from the points $x_2$ and $x_3$, and simultaneously, the points $x_2$ and $x_3$
approach each other with the speed $\sim||p||^{-1}$ so that the product
$||p||\cdot ||p_{23}||$ is a well-defined {\em finite}\, number;
 in the second compactification  such  limit configurations
fill  in the stratum $ \Ba{c}
 \xy
 (0,0)*{\circ}="a",
(0,4)*{}="0",
(-3,-4)*{\bu}="b_1",
(3,-4)*{\bu}="b_3",
(-3,-8)*{}="c_1",
(0,-8)*{}="c_2",
(6,-8)*{}="c_3",
\ar @{.} "a";"0" <0pt>
\ar @{.} "a";"b_3" <0pt>
\ar @{.} "a";"b_1" <0pt>
\ar @{-} "b_1";"c_1" <0pt>
\ar @{-} "b_3";"c_2" <0pt>
\ar @{-} "b_3";"c_3" <0pt>
\endxy
\Ea
$ but in the first compactification all such configurations tend to one and the same
point in the boundary given by the graph  $\xy
 (0,0)*{\circ}="a",
(0,3)*{}="0",
(-3,-3)*{\bu}="b_1",
(3,-7)*{\circ}="e",
(3,-3)*{\bu}="b_3",
(-3,-6)*{}="c_1",
(0,-10)*{}="c_2",
(6,-10)*{}="c_3",
\ar @{.} "a";"0" <0pt>
\ar @{.} "a";"b_3" <0pt>
\ar @{.} "a";"b_1" <0pt>
\ar @{-} "b_1";"c_1" <0pt>
\ar @{-} "e";"c_2" <0pt>
\ar @{-} "b_3";"e" <0pt>
\ar @{-} "e";"c_3" <0pt>
\endxy
$, see picture (\ref{2: picture for J_3}).
\Ei


Using metric graphs we can make the second topological compactification, $\widehat{\fC}_\bu(\R)$,
 into a smooth manifold with corners
 by constructing a coordinate chart, $\cU_T$, near the boundary stratum corresponding
to an arbitrary tree $T\in \cM or(A_\infty)$ as follows:
\sip

{\bf (i)} Associate to $T$ a metric graph, $T_{\ccm \cce \cct\ccr \cci \ccc}$, by
assigning
\Bi
\item[(a)]
to every internal edge of $T$ of the form
$
\xy (0,3)*{\bullet}="a", (0,-3)*{\circ}="b"
           \ar @{-} "b";"a" <0pt>\endxy$  a {\em small}\, positive real number $\var \ll +\infty$;

\item[(b)] to every (if any) white vertex of a dashed corolla in $T$ a {\em  large}\, positive real number $\tau \gg 0$,
and to its every incoming edge a {\em small}\, parameter $\tau^{-1}$,
\Beq\label{2: tau decorated c orolla}
\xy
(3,0)*{^\tau},
(-10,-3)*{^{\frac{1}{\tau}}},
(-0.5,-4)*{^{\frac{1}{\tau}}},
(10,-3)*{^{\frac{1}{\tau}}},
(2,-6)*{\ldots},
 (0,0)*{\circ}="a",
(0,5)*{}="0",
(-12,-6)*{}="b_1",
(-5,-6)*{}="b_2",
(12,-6)*{}="b_3",
\ar @{.} "a";"0" <0pt>
\ar @{.} "a";"b_2" <0pt>
\ar @{.} "a";"b_1" <0pt>
\ar @{.} "a";"b_3" <0pt>
\endxy,
\Eeq

\item[(c)] to every (if any) two vertex subgraph of $T_{\ccm \cce \cct\ccr \cci \ccc}$ of the form
$\xy
(3,4)*{^{\tau_1}};
(3,-4)*{^{\tau_2}};
(3,0)*{^{\tau_1^{-1}}};
(0,4)*{\circ}="a",
(0,-4)*{\circ}="b",
\ar @{.} "b";"a" <0pt>
\endxy$
an inequality  $\tau_1\gg \tau_2\gg 0$. This can be understood as a relation $\tau_2=\var_{12} \tau_1$ for some {\em small}\,
parameter $\var_{12}$.
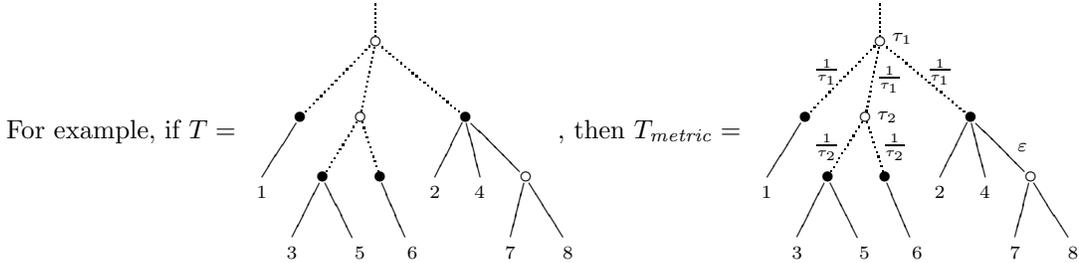
\Ei
For example, if
$
T=\Ba{c}
\xy
(-15,-10)*{_1},
(-11,-18)*{_3},
(-2,-18)*{_5},
(5,-18)*{_6},
(8,-10)*{_2},
(14,-10)*{_4},
(18,-18)*{_7},
(25.7,-18)*{_8},
(0,15)*{}="0",
 (0,10)*{\circ}="a",
(-10,0)*{\bullet}="b_1",
(-2,0)*{\circ}="b_2",
(12,0)*{\bullet}="b_3",
(-15,-8)*{}="c_0",
(0.6,-8)*{\bullet}="c_1",
(-7,-8)*{\bullet}="c_2",
(8,-8)*{}="c_3",
(14,-8)*{}="c_4",
(20,-8)*{\circ}="c_5",
(-11,-16)*{}="d_1",
(-3,-16)*{}="d_2",
(4,-16)*{}="d_3",
(18,-16)*{}="d_4",
(25,-16)*{}="d_5",
\ar @{.} "a";"0" <0pt>
\ar @{.} "a";"b_1" <0pt>
\ar @{.} "a";"b_2" <0pt>
\ar @{.} "a";"b_3" <0pt>
\ar @{-} "b_1";"c_0" <0pt>
\ar @{.} "b_2";"c_1" <0pt>
\ar @{.} "b_2";"c_2" <0pt>
\ar @{-} "b_3";"c_3" <0pt>
\ar @{-} "b_3";"c_4" <0pt>
\ar @{-} "b_3";"c_5" <0pt>
\ar @{-} "c_2";"d_1" <0pt>
\ar @{-} "c_2";"d_2" <0pt>
\ar @{-} "c_1";"d_3" <0pt>
\ar @{-} "c_5";"d_4" <0pt>
\ar @{-} "c_5";"d_5" <0pt>
\endxy
\Ea
$\hspace{-4mm},
then
$
T_{\ccm \cce \cct\ccr \cci \ccc}=
\Ba{c}
\xy
(1.3,5)*{_{\frac{1}{\tau_{1}}}},
(-7,6)*{_{\frac{1}{\tau_{1}}}},
(8,6)*{_{\frac{1}{\tau_{1}}}},
(19,-4)*{_{\var}},
(-7,-4)*{_{\frac{1}{\tau_{2}}}},
(2,-4)*{_{\frac{1}{\tau_{2}}}},
(3,10)*{^{\tau_1}},
(1,0)*{_{\tau_2}},
(-15,-10)*{_1},
(-11,-18)*{_3},
(-2,-18)*{_5},
(5,-18)*{_6},
(8,-10)*{_2},
(14,-10)*{_4},
(18,-18)*{_7},
(25.7,-18)*{_8},
(0,15)*{}="0",
 (0,10)*{\circ}="a",
(-10,0)*{\bullet}="b_1",
(-2,0)*{\circ}="b_2",
(12,0)*{\bullet}="b_3",
(-15,-8)*{}="c_0",
(0.6,-8)*{\bullet}="c_1",
(-7,-8)*{\bullet}="c_2",
(8,-8)*{}="c_3",
(14,-8)*{}="c_4",
(20,-8)*{\circ}="c_5",
(-11,-16)*{}="d_1",
(-3,-16)*{}="d_2",
(4,-16)*{}="d_3",
(18,-16)*{}="d_4",
(25,-16)*{}="d_5",
\ar @{.} "a";"0" <0pt>
\ar @{.} "a";"b_1" <0pt>
\ar @{.} "a";"b_2" <0pt>
\ar @{.} "a";"b_3" <0pt>
\ar @{-} "b_1";"c_0" <0pt>
\ar @{.} "b_2";"c_1" <0pt>
\ar @{.} "b_2";"c_2" <0pt>
\ar @{-} "b_3";"c_3" <0pt>
\ar @{-} "b_3";"c_4" <0pt>
\ar @{-} "b_3";"c_5" <0pt>
\ar @{-} "c_2";"d_1" <0pt>
\ar @{-} "c_2";"d_2" <0pt>
\ar @{-} "c_1";"d_3" <0pt>
\ar @{-} "c_5";"d_4" <0pt>
\ar @{-} "c_5";"d_5" <0pt>
\endxy
\Ea
$

\sip

with $\tau_1,\tau_2=\var_{12}\tau_1 \in (l,+\infty)$ and $\var_{12},\var\in (0,s)$ for some
large $l\gg 0$ and small $s\ll +\infty$ real numbers.
\mip

{\bf (ii)} Choose an equivariant section, $s: \fC_n(\R)\rar \Conf_n(\R)$, of the natural projection
$\Conf_n(\R)\rar  \fC_n(\R)$ as well as an arbitrary smooth structure on its image, $\fC_n^{st}(\R):=s(\fC_n(\R))$,
which is  called the space of configurations in the {\em standard position}; for example,
$\fC_n^{st}(\R)$ can be chosen to be the subspace of $\Conf_n(\R)$ satisfying the condition
$\sum_{i=1}^n x_i=0$; in particular, $\fC_1^{st}(\R)=0\in \R$.

\sip

{\bf (iii)} the required coordinate chart $\cU_T\subset \overline{\fC}_n(\R)$ is, by definition, isomorphic to the manifold with corners,
    $$(l,+\infty]^{\# V_\circ^{d}(T)}\times  [0,s)^{\# E_{\circ}^\bu(T)}\times
    \prod_{v\in V_\circ (T)} {C}_{\# In(v)}(\R)\times  \prod_{v\in V_\bu (T)} {\fC}_{\# In(v)}(\R)
    $$
where $V_\circ(T)$ is the set of white vertices of $T$, $V^d_\circ(T)\subset V_\circ(T)$ a subset
corresponding to dashed corollas, $V_\bu(T)$ the set of black vertices, and  $E_{\circ}^\bu(T)$ the set of internal edges of the type $
\xy (0,3)*{\bullet}="a", (0,-3)*{\circ}="b"
           \ar @{-} "b";"a" <0pt>\endxy$.
The isomorphism $\al_T$ between $\cU_T$ and the latter product
of manifolds with corners can be read from the metric graph via a simple procedure which we explain
on the particular example. For the tree $T$ shown above the map $\al_T$ is defined by
a formal extension
of the domain of the following continuous  map\footnote{We ordered factors $C_{\# In(v)}(\R)$ and $\fC_{\# In(v)}(\R)$
in the formula below in accordance with
a  natural ``from top to the bottom" and ``from left to the right" ordering of the vertices $v$ of the planar
tree $T$. We have also thrown away factors $\fC_1^{st}(\R)$ as they are just single points identified with $0\in \R$ .}
$$
\Ba{ccccccccccccccccc}
 \hspace{-2mm}(l,+\infty)^2&\hspace{-2mm}\times\hspace{-3mm} & (0,s) & \hspace{-3mm} \times \hspace{-3mm}&
 C_3^{st}(\R) &\hspace{-3mm} \times\hspace{-3mm} &  \fC_1^{st}(\R) & \hspace{-3mm} \times\hspace{-3mm} & C_2^{st}(\R)
 &\hspace{-3mm} \times\hspace{-3mm}&
 \fC_3^{st}(\R)  & \hspace{-3mm} \times\hspace{-3mm} & \fC_2^{st}(\R) & \hspace{-3mm} \times \hspace{-3mm} & \fC_1^{st}(\R)
 & \hspace{-2mm} \times \hspace{-2mm} & C_2^{st}(\R)
\\
 \hspace{-2mm}(\tau_1,\tau_2)&\hspace{-3mm}\times \hspace{-3mm} & \var & \hspace{-3mm} \times \hspace{-3mm} & (x', x'',x''') & \hspace{-3mm}\times
 \hspace{-3mm}& (x_1\hspace{-1mm}=\hspace{-1mm}0) & \hspace{-3mm} \times\hspace{-3mm} &  (t' \hspace{-1mm}= \hspace{-1mm}\frac{-1}{\sqrt{2}},t''\hspace{-1mm}=\hspace{-1mm}\frac{1}{\sqrt{2}})
  & \hspace{-3mm} \times \hspace{-3mm}& (x_2,x_4,u) & \hspace{-3mm}\times
 \hspace{-3mm}& (x_3, x_5)& \hspace{-3mm} \times \hspace{-3mm}& (x_6\hspace{-1mm}=\hspace{-1mm} 0)  &\hspace{-3mm} \times \hspace{-3mm}&
(x_7\hspace{-1mm}=\hspace{-1mm}\frac{-1}{\sqrt{2}},x_8\hspace{-1mm}=\hspace{-1mm}\frac{1}{\sqrt{2}}) \\
\Ea
$$
$$
\hspace{120mm}
\lon
\Ba{c} \fC_8(\R)\\
(y_1,\ldots, y_8)
\Ea
$$

with
$$
\Ba{lllrrrl}
y_1 &=& \tau_1 x' +\frac{1}{\tau_1}x_1 & &   y_2 &=&   \tau_1  x''' + \frac{1}{\tau_1}x_2  \\
y_3 &=&  \tau_1 x'' + \frac{1}{\tau_1}(\tau_2 t' + \frac{1}{\tau_2}x_3)   &&  y_4 &=& \tau_1 x''' +  \frac{1}{\tau_1}x_4\\
y_5 &=&  \tau_1 x'' + \frac{1}{\tau_1} (\tau_2 t' + \frac{1}{\tau_2}x_5)  &&  y_7 &=&\tau_1 x''' + \frac{1}{\tau_1}(u + \var x_7)\\
y_6 &=&   \tau_1 x'' + \frac{1}{\tau_1}(\tau_2 t'' + \frac{1}{\tau_2}x_6) , &&    y_8 &=&   \tau_1 x'''  +  \frac{1}{\tau_1}(u + \var x_8)
\Ea
$$
The boundary strata in $\cU_T$ are given by allowing formally $\tau_1=+\infty$, $\tau_2=+\infty$ with $\tau_1/\tau_2=0$
and/or $\var=0$. Therefore, the only novelty comparing to the case of associahedra discussed in
Sect.\ {\ref{2.1: associahedra}}
comes from the dashed corollas decorated by a large parameter $\tau$; such a corolla tacitly assumes
{\em two}\, rescaling operations: the first one is a
{\em magnification}\, of the standard configuration used to decorate its vertex by the parameter $\tau$,
and the second is a {\em compression} by the factor
$\tau^{-1}$ of  the standard configurations which correspond to all the corollas attached to its legs.

It is this second smooth structure on $\widehat{\fC}_\bu(\R)$ (and
its higher dimensional analogs,  $\widehat{\fC}_\bu(\R^d)$ and $\widehat{\fC}_\bu(\bbH)$, see below)
which we shall be interested in applications. We have no interesting propagators to show in
 the case of smooth structures of the first type at present.

\subsection{One more configuration space model for $\cM or(\cA_\infty)$} Let $\R^+:=\{x\in \R\ | \ x>0\}$ and, for a finite set $A$,
 consider a configuration space,
$$
\Conf_A(\R^+)=\left\{A\hook \R^+\right\}    
$$
of all injections  of $A$ into $\R^+$. The $1$-dimensional Lie group $\R^+$ acts on $\Conf_A(\R)$ by dilations,
$$
\Ba{ccc}
\Conf_A(\R^+) \times \R^+& \lon & \Conf_A(\R^+)\\
(p=\{x_i\}_{i\in A},\la) &\lon & \la p:= \{\la x_i\}_{i\in A}
\Ea
$$
This action is free so that the quotient space (cf.\ (\ref{2': def of fC(R)})),
\Beq\label{2': def of fC(R+)}
{\mathfrak C}_A(\R^+):= \Conf_A(\R^+)/\R^+,
\Eeq
is a naturally oriented $(\# A-1)$-dimensional real manifold equipped with a smooth action of the permutation group $Aut(A)$.
The space $\fC_1(\R^+)$ is a point and hence closed. For any $p\in \Conf_A(\R^+)$ we set $x_{min}(p):=\inf_{i\in A} x_i$.
Define the following section,
$$
C^{st}_{A}(\R^+):=\left\{p\in  \Conf_A(\R^+)\ |\  x_{min}(p)=1\ \mbox{and}\ |p-x_{min}(p)|=1             \right\}
$$
of the natural projection $\Conf_A(\R^+)\rar C_A(\R)$. Then, for any finite set $A$ with $\#A \geq 2,$ we have an equivariant isomorphism,
$$
\Ba{rccccc}
\Psi_A: &\fC_{A}(\R^+) & \lon & C^{st}_{A}(\R^+) \times \R^+ &\simeq & C_{A}(\R) \times (0,+\infty)   \\
& p & \lon & (\frac{p-x_{min}(p)}{|p-x_{min}(p)|}+1, |p-x_{min}(p)|).&&
\Ea
$$
Next one can define a compactification, $\widehat{\fC}_n(\R^+)$, of the space $\fC_{n}(\R^+)$ using either formulae
analogous to (\ref{2: first compactifn of fC(R)}) or to (\ref{2: second compactifn of fC(R)});
the face complex of $\widehat{\fC}_n(\R^+)$ is precisely the operad $\cM or(\cA_\infty)$.

\subsection{Metric graphs and smooth structures on compactified configuration spaces}\label{2: Remark on metric trees}
To save  space-time below in this paper we devote this subsection to a formalization of the Kontsevich type
construction of a smooth atlas on a compactified configuration space  with the help of metric trees which was used in
Sections ~{\ref{2: subsubs smooth atlas on associh}} and {\ref{2:  A different smooth structure on fC(R)}}. This subsection can be skipped as its only purpose is to give a rigorous
(and obvious) definition of the words ``by analogy to \S {\ref{2: subsubs smooth atlas on associh}} or to
{\ref{2:  A different smooth structure on fC(R)}}"
which we use several times below.

\sip

 All  compactified configurations spaces we work with in this paper
 come equipped with a structure of topological (coloured) operad,  $\overline{C}=\{\overline{C}_n\}_{n\geq 1}$ which, as an operad in the category of sets, is   free, $\overline{C}=\cF ree \langle C \rangle$, and  generated by an $\bS$-space,
$$
C=\{C_p:=\Conf_p(\V)/G\}_{p\geq 1},
$$
such that each $C_p$ is the quotient of the configuration space of $p$ pairwise distinct numbered  points in a subset\footnote{Say, $\V$ can be the ``upper half" of  $\R^d$ with respect to the coordinate $x^d$.}
 $\V\subset \R^d$
with respect to an action of a subgroup $G$ of the group $\R^+\ltimes \R^d$ (with $\R^+$ acting on $\R^d$ by dilations) which preserves $\V$ and commutes with the natural action of $\bS_p$ on
$\Conf_p(\R^d)$. As
$
\overline{C}=\cF ree \langle C \rangle$,  each topological space $\overline{C}_n$ is canonically stratified,
$$
\overline{C}_n=\coprod_{T\in {\cT_n}} C_T \ \ \ \ \ (C_T\hook \overline{C}_n\ \mbox{is continous}),
$$
into a disjoint union of cartesian products,
$$
C_T=\prod_{v\in V(T)} {C}_{\# In(v)}.
$$
Here ${\cT_n}$ stands for the family of trees whose input legs are in a fixed bijection with $[n]$,
$V(T)$ stands for the set of vertices of a tree $T\in \cT_n$, and, for a $v\in V(T)$, $In(v)$ stands for its set
of input half-edges.
For example if $T$ is given by (\ref{2: tree T}), then $C_T\simeq C_3^{\times 2}\times C_2^{\times 2}$.

\sip

In all cases of interest
in this paper the structure of a smooth manifold with corners on $\overline{C}_n$
will be defined via an explicit construction of a coordinate chart $\cU_T$ at each stratum $C_T$
with the help of an associated metric tree, $T_{\ccm\cce\cct\ccr\cci\ccc}$, obtained from $T$
\Bi
\item[(i)] by
 assigning
to all vertices, $v$, of some fixed colour a {\em large}\, parameter $\tau\in \R^+$ and simultaneously to all input legs
of $v$ the small parameter\footnote{A smooth structure on the compactification (\ref{2: first compactifn of fC(R)})
should be described with the help of metric corollas (\ref{2: tau decorated corolla 1}) {\em without}\, assigning to its input legs
the small paramter $\tau^{-1}$; however, we never use such a smooth structure in applications and, therefore, exclude it
from now on from the consideration in this paper.}
 $\tau^{-1}$, see (\ref{2: tau decorated c orolla}), and
\item[(ii)] by assigning to all other internal edges
a {\em small}\, parameter $\var$, see (\ref{2: metric tree T}) for an example.
\Ei
As a result {\em every}\, internal edge of  $T_{\ccm\cce\cct\ccr\cci\ccc}$ gets assigned
 a {\em small}\, parameter $\var$ or $\tau^{-1}$. To read the coordinate chart $\cU_T$ from such a metric tree one has to choose a suitable $\bS_n$-equivariant section,
$$
s: C_n \lon \Conf_n(\V)\subset \Conf_n(\R^d),
$$
of the natural projection  $\Conf_n(\V)\rar C_n$. The subspace $s(C_\bu)\subset \Conf_\bu(\R)$ is denoted by $C_\bu^{st}$ and  called the {\em  space of standard configurations}. Then one can use a natural
action of $\R^+\ltimes \V$ on  $\Conf_n(\R^d)$ to define a suitable translation map,
$$
\Ba{rccc}
T:& \V \times C^{st}_n & \lon & \Conf_n(\V)\\
  &(z_0,p)  &\lon & T_{z_0}(p)
\Ea
$$
and a {\em rescaling map},
$$
\Ba{rccc}
\centerdot:& \R^+ \times C^{st}_n & \lon & \Conf_n(\V)\\
  & (\la,p)   &\lon & \la\centerdot p.
\Ea
$$
Having made choices for $C_n^{st}$ and the maps  $T_{z_0}$ and $\la\centerdot$, one
constructs out of a metric tree  an open subset, $\cU_T$, containing the boundary stratum $T$ together with a homeomorphism,
\Beq\label{2: phi_T map}
\phi_T: \cU_T \lon [0, \delta)^m\times \prod_{v\in V(T)} C_{\# In(v)}^{st},
\Eeq
for some sufficiently small positive real number $\delta$ and natural number $m$. This homeomorphism together with a suitable choice
of an $\bS_\bu$-equivariant structure of a smooth manifold with corners on each $C_\bu^{st}$  makes $\cU_T$ itself
into a smooth manifold with corners; the final step in the construction of a smooth atlas $\{\cU_T\}$ on
$\overline{C}_\bu$ is to check smoothness of the transition functions at the non-empty intersections $\cU_T\cap
\cU_{T'}$ of the coordinate charts (which is often straighforward).

\sip

The construction of (\ref{2: phi_T map}) is universal and is best explained
in a particular example: if
$$
T_{\ccm \cce \cct\ccr \cci \ccc}=
\Ba{c}
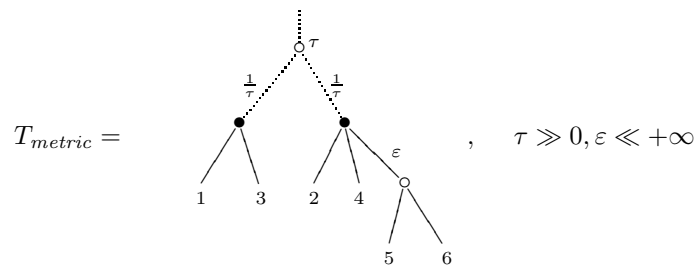
\xy
(-0.5,5)*{_{\frac{1}{\tau}}},
(11,5)*{_{\frac{1}{\tau}}},
(19,-4)*{_{\var}},
%
(8,10)*{^{\tau}},
(-7,-10)*{_1},
(1,-10)*{_3},
(8,-10)*{_2},
(14,-10)*{_4},
(18,-18)*{_5},
(25.7,-18)*{_6},
(6,15)*{}="0",
 (6,10)*{\circ}="a",
(-2,0)*{\bu}="b_2",
(12,0)*{\bullet}="b_3",
(-15,-8)*{}="c_0",
(0.6,-8)*{}="c_1",
(-7,-8)*{}="c_2",
(8,-8)*{}="c_3",
(14,-8)*{}="c_4",
(20,-8)*{\circ}="c_5",
(-11,-16)*{}="d_1",
(-3,-16)*{}="d_2",
(4,-16)*{}="d_3",
(18,-16)*{}="d_4",
(25,-16)*{}="d_5",
\ar @{.} "a";"0" <0pt>
\ar @{.} "a";"b_2" <0pt>
\ar @{.} "a";"b_3" <0pt>
\ar @{-} "b_2";"c_1" <0pt>
\ar @{-} "b_2";"c_2" <0pt>
\ar @{-} "b_3";"c_3" <0pt>
\ar @{-} "b_3";"c_4" <0pt>
\ar @{-} "b_3";"c_5" <0pt>
\ar @{-} "c_5";"d_4" <0pt>
\ar @{-} "c_5";"d_5" <0pt>
\endxy
\Ea, \ \ \ \ \tau\gg 0, \var\ll +\infty
$$
then $\cU_T\cap C_6$ is given, by definition, by the image under projection $\Conf_6(\V)\rar C_6$  of
a subset consisting of all possible configurations, $\{z_1,\ldots, z_6\}\in \Conf_6(\V)$,
which can be obtained as follows
\Bi
\item[Step 1:] Take an arbitrary standardly positioned configuration, $p^{(1)}=(z', z'')\in C_2^{st}$
and apply $\tau$-rescaling, $p^{(1)}\rar \tau\centerdot p^{(1)}=:( \tau\centerdot z',  \tau\centerdot z'')$;
\item[Step 2:] Take  arbitrary standard configurations,
$p^{(2)}=(z_1^{st}, z^{st}_3)\in C_2^{st}$ and  $p^{(3)}=(z^{st}_2, z^{st}_4, z''')\in C_3^{st}$,
$\tau^{-1}$-rescale them,
\Beqrn
p^{(2)}=(z_1^{st}, z^{st}_3) &\lon & \tau^{-1}\centerdot p^{(2)}=:(z_1, z_3),\\
p^{(3)}=(z_2^{st}, z^{st}_4, z''')& \lon & \tau^{-1}\centerdot p^{(3)}=:( z_2, z_4,
  z''''),
\Eeqrn
and then place the results at the points $\tau\centerdot z'$ and $\tau\centerdot z'$ respectively,
i.e.\ consider a configuration
$$
T_{\tau\centerdot z'}\left(  \tau^{-1}\centerdot p^{(2)}\right)\coprod T_{\tau\centerdot z''}
\left(  \tau^{-1}\centerdot p^{(3)}\right)=:\left(z_1,z_3,z_2,z_4, z'''''\right)\in \Conf_5(\V);
$$
\item[Step 3:] Finally, take an arbitrary standardly positioned configuration,
$p^{(4)}=(z^{st}_5, z^{st}_6)\in C_2^{st}$, apply $\var\tau^{-1}$-rescaling,
$p^{(4)}\rar \var\tau^{-1}\centerdot p^{(4)}$, and place the result into the point  $z'''''$. i.e. consider
    $T_{z'''''}( \var\tau^{-1}\centerdot p^{(4)})=:(z_5,z_6)$.
\Ei
If the constructed continuous map,
$$
\phi_T:
(0,\delta)^2 \times (C_2^{st})^{\times 3}\times C_3^{st} \rar \Conf_6(\V)
$$
is an injection for a sufficiently small $\delta\in \R^+$ (and this will be the case in all cases of interest in this paper), then its image gives us the desired
smooth coordinate chart $\cU_T$; the boundary stratum $C_T$ is given in this chart by setting formally the small
parameters $\var, \tau^{-1}\in (0,\delta)$ to zero.

\sip

The above construction is applicable to all operads, $\overline{C}$, of compactified configuration spaces studied in this paper. In each concrete case we have to specify  only three things,
\Bi
\item[(i)] an association $T\rar T_{metric}$,
\item[(ii)]
a definition of the space, $C_n^{st}\subset \Conf_n(\V)$, of standard positions, and
\item[(iii)] a rescaling map, $\centerdot: \R^+\times C_n^{st}\rar \Conf_n(\V)$, and a
translation map $T:  \V \times C_n^{st}\rar \Conf_n(\V)$.
\Ei
the rest of the construction of a smooth atlas goes along the lines formalized in this
subsection.


\bip

{\large
\section{\bf Kontsevich configuration spaces  and open-closed homotopy algebras}
\label{3: Section}
}
\mip

\subsection{Fulton-MacPherson compactification of points on the complex plane \cite{Ko}}
Let
$$
\Conf_n(\C):=\{z_1, \ldots, z_n\in \C\ |\, z_i\neq z_j\ \mbox{for}\ i\neq j\}
$$
be the configuration space of $n$ pairwise distinct points on the complex plane $\C$.
The space $C_{n}(\C)$ is a smooth $(2n-3)$-dimensional real manifold defined as the orbit space \cite{Ko},
 $$
C_{n}(\C):=\Conf_n(\C)/{G_{(3)}},
$$
with respect to the following action of a real 3-dimensional Lie group,
$$
G_{(3)}= \{z\rar az+b \ |\ a\in \R^+, b\in \C\}.
$$
The space $C_2(\C)$ is homeomorphic to the circle $S^1$ and hence is compact.
The compactification, $\overline{C}_n(\C)$, of ${C}_n(\C)$ for $n\geq 3$
can be defined  \cite{Ko, Ga}) as the closure of an embedding,
$$
\Ba{ccccc}
C_n(\C) & \lon & (\R/2\pi\Z )^{n(n-1)} &\times& (\R\P^2)^{n(n-1)(n-2)}\\
(z_1, \ldots, z_n) & \lon & \prod_{i\neq j} Arg(z_i - z_j) &\times&
 \prod_{i\neq j\neq k\neq i}\left[|z_{i}-z_{j}| :
|z_{j}-x_{k}|: |z_{i}-z_{k}|\right]
\Ea.
$$
The space $\overline{C}_n(\C)$ is a smooth (naturally oriented) manifold with corners.
Its codimension 1 strata is given by
$$
\p \overline{C}_n(\C) = \bigsqcup_{A\subset [n]\atop \# A\geq 2} C_{n - \# A + 1}(\C)\times
 C_{\# A}(\C)
$$
where the summation runs over all possible  proper subsets of $[n]$ with cardinality $\geq 2$.
Geometrically, each such  stratum corresponds to the $A$-labeled elements of the set $\{z_1, \ldots, z_n\}$ moving
very close
to each other. If we represent $\overline{C}_n(\C)$ by the symmetric $n$-corolla of
degree\footnote{We prefer working with {\em co}chain complexes, and hence always adopt gradings accordingly.}
 $3-2n$
\Beq\label{3: Lie_inf corolla}
 \xy
(1,-5)*{\ldots},
(-13,-7)*{_1},
(-8,-7)*{_2},
(-3,-7)*{_3},
(7,-7)*{_{n-1}},
(13,-7)*{_n},
 (0,0)*{\circ}="a",
(0,5)*{}="0",
(-12,-5)*{}="b_1",
(-8,-5)*{}="b_2",
(-3,-5)*{}="b_3",
(8,-5)*{}="b_4",
(12,-5)*{}="b_5",
\ar @{-} "a";"0" <0pt>
\ar @{-} "a";"b_2" <0pt>
\ar @{-} "a";"b_3" <0pt>
\ar @{-} "a";"b_1" <0pt>
\ar @{-} "a";"b_4" <0pt>
\ar @{-} "a";"b_5" <0pt>
\endxy=
\xy
(1,-6)*{\ldots},
(-13,-7)*{_{\sigma(1)}},
(-6.7,-7)*{_{\sigma(2)}},
(13,-7)*{_{\sigma(n)}},
 (0,0)*{\circ}="a",
(0,5)*{}="0",
(-12,-5)*{}="b_1",
(-8,-5)*{}="b_2",
(-3,-5)*{}="b_3",
(8,-5)*{}="b_4",
(12,-5)*{}="b_5",
\ar @{-} "a";"0" <0pt>
\ar @{-} "a";"b_2" <0pt>
\ar @{-} "a";"b_3" <0pt>
\ar @{-} "a";"b_1" <0pt>
\ar @{-} "a";"b_4" <0pt>
\ar @{-} "a";"b_5" <0pt>
\endxy,
\ \ \ \forall \sigma\in \bS_n,\ n\geq2
\Eeq
then the boundary operator in the associated face complex of $\overline{C}_\bu(\C)$ takes a familiar form
\Beq\label{3: Lie_infty differential}
\p\hspace{-3mm}
 \xy
(1,-5)*{\ldots},
(-13,-7)*{_1},
(-8,-7)*{_2},
(-3,-7)*{_3},
(7,-7)*{_{n-1}},
(13,-7)*{_n},
 (0,0)*{\circ}="a",
(0,5)*{}="0",
(-12,-5)*{}="b_1",
(-8,-5)*{}="b_2",
(-3,-5)*{}="b_3",
(8,-5)*{}="b_4",
(12,-5)*{}="b_5",
\ar @{-} "a";"0" <0pt>
\ar @{-} "a";"b_2" <0pt>
\ar @{-} "a";"b_3" <0pt>
\ar @{-} "a";"b_1" <0pt>
\ar @{-} "a";"b_4" <0pt>
\ar @{-} "a";"b_5" <0pt>
\endxy
=
\sum_{A\varsubsetneq [n]\atop
\# A\geq 2}\hspace{-2mm}
\Ba{c}
\begin{xy}
<10mm,0mm>*{\circ},
<10mm,0.8mm>*{};<10mm,5mm>*{}**@{-},
<0mm,-10mm>*{...},
<14mm,-5mm>*{\ldots},
<13mm,-7mm>*{\underbrace{\ \ \ \ \ \ \ \ \ \ \ \ \  }},
<14mm,-10mm>*{_{[n]\setminus A}};
<10.3mm,0.1mm>*{};<20mm,-5mm>*{}**@{-},
<9.7mm,-0.5mm>*{};<6mm,-5mm>*{}**@{-},
<9.9mm,-0.5mm>*{};<10mm,-5mm>*{}**@{-},
<9.6mm,0.1mm>*{};<0mm,-4.4mm>*{}**@{-},
<0mm,-5mm>*{\circ};
<-5mm,-10mm>*{}**@{-},
<-2.7mm,-10mm>*{}**@{-},
<2.7mm,-10mm>*{}**@{-},
<5mm,-10mm>*{}**@{-},
<0mm,-12mm>*{\underbrace{\ \ \ \ \ \ \ \ \ \ }},
<0mm,-15mm>*{_{A}},
\end{xy}
\Ea
\Eeq
implying the following  useful observation.

\subsubsection{\bf Proposition \cite{GJ}}  {\em The face complex
 of the family of
compactified configurations spaces,
$\{\overline{C}_n(\C)\}_{n\geq 2}$, has a structure of a dg free non-unital  pseudo-operad canonically isomorphic
to the operad, $L_\infty\{1\}$, of strong homotopy Lie algebras with degree shifted by one\footnote{Denote the endomorphism operad,
$\cE nd_{\K[m]}$, of the one 1-dimensional graded vector space $\K[m]$ by $\{m\}$. Then for any dg operad $P$ the tensor
product $P\ot \{m\}=: P\{m\}$ is again an operad whose representations in a graded vector space $V$ are in one-to-one
correspondence with representations of the operad $P$ in $V[m]$. Therefore, the association $P\rar P\{m\}$ is a kind of degree
shifting in the world of operads}.}
\mip
\subsubsection{\bf Smooth atlas on $\overline{C}_\bu(\C)$} The coordinate chart $\cU_T$ near the boundary stratum in  $\overline{C}_\bu(\C)$
corresponding to an arbitrary tree $T$ built from corollas ({\ref{3: Lie_inf corolla}}) is constructed as in
Section~{\ref{2: subsubs smooth atlas on associh}} by associating to $T$ a metric tree  $T_{\ccm\cce\cct\ccr\cci\ccc}$
whose every internal edge is assigned a small positive real number and whose every vertex, $v$, is decorated
with an element of $C_{\# In(v)}^{st}(\C)$ which is defined as the subset
of $\Conf_{\# In(v)}(\C)$ consisting of all configurations $(z_1,\ldots, z_{\# In(v)})$ satisfying two conditions,
$\sum_{i=1}^{\#In(v)} z_i=0$ and $\sum^{\#In(v)}_{i=1} |z_i|^2=1$.
The rescaling (resp.\ translation) map on $C^{st}_\bu(\C)$ is defined to be the ordinary dilation (resp.\ translation)
map, see Remark~{\ref{2: Remark on metric trees}}.

\subsubsection{\bf An equivalent construction of the compactification
$\overline{C}_\bu(\C)$ \cite {AT}}\label{3: second FM compact in C}
 Let $\Conf_A(\C)$ stand for the
 space of immersions, $A\hook \C$, of a finite non-empty set $A$ into the complex plane and
$\widetilde{\Conf}_A(\C)$ for the space of all possible maps. We define $C_A(\C)=\Conf_A(\C)/G_{(3)}$
and, for a configuration $p=\{z_i\}_{i\in A}\in \Conf_A(\C)$, we  set,
$$
z_c(p):=\frac{1}{\# A}\sum_{i\in A} z_i, \ \ \ \ |p-z_c(p)|:=\sqrt{\sum_{i\in A}|z_i-z_c(p)|^2}.
$$
Recall that $C_A(\C)$ can be equivariantly   identified with
$$
C_A^{st}(\C) = \{p\in \Conf_n(\C)\ |\ z_c(p)=0,\ |p-z_c(p)|=1\}
$$
Let us also consider a space,
$$
\widetilde{C}_A^{st}(\C)= \{p\in \widetilde{\Conf}_A(\C)\ |\ z_c(p)=0,\ |p-z_c(p)|=1\},
$$
which is a {\em compact}\, $(2\# A-3)$-dimensional manifold with boundary. The
compactification $\overline{C}_\bu(\C)$ can be defined as the closure of an embedding,
$$
C_n(\C) \stackrel{\prod \pi_A}{\lon} \prod_{A\subseteq [n]\atop \# A\geq 2} C_A(\C)
 \stackrel{\simeq }{\lon} \prod_{A\subseteq [n]\atop \# A\geq 2} C_A^{st}(\C)
\hook  \prod_{A\subseteq [n]\atop \# A\geq 2} \widetilde{C}^{st}_A(\C).
$$
where the product runs over all possible subsets $A$ of $[n]$ with $\# A\geq 2$, and
$$
\Ba{rccc}
\pi_A : & C_n(\C) & \lon & C_A(\C)\\
        & p=\{z_i\}_{i\in [n]} & \lon & p_A:=\{z_i\}_{i\in A}
\Ea
$$
is the natural forgetful  map.

\subsubsection{\bf Higher dimensional version}\label{3: subsection on C_nR^d} One sets, for $d\geq 2$,
$$
C_n(\R^{d}):=\frac{\{p_1, \ldots, p_n\in \R^{d}\ |\, p_i\neq p_j\ \mbox{for}\ i\neq j\}}{G_{(k+1)}},
$$
where $G_{(d+1)}:= \{p\rar \la p + \nu \ |\ \la\in \R^+, \nu\in \R^{d}\}$. The map
$$
\Ba{ccc}
C_2(\R^{d}) & \lon & S^{d-1}\\
(p_1,p_2) & \lon & \frac{p_1-p_2}{|p_1-p_2|}
\Ea
$$
is an isomorphism so that $C_2(\R^d)$ is compact. For $n\geq 3$ the compactified configuration
space $\overline{C}_n(\R^{d})$ is defined as the closure of an embedding
$$
\Ba{ccccc}
C_n(\R^{d}) & \lon &  (S^{d-1} )^{n(n-1)} &\times& (\R\P^2)^{n(n-1)(n-2)}\\
(p_1, \ldots, p_n) & \lon &  \prod_{i\neq j} \frac{p_i-p_j}{|p_i-p_j|} &\times&
 \prod_{i\neq j\neq k\neq i}\left[|p_{i}-p_{j}| :
|p_{j}-p_{k}|: |p_{i}-p_{k}|\right]
\Ea
$$

The face complex of $\{\overline{C}_n(\R^{d})\}_{n\geq 2}$ has a structure of a dg free operad canonically
isomorphic
to the operad of $\cA_\infty$-algebras for $d=1$, and, for $d\geq 2$,  to the operad,
$\caL_\infty\{d-1\}$, of strong homotopy Lie algebras with degree shifted by $d-1$ \cite{GJ}.
If $d$ is odd, then the action of an element $\sigma\in \bS_n$ on
$C_n(\R^{d})$ preserves  its natural orientation if the permutation $\sigma$ is even and
reverses the orientation if $\sigma$ is odd. Therefore, the generating corollas
(\ref{3: Lie_inf corolla}) of the operad $\overline{C}_\bu(\R^{d})$
are symmetric for $d$ even, and skewsymmetric for $d$ odd.

\subsection{Kontsevich's compactification of points in the upper half plane \cite{Ko}}
Let
$$
\Conf_{n,m}(\bbH):=\{z_1, \ldots, z_n\in \bbH, x_{\bar{1}},\ldots, x_{\bar{m}}\in \R\subset \p\overline{\bbH}\ |\, z_i\neq z_j,
x_{\bar{i}}\neq x_{\bar{j}}\ \mbox{for}\ i\neq j\}
$$
be the configuration space of $n+m$ pairwise distinct points on the closed upper half plane $\overline\bbH$.
For future reference we also define,
$$
\widetilde{\Conf}_{n,m}(\bbH):=\{z_1, \ldots, z_n\in \overline{\bbH}, x_{\bar{1}},
\ldots, x_{\bar{m}}\in \R\subset \p\overline{\bbH}\},
$$
where the condition on the points being distinct is dropped.

\sip

For any configuration $p= (z_1, \ldots, z_n, x_{\bar{1}},\ldots, x_{\bar{m}})\in \Conf_{n,m}(\bbH)$ we set
$$
x_c(p):=\frac{1}{n+m}\left(\sum_{i=1}^{n}\Re(z_i)+ \sum_{i=1}^m x_i\right), \ \ \ \ \mbox{and}\ \ \
|p- x_c(p)|:=\sqrt{\sum_{i=1}^{p}|z_i- x_c(p)|^2+ \sum_{i=1}^q |x_i- x_c(p)|^2}.
$$
The Lie group $G_{(2)}$ acts freely on $\Conf_{n,m}(\bbH)$ with the space of orbits,
 $$
C_{n,m}(\bbH):=\Conf_{n,m}(\bbH)/{G_{(2)}}, \ \ \ \ 2n+m\geq 2,
$$
being a $(2n+m-2)$-dimensional naturally oriented manifold. 
A compactification, $\overline{C}_{n,m}(\bbH)$, of  ${C}_{n,m}(\bbH)$  has been defined in
\cite{Ko} as the closure of an embedding
$$
\Ba{ccccc}
C_{n,m}(\bbH) & \lon & C_{2n+m}(\C) &\hook  & \overline{C}_{2n+m}(\C)\\
(z_1, \ldots, z_n, x_{\bar{1}},\ldots, x_{\bar{m}}) & \lon & (z_1, \ldots, z_n, \overline{z_1},\ldots,  \overline{z_n}, x_{\bar{1}},\ldots, x_{\bar{m}})
 &&
\Ea.
$$
The face complex of a disjoint union,
\Beq\label{3: Konts 2-coloured operad}
\overline{C}(\bbH):=\overline{C}_\bu(\C)\bigsqcup \overline{C}_{\bu,\bu}(\bbH),
\Eeq
has a natural structure of a dg free 2-colored operad \cite{KS} generated by degree $3-2n$ corollas (\ref{3: Lie_inf corolla})
representing $\overline{C}_n(\C)$, $n\geq 2$, and degree $2-2n-m$ corollas,
\Beq\label{3: OCHA corolla}
\Ba{c}
\begin{xy}
 <0mm,-0.5mm>*{\blacktriangledown};
 <0mm,0mm>*{};<0mm,5mm>*{}**@{.},
 <0mm,0mm>*{};<-16mm,-5mm>*{}**@{-},
 <0mm,0mm>*{};<-11mm,-5mm>*{}**@{-},
 <0mm,0mm>*{};<-3.5mm,-5mm>*{}**@{-},
 <0mm,0mm>*{};<-6mm,-5mm>*{...}**@{},
   <0mm,0mm>*{};<-16mm,-8mm>*{^{1}}**@{},
   <0mm,0mm>*{};<-11mm,-8mm>*{^{2}}**@{},
   <0mm,0mm>*{};<-3mm,-8mm>*{^{n}}**@{},
 <0mm,0mm>*{};<16mm,-5mm>*{}**@{.},
 <0mm,0mm>*{};<8mm,-5mm>*{}**@{.},
 <0mm,0mm>*{};<3.5mm,-5mm>*{}**@{.},
 <0mm,0mm>*{};<11.6mm,-5mm>*{...}**@{},
   <0mm,0mm>*{};<19mm,-8mm>*{^{\bar{m}}}**@{},
<0mm,0mm>*{};<10mm,-8mm>*{^{\bar{2}}}**@{},
   <0mm,0mm>*{};<5mm,-8mm>*{^{\bar{1}}}**@{},
 \end{xy}
\Ea=
\Ba{c}
\begin{xy}
 <0mm,-0.5mm>*{\blacktriangledown};
 <0mm,0mm>*{};<0mm,5mm>*{}**@{.},
 <0mm,0mm>*{};<-16mm,-5mm>*{}**@{-},
 <0mm,0mm>*{};<-11mm,-5mm>*{}**@{-},
 <0mm,0mm>*{};<-3.5mm,-5mm>*{}**@{-},
 <0mm,0mm>*{};<-6mm,-5mm>*{...}**@{},
   <0mm,0mm>*{};<-18mm,-8mm>*{^{\sigma(1)}}**@{},
   <0mm,0mm>*{};<-11mm,-8mm>*{^{\sigma(2)}}**@{},
   <0mm,0mm>*{};<-3mm,-8mm>*{^{\sigma(n)}}**@{},
 <0mm,0mm>*{};<16mm,-5mm>*{}**@{.},
 <0mm,0mm>*{};<8mm,-5mm>*{}**@{.},
 <0mm,0mm>*{};<3.5mm,-5mm>*{}**@{.},
 <0mm,0mm>*{};<11.6mm,-5mm>*{...}**@{},
   <0mm,0mm>*{};<19mm,-8mm>*{^{\bar{m}}}**@{},
<0mm,0mm>*{};<10mm,-8mm>*{^{\bar{2}}}**@{},
   <0mm,0mm>*{};<5mm,-8mm>*{^{\bar{1}}}**@{},
 \end{xy}
\Ea
, \ \ \ \ 2n+m\geq 2, \forall\ \sigma\in \bS_n
\Eeq
representing  $\overline{C}_{n,m}(\bbH)$. The boundary differential in the associated face complex is given on the generators
by (\ref{3: Lie_infty differential}) and the following formula \cite{Ko, KS}
\Beqr
\p
\Ba{c}
\begin{xy}
 <0mm,-0.5mm>*{\blacktriangledown};
 <0mm,0mm>*{};<0mm,5mm>*{}**@{.},
 <0mm,0mm>*{};<-16mm,-5mm>*{}**@{-},
 <0mm,0mm>*{};<-11mm,-5mm>*{}**@{-},
 <0mm,0mm>*{};<-3.5mm,-5mm>*{}**@{-},
 <0mm,0mm>*{};<-6mm,-5mm>*{...}**@{},
   <0mm,0mm>*{};<-16mm,-8mm>*{^{1}}**@{},
   <0mm,0mm>*{};<-11mm,-8mm>*{^{2}}**@{},
   <0mm,0mm>*{};<-3mm,-8mm>*{^{n}}**@{},
 <0mm,0mm>*{};<16mm,-5mm>*{}**@{.},
 <0mm,0mm>*{};<8mm,-5mm>*{}**@{.},
 <0mm,0mm>*{};<3.5mm,-5mm>*{}**@{.},
 <0mm,0mm>*{};<11.6mm,-5mm>*{...}**@{},
   <0mm,0mm>*{};<17mm,-8mm>*{^{\bar{m}}}**@{},
<0mm,0mm>*{};<10mm,-8mm>*{^{\bar{2}}}**@{},
   <0mm,0mm>*{};<5mm,-8mm>*{^{\bar{1}}}**@{},
 \end{xy}
\Ea
&=&
\sum_{A\varsubsetneq [n]\atop
\# A\geq 2}\
\Ba{c}
\begin{xy}
 <0mm,-0.5mm>*{\blacktriangledown};
 <0mm,0mm>*{};<0mm,5mm>*{}**@{.},
 <0mm,0mm>*{};<-16mm,-5mm>*{}**@{-},
 <0mm,0mm>*{};<-11mm,-5mm>*{}**@{-},
 <0mm,0mm>*{};<-3.5mm,-5mm>*{}**@{-},
 <0mm,0mm>*{};<-6mm,-5mm>*{...}**@{},
 <0mm,0mm>*{};<16mm,-5mm>*{}**@{.},
 <0mm,0mm>*{};<8mm,-5mm>*{}**@{.},
 <0mm,0mm>*{};<3.5mm,-5mm>*{}**@{.},
 <0mm,0mm>*{};<11.6mm,-5mm>*{...}**@{},
   <0mm,0mm>*{};<17mm,-8mm>*{^{\bar{m}}}**@{},
<0mm,0mm>*{};<10mm,-8mm>*{^{\bar{2}}}**@{},
   <0mm,0mm>*{};<5mm,-8mm>*{^{\bar{1}}}**@{},
<-17mm,-12mm>*{\underbrace{\ \ \ \ \ \ \ \ \ \   }},
<-17mm,-14.9mm>*{_A};
<-6mm,-7mm>*{\underbrace{\ \ \ \ \ \ \  }},
<-6mm,-10mm>*{_{[n]\setminus A}};
 (-16.5,-5.5)*{\circ}="a",
(-23,-10)*{}="b_1",
(-20,-10)*{}="b_2",
(-16,-10)*{...}="b_3",
(-12,-10)*{}="b_4",
\ar @{-} "a";"b_2" <0pt>
\ar @{-} "a";"b_1" <0pt>
\ar @{-} "a";"b_4" <0pt>
 \end{xy}
\Ea \ \ \nonumber  \\
\label{3: differential on Konts corollas}
&+& \sum_{k, l, [n]=I_1\sqcup I_2\atop
{2\# I_1 + m \geq l+1 \atop
2\#I_2 + l\geq 2}}
(-1)^{k+l(n-k-l)}
\Ba{c}
\begin{xy}
 <0mm,-0.5mm>*{\blacktriangledown};
 <0mm,0mm>*{};<0mm,6mm>*{}**@{.},
 <0mm,0mm>*{};<-16mm,-6mm>*{}**@{-},
 <0mm,0mm>*{};<-11mm,-6mm>*{}**@{-},
 <0mm,0mm>*{};<-3.5mm,-6mm>*{}**@{-},
 <0mm,0mm>*{};<-6mm,-6mm>*{...}**@{},
<0mm,0mm>*{};<2mm,-9mm>*{^{\bar{1}}}**@{},
<0mm,0mm>*{};<6mm,-9mm>*{^{\bar{k}}}**@{},
<0mm,0mm>*{};<19mm,-9mm>*{^{\overline{k+l+1}}}**@{},
<0mm,0mm>*{};<28mm,-9mm>*{^{\overline{m}}}**@{},
<0mm,0mm>*{};<13mm,-16.6mm>*{^{\overline{k+1}}}**@{},
<0mm,0mm>*{};<20mm,-16.6mm>*{^{\overline{k+l}}}**@{},
 <0mm,0mm>*{};<11mm,-6mm>*{}**@{.},
 <0mm,0mm>*{};<6mm,-6mm>*{}**@{.},
 <0mm,0mm>*{};<2mm,-6mm>*{}**@{.},
 <0mm,0mm>*{};<17mm,-6mm>*{}**@{.},
 <0mm,0mm>*{};<25mm,-6mm>*{}**@{.},
 <0mm,0mm>*{};<4mm,-6mm>*{...}**@{},
<0mm,0mm>*{};<20mm,-6mm>*{...}**@{},
<6.5mm,-16mm>*{\underbrace{\ \ \ \ \   }_{I_2}},
<-10mm,-9mm>*{\underbrace{\ \ \ \ \ \ \ \ \ \ \ \   }_{I_1}},
 %
 (11,-7)*{\blacktriangledown}="a",
(4,-13)*{}="b_1",
(9,-13)*{}="b_2",
(16,-13)*{...},
(7,-13)*{...},
(13,-13)*{}="b_3",
(19,-13)*{}="b_4",
\ar @{-} "a";"b_2" <0pt>
\ar @{.} "a";"b_3" <0pt>
\ar @{-} "a";"b_1" <0pt>
\ar @{.} "a";"b_4" <0pt>
 \end{xy}
\Ea
\Eeqr
This operad was denoted in \cite{KS} by\footnote{This notation may be misleading as this operad is neither a minimal
resolution of some operad $\f\cC$ nor a cobar construction of some cooperad which is Koszul dual to a quadratic operad
$\f\cC$.}
 $\f\cC_\infty$ and its representations in a pair of dg vector spaces
$(X_c,X_o)$  were called {\em open-closed homotopy algebras}\, or OCHA for short.
\sip

 It was shown in
\cite{H} that representations of $\f\cC_\infty$ are in one-to-one correspondence with degree one codifferentials
in the tensor product, $\odot^\bu( X_c[2])\bigotimes \ot^\bu (X_o[1])$, of
the free graded cocommutative coalgebra cogenerated by $X_c[2]$ and the free coalgebra cogenerated by $X_o[1]$.
As  $\f\cC_\infty$ is a {\em free}\, operad, its arbitrary representation, $\rho$, is uniquely determined
by the values on the generators,
\Beqrn
\nu_n&:=& \rho  \left(  \xy
(1,-5)*{\ldots},
(-13,-7)*{_1},
(-8,-7)*{_2},
(-3,-7)*{_3},
(7,-7)*{_{n-1}},
(13,-7)*{_n},
 (0,0)*{\circ}="a",
(0,5)*{}="0",
(-12,-5)*{}="b_1",
(-8,-5)*{}="b_2",
(-3,-5)*{}="b_3",
(8,-5)*{}="b_4",
(12,-5)*{}="b_5",
\ar @{-} "a";"0" <0pt>
\ar @{-} "a";"b_2" <0pt>
\ar @{-} "a";"b_3" <0pt>
\ar @{-} "a";"b_1" <0pt>
\ar @{-} "a";"b_4" <0pt>
\ar @{-} "a";"b_5" <0pt>
\endxy\right) \in \Hom(\odot^n X_c, X_c[3-2n]), \ \ \ n\geq 2, \\
\mu_{n,m} & := & \rho  \left(
\Ba{c}
\begin{xy}
 <0mm,-0.5mm>*{\blacktriangledown};
 <0mm,0mm>*{};<0mm,5mm>*{}**@{.},
 <0mm,0mm>*{};<-16mm,-5mm>*{}**@{-},
 <0mm,0mm>*{};<-11mm,-5mm>*{}**@{-},
 <0mm,0mm>*{};<-3.5mm,-5mm>*{}**@{-},
 <0mm,0mm>*{};<-6mm,-5mm>*{...}**@{},
   <0mm,0mm>*{};<-16mm,-8mm>*{^{1}}**@{},
   <0mm,0mm>*{};<-11mm,-8mm>*{^{2}}**@{},
   <0mm,0mm>*{};<-3mm,-8mm>*{^{n}}**@{},
 <0mm,0mm>*{};<16mm,-5mm>*{}**@{.},
 <0mm,0mm>*{};<8mm,-5mm>*{}**@{.},
 <0mm,0mm>*{};<3.5mm,-5mm>*{}**@{.},
 <0mm,0mm>*{};<11.6mm,-5mm>*{...}**@{},
   <0mm,0mm>*{};<17mm,-8mm>*{^{\bar{m}}}**@{},
<0mm,0mm>*{};<10mm,-8mm>*{^{\bar{2}}}**@{},
   <0mm,0mm>*{};<5mm,-8mm>*{^{\bar{1}}}**@{},
 \end{xy}
\Ea
\right) \in \Hom(\odot^n X_c\bigotimes \ot^m X_o, X_o[2-2n-m]), \ \ 2n +m \geq 2,
\Eeqrn
which satisfy quadratic relations given by the formulae for the differential $\p$. One often denotes
in this context the given differential in  the dg space $X_c$ by
$\nu_1$ and the one in $X_o$ by $\mu_{0,1}$.

\subsubsection{\bf An interpretation of OCHA}\label{3: remark on repr of OC_infty}
 Let
$\mbox{Coder}(\ot^\bu (X_o[1]), [\ ,\ ])$ be the Lie algebra of
coderivations
of the free coalgebra, $\ot^\bu (X_o[1])$, cogenerated by $X_o[1]$.
We do {\em not}\,  assume that coderivations
preserve the co-unit so that MC elements, $\ga$, in this Lie algebra describe, in general, non-flat
$\cA_\infty$-structures on $X_o$.
We have an isomorphism of vector spaces,
$$
\mbox{Coder}(\ot^\bu (X_o[1]))=\oplus_{m\geq 0} \Hom(\ot^m X_o, X_o[1-m]).
$$

It is not hard to check that a representation, $\rho$, of the dg operad $\f\cC_\infty$ in
a pair of dg vector spaces
$(X_c,X_o)$ is equivalent to the
following data
\Bi
\item[(i)] a $\caL_\infty\{1\}$-algebra structure, $\nu=\{\nu_n: \odot^n X_c\rar X_c[3-2n]\}_{n\geq 1}$, in $X_c$, i.e.\
 an ordinary $\caL_\infty$-structure in $X_c[1]$;
\item[(ii)] a $\cA_\infty$-algebra structure, $\mu=\{\mu_{0,m}: \ot^m X_o\rar X_c[2-m]\}_{m\geq 1}$, in $X_o$; the associated
MC element, $\mu$,  of the Lie algebra $\mbox{Coder}(\ot^\bu (X_o[1])), [\ ,\ ])$ makes the latter into a {\em dg}\, Lie
algebra with the differential $d_\mu:=[\mu,\ ]$;
\item[(iii)] a  $\caL_\infty$-morphism, $F$, from the $L_\infty$-algebra $(X_c[1], \nu)$ to
the dg Lie algebra $(\mbox{Coder}(\ot^\bu (X_o[1])), [\ ,\ ], d_\mu)$,
$$
F=\left\{F_n: \odot^n X_c   \lon  \mbox{Coder}(\ot^\bu (X_o[1]))[1-2n]            \right\}_{n\geq 1}
$$
such that the composition
$$
 \odot^n X_c   \stackrel{F_n}{\lon}  \mbox{Coder}(\ot^\bu (X_o[1]))[1-2n]  \stackrel{proj}{\lon}
 \Hom(\ot^m X_o, X_o[2-2n-m])
$$
coincides  precisely with $\mu_{n,m}$ for any $n\geq 1$, $m\geq 0$.
\Ei

\sip

If $\rho$ is an arbitrary representation of $\f\cC_\infty$ and $\ga\in X_c$ is an arbitrary Maurer-Cartan
element\footnote{We tacitly assume here that the $L_\infty$-algebra
$(X_c,\nu_\bu)$ is appropriately filtered so that
the MC equation makes sense. In our applications below $\nu_{n\geq 3}=0$
 so that one has no problems with convergence of the infinite sum.},
$$
\sum_{n\geq 0}\frac{1}{n!} \nu_n(\ga^{\ot n})=0, \ \ \ |\ga|=2,
$$
of the associated $\caL_\infty$-algebra $(X_c,\nu)$, then the maps
$$
\mu_m:= \sum_{n\geq 1}\frac{\hbar^n}{n!}
\mu_{n,m}(\ga^{\ot n}\ot \ldots): \ot^m X_o\lon X_o[[\hbar]][2-m], \ \ m\geq 0,
$$
make the topological (with respect to the adic topology) vector space $X_o[[\hbar]]$ into a
continuous and, in general, non-flat $\cA_\infty$-algebra.  Here $\hbar$ is a formal parameter, and
$X_c[[\hbar]:=X_c\ot\K[[\hbar]]$.

\subsubsection{\bf Example} Kontsevich's formality construction \cite{Ko} gives a
non-trivial representation of $\f\cC_\infty$
in the pair $(X_c:=\cT_{poly}(\R^d), X_o:=C^\infty(\R^d))$, consisting of the space of polyvector
fields and the space of smooth functions
on $\R^d$ for any $d$.

\subsubsection{\bf Smooth atlas on $\overline{C}(\C,\bbH)$}\label{3: Smooth atlas on Konts config spaces}
A generic boundary stratum in  $\overline{C}(\bbH)=\overline{C}_\bu(\C)\coprod \overline{C}_{\bu,\bu}(\bbH)$ is given by a tree $T$
constructed from corollas (\ref{3: Lie_inf corolla}) and (\ref{3: OCHA corolla}).  A
smooth coordinate chart, $\cU_T$, containing that boundary stratum is constructed as in
Section~{\ref{2: Remark on metric trees}} by associating to $T$ a metric tree  $T_{\ccm\cce\cct\ccr\cci\ccc}$
whose
every (of any colour) internal   edge is assigned a small positive real number $\var$ (i.e.\ there are no metric
corollas  type (\ref{2: tau decorated c orolla})).
The spaces of {\em standard configurations}\, associated with $\circ$-vertices are set to be $C_{\bu}^{st}(\C)$ and with $\blacktriangledown$-vertices are set to be $C_{\bu,\bu}^{st}(\bbH)$  which are, by definition,  the subsets
of $\Conf_{\bu,\bu}(\bbH)$ consisting of all configurations $p$ satisfying two conditions,
$x_c(p)=0$ and $||p||=1$. The rescaling and translation maps are defined to be the ordinary dilation,
 $z\rar \lambda z$, and translation, $T_{z_o}: z\rar z+ z_o$, maps on both
$C_{\bu}^{st}(\C)$ and $C_{\bu,\bu}^{st}(\bbH)$. The latter two groups of spaces can be equipped with arbitrary
$\bS$-equivariant smooth structures.

\subsubsection{\bf Example: Kontsevich eye}
$$
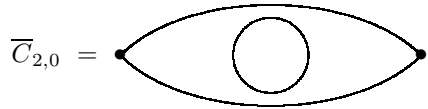
\overline{C}_{2,0}\ = \
{
\xy
 (0,0)*{
\xycircle(5,5){}};
(-20,0)*{\bullet},
(20,0)*{\bullet},
(-9,0)*+{};(8.5,0)*+{};
%
   \ar@{-}@(ur,lu) (-20.0,0.0)*{};(20.0,0.0)*{},
   \ar@{-}@(dr,ld) (-20.0,0.0)*{};(20.0,0.0)*{},
\endxy}
$$
The codimension 1 boundary splits into the union of three strata: the inner circle $C_2(\C)$ (``pupil")
represents limit configurations when two points $z_1,z_2\in \bbH$ collapse into a single point in $\bbH$,
the upper (resp. lower)
lid $C_{1,1}(\C)$ represents limit configurations of the form $(z_1\in \R ,z_2\in \bbH)$ (resp.\
$(z_1\in \bbH ,z_2\in \R)$).

\subsection{Higher dimensional version of $C_{n,m}(\bbH)$ and  open-closed homotopy Lie algebras}
\label{3.3 higher dim version of Konts spaces}
 Let $\bbH^{d}$ stand for a subspace of $\R^{d}=\{x_1, \ldots, x_{d-1}, x_{d}\}$
 such that $x_{d}> 0$, $d\geq 2$. The space $\bbH^2$ is just another notation for the
 the upper-half-plane $\bbH$. In a full analogy
 to the case $\overline{C}_{\bu,\bu}(\bbH^2)$ one defines a compactification $\overline{C}_{\bu,\bu}(\bbH^d)$
 for any  $d\geq 3$ and observes that the disjoint union,
$$
 \overline{C}(\bbH^{d}):=\overline{C}_\bu(\R^d)\coprod \overline{C}_{\bu,\bu}(\bbH^{d})
 $$
has a natural structure of a 2-coloured operad in the category of semialgebraic sets. Representations
of the associated  dg operad of fundamental chains, $\cF \cC hains(\overline{C}(\bbH^{d})$, are called
{\em open-closed homotopy Lie algebras} or, shortly, OCHLA. One can describe this operad  in
terms of  generators of degree $d(1-n)-(d-1)m$,
$$
\Ba{c}
\begin{xy}
 <0mm,-0.5mm>*{\blacktriangledown};
 <0mm,0mm>*{};<0mm,5mm>*{}**@{.},
 <0mm,0mm>*{};<-16mm,-5mm>*{}**@{-},
 <0mm,0mm>*{};<-11mm,-5mm>*{}**@{-},
 <0mm,0mm>*{};<-3.5mm,-5mm>*{}**@{-},
 <0mm,0mm>*{};<-6mm,-5mm>*{...}**@{},
   <0mm,0mm>*{};<-16mm,-8mm>*{^{1}}**@{},
   <0mm,0mm>*{};<-11mm,-8mm>*{^{2}}**@{},
   <0mm,0mm>*{};<-3mm,-8mm>*{^{n}}**@{},
 <0mm,0mm>*{};<16mm,-5mm>*{}**@{.},
 <0mm,0mm>*{};<8mm,-5mm>*{}**@{.},
 <0mm,0mm>*{};<3.5mm,-5mm>*{}**@{.},
 <0mm,0mm>*{};<11.6mm,-5mm>*{...}**@{},
   <0mm,0mm>*{};<19mm,-8mm>*{^{\bar{m}}}**@{},
<0mm,0mm>*{};<10mm,-8mm>*{^{\bar{2}}}**@{},
   <0mm,0mm>*{};<5mm,-8mm>*{^{\bar{1}}}**@{},
 \end{xy}
\Ea\hspace{-2mm}
=(-1)^{d sgn(\sigma) + (d-1)sgn(\tau)}\hspace{-5mm}
\Ba{c}
\begin{xy}
 <0mm,-0.5mm>*{\blacktriangledown};
 <0mm,0mm>*{};<0mm,5mm>*{}**@{.},
 <0mm,0mm>*{};<-16mm,-5mm>*{}**@{-},
 <0mm,0mm>*{};<-11mm,-5mm>*{}**@{-},
 <0mm,0mm>*{};<-3.5mm,-5mm>*{}**@{-},
 <0mm,0mm>*{};<-6mm,-5mm>*{...}**@{},
   <0mm,0mm>*{};<-18mm,-8mm>*{^{\sigma(1)}}**@{},
   <0mm,0mm>*{};<-11mm,-8mm>*{^{\sigma(2)}}**@{},
   <0mm,0mm>*{};<-3mm,-8mm>*{^{\sigma(n)}}**@{},
 <0mm,0mm>*{};<16mm,-5mm>*{}**@{.},
 <0mm,0mm>*{};<8mm,-5mm>*{}**@{.},
 <0mm,0mm>*{};<3.5mm,-5mm>*{}**@{.},
 <0mm,0mm>*{};<11.6mm,-5mm>*{...}**@{},
   <0mm,0mm>*{};<19mm,-8mm>*{^{\tau(\bar{m})}}**@{},
<0mm,0mm>*{};<10mm,-8mm>*{^{\tau(\bar{2})}}**@{},
   <0mm,0mm>*{};<5mm,-8mm>*{^{\tau(\bar{1})}}**@{},
 \end{xy}
\Ea
,  \ dn+(d-1)m\geq d, \forall\ \sigma\in \bS_n, \tau
\in\bS_m,
$$
  or, equivalently,
in terms of its representations in an arbitrary pair, $(X_o, X_c)$, of dg vector spaces. We choose here
a second more compact option.
\sip

As  $\cF \cC hains(\overline{C}(\bbH^{d})$ is a {\em free}\, operad, its arbitrary representation, $\rho$,
in $(X_o, X_c)$
is uniquely determined by the values on the generators,
\Beqrn
\nu_n&:=& \rho  \left(  \xy
(1,-5)*{\ldots},
(-13,-7)*{_1},
(-8,-7)*{_2},
(-3,-7)*{_3},
(7,-7)*{_{n-1}},
(13,-7)*{_n},
 (0,0)*{\circ}="a",
(0,5)*{}="0",
(-12,-5)*{}="b_1",
(-8,-5)*{}="b_2",
(-3,-5)*{}="b_3",
(8,-5)*{}="b_4",
(12,-5)*{}="b_5",
\ar @{-} "a";"0" <0pt>
\ar @{-} "a";"b_2" <0pt>
\ar @{-} "a";"b_3" <0pt>
\ar @{-} "a";"b_1" <0pt>
\ar @{-} "a";"b_4" <0pt>
\ar @{-} "a";"b_5" <0pt>
\endxy\right) \in \displaystyle\Hom\left(\odot^n (X_c[d]), X_c[d+1]\right), \ \ \ n\geq 2, \\
\mu_{n,m} & := & \rho  \left(
\Ba{c}
\begin{xy}
 <0mm,-0.5mm>*{\blacktriangledown};
 <0mm,0mm>*{};<0mm,5mm>*{}**@{.},
 <0mm,0mm>*{};<-16mm,-5mm>*{}**@{-},
 <0mm,0mm>*{};<-11mm,-5mm>*{}**@{-},
 <0mm,0mm>*{};<-3.5mm,-5mm>*{}**@{-},
 <0mm,0mm>*{};<-6mm,-5mm>*{...}**@{},
   <0mm,0mm>*{};<-16mm,-8mm>*{^{1}}**@{},
   <0mm,0mm>*{};<-11mm,-8mm>*{^{2}}**@{},
   <0mm,0mm>*{};<-3mm,-8mm>*{^{n}}**@{},
 <0mm,0mm>*{};<16mm,-5mm>*{}**@{.},
 <0mm,0mm>*{};<8mm,-5mm>*{}**@{.},
 <0mm,0mm>*{};<3.5mm,-5mm>*{}**@{.},
 <0mm,0mm>*{};<11.6mm,-5mm>*{...}**@{},
   <0mm,0mm>*{};<17mm,-8mm>*{^{\bar{m}}}**@{},
<0mm,0mm>*{};<10mm,-8mm>*{^{\bar{2}}}**@{},
   <0mm,0mm>*{};<5mm,-8mm>*{^{\bar{1}}}**@{},
 \end{xy}
\Ea
\right) \in \Hom\left(\odot^n (X_c[d])\bigotimes \odot^m (X_o[d-1]), X_o[d]\right), \ \ dn +(d-1)m \geq d.
\Eeqrn
which satisfy quadratic relations which we explain in the  definition of OCHLA.
Let us denote  the given differential in  the dg space $X_c$ by
$\nu_1$ and the one in $X_o$ by $\mu_{0,1}$.

\sip

A structure of an {\em open-closed homotopy Lie d-algebra}\, in  a pair of dg vector spaces
$(X_c,X_o)$ is, for $d\geq 3$, the data,
\Bi
\item[(i)] a $\caL_\infty\{d-1\}$-algebra structure,
$\nu=\{\nu_n: \odot^n (X_c[d])\rar X_c[d+1]\}_{n\geq 1}$, on $X_c$, i.e.\ an ordinary
$\caL_\infty$-structure on $X_c[d-1]$;
\item[(ii)] a $\caL_\infty\{d-2\}$-algebra structure,
$\mu=\{\mu_{0,n}: \odot^n (X_o[d-1])\rar X_o[d]\}_{n\geq 1}$, on $X_o$;
 the associated
MC element, $\mu$,  of the Lie algebra $(\mbox{Coder}(\odot^\bu (X_o[d-1])), [\ ,\ ])$ makes
the latter into a {\em dg}\, Lie
algebra with the differential $d_\mu:=[\mu,\ ]$;
\item[(iii)] a  $\caL_\infty$-morphism, $F$, from the $\caL_\infty$-algebra $(X_c[d-1], \nu)$ to
the dg Lie algebra $(\mbox{Coder}(\odot^\bu (X_o[d-1])), [\ ,\ ], d_\mu)$,
$$
F=\left\{F_n: \odot^n (X_c[d])   \lon  \mbox{Coder}\left(\odot^\bu (X_o[d-1])[1]\right)            \right\}_{n\geq 1}
$$
such that the composition
$$
 \odot^n (X_c[d])   \stackrel{F_n}{\lon}  \mbox{Coder}(\odot^\bu (X_o[d-1]))[1]  \stackrel{proj}{\lon}
 \Hom(\odot^m (X_o[d-1]), X_o[d-1])[1]
$$
coincides  precisely with $\mu_{n,m}$ for any $n\geq 1$, $m\geq 0$.
\Ei

If $\rho$ is an arbitrary representation of $\cF \cC hains(\overline{C}(\bbH^{d})$ for $d\geq 3$
and $\ga\in X_c$ is an arbitrary Maurer-Cartan
element,
$$
\sum_{n\geq 0}\frac{1}{n!} \nu_n(\ga^{\ot n})=0, \ \ \ |\ga|=d,
$$
of the associated  $\caL_\infty$-algebra $(X_c,\nu_\bu)$, then the element
$$
F(\ga):= \sum_{n\geq 1}\frac{\hbar^n}{n!}F_n(\ga^{\ot n}) \in  \mbox{Coder}(\odot^\bu (X^\hbar_o[d-1])[1],
$$
make the topological  vector space $X_o^\hbar:=X_o[[\hbar]]$ into a
continuous (in general,  non-flat) $\caL_\infty\{d-2\}$-algebra.  We show an explicit and
 non-trivial example of such an
open-closed homotopy Lie algebra below in Corollary {\ref{9: corollary on extension of Schouten}}.
below.


\bip

{\large
\section{\bf Configuration space models for the 2-coloured operad of $L_\infty$-morphisms}
\label{4': Section on Mor(L_infty)}
}

\mip

\subsection{The complex plane models}\label{4': C models for Mor(Linfty)}
The 2-dimensional Lie group, $G_{(2)}'$,
of complex translations, $z\rar z + \nu$, $\nu\in \C$, acts freely on $\Conf_A(\C)$
(see
\S {\ref{3: second FM compact in C}} for notations used in this subsection),
$$
\Ba{ccc}
\Conf_A(\C) \times \C& \lon & \Conf_A(\C)\\
(p=\{z_i\}_{i\in A},\nu) &\lon & p+\nu:= \{z_i+\nu\}_{i\in A}
\Ea
$$
 so that
the quotient
$$
\fC_A(\C):=\Conf_A(\C)/G_{(2)}'
$$
is a $(2\# A -2)$-dimensional manifold; as usual, we abbreviate $\fC_{[n]}(\C)$ to $\fC_{n}(\C)$.
There is a diffeomorphism,
$$
\Ba{rccccc}
\Psi_A: & \fC_A(\C) & \lon & C_A^{st}(\C) & \times & (0,+\infty)\\
  &           p     & \lon & \frac{p-z_c(p)}{|p-z_c(p)|} && |p-z_c(p)|
\Ea
$$
Note that the configuration  $\frac{p-z_c(p)}{|p-z_c(p)|}$ is invariant under $\R^+\ltimes \C$
and hence gives a well-defined point in $C_A^{st}(\C)\simeq C_A(\C)$. For any  non-empty subset
$A\subseteq [n]$ there is a natural map
$$
\Ba{rccc}
\pi_A : & \fC_n(\C) & \lon & \fC_A(\C)\\
        & p=\{z_i\}_{i\in [n]} & \lon & p_A:=\{z_i\}_{i\in A}
\Ea
$$
which forgets all the points labeled by elements of the complement $[n]\setminus A$.

\sip

A {\em topological compactification}, $\widehat{\fC}_n(\C)$, of $\fC_{n}(\C)$
can be defined as the closure of a composition (cf.\ \cite{Me-Auto}),
\Beq\label{4': first compactifn of fC(C)}
\fC_{n}(\C)\stackrel{\prod \pi_A}{\lon} \prod_{A\subseteq [n]\atop \# A\geq 2} \fC_{A}(\C)
\stackrel{\prod \Psi_A}{\lon} \prod_{A\subseteq [n]\atop  \# A\geq 2}
 C_{A}^{st}(\C)\times (0, +\infty) \hook \prod_{A\subseteq [n]\atop  \# A\geq 2}
 \widetilde{C}^{st}_{A}(\C)\times [0, +\infty].
\Eeq
Thus all the limiting points in this compactification
come  from configurations when a group or groups of points move too {\em close}\,  to each other
 within each group (as in the case of $\overline{C}_n(\R)$) and/or a group or
groups of points are moving too {\em far}\, (with respect to the relative Euclidean distances inside each group) away
from each other (cf.\ \S\ref{2'': section}).

The boundary strata in $\widehat{\fC}_{n}(\C)$ are given by the limit values $0$ or $+\infty$
of the parameters $|p_A-z_c(p_A)|$, $A\subseteq [n]$, and it is an easy (and fully analogous to
\S\ref{2'': section})
exercise to find all the codimension 1 boundary strata,

\Beq\label{4':codimension 1 boundary in widehat{C}_n(C)}
\displaystyle
\p \widehat{\fC}_{n}(\C) = \bigsqcup_{A\subseteq [n]\atop \# A\geq 2} \left(\widehat{\fC}_{n - \# A + 1}(\C)\times
 \overline{C}_{\# A}(\C)\right)\
 \bigsqcup_{[n]=B_1\ \sqcup \ldots \sqcup B_k\atop{
 2\leq k\leq n \atop \#B_1,\ldots, \#B_k\geq 1}}\left( \overline{C}_{k}(\C)\times \widehat{\fC}_{\# B_1,0}(\C)\times \ldots\times
 \widehat{\fC}_{\# B_k,0}(\C)
 \right)
\Eeq
where the first summation runs over all possible  subsets, $A$, of $[n]$
 with cardinality  at least two, and the
second summation runs over all possible decompositions of $[n]$ into (at least two)  disjoint
non-empty subsets $B_1, \ldots, B_k$.
Geometrically,  a stratum in the first group of summands  corresponds to  $A$-labeled elements of the set $\{z_1, \ldots, z_n\}$ moving {\em close}\,
to each other, while a stratum in the second group of summands corresponds to $k$ clusters of points (labeled, respectively,
 by disjoint ordered subsets $B_1, \ldots B_k$ of $[n]$) moving {\em far}\,  from each other while keeping relative
  distances within each group $B_i$ finite.

\sip

Note that the faces of the type $C_\bu(\C)$ appear in the natural stratification
of $\widehat{\fC}_{n}(\C)$ in two ways --- as the strata describing collapsing points and as the strata controlling groups of
points at ``infinity" --- and they {\em never intersect}\, in $\widehat{\fC}_{n}(\C)$ (cf.\
\S \ref{2'': section}). For that reason we have to assign to these two groups
of faces  different colours and represent collapsing $\overline{C}_n(\C)$-stratum  by, say,  white corolla
with straight legs as in (\ref{3: Lie_inf corolla}), the $\overline{C}_n(\R)$-stratum at ``infinity" by, say,
   a version of (\ref{3: Lie_inf corolla}) with ``broken" legs,
$
\hspace{-3mm}
\xy
(1,-5)*{\ldots},
(-13,-7)*{_{i_1}},
(-8,-7)*{_{i_2}},
(-3,-7)*{_{i_3}},
(7,-7)*{_{i_{q-1}}},
(13,-7)*{_{i_q}},
 (0,0)*{\circ}="a",
(0,5)*{}="0",
(-12,-5)*{}="b_1",
(-8,-5)*{}="b_2",
(-3,-5)*{}="b_3",
(8,-5)*{}="b_4",
(12,-5)*{}="b_5",
\ar @{--} "a";"0" <0pt>
\ar @{--} "a";"b_2" <0pt>
\ar @{--} "a";"b_3" <0pt>
\ar @{--} "a";"b_1" <0pt>
\ar @{--} "a";"b_4" <0pt>
\ar @{--} "a";"b_5" <0pt>
\endxy
$,
and the face $\fC_{n}(\C)$ by the black corolla
$\hspace{-3mm}
\xy
(1,-5)*{\ldots},
(-13,-7)*{_{i_1}},
(-8,-7)*{_{i_2}},
(-3,-7)*{_{i_3}},
(7,-7)*{_{i_{n-1}}},
(13,-7)*{_{i_n}},
 (0,0)*{\bullet}="a",
(0,5)*{}="0",
(-12,-5)*{}="b_1",
(-8,-5)*{}="b_2",
(-3,-5)*{}="b_3",
(8,-5)*{}="b_4",
(12,-5)*{}="b_5",
\ar @{--} "a";"0" <0pt>
\ar @{-} "a";"b_2" <0pt>
\ar @{-} "a";"b_3" <0pt>
\ar @{-} "a";"b_1" <0pt>
\ar @{-} "a";"b_4" <0pt>
\ar @{-} "a";"b_5" <0pt>
\endxy
$ of degree $2-2n$.

\mip

\subsubsection{\bf Proposition (cf.\ \cite{Me-Auto})}\label{4': Propos on the face complex of Mor(Lie_infty)}
{\em The face complex of the disjoint union
\Beq\label{4: My Lie_infty config topol operad}
\widehat{\fC}(\C):=\overline{C}_\bu(\C)\sqcup \widehat{\fC}_{\bu}(\C)\sqcup
\overline{C}_\bu(\C)
\Eeq
has naturally a structure of a dg free  non-unital 2-coloured operad of transformation type,
$$
\cM or(L_\infty):= \cF ree
\left\langle
\xy
(1,-5)*{\ldots},
(-13,-7)*{_1},
(-8,-7)*{_2},
(-3,-7)*{_3},
(7,-7)*{_{p-1}},
(13,-7)*{_p},
 (0,0)*{\circ}="a",
(0,5)*{}="0",
(-12,-5)*{}="b_1",
(-8,-5)*{}="b_2",
(-3,-5)*{}="b_3",
(8,-5)*{}="b_4",
(12,-5)*{}="b_5",
\ar @{-} "a";"0" <0pt>
\ar @{-} "a";"b_2" <0pt>
\ar @{-} "a";"b_3" <0pt>
\ar @{-} "a";"b_1" <0pt>
\ar @{-} "a";"b_4" <0pt>
\ar @{-} "a";"b_5" <0pt>
\endxy,
\ \ \
\xy
(1,-5)*{\ldots},
(-13,-7)*{_1},
(-8,-7)*{_2},
(-3,-7)*{_3},
(7,-7)*{_{n-1}},
(13,-7)*{_n},
 (0,0)*{\bullet}="a",
(0,5)*{}="0",
(-12,-5)*{}="b_1",
(-8,-5)*{}="b_2",
(-3,-5)*{}="b_3",
(8,-5)*{}="b_4",
(12,-5)*{}="b_5",
\ar @{--} "a";"0" <0pt>
\ar @{-} "a";"b_2" <0pt>
\ar @{-} "a";"b_3" <0pt>
\ar @{-} "a";"b_1" <0pt>
\ar @{-} "a";"b_4" <0pt>
\ar @{-} "a";"b_5" <0pt>
\endxy
\ \ \ ,
\xy
(1,-5)*{\ldots},
(-13,-7)*{_1},
(-8,-7)*{_2},
(-3,-7)*{_3},
(7,-7)*{_{q-1}},
(13,-7)*{_q},
 (0,0)*{\circ}="a",
(0,5)*{}="0",
(-12,-5)*{}="b_1",
(-8,-5)*{}="b_2",
(-3,-5)*{}="b_3",
(8,-5)*{}="b_4",
(12,-5)*{}="b_5",
\ar @{--} "a";"0" <0pt>
\ar @{--} "a";"b_2" <0pt>
\ar @{--} "a";"b_3" <0pt>
\ar @{--} "a";"b_1" <0pt>
\ar @{--} "a";"b_4" <0pt>
\ar @{--} "a";"b_5" <0pt>
\endxy
\right\rangle_{p,q\geq 2, n\geq 1}
$$
equipped with a differential which is given on white corollas of both colours by
formula (\ref{3: Lie_infty differential}) and on  black corollas by the following formula
  \Beqr\label{Ch3: d on black corollas}
\p
\xy
(1,-5)*{\ldots},
(-13,-7)*{_1},
(-8,-7)*{_2},
(-3,-7)*{_3},
(7,-7)*{_{n-1}},
(13,-7)*{_n},
 (0,0)*{\bullet}="a",
(0,5)*{}="0",
(-12,-5)*{}="b_1",
(-8,-5)*{}="b_2",
(-3,-5)*{}="b_3",
(8,-5)*{}="b_4",
(12,-5)*{}="b_5",
\ar @{--} "a";"0" <0pt>
\ar @{-} "a";"b_2" <0pt>
\ar @{-} "a";"b_3" <0pt>
\ar @{-} "a";"b_1" <0pt>
\ar @{-} "a";"b_4" <0pt>
\ar @{-} "a";"b_5" <0pt>
\endxy
&=&
- \sum_{A\varsubsetneq [n]\atop
\# A\geq 2}
\Ba{c}
\begin{xy}
<10mm,0mm>*{\bu},
<10mm,0.8mm>*{};<10mm,5mm>*{}**@{--},
<0mm,-10mm>*{...},
<12mm,-5mm>*{\ldots},
<13mm,-7mm>*{\underbrace{\ \ \ \ \ \ \ \ \ \ \ \ \  }},
<14mm,-10mm>*{_{[n]\setminus A}};
<10.0mm,0mm>*{};<20mm,-5mm>*{}**@{-},
<10.0mm,-0mm>*{};<5mm,-5mm>*{}**@{-},
<10.0mm,-0mm>*{};<8mm,-5mm>*{}**@{-},
<10.0mm,0mm>*{};<0mm,-4.4mm>*{}**@{-},
<10.0mm,0mm>*{};<16.5mm,-5mm>*{}**@{-},
<0mm,-5mm>*{\circ};
<-5mm,-10mm>*{}**@{-},
<-2.7mm,-10mm>*{}**@{-},
<2.7mm,-10mm>*{}**@{-},
<5mm,-10mm>*{}**@{-},
<0mm,-12mm>*{\underbrace{\ \ \ \ \ \ \ \ \ \ }},
<0mm,-15mm>*{_{A}},
\end{xy}
\Ea
\nonumber\\
&& +\ \, \sum_{k=2}^n \sum_{[n]=B_1\sqcup\ldots\sqcup B_k
\atop \inf B_1<\ldots< \inf B_k}
\Ba{c}
\xy
(-15.5,-7)*{...},
(19,-7)*{...},
(7.5,0)*{\ldots},
(-17.8,-12)*{_{B_1}},
(-3.2,-12)*{_{B_2}},
(17.8,-12)*{_{B_k}},
(-1.8,-7)*{...},
%
(-3.2,-9)*{\underbrace{\ \ \ \ \ \  \ \ \ \   }},
%
(-17.8,-9)*{\underbrace{\ \ \ \ \ \  \ \ \ \   }},
%
(16.8,-9)*{\underbrace{\ \ \ \ \ \  \ \ \ \ \  }},
%
 (0,7)*{\circ}="a",
(-14,0)*{\bullet}="b_0",
(-4.5,0)*{\bullet}="b_2",
(14,0)*{\bullet}="b_3",
(0,13)*{}="0",
(1,-7)*{}="c_1",
(-8,-7)*{}="c_2",
(-5,-7)*{}="c_3",
(-22,-7)*{}="d_1",
(-19,-7)*{}="d_2",
(-13,-7)*{}="d_3",
(12,-7)*{}="e_1",
(15,-7)*{}="e_2",
(22,-7)*{}="e_3",
\ar @{--} "a";"0" <0pt>
\ar @{--} "a";"b_0" <0pt>
\ar @{--} "a";"b_2" <0pt>
\ar @{--} "a";"b_3" <0pt>
\ar @{-} "b_2";"c_1" <0pt>
\ar @{-} "b_2";"c_2" <0pt>
\ar @{-} "b_2";"c_3" <0pt>
\ar @{-} "b_0";"d_1" <0pt>
\ar @{-} "b_0";"d_2" <0pt>
\ar @{-} "b_0";"d_3" <0pt>
\ar @{-} "b_3";"e_1" <0pt>
\ar @{-} "b_3";"e_2" <0pt>
\ar @{-} "b_3";"e_3" <0pt>
\endxy
\Ea.
\Eeqr
Representations of this operad in a pair of dg vector spaces, $V_{in}$ and $V_{out}$, is the same
as a triple, $(\mu_{in}, \mu_{out}, F)$, consisting of  $\caL_\infty\{1\}$ structures, $\mu_{in}$ on $V_{in}$
and $\mu_{out}$ on $V_{out}$, and  of a $\caL_\infty\{1\}$ morphism, $F:(V_{in},\mu_{in})\rar (V_{out},\mu_{out})$,
between them.
}

\subsubsection{\bf Example} As $C_2^{st}(\C)=\widetilde{C}_2(\C)=S^1$, the space
 $\widehat{\fC}_2(\C)$ is the closure
of the embedding
$$
\Ba{ccccccc}
\fC_2(\C) & \lon &  S^1 &\times & (0,+\infty) &\hook &  S^1 \times  [0,+\infty]\\
(z_1,z_2) &\lon &   Arg(z_1-z_2) && |z_1-z_2| &&
\Ea
$$
and hence can be identified with the closed cylinder
\Beq\label{5: fC_2(C) cylinder}
\widehat{\fC}_2(\C)=\
\Ba{c}
\xy
(-8,0)*{}="a",
(8,0)*{}="b",
(-8,-25)*{}="a1",
(8,-25)*{}="b1",
(-8,0)*-{};(8,0)*-{};
**\crv{(0,6)}
**\crv{(0,-6)};
\ar @{-} "a";"a1" <0pt>
\ar @{-} "b";"b1" <0pt>
\endxy\vspace{-3mm}\\
\xy
(-8,0)*{},
(8,0)*{},
(-8,0)*-{};(8,0)*-{};
**\crv{(0,6)}
**\crv{(0,-6)};
\endxy
\Ea.
\Eeq

\subsubsection{\bf Smooth (or semialgebraic) structure} The embedding formula (\ref{4': first compactifn of fC(C)})
makes $\widehat{\fC}(\C)$ into an operad in the category of semialgebraic manifolds. We can make it also into an operad
in the category of smooth manifolds with corners using metric trees in almost exactly the same way as in \S
{\ref{2: subsubs smooth atlas on associh}}.

\subsubsection{\bf A second complex space model for $\cM or(L_\infty)$}  In a full analogy to
 \S {\ref{2:  A different smooth structure on fC(R)}}
 we can introduce on $\widehat{\fC}(\C)$ a different smooth structure using a different compactification formula.
For a pair of subspaces $B\subsetneq A\subseteq [n]$ we consider
$$
\Ba{rcccccc}
\pi_{A,B}: & \fC_n(\C) & \lon &  C_B^{st}(\C)&\times & (0,+\infty)\\
& p & \lon &  \frac{p_B-z_c(p_B)}{|p_B-z_c(p_B)|}
&& |p_A-z_c(p_B)|\cdot |p_B-z_c(p_B)|
\Ea
$$
and then define a topological compactification $\widehat{\fC}_\bu(\C)$ as the closure of
the following composition of embeddings,
\Beq\label{4': second compactifn of fC(C)}
\fC_{n}(\C)\stackrel{\Psi_n\prod \pi_{A,B}}{\lon} C_n^{st}(\C)\times (0,+\infty)
\hspace{-3mm}  \prod_{B\subsetneq A\subseteq [n]\atop \# B\geq 2} \hspace{-4mm} C_B^{st}(\C) \times (0,+\infty)
 \hook \widetilde{C}_n^{st}(\C)\times [0,+\infty]\hspace{-3mm}\prod_{B\subsetneq A\subseteq [n]\atop \# B\geq 2}
 \hspace{-4mm}  \widetilde{C}_B^{st}(\C) \times [0,+\infty].
\Eeq
The boundary strata in $\widehat{\fC}_{n}(\C)$ are given by the limit values $0$ or $+\infty$
of the parameters $|p_{[n]}-z_c(p_{[n]})|$ and  $|p_A-z_c(p_B)|\cdot |p_B-z_c(p_B)|$, and the combinatorics of
its face complex is again described
by Proposition {\ref{4': Propos on the face complex of Mor(Lie_infty)}}. However, this compactification has a different
geometric
meaning from the one given by the embedding  formula (\ref{4': first compactifn of fC(C)}); we refer to \S
{\ref{2:  A different smooth structure on fC(R)}} for detailed discussion of the 1-dimensional version of this phenomenon.
Smooth structure on this compactification can be introduced in a complete analogy to \S
{\ref{2:  A different smooth structure on fC(R)}}.

\subsection{Upper half space models for $\cM or(L_\infty)$ \cite{Me-Auto}}
\label{4': Subsec on my compact of C_n,0}  Let $\Conf_{A,0}(\bbH)$ stand for the space of
 injections, $A\hook \bbH$, of a finite set $A$ into the upper half-plane, and  $\widetilde{\Conf}_{A,0}(\bbH)$
for the space of  maps, $A\rar \bbH$. In this section we remind a compactification, $\widehat{C}_{A,0}(\bbH)$, of
Kontsevich's configuration space,
$$
C_{A,0}(\bbH)=\Conf_{A,0}(\bbH)/G_{(2)},\ \ \ \ \# A\geq 1,
$$
which is different from Kontsevich's one and which
 gives us an upper half space model for the 2-coloured operad of homotopy morphisms of $\caL_\infty$-algebras.
It is worth noting that  the group $G_{(2)}'$ used earlier to construct a complex space model, $\widehat{\fC}_\bu(\C)$,
for  $\cM or(\caL_\infty)$ is obtained from the group $G_{(3)}=\R^+\ltimes \C$ (which was used
to construct a configuration
space model for the operad of $\caL_\infty$-algebras) by taking away {\em dilations}\, $\R^+$, while this
time we use points in the upper half plane together with the group $G_{(2)}$ obtained  from
$G_{(3)}$ by taking away the semigroup of {\em vertical translations}\, $\R^+$.

\sip

Define a section,
$$
\Ba{rccc}
s: &C_{A,0}(\bbH) & \lon & \Conf_{A,0}(\bbH)\\
\displaystyle
&\displaystyle p={\{z_i=x_i+\ii y_i\in \bbH\}_{i\in A}} & \lon &
\displaystyle p^{st}:=  \frac{p - x_c(p)}{\inf_{i\in A}y_i}.
\Ea
$$
where $x_c(p):=\sum_{i=1}^{\# A}\frac{1}{\# A}x_{i}$, and set $C_{A,0}^{st}(\bbH):=\Img s$.  Note that every point in the configuration $p^{st}$ lies
in the subspace $\Im z \geq 1 \subset \bbH$ and at least one point lies on the line $\Im z=1$.
Thus
$$
C_{A,0}^{st}(\bbH)=\left\{p=\{z_i\}_{i\in A} \in  \Conf_{A,0}(\bbH)\ |\ x_c(p)=0, \ \inf_{i\in A}y_i=1
\right\}.
$$

It is an elementary exercise to check that the subspace $C_A^{st}(\bbH)\subset C_{A,0}^{st}(\bbH)$
consisting of elements $p^{st}$ with $$
|p^{st}-\ii|=1
$$
gives
a global section of the surjective forgetful map  $ \Conf_{A,0}(\bbH)\rar C_A(\C)$
and hence is homeomorphic to $C_A(\C)$. Note that
both spaces $C_{n,0}^{st}(\bbH)$ and $C_n^{st}(\bbH)$ have natural structures of smooth manifolds with corners
(and also of semialgebraic sets); for example,
$$
C_2^{st}(\bbH)=
\xy
(-9,0)*-{};(8.5,0)*-{};
**\crv{(0,6)}
**\crv{(0,-6)}
\endxy
$$
rather than an ordinary smooth circle $S^1$.
Thus
$$
C_A^{st}(\bbH)=\left\{p=\{z_i\}_{i\in A} \in  \Conf_{A,0}(\bbH)\ |\ x_c(p)=0, \ \inf_{i\in A}y_i=1,\ |p-\ii|=1
\right\}.
$$
We also define
$$
\widetilde{C}_A^{st}(\bbH)=\left\{p=\{z_i\}_{i\in A} \in  \widetilde{\Conf}_{A,0}(\bbH)\ |\ x_c(p)=0, \ \inf_{i\in A}y_i=1,\ |p-\ii|=1
\right\}.
$$
which is a {\em compact}\, manifold with corners.
 There is a homeomorphism,
\Beq\label{5': X_n iso of C(H)}
\Ba{rccccc}
\Xi_n: & C_{n,0}(\bbH) & \lon &  C_{n}^{st}(\bbH)\simeq C_n(\C) &\times & (0,+\infty) \\
& p & \lon & \frac{p^{st}-\ii }{|p^{st}-\ii|} +\ii &\times & |p^{st}-\ii|.
\Ea
\Eeq
Let, for a subset $A\subset [n]$,
$$
\Ba{rccc}
\pi_A: &  C_{n,0}(\bbH) & \lon & C_{A,0}(\bbH)\\
       &  p=\{z_i\}_{i\in [n]} & \lon &  p_A=\{z_i\}_{i\in A}
\Ea
$$
stand for the natural forgetful map.
For a pair of subspaces $B\subsetneq A\subseteq [n]$ we  consider a map
$$
\Ba{rcccccc}
\Xi_{A,B}: & C_{n,0}(\bbH) & \lon &  C_B^{st}(\bbH)&\times & (0,+\infty)\\
& p & \lon &  \frac{p_B- z_{min}(p_B)}{|p_B-z_{min}(p_B)|} +\ii
&& ||p_{A,B}||:=  \frac{|p_B-z_{min}(p_B)|}{y_{min}(p_A)}.
\Ea
$$

Depending on application needs,
a topological {\em compactification}, $\widehat{C}_{n,0}(\bbH)$, of $C_{n,0}(\bbH)$
can be defined either as the closure of a composition
(cf.\ (\ref{2: first compactifn of fC(R)}) and (\ref{4': first compactifn of fC(C)})),
\Beq\label{4': first compactifn of fC(H)}
C_{n,0}(\bbH)\stackrel{\prod \pi_A}{\lon} \prod_{A\subseteq [n]\atop A\neq \emptyset} C_{A,0}(\bbH)
\stackrel{\prod \Xi_A}{\lon} \prod_{A\subseteq [n]\atop A\neq \emptyset}
 C_{A}^{st}(\bbH)\times (0, +\infty) \lon \prod_{A\subseteq [n]\atop A\neq \emptyset}
 \widetilde{C}^{st}_{A}(\bbH)\times [0, +\infty].
\Eeq
or as the closure on the following embedding,
\Beq\label{4': second compactifn of fC(H)}
C_{n,0}(\bbH)\stackrel{\Psi_n\prod \Xi_{A,B}}{\lon} C_n^{st}(\bbH)\times (0,+\infty)
\hspace{-3mm}  \prod_{B\subsetneq A\subseteq [n]\atop \# B\geq 2} \hspace{-4mm} C_B^{st}(\bbH) \times (0,+\infty)
 \hook \widetilde{C}_n^{st}(\bbH)\times [0,+\infty]\hspace{-3mm}\prod_{B\subsetneq A\subseteq [n]\atop \# B\geq 2}
 \hspace{-4mm}  \widetilde{C}_B^{st}(\bbH) \times [0,+\infty].
\Eeq
The boundary strata in both  cases are given   by the limit values $0$ or $+\infty$
of the parameters $|p^{st}-\ii|$ and  $|p^{st}_{A}-\ii|$ (respectively, $|p^{st}-\ii|$ and $||p_{A,B}||$).
It is not hard to see that the combinatorics
of the
face complex of $\widehat{C}_{\bu,0}(\bbH)$ is the same as in the case of $\widehat{\fC}_\bu(\C)$
so that  $\widehat{C}_{\bu,0}(\bbH)$ gives us  a configuration space model for the 2-coloured operad of
$\caL_\infty$-algebras and their homotopy morphisms \cite{Me-Auto}. However, the geometric meaning and the natural smooth structure
on $\widehat{\fC}_\bu(\bbH)$ are  not equivalent to the ones studied above. Note that contrary to the Kontsevich
compactification, $\overline{C}_{n,0}(\bbH)$, of $C_{n,0}(\bbH)$ limit configurations in  $\widehat{C}_{n,0}(\bbH)$
never approach the real line in $\overline{\bbH}$. From now on we use the symbol $\widehat{C}_{n,0}(\bbH)$
to denote the closure of the embedding (\ref{4': second compactifn of fC(H)}).

\sip

\subsubsection{\bf Smooth atlas on $\widehat{C}_{\bu,0}(\bbH)$}\label{4: Smooth atals on Mor(Lie_infty)}
An atlas on the topological operad
$$
\widehat{C}(\bbH):=
\overline{C}_\bu(\C)\sqcup \widehat{C}_{\bu,0}(\bbH)\sqcup
\overline{C}_\bu(\C)
$$
 can be constructed with the help of exactly the same kind of metric trees as the  ones used in
Section~{\ref{2:  A different smooth structure on fC(R)}} (see also \S{\ref{2: Remark on metric trees}}):
\Bi
\item
 the spaces of standard positions associated with white $n$-corollas of both colours are set to be $C_n^{st}(\bbH)$,
 and the space of standard positions associated with the black $n$-corolla is  $C_{n,0}^{st}(\bbH)$;
\item
the rescaling operation is defined on
$C_n^{st}(\bbH)$ and $C_{n,0}^{st}(\bbH)$ by the map,
$z\rar \la(z-i) +i$;
\item for a point $z_0=x_0+\ii y_0$ the associated translation map $T_{z_0}: \Conf_{n,0}(\bbH) \rar  \Conf_{n,0}(\bbH)$
is defined to be $p\rar p + z_0$.
\Ei
This atlas makes $\widehat{C}(\bbH)$ into an operad
in the category of smooth manifolds with corners. For example, the space $\widehat{C}_{2,0}(\bbH)$ is the closure
of an embedding,
$$
\Ba{rcc}
C_{2,0}(\bbH) & \lon & C_2^{st}(\bbH) \times [0, +\infty]\\
\Ea
$$
 and hence is diffeomorphic to the following manifold with corners
$$
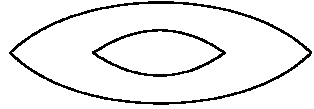
{
\xy
(-9,0)*-{};(8.5,0)*-{};
**\crv{(0,6)}
**\crv{(0,-6)}
   \ar@{-}@(ur,lu) (-20.0,0.0)*{};(20.0,0.0)*{},
   \ar@{-}@(dr,ld) (-20.0,0.0)*{};(20.0,0.0)*{},
\endxy}
$$
whose inner topological circle represents the boundary component, $C_{1,0}(\bbH)\times C_2^{st}(\bbH)$,
describing two point moving very close to each other while the outer topological circle describes the
 boundary component describing two points moving very far ---
in the Euclidean or Poincar\'{e} metric --- from each other.

\subsubsection{\bf Remark} We can define a slightly different smooth structure on  $\widehat{C}_\bu(\bbH)$
by associating $C_n^{st}(C)$  to white corollas with solid legs and
$C_n^{st}(\bbH)$ to white corollas with broken legs. Then the rescaling operation on  $C_n^{st}(\C)$ has to be defined
as an ordinary dilation, $z\rar \lambda z$. In this smooth structure on $\widehat{\fC}_\bu(\bbH)$ the space
 $\widehat{C}_{2,0}(\bbH)$ is precisely the ``Kontsevich eye",
$$
\widehat{C}_{2,0}(\bbH)\ = \
{
\xy
 (0,0)*{
\xycircle(5,5){}};
(-20,0)*{\bullet},
(20,0)*{\bullet},
(-9,0)*+{};(8.5,0)*+{};
%
   \ar@{-}@(ur,lu) (-20.0,0.0)*{};(20.0,0.0)*{},
   \ar@{-}@(dr,ld) (-20.0,0.0)*{};(20.0,0.0)*{},
\endxy}
$$
There is no big difference between these two smooth structures as in both cases the Kontsevich propagator
 $\om_K=d Arg\frac{z_1-z_2}{\overline{z}_1-z_2}$ is a smooth differential 1-form on  $\widehat{C}_{2,0}(\bbH)$.

\subsubsection{\bf Higher dimensional versions} In a full analogy  one can define a compactification,
$\widehat{C}_{n,0}(\bbH^k)$, of the orbit space $C_{n,0}(\bbH^k)$ for any $k\geq 2$
and check that the face complex of the disjoint union
$$
\widehat{C}(\bbH^d):=
\overline{C}_\bu(\R^{d})\sqcup \widehat{C}_{\bu,0}(\bbH^d)\sqcup
\overline{C}_\bu(\R^{d})$$
 is a dg free 2-coloured operad of morphisms of $\caL_{\infty}\{k\}$-algebras.
In fact, we can talk about this family of operads in the range $k\geq 1$:
the case $k=1$ gives us the 2-coloured operad of $\cA_\infty$-morphisms, and all the other cases
give us the (degree shifted) 2-coloured operad of $\caL_\infty$-morphisms; the topological reason for this phenomenon
 is clear.


\bip

{\large
\section{\bf Configuration space model for the 4-coloured operad
of OCHA morphsisms}
\label{6: Section}
}

\mip

\subsection{New compactified configuration spaces $\widehat{\fC}_{n,m}(\bbH)$} Let us define,
for $2n+m\geq 1$,
 $$
\fC_{n,m}(\bbH):=\Conf_{n,m}(\bbH)/{G_{(1)}},
$$
where the Lie group $G_{(1)}=\R$ acts on $\overline{\bbH}$ by translations,
$$
G_{(1)}= \{z\rar  z+ \nu\ |\ \nu\in \R\}.
$$
This is a $(2n+m-1)$-dimensional naturally oriented manifold which is isomorphic to $C_{n,m}(\bbH)\times \R^+$
via the following map,
$$
\Ba{rccc}
\Phi_{n,m}: & \fC_{n,m}(\bbH) & \lon &  C_{n,m}(\bbH)\times \R^+ \\
& p & \lon & \frac{p-x_c(p)}{|p-x_c(p)|} \times |p-x_c(p)|.
\Ea
$$
Note that the fraction  $(p-x_c(p))/|p-x_c(p)|$ is
$G_{(2)}$-invariant and hence gives a well-defined element in  $C_{n,m}(\bbH)$.
Recall that
$$
C_{n,m}^{st}(\bbH)=\left\{p\in \Conf_{n,m}(\bbH)\ |\ x_c(p)=0,\ \ |p|=1       \right\}
$$
gives a section of the natural projection $\Conf_{n,m}(\bbH)\rar C_{n,m}(\bbH)$. We also consider
$$
\widetilde{C}_{n,m}^{st}(\bbH)=\left\{p\in \widetilde{\Conf}_{n,m}(\bbH)\ |\ x_c(p)=0,\ \ |p|=1
   \right\}
$$
which is a compact manifold with boundary.

\sip

For a pair of subsets $A\subset [n]$ and ${B}\subset [m]$, let
$$
\Ba{rccc}
\pi_{A,{B}}: &  \fC_{n,m}(\bbH) & \lon & \fC_{A,{B}}(\bbH)\\
       &  p=\{z_i, x_j\}_{i\in [n], j\in [m]} & \lon &  p_{A, {B}}
       :=\{z_i, x_j\}_{i\in A, j\in {B}}
\Ea
$$
be the forgetful map. We also consider a map
$$
\Ba{rccccc}
\Xi_{A,0}: &  \fC_{A,B}(\bbH) & \stackrel{\pi_{A\sqcup \emptyset}}{\lon}
& \fC_{A,{0}}(\bbH) &\lon &  C_A(\C) \times (0+\infty)  \\
       &  p=\{z_i, x_j\}_{i\in [A], j\in [B]} & \lon &  p_{A}
       =\{z_i\}_{i\in A} &\lon & \displaystyle \left(\frac{p_A-z_c(p_A)}{|p_A-z_c(p_A)|}, |p_A-z_c(p_A)|\right)
\Ea
$$
where $z_c(p):=\frac{1}{\# A} \sum_{i\in A} z_i$. Note that the fraction  $(p_A-z_c(p_A))/|p_A-z_c(p_A)|$ is
$G_{(3)}$-invariant and hence gives a well-defined element in  $C^{st}_{A}(\C)\simeq C_A(\C)$.


\subsubsection{\bf Definition}
A topological {\em compactification}, $\widehat{\fC}_{n,m}(\bbH)$, of $\fC_{n,m}(\bbH)$
can be defined as the closure of a composition
(cf.\ (\ref{2: first compactifn of fC(R)}) and (\ref{4': first compactifn of fC(C)}))
\Beq\label{4: first compctfn fC_n,m(H)}
\Ba{ccccc}
\fC_{n,m}(\bbH)\hspace{-1mm} & \xrightarrow{\prod \pi_{A,B }} & \hspace{-2mm}
\displaystyle\prod_{A\subseteq [n], B\subset [m]\atop \# 2A+\# B\geq 1}
\fC_{A,B}(\bbH) \hspace{-2mm} &
\xrightarrow{\prod \Phi_{A, B}\times \prod \Xi_{A,0}} & \hspace{-2mm}
\displaystyle\prod_{A\subseteq [n], B\subseteq [m]\atop \# 2A+\# B\geq 1}
 \left(C^{st}_{A,B}(\bbH)\times \R^+ \right)\times
 \prod_{A\subseteq [n] \atop\# A\geq 2}\left(  C^{st}_A(\C)\times \R^+\right)\\
&&&&\Big\downarrow\\
&&&&
\displaystyle\prod_{A\subseteq [n], B\subseteq [m]\atop \# 2A+\# B\geq 1}\hspace{-3mm}
 \left(\widetilde{C}^{st}_{A,B}(\bbH)\times \overline{\R^+} \right)\times
 \prod_{A\subseteq [n] \atop\# A\geq 2}\left(  \widetilde{C}^{st}_A(\C)\times \overline{\R^+}\right)
\Ea
\Eeq

The boundary strata in both definitions  are  given   by the limit values, $0$ or $+\infty$,
of the parameters,
$
\left\{|p_{A,B}-x_c(p_{A,B})|, \ \ |p_A-z_c(p_A)|\right\}_{A\subseteq [n], B\subseteq [m]}
$
so that all the limiting points in this compactification
come  from configurations when a group or groups of points move too {\em close}\,  to each other
 within each group and/or a group or
groups of points are moving too {\em far}\, (with respect to the relative Euclidean distances
inside each group) away
from each other.
It is obvious that the family $\widehat{\fC}_{0,\bu}(\bbH)$ is precisely the family of
compactified configuration spaces
defined in \S \ref{2'': section}  and hence describes the 2-coloured operad of $A_\infty$-morphisms.
We claim that $\widehat{\fC}_{\bu,\bu}(\bbH)$ unifies this 2-coloured operad with the 2-coloured
operad $\widehat{C}_{\bu,0}(\bbH)$ describing $\cM or(L_\infty)$ into a 4-coloured operad with an
expected meaning ---
it gives a geometric model for the operad, $\cM or(\f\cC_\infty)$, of homotopy morphisms of OCHA algebras introduced
in the analytic form in \cite{KS}. Let us first give a precise definition of  $\cM or(\f\cC_\infty)$ and then prove the claim.

\subsubsection{\bf Morphisms of OCHA algebras \cite{KS}} The 4-coloured operad, $\cM or(\f\cC_\infty)$, is a dg free
operad generated by two copies,
$$\label{4: first copy of OCHA}
\left\langle
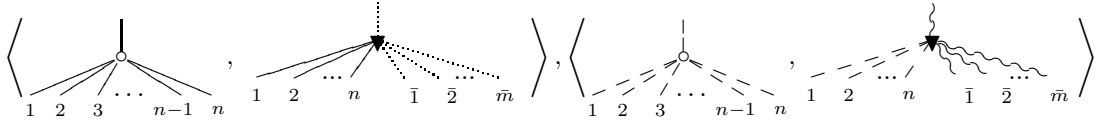
\xy
(1,-5)*{\ldots},
(-12,-7)*{_1},
(-8,-7)*{_2},
(-3,-7)*{_3},
(7,-7)*{_{n-1}},
(13,-7)*{_n},
 (0,0)*{\circ}="a",
(0,5)*{}="0",
(-12,-5)*{}="b_1",
(-8,-5)*{}="b_2",
(-3,-5)*{}="b_3",
(8,-5)*{}="b_4",
(12,-5)*{}="b_5",
\ar @{-} "a";"0" <0pt>
\ar @{-} "a";"b_2" <0pt>
\ar @{-} "a";"b_3" <0pt>
\ar @{-} "a";"b_1" <0pt>
\ar @{-} "a";"b_4" <0pt>
\ar @{-} "a";"b_5" <0pt>
\endxy,
\Ba{c}
\begin{xy}
 <0mm,-0.5mm>*{\blacktriangledown};
 <0mm,0mm>*{};<0mm,5mm>*{}**@{.},
 <0mm,0mm>*{};<-16mm,-5mm>*{}**@{-},
 <0mm,0mm>*{};<-11mm,-5mm>*{}**@{-},
 <0mm,0mm>*{};<-3.5mm,-5mm>*{}**@{-},
 <0mm,0mm>*{};<-6mm,-5mm>*{...}**@{},
   <0mm,0mm>*{};<-16mm,-8mm>*{^{1}}**@{},
   <0mm,0mm>*{};<-11mm,-8mm>*{^{2}}**@{},
   <0mm,0mm>*{};<-3mm,-8mm>*{^{n}}**@{},
 <0mm,0mm>*{};<16mm,-5mm>*{}**@{.},
 <0mm,0mm>*{};<8mm,-5mm>*{}**@{.},
 <0mm,0mm>*{};<3.5mm,-5mm>*{}**@{.},
 <0mm,0mm>*{};<11.6mm,-5mm>*{...}**@{},
   <0mm,0mm>*{};<17mm,-8mm>*{^{\bar{m}}}**@{},
<0mm,0mm>*{};<10mm,-8mm>*{^{\bar{2}}}**@{},
   <0mm,0mm>*{};<5mm,-8mm>*{^{\bar{1}}}**@{},
 \end{xy}
\Ea
\right\rangle
,
\left\langle
\xy
(1,-5)*{\ldots},
(-12,-7)*{_1},
(-8,-7)*{_2},
(-3,-7)*{_3},
(7,-7)*{_{n-1}},
(13,-7)*{_n},
 (0,0)*{\circ}="a",
(0,5)*{}="0",
(-12,-5)*{}="b_1",
(-8,-5)*{}="b_2",
(-3,-5)*{}="b_3",
(8,-5)*{}="b_4",
(12,-5)*{}="b_5",
\ar @{--} "a";"0" <0pt>
\ar @{--} "a";"b_2" <0pt>
\ar @{--} "a";"b_3" <0pt>
\ar @{--} "a";"b_1" <0pt>
\ar @{--} "a";"b_4" <0pt>
\ar @{--} "a";"b_5" <0pt>
\endxy,\hspace{-1mm}
\Ba{c}
\begin{xy}
 <0mm,-0.5mm>*{\blacktriangledown};
 <0mm,0mm>*{};<0mm,5mm>*{}**@{~},
 <0mm,0mm>*{};<-16mm,-5mm>*{}**@{--},
 <0mm,0mm>*{};<-11mm,-5mm>*{}**@{--},
 <0mm,0mm>*{};<-3.5mm,-5mm>*{}**@{--},
 <0mm,0mm>*{};<-6mm,-5mm>*{...}**@{},
   <0mm,0mm>*{};<-16mm,-8mm>*{^{1}}**@{},
   <0mm,0mm>*{};<-11mm,-8mm>*{^{2}}**@{},
   <0mm,0mm>*{};<-3mm,-8mm>*{^{n}}**@{},
 <0mm,0mm>*{};<16mm,-5mm>*{}**@{~},
 <0mm,0mm>*{};<8mm,-5mm>*{}**@{~},
 <0mm,0mm>*{};<3.5mm,-5mm>*{}**@{~},
 <0mm,0mm>*{};<11.6mm,-5mm>*{...}**@{},
   <0mm,0mm>*{};<17mm,-8mm>*{^{\bar{m}}}**@{},
<0mm,0mm>*{};<10mm,-8mm>*{^{\bar{2}}}**@{},
   <0mm,0mm>*{};<5mm,-8mm>*{^{\bar{1}}}**@{},
 \end{xy}
\Ea
\hspace{-1mm}
\right\rangle
$$
of the operad $\f\cC_\infty$, one copy,
$
\left\langle
\xy
(1,-5)*{\ldots},
(-13,-7)*{_1},
(-8,-7)*{_2},
(-3,-7)*{_3},
(7,-7)*{_{n-1}},
(13,-7)*{_n},
 (0,0)*{\bullet}="a",
(0,5)*{}="0",
(-12,-5)*{}="b_1",
(-8,-5)*{}="b_2",
(-3,-5)*{}="b_3",
(8,-5)*{}="b_4",
(12,-5)*{}="b_5",
\ar @{--} "a";"0" <0pt>
\ar @{-} "a";"b_2" <0pt>
\ar @{-} "a";"b_3" <0pt>
\ar @{-} "a";"b_1" <0pt>
\ar @{-} "a";"b_4" <0pt>
\ar @{-} "a";"b_5" <0pt>
\endxy
\right\rangle
$
of the generators of $\cM or(\caL_\infty)$,
 and the following family of $\bS_N$-modules, $N\geq 1$,
$$
\bigoplus_{N=n+m\atop 2n+m\geq 1}\K[\bS_N]\ot_{\bS_n\times \bS_m} \left(\id_n\ot \K[\bS_m]\right)[2n+m-1] =:
\mbox{span}\left\langle
\Ba{c}
\begin{xy}
 <0mm,-0.5mm>*{\blacklozenge};
 <0mm,0mm>*{};<0mm,5mm>*{}**@{~},
 <0mm,0mm>*{};<-16mm,-5mm>*{}**@{-},
 <0mm,0mm>*{};<-11mm,-5mm>*{}**@{-},
 <0mm,0mm>*{};<-3.5mm,-5mm>*{}**@{-},
 <0mm,0mm>*{};<-6mm,-5mm>*{...}**@{},
   <0mm,0mm>*{};<-16mm,-8mm>*{^{i_1}}**@{},
   <0mm,0mm>*{};<-11mm,-8mm>*{^{i_2}}**@{},
   <0mm,0mm>*{};<-3mm,-8mm>*{^{i_n}}**@{},
 <0mm,0mm>*{};<16mm,-5mm>*{}**@{.},
 <0mm,0mm>*{};<8mm,-5mm>*{}**@{.},
 <0mm,0mm>*{};<3.5mm,-5mm>*{}**@{.},
 <0mm,0mm>*{};<11.6mm,-5mm>*{...}**@{},
   <0mm,0mm>*{};<17mm,-8mm>*{^{i_{\bar{m}}}}**@{},
<0mm,0mm>*{};<10mm,-8mm>*{^{i_{\bar{2}}}}**@{},
   <0mm,0mm>*{};<5mm,-8mm>*{^{i_{\bar{1}}}}**@{},
 \end{xy}
\Ea
\right\rangle
$$
where $\id_n$ stands for the trivial representation of $\bS_n$ (implying that the solid input
legs of the  $\blacklozenge$-corolla are ``symmetric" as in (\ref{3: OCHA corolla})). The differential $\p$
is given on $\circ$-, $\blacktriangledown$- and $\bu$-corollas by formulae (\ref{3: Lie_infty differential}),
(\ref{3: differential on Konts corollas}),
and, respectively, (\ref{Ch3: d on black corollas}), and on $\blacklozenge$-corollas
by the following formula,
\Beqr
\p\hspace{-2mm}
\Ba{c}
\begin{xy}
 <0mm,-0.5mm>*{\blacklozenge};
 <0mm,0mm>*{};<0mm,5mm>*{}**@{~},
 <0mm,0mm>*{};<-16mm,-5mm>*{}**@{-},
 <0mm,0mm>*{};<-11mm,-5mm>*{}**@{-},
 <0mm,0mm>*{};<-3.5mm,-5mm>*{}**@{-},
 <0mm,0mm>*{};<-6mm,-5mm>*{...}**@{},
   <0mm,0mm>*{};<-16mm,-8mm>*{^{1}}**@{},
   <0mm,0mm>*{};<-11mm,-8mm>*{^{2}}**@{},
   <0mm,0mm>*{};<-3mm,-8mm>*{^{n}}**@{},
 <0mm,0mm>*{};<16mm,-5mm>*{}**@{.},
 <0mm,0mm>*{};<8mm,-5mm>*{}**@{.},
 <0mm,0mm>*{};<3.5mm,-5mm>*{}**@{.},
 <0mm,0mm>*{};<11.6mm,-5mm>*{...}**@{},
   <0mm,0mm>*{};<17mm,-8mm>*{^{{\bar{m}}}}**@{},
<0mm,0mm>*{};<10mm,-8mm>*{^{{\bar{2}}}}**@{},
   <0mm,0mm>*{};<5mm,-8mm>*{^{{\bar{1}}}}**@{},
 \end{xy}
\Ea\hspace{-3mm}                                &=& -
\sum_{[n]=I_1\sqcup I_2\atop
\# I_1\geq 2, \# I_2\geq 1}
\Ba{c}
\begin{xy}
 <0mm,-0.5mm>*{\blacklozenge};
 <0mm,0mm>*{};<0mm,5mm>*{}**@{~},
 <0mm,0mm>*{};<-16mm,-5mm>*{}**@{-},
 <0mm,0mm>*{};<-11mm,-5mm>*{}**@{-},
 <0mm,0mm>*{};<-3.5mm,-5mm>*{}**@{-},
 <0mm,0mm>*{};<-6mm,-5mm>*{...}**@{},
 <0mm,0mm>*{};<16mm,-5mm>*{}**@{.},
 <0mm,0mm>*{};<8mm,-5mm>*{}**@{.},
 <0mm,0mm>*{};<3.5mm,-5mm>*{}**@{.},
 <0mm,0mm>*{};<11.6mm,-5mm>*{...}**@{},
   <0mm,0mm>*{};<17mm,-8mm>*{^{{\bar{m}}}}**@{},
<0mm,0mm>*{};<10mm,-8mm>*{^{{\bar{2}}}}**@{},
   <0mm,0mm>*{};<5mm,-8mm>*{^{{\bar{1}}}}**@{},
<-16mm,-13mm>*{\underbrace{\ \ \ \ \ \ \ \ }_{I_1}},
<-7mm,-8mm>*{\underbrace{\ \ \ \ \ \ \ \ }_{I_2}},
 (-16,-5)*{\circ}="a",
(-20,-10)*{}="b_1",
(-17,-10)*{}="b_2",
(-15,-9)*{...}="b_3",
(-12,-10)*{}="b_4",
\ar @{-} "a";"b_2" <0pt>
\ar @{-} "a";"b_1" <0pt>
\ar @{-} "a";"b_4" <0pt>
 \end{xy}\Ea                    \nonumber       \\
&-& \sum_{k, l, [n]=I_1\sqcup I_2\atop
{2\# I_1 + m \geq l\atop
2\#I_2 + l\geq 2}}
(-1)^{k+l(n-k-l)}
\Ba{c}
\begin{xy}
 <0mm,-0.5mm>*{\blacklozenge};
 <0mm,0mm>*{};<0mm,5mm>*{}**@{~~},
 <0mm,0mm>*{};<-16mm,-6mm>*{}**@{-},
 <0mm,0mm>*{};<-11mm,-6mm>*{}**@{-},
 <0mm,0mm>*{};<-3.5mm,-6mm>*{}**@{-},
 <0mm,0mm>*{};<-6mm,-6mm>*{...}**@{},
<0mm,0mm>*{};<2mm,-9mm>*{^{\bar{1}}}**@{},
<0mm,0mm>*{};<6mm,-9mm>*{^{\bar{k}}}**@{},
<0mm,0mm>*{};<19mm,-9mm>*{^{\overline{k+l+1}}}**@{},
<0mm,0mm>*{};<28mm,-9mm>*{^{\overline{m}}}**@{},
<0mm,0mm>*{};<13mm,-16.6mm>*{^{\overline{k+1}}}**@{},
<0mm,0mm>*{};<20mm,-16.6mm>*{^{\overline{k+l}}}**@{},
 <0mm,0mm>*{};<11mm,-6mm>*{}**@{.},
 <0mm,0mm>*{};<6mm,-6mm>*{}**@{.},
 <0mm,0mm>*{};<2mm,-6mm>*{}**@{.},
 <0mm,0mm>*{};<17mm,-6mm>*{}**@{.},
 <0mm,0mm>*{};<25mm,-6mm>*{}**@{.},
 <0mm,0mm>*{};<4mm,-6mm>*{...}**@{},
<0mm,0mm>*{};<20mm,-6mm>*{...}**@{},
<6.5mm,-16mm>*{\underbrace{\ \ \ \ \   }_{I_2}},
<-10mm,-9mm>*{\underbrace{\ \ \ \ \ \ \ \ \ \ \ \   }_{I_1}},
 (11,-7)*{\blacktriangledown}="a",
(4,-13)*{}="b_1",
(9,-13)*{}="b_2",
(16,-13)*{...},
(7,-13)*{...},
(13,-13)*{}="b_3",
(19,-13)*{}="b_4",
\ar @{-} "a";"b_2" <0pt>
\ar @{.} "a";"b_3" <0pt>
\ar @{-} "a";"b_1" <0pt>
\ar @{.} "a";"b_4" <0pt>
 \end{xy}
\Ea
                                                   \nonumber \\
&+& \hspace{-3mm} \sum_{k,l\geq 0\atop
2k+l\geq 2} \sum_{[n]=I_{1}\sqcup ...\sqcup J_l
\atop m=m_1+...+m_l}\hspace{-4mm}
(-1)^{\sum_{i=1}^l(l-i)(m_i-1)}\hspace{-4mm}
\Ba{c}
\begin{xy}
 <0mm,-0.5mm>*{\blacktriangledown};
 <0mm,0mm>*{};<0mm,5mm>*{}**@{~},
 <0mm,0mm>*{};<-16mm,-6mm>*{}**@{--},
 <0mm,0mm>*{};<-5mm,-6mm>*{}**@{--},
 <-9.9mm,-6mm>*{\ldots}**@{},
 <0mm,0mm>*{};<22mm,-6mm>*{}**@{~},
 <0mm,0mm>*{};<6mm,-6mm>*{}**@{~},
<12mm,-6mm>*{\ldots}**@{},
<-17mm,-15mm>*{\underbrace{\ \ \ \ \ \ }_{I_1}},
<-6mm,-15mm>*{\underbrace{\ \ \ \ \ \ }_{I_k}},
<-11mm,-16mm>*{\ldots},
<2mm,-15mm>*{\underbrace{\ \   }_{J_1}},
<18mm,-15mm>*{\underbrace{\   }_{J_l}},
<8mm,-14.8mm>*{^{\bar{1}}}**@{},
<12.8mm,-14.8mm>*{^{{\bar{m}_1}}}**@{},
<29mm,-14.8mm>*{^{\bar{m}}}**@{},
(-16,-6)*{\bu}="a_1",
(-5,-6)*{\bu}="a_2",
(-20,-12)*{}="b_1",
(-18,-12)*{}="b_2",
(-15,-12)*{...},
(-13,-12)*{}="b_3",
(-9,-12)*{}="b^1",
(-7,-12)*{}="b^2",
(-4.9,-12)*{...},
(-3,-12)*{}="b^3",
(22,-6)*{\blacklozenge}="a_3",
(6,-6)*{\blacklozenge}="a_4",
(28,-12)*{}="c_1",
(24,-12)*{}="c_2",
(26,-12)*{...},
(18,-12)*{...},
(20,-12)*{}="c_3",
(16,-12)*{}="c_4",
(0,-12)*{}="c^1",
(4,-12)*{}="c^2",
(2,-12)*{...},
(10,-12)*{...},
(8,-12)*{}="c^3",
(12,-12)*{}="c^4",
\ar @{-} "a_1";"b_2" <0pt>
\ar @{-} "a_1";"b_3" <0pt>
\ar @{-} "a_1";"b_1" <0pt>
\ar @{-} "a_2";"b^2" <0pt>
\ar @{-} "a_2";"b^3" <0pt>
\ar @{-} "a_2";"b^1" <0pt>
\ar @{.} "a_3";"c_2" <0pt>
\ar @{-} "a_3";"c_3" <0pt>
\ar @{.} "a_3";"c_1" <0pt>
\ar @{-} "a_3";"c_4" <0pt>
\ar @{-} "a_4";"c^2" <0pt>
\ar @{.} "a_4";"c^3" <0pt>
\ar @{-} "a_4";"c^1" <0pt>
\ar @{.} "a_4";"c^4" <0pt>
 \end{xy}
\Ea
\label{4: d on Mor(OCHA) corollas}
\Eeqr
Representations of this 4-coloured operad in a 4-tuple of dg vector spaces, $V^{in}_c$, $V^{in}_o$, $V_c^{out}$
 and $V_o^{out}$, is the same
as a pair,  $(V^{in}_c, V^{in}_0)$ and $(V^{out}_c, V^{out}_0)$, of homotopy open-closed algebras,
and a homotopy morphism, $F: (V^{in}_c, V^{in}_0)\rar (V^{out}_c, V^{out}_o)$, between them as defined in \cite{KS}.

\subsection{Theorem on the face complex of  $\widehat{\fC}_{\bu,\bu}$}
\label{4: Theorem on Mor(CO_infty) operad}
 {\em The face complex of the disjoint union,
\Beq\label{4: Mor(CO_infty) config spaces topol operad}
\widehat{\fC}(\bbH):=
\underbrace{\overline{C}_\bu(\C)\bigsqcup \overline{C}_{\bu,\bu}(\bbH)}_{in}
\bigsqcup \widehat{\fC}_{\bu,\bu}(\bbH) \bigsqcup \widehat{\fC}_{\bu}(\C)\bigsqcup \underbrace{
\overline{C}_\bu(\C)\bigsqcup \overline{C}_{\bu,\bu}(\bbH)}_{out}
\Eeq
has a natural structure of a dg free 4-coloured operad canonically isomorphic to the operad
$\cM or(\f\cC_\infty)$.
The canonical isomorphism is given by the following identifications,
$$
\underbrace{\left\langle\overline{C}_n(\C),\ \ \overline{C}_{n,m}(\bbH)\right\rangle}_{in} =
\left\langle
\xy
(1,-5)*{\ldots},
(-12,-7)*{_1},
(-8,-7)*{_2},
(-3,-7)*{_3},
(7,-7)*{_{n-1}},
(13,-7)*{_n},
 (0,0)*{\circ}="a",
(0,5)*{}="0",
(-12,-5)*{}="b_1",
(-8,-5)*{}="b_2",
(-3,-5)*{}="b_3",
(8,-5)*{}="b_4",
(12,-5)*{}="b_5",
\ar @{-} "a";"0" <0pt>
\ar @{-} "a";"b_2" <0pt>
\ar @{-} "a";"b_3" <0pt>
\ar @{-} "a";"b_1" <0pt>
\ar @{-} "a";"b_4" <0pt>
\ar @{-} "a";"b_5" <0pt>
\endxy,
\Ba{c}
\begin{xy}
 <0mm,-0.5mm>*{\blacktriangledown};
 <0mm,0mm>*{};<0mm,5mm>*{}**@{.},
 <0mm,0mm>*{};<-16mm,-5mm>*{}**@{-},
 <0mm,0mm>*{};<-11mm,-5mm>*{}**@{-},
 <0mm,0mm>*{};<-3.5mm,-5mm>*{}**@{-},
 <0mm,0mm>*{};<-6mm,-5mm>*{...}**@{},
   <0mm,0mm>*{};<-16mm,-8mm>*{^{1}}**@{},
   <0mm,0mm>*{};<-11mm,-8mm>*{^{2}}**@{},
   <0mm,0mm>*{};<-3mm,-8mm>*{^{n}}**@{},
 <0mm,0mm>*{};<16mm,-5mm>*{}**@{.},
 <0mm,0mm>*{};<8mm,-5mm>*{}**@{.},
 <0mm,0mm>*{};<3.5mm,-5mm>*{}**@{.},
 <0mm,0mm>*{};<11.6mm,-5mm>*{...}**@{},
   <0mm,0mm>*{};<17mm,-8mm>*{^{\bar{m}}}**@{},
<0mm,0mm>*{};<10mm,-8mm>*{^{\bar{2}}}**@{},
   <0mm,0mm>*{};<5mm,-8mm>*{^{\bar{1}}}**@{},
 \end{xy}
\Ea
\right\rangle
$$
$$
\underbrace{\left\langle\overline{C}_n(\C),\ \ \overline{C}_{n,m}(\bbH)\right\rangle}_{out} =
\left\langle
\xy
(1,-5)*{\ldots},
(-12,-7)*{_1},
(-8,-7)*{_2},
(-3,-7)*{_3},
(7,-7)*{_{n-1}},
(13,-7)*{_n},
 (0,0)*{\circ}="a",
(0,5)*{}="0",
(-12,-5)*{}="b_1",
(-8,-5)*{}="b_2",
(-3,-5)*{}="b_3",
(8,-5)*{}="b_4",
(12,-5)*{}="b_5",
\ar @{--} "a";"0" <0pt>
\ar @{--} "a";"b_2" <0pt>
\ar @{--} "a";"b_3" <0pt>
\ar @{--} "a";"b_1" <0pt>
\ar @{--} "a";"b_4" <0pt>
\ar @{--} "a";"b_5" <0pt>
\endxy,\hspace{-1mm}
\Ba{c}
\begin{xy}
 <0mm,-0.5mm>*{\blacktriangledown};
 <0mm,0mm>*{};<0mm,5mm>*{}**@{~},
 <0mm,0mm>*{};<-16mm,-5mm>*{}**@{--},
 <0mm,0mm>*{};<-11mm,-5mm>*{}**@{--},
 <0mm,0mm>*{};<-3.5mm,-5mm>*{}**@{--},
 <0mm,0mm>*{};<-6mm,-5mm>*{...}**@{},
   <0mm,0mm>*{};<-16mm,-8mm>*{^{1}}**@{},
   <0mm,0mm>*{};<-11mm,-8mm>*{^{2}}**@{},
   <0mm,0mm>*{};<-3mm,-8mm>*{^{n}}**@{},
 <0mm,0mm>*{};<16mm,-5mm>*{}**@{~},
 <0mm,0mm>*{};<8mm,-5mm>*{}**@{~},
 <0mm,0mm>*{};<3.5mm,-5mm>*{}**@{~},
 <0mm,0mm>*{};<11.6mm,-5mm>*{...}**@{},
   <0mm,0mm>*{};<17mm,-8mm>*{^{\bar{m}}}**@{},
<0mm,0mm>*{};<10mm,-8mm>*{^{\bar{2}}}**@{},
   <0mm,0mm>*{};<5mm,-8mm>*{^{\bar{1}}}**@{},
 \end{xy}
\Ea
\hspace{-1mm}
\right\rangle
$$

$$
\widehat{C}_{n,0}(\bbH)=
\Ba{c}
\xy
(1,-5)*{\ldots},
(-13,-7)*{_1},
(-8,-7)*{_2},
(-3,-7)*{_3},
(7,-7)*{_{n-1}},
(13,-7)*{_n},
 (0,0)*{\bullet}="a",
(0,5)*{}="0",
(-12,-5)*{}="b_1",
(-8,-5)*{}="b_2",
(-3,-5)*{}="b_3",
(8,-5)*{}="b_4",
(12,-5)*{}="b_5",
\ar @{--} "a";"0" <0pt>
\ar @{-} "a";"b_2" <0pt>
\ar @{-} "a";"b_3" <0pt>
\ar @{-} "a";"b_1" <0pt>
\ar @{-} "a";"b_4" <0pt>
\ar @{-} "a";"b_5" <0pt>
\endxy
\Ea, \ \ \ \ \
\widehat{\fC}_{n,m}=
\Ba{c}
\xy
 <0mm,-0.5mm>*{\blacklozenge};
 <0mm,0mm>*{};<0mm,5mm>*{}**@{~},
 <0mm,0mm>*{};<-16mm,-5mm>*{}**@{-},
 <0mm,0mm>*{};<-11mm,-5mm>*{}**@{-},
 <0mm,0mm>*{};<-3.5mm,-5mm>*{}**@{-},
 <0mm,0mm>*{};<-6mm,-5mm>*{...}**@{},
   <0mm,0mm>*{};<-16mm,-8mm>*{^{1}}**@{},
   <0mm,0mm>*{};<-11mm,-8mm>*{^{2}}**@{},
   <0mm,0mm>*{};<-3mm,-8mm>*{^{n}}**@{},
 <0mm,0mm>*{};<16mm,-5mm>*{}**@{.},
 <0mm,0mm>*{};<8mm,-5mm>*{}**@{.},
 <0mm,0mm>*{};<3.5mm,-5mm>*{}**@{.},
 <0mm,0mm>*{};<11.6mm,-5mm>*{...}**@{},
   <0mm,0mm>*{};<17mm,-8mm>*{^{{\bar{m}}}}**@{},
<0mm,0mm>*{};<10mm,-8mm>*{^{{\bar{2}}}}**@{},
   <0mm,0mm>*{};<5mm,-8mm>*{^{{\bar{1}}}}**@{},
 \endxy
\Ea
$$
}
\begin{proof}  The codimension 1 boundary strata in $\widehat{\fC}_{n,m}(\bbH)$ are given by
the limit values, $0$ and $+\infty$, of the parameters
$$
\left\{||p_{A,B}||:=|p_{A,B}-x_c(p_{A,B})|,\ \ \ \
 ||p_A||_0:=|p_{A}-z_{c}(p_{A})| \right\}_{A\subset [n]
,  B\subseteq [m]}.
$$

(i) The limit configurations, $p\in \widehat{\fC}_{n,m}(\bbH)$, filling in the boundary stratum
$\Ba{c}
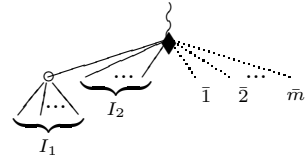
\begin{xy}
 <0mm,-0.5mm>*{\blacklozenge};
 <0mm,0mm>*{};<0mm,5mm>*{}**@{~},
 <0mm,0mm>*{};<-16mm,-5mm>*{}**@{-},
 <0mm,0mm>*{};<-11mm,-5mm>*{}**@{-},
 <0mm,0mm>*{};<-3.5mm,-5mm>*{}**@{-},
 <0mm,0mm>*{};<-6mm,-5mm>*{...}**@{},
 <0mm,0mm>*{};<16mm,-5mm>*{}**@{.},
 <0mm,0mm>*{};<8mm,-5mm>*{}**@{.},
 <0mm,0mm>*{};<3.5mm,-5mm>*{}**@{.},
 <0mm,0mm>*{};<11.6mm,-5mm>*{...}**@{},
   <0mm,0mm>*{};<17mm,-8mm>*{^{{\bar{m}}}}**@{},
<0mm,0mm>*{};<10mm,-8mm>*{^{{\bar{2}}}}**@{},
   <0mm,0mm>*{};<5mm,-8mm>*{^{{\bar{1}}}}**@{},
<-16mm,-13mm>*{\underbrace{\ \ \ \ \ \ \ \ }_{I_1}},
<-7mm,-8mm>*{\underbrace{\ \ \ \ \ \ \ \ }_{I_2}},
 (-16,-5)*{\circ}="a",
(-20,-10)*{}="b_1",
(-17,-10)*{}="b_2",
(-15,-9)*{...}="b_3",
(-12,-10)*{}="b_4",
\ar @{-} "a";"b_2" <0pt>
\ar @{-} "a";"b_1" <0pt>
\ar @{-} "a";"b_4" <0pt>
 \end{xy}\Ea$ are given by $||p_{I_1}||_0=0$, $||p||$ is a finite number.
 \sip

 (ii) The limit configurations, $p\in \widehat{\fC}_{n,m}(\bbH)$, filling in the boundary stratum
$\Ba{c}
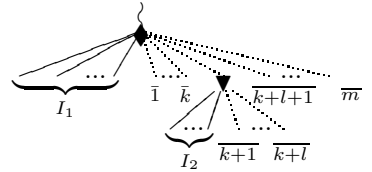
\begin{xy}
 <0mm,-0.5mm>*{\blacklozenge};
 <0mm,0mm>*{};<0mm,5mm>*{}**@{~~},
 <0mm,0mm>*{};<-16mm,-6mm>*{}**@{-},
 <0mm,0mm>*{};<-11mm,-6mm>*{}**@{-},
 <0mm,0mm>*{};<-3.5mm,-6mm>*{}**@{-},
 <0mm,0mm>*{};<-6mm,-6mm>*{...}**@{},
<0mm,0mm>*{};<2mm,-9mm>*{^{\bar{1}}}**@{},
<0mm,0mm>*{};<6mm,-9mm>*{^{\bar{k}}}**@{},
<0mm,0mm>*{};<19mm,-9mm>*{^{\overline{k+l+1}}}**@{},
<0mm,0mm>*{};<28mm,-9mm>*{^{\overline{m}}}**@{},
<0mm,0mm>*{};<13mm,-16.6mm>*{^{\overline{k+1}}}**@{},
<0mm,0mm>*{};<20mm,-16.6mm>*{^{\overline{k+l}}}**@{},
 <0mm,0mm>*{};<11mm,-6mm>*{}**@{.},
 <0mm,0mm>*{};<6mm,-6mm>*{}**@{.},
 <0mm,0mm>*{};<2mm,-6mm>*{}**@{.},
 <0mm,0mm>*{};<17mm,-6mm>*{}**@{.},
 <0mm,0mm>*{};<25mm,-6mm>*{}**@{.},
 <0mm,0mm>*{};<4mm,-6mm>*{...}**@{},
<0mm,0mm>*{};<20mm,-6mm>*{...}**@{},
<6.5mm,-16mm>*{\underbrace{\ \ \ \ \   }_{I_2}},
<-10mm,-9mm>*{\underbrace{\ \ \ \ \ \ \ \ \ \ \ \   }_{I_1}},
 (11,-7)*{\blacktriangledown}="a",
(4,-13)*{}="b_1",
(9,-13)*{}="b_2",
(16,-13)*{...},
(7,-13)*{...},
(13,-13)*{}="b_3",
(19,-13)*{}="b_4",
\ar @{-} "a";"b_2" <0pt>
\ar @{.} "a";"b_3" <0pt>
\ar @{-} "a";"b_1" <0pt>
\ar @{.} "a";"b_4" <0pt>
 \end{xy}
\Ea$
are given by the equation  $||p_{I_2,\{\overline{k+1}\ldots, \overline{k+l}\}}||=0$, $||p||$ is a finite number.

\sip

(iii) The limit configurations, $p\in \widehat{\fC}_{n,m}(\bbH)$, filling in the boundary stratum
$\Ba{c}
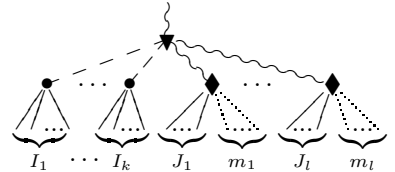
\begin{xy}
 <0mm,-0.5mm>*{\blacktriangledown};
 <0mm,0mm>*{};<0mm,5mm>*{}**@{~},
 <0mm,0mm>*{};<-16mm,-6mm>*{}**@{--},
 <0mm,0mm>*{};<-5mm,-6mm>*{}**@{--},
 <-9.9mm,-6mm>*{\ldots}**@{},
 <0mm,0mm>*{};<22mm,-6mm>*{}**@{~},
 <0mm,0mm>*{};<6mm,-6mm>*{}**@{~},
<12mm,-6mm>*{\ldots}**@{},
<-17mm,-15mm>*{\underbrace{\ \ \ \ \ \ }_{I_1}},
<-6mm,-15mm>*{\underbrace{\ \ \ \ \ \ }_{I_k}},
<-11mm,-16mm>*{\ldots},
<2mm,-15mm>*{\underbrace{\ \   }_{J_1}},
<18mm,-15mm>*{\underbrace{\   }_{J_l}},
<10mm,-15mm>*{\underbrace{\ \   }_{m_1}},
<26mm,-15mm>*{\underbrace{\   }_{m_l}},
(-16,-6)*{\bu}="a_1",
(-5,-6)*{\bu}="a_2",
(-20,-12)*{}="b_1",
(-18,-12)*{}="b_2",
(-15,-12)*{...},
(-13,-12)*{}="b_3",
(-9,-12)*{}="b^1",
(-7,-12)*{}="b^2",
(-4.9,-12)*{...},
(-3,-12)*{}="b^3",
(22,-6)*{\blacklozenge}="a_3",
(6,-6)*{\blacklozenge}="a_4",
(28,-12)*{}="c_1",
(24,-12)*{}="c_2",
(26,-12)*{...},
(18,-12)*{...},
(20,-12)*{}="c_3",
(16,-12)*{}="c_4",
(0,-12)*{}="c^1",
(4,-12)*{}="c^2",
(2,-12)*{...},
(10,-12)*{...},
(8,-12)*{}="c^3",
(12,-12)*{}="c^4",
\ar @{-} "a_1";"b_2" <0pt>
\ar @{-} "a_1";"b_3" <0pt>
\ar @{-} "a_1";"b_1" <0pt>
\ar @{-} "a_2";"b^2" <0pt>
\ar @{-} "a_2";"b^3" <0pt>
\ar @{-} "a_2";"b^1" <0pt>
\ar @{.} "a_3";"c_2" <0pt>
\ar @{-} "a_3";"c_3" <0pt>
\ar @{.} "a_3";"c_1" <0pt>
\ar @{-} "a_3";"c_4" <0pt>
\ar @{-} "a_4";"c^2" <0pt>
\ar @{.} "a_4";"c^3" <0pt>
\ar @{-} "a_4";"c^1" <0pt>
\ar @{.} "a_4";"c^4" <0pt>
 \end{xy}
\Ea$
are characterized by the following data: $||p_{[n],[m]}||=+\infty$,  $||p_{I_i}||_0$
is finite for all $i\in [k]$,
 $||p_{J_i, m_j}||$ is finite for all $j\in [l]$, and the image of $p$ under the
 projection $\widehat{\fC}_{n,m}(\bbH)\rar \widetilde{C}^{st}_{n,m}(\bbH)$ consists of $k$ different points
 in the upper-half-plane and $l$ different points on the real line.
  This is the case when $k$ groups
 of points in $\bbH$  parameterized by sets $I_1, \ldots, I_k$, and $l$-groups of points
 in $\overline{\bbH}$
parameterized by sets $J_1\sqcup m_1, \ldots, J_l\sqcup m_l$
are moving far away from each other in such a way that their sizes (measured by the parameters
$||p_{I_i}||_0$ and $||p_{J_i, m_j}||$)  stay finite.
\sip

Finally, it is an elementary calculation to check that all the  boundary strata defined above have codimension $1$, and
these are
the only boundary strata satisfying this condition.
\end{proof}

Let us illustrate the above theorem with several explicit examples.

\sip

\subsubsection{\bf The case $\widehat{\fC}_{0,\bu}(\bbH)$} The configuration spaces
$\{\widehat{\fC}_{0,m}(\bbH)\}_{m\geq 1}$ are precisely Stasheff's multiplihedra
(see \S \ref{2'': section}), and the formula (\ref{4: d on Mor(OCHA) corollas}) indeed reduces in this case
to (\ref{2: differential A_infty morphism}).

\subsubsection{\bf The case $\widehat{\fC}_{1,0}(\bbH)$}  In the case $n=1$, $m=0$ formula
(\ref{4: d on Mor(OCHA) corollas})
gives,
$$
\p
\Ba{c}
\begin{xy}
 <0mm,-0.5mm>*{\blacklozenge};
 <0mm,0mm>*{};<0mm,5mm>*{}**@{~},
 <0mm,0mm>*{};<-1mm,-5mm>*{}**@{-},
 \end{xy}
\Ea
=-\Ba{c}
\begin{xy}
 <0mm,-0.5mm>*{\blacklozenge};
 <0mm,0mm>*{};<0mm,5mm>*{}**@{~},
 <0mm,0mm>*{};<1mm,-5mm>*{}**@{.},
<1mm,-5.5mm>*{\blacktriangledown};
 <1mm,-5.5mm>*{};<0mm,-10mm>*{}**@{-},
 \end{xy}
\Ea
+
\Ba{c}
\begin{xy}
 <0mm,-0.5mm>*{\blacktriangledown};
 <0mm,0mm>*{};<0mm,5mm>*{}**@{~},
 <0mm,0mm>*{};<-1mm,-4.5mm>*{}**@{--},
<-1mm,-5.5mm>*{\bu};
 <-1mm,-5.5mm>*{};<-1mm,-10mm>*{}**@{-},
 \end{xy}
\Ea
$$
On the other hand, $\fC_{1,0}(\bbH)$ is isomorphic to $(0,+\infty)$, the $y$-axis in $\bbH$, and
$\widehat{\fC}_{1,0}(\bbH)$ is  the closure of the embedding
$(0, +\infty)\hook
[0,+\infty]$. Hence $\widehat{\fC}_{1,0}(\bbH)$ is the closed interval $[0,+\infty]$ with the boundary operator
$\p$ coinciding
precisely with the above formula if we use the  identifications of configuration spaces with graphs
given in Theorem~\ref{4: Mor(CO_infty) config spaces topol operad}.

\subsubsection{\bf The case $\widehat{\fC}_{1,1}(\bbH)$}  In the case $n=1$, $m=1$ formula
(\ref{4: d on Mor(OCHA) corollas})
gives,
\Beq\label{4: p for fC_1,1}
\p
\Ba{c}
\begin{xy}
 <0mm,-0.5mm>*{\blacklozenge};
 <0mm,0mm>*{};<0mm,5mm>*{}**@{~},
 <0mm,0mm>*{};<-2mm,-5.5mm>*{}**@{-},
<0mm,0mm>*{};<2mm,-5.5mm>*{}**@{.},
 \end{xy}
\Ea
=-
\underbrace{\Ba{c}
\begin{xy}
 <0mm,-0.5mm>*{\blacklozenge};
 <0mm,0mm>*{};<0mm,5mm>*{}**@{~},
 <0mm,0mm>*{};<1mm,-5mm>*{}**@{.},
<1mm,-5.5mm>*{\blacktriangledown};
 <1mm,-5.5mm>*{};<-1mm,-10.5mm>*{}**@{-},
<1mm,-5.5mm>*{};<3mm,-10.5mm>*{}**@{.},
 \end{xy}
\Ea}_a
+
\underbrace{
\Ba{c}
\begin{xy}
 <0mm,-0.5mm>*{\blacktriangledown};
 <0mm,0mm>*{};<0mm,5mm>*{}**@{~},
 <0mm,0mm>*{};<-2mm,-4.5mm>*{}**@{--},
<0mm,0mm>*{};<3mm,-5.5mm>*{}**@{~},
<-2mm,-5.5mm>*{\bu};
 <-2mm,-5.5mm>*{};<-2mm,-10.5mm>*{}**@{-},
<3mm,-5.5mm>*{\blacklozenge};
 <3mm,-5.5mm>*{};<3mm,-10.5mm>*{}**@{.},
 \end{xy}
\Ea}_b
+
\underbrace{
\Ba{c}
\begin{xy}
 <0mm,-0.5mm>*{\blacktriangledown};
 <0mm,0mm>*{};<0mm,5mm>*{}**@{~},
 <0mm,0mm>*{};<2mm,-5.5mm>*{}**@{~},
<0mm,0mm>*{};<6mm,-5.5mm>*{}**@{~},
<2mm,-5.5mm>*{\blacklozenge};
 <2mm,-5.5mm>*{};<1mm,-10.5mm>*{}**@{-},
<6mm,-5.5mm>*{\blacklozenge};
 <6mm,-5.5mm>*{};<7mm,-10.5mm>*{}**@{.},
 \end{xy}
\Ea}_c
-
\underbrace{
\Ba{c}
\begin{xy}
 <0mm,-0.5mm>*{\blacktriangledown};
 <0mm,0mm>*{};<0mm,5mm>*{}**@{~},
 <0mm,0mm>*{};<2mm,-5.5mm>*{}**@{~},
<0mm,0mm>*{};<6mm,-5.5mm>*{}**@{~},
<2mm,-5.5mm>*{\blacklozenge};
 <2mm,-5.5mm>*{};<3mm,-10.5mm>*{}**@{.},
<6mm,-5.5mm>*{\blacklozenge};
 <6mm,-5.5mm>*{};<5mm,-10.5mm>*{}**@{-},
 \end{xy}
\Ea}_d
\Eeq
On the other hand, the compactifying embedding takes the following explicit form
$$
\Ba{rcccc}
\fC_{1,1}(\bbH) & \lon & \left(\widetilde{C}_{1,1} \times [0,+\infty]\right)\hspace{-3mm} &
\times &\hspace{-3mm} \left(\widetilde{C}_{1,0} \times [0,+\infty]\right)\\
p=(\underbrace{x_1+\ii y_1}_{p_1} , x_2)\hspace{-3mm}& \lon &\hspace{-3mm}
\displaystyle\left( \frac{p-x_c(p)}{||p||}=
 \frac{\frac{1}{2}(x_1-x_2)+\ii y_1, -\frac{1}{2}(x_1-x_2)}{\sqrt{y_1^2 +\frac{1}{2}(x_1-x_2)^2}}, ||p||\right)
\hspace{-3mm}&&\hspace{-3mm} \left(i, y_1\right)
\Ea
$$
In this approach
\Bi
\item the boundary stratum $a=\overline{C}_{1,1}(\bbH)$ is given by the limit configurations
with $||p||\rar 0$;
\item the boundary stratum $b=\overline{C}_{1,1}(\bbH)$ is given by the limit configurations
with $||y_1||\rar +\infty $,
\item the boundary stratum $c$ and $d$ are given by the limit configurations with
$||p||\rar +\infty $ and
$y_1$ finite, i.e.\ as the limit configurations
$(\frac{\la}{2}(x_1-x_2)+ \ii y_1, -\frac{\la}{2}(x_1-x_2))$
when $|\la|\rar +\infty$; the case $\la\rar +\infty$ corresponds to $c$ and the case
$\la\rar -\infty$  to $d$.
\Ei

Each term on the r.h.s.\  of (\ref{4: p for fC_1,1}) stands therefore for a closed interval.
We finally get the following picture,
$$
\overline{\fC}_{1,1}=\Ba{c}
\xy
 (0,0)*\ellipse(6,6),=:a(-180){-};
 (0,0)*\ellipse(12,12),=:a(-180){-};
(0,13.5)*{^{b}},
(0,3.5)*{^{a}},
 (-6,0)*{}="a",
(-12,0)*{}="b",
(-9,-3)*{^{c}},
(9,-3)*{^{d}},
(6,0)*{}="c",
(12,0)*{}="d",
\ar @{-} "a";"b" <0pt>
\ar @{-} "c";"d" <0pt>
\endxy\Ea \simeq \Ba{c}\xy
(0,3)*{^{b}},
(0,-11)*{^{a}},
(-9,-5)*{^{c}},
(9,-5)*{^{d}},
(-7,0)*-{};(7.0,-0)*-{}
**\crv{(0,3.3)};
(-7,-10)*-{};(7.0,-10)*-{}
**\crv{(0,-6.7)};
<7mm,0mm>*{};<7mm,-10mm>*{}**@{-},
<-7mm,0mm>*{};<-7mm,-10mm>*{}**@{-},
\endxy\Ea
$$

\subsubsection{\bf The case $\widehat{\fC}_{2,0}(\bbH)$}  In the case $n=2$, $m=0$ formula
(\ref{4: d on Mor(OCHA) corollas})
gives (see picture (\ref{5: picture for fC(H)_2,0}) below which visualizes   each
summand's contribution into the boundary of $\widehat{\fC}_{2,0}(\bbH)$) ,
\Beqrn
\p
\Ba{c}
\begin{xy}
 <0mm,-0.5mm>*{\blacklozenge};
 <0mm,0mm>*{};<0mm,5mm>*{}**@{~},
 <0mm,0mm>*{};<-5mm,-5.5mm>*{}**@{-},
<0mm,0mm>*{};<-2mm,-5.5mm>*{}**@{-},
<-5mm,-7.5mm>*{_1};
<-2mm,-7.5mm>*{_2};
 \end{xy}
\Ea
&=&-
\underbrace{\Ba{c}
\begin{xy}
 <0mm,-0.5mm>*{\blacklozenge};
 <0mm,0mm>*{};<0mm,5mm>*{}**@{~},
 <0mm,0mm>*{};<1mm,-5mm>*{}**@{.},
<1mm,-5.5mm>*{\blacktriangledown};
 <1mm,-5.5mm>*{};<-1mm,-10.5mm>*{}**@{-},
<1mm,-5.5mm>*{};<-4mm,-10.5mm>*{}**@{-},
<-1mm,-12.5mm>*{_2};
<-4mm,-12.5mm>*{_1};
 \end{xy}
\Ea}_{\Ba{c}\xy
(0,0)*\ellipse(2.9,1){.};
(-8,0)*-{};(8.0,-0)*-{}
**\crv{~*=<4pt>{.}(0,3.3)}
**\crv{(0,-3.3)};
\endxy\Ea}
-
\underbrace{\Ba{c}
\begin{xy}
 <0mm,-0.5mm>*{\blacklozenge};
 <0mm,0mm>*{};<0mm,5mm>*{}**@{~},
 <0mm,0mm>*{};<-1mm,-5mm>*{}**@{-},
<-1mm,-5.3mm>*{\circ};
 <-1mm,-5.5mm>*{};<-3mm,-10.5mm>*{}**@{-},
<-1mm,-5.5mm>*{};<1mm,-10.5mm>*{}**@{-},
<-3mm,-12.5mm>*{_1};
<1mm,-12.5mm>*{_2};
 \end{xy}
\Ea}_{\Ba{c}\xy (0,0)*\ellipse(3,1){.};(0,-5)*\ellipse(3,1){.};
<-3mm,0mm>*{};<-3mm,-10.2mm>*{}**@{.},
<3mm,0mm>*{};<3mm,-10.2mm>*{}**@{.},
\endxy\Ea}
-
\underbrace{
\Ba{c}
\begin{xy}
 <0mm,-0.5mm>*{\blacklozenge};
 <0mm,0mm>*{};<0mm,5mm>*{}**@{~},
 <0mm,0mm>*{};<-3mm,-5.5mm>*{}**@{-},
<0mm,0mm>*{};<4mm,-5.5mm>*{}**@{.},
<4mm,-5.5mm>*{\blacktriangledown};
 <4mm,-5mm>*{};<1mm,-10mm>*{}**@{-},
<-3.5mm,-6.7mm>*{_1};
<0.5mm,-11.7mm>*{_2};
 \end{xy}
\Ea}_{\Ba{c}\xy
(-7,0)*-{};(7.0,-0)*-{}
**\crv{(0,-3.3)};
(-4,-10)*-{};(10,-10)*-{}
**\crv{(0,-13.3)};
<7mm,0mm>*{};<10mm,-10mm>*{}**@{-},
<-7mm,0mm>*{};<-4mm,-10mm>*{}**@{-},
\endxy\Ea}
-
\underbrace{
\Ba{c}
\begin{xy}
 <0mm,-0.5mm>*{\blacklozenge};
 <0mm,0mm>*{};<0mm,5mm>*{}**@{~},
 <0mm,0mm>*{};<-3mm,-5.5mm>*{}**@{-},
<0mm,0mm>*{};<4mm,-5.5mm>*{}**@{.},
<4mm,-5.5mm>*{\blacktriangledown};
 <4mm,-5mm>*{};<1mm,-10mm>*{}**@{-},
<-3.5mm,-6.7mm>*{_2};
<0.5mm,-11.7mm>*{_1};
 \end{xy}
\Ea}_{\Ba{c}\xy
(-7,0)*+{};(7.0,-0)*-{}
**\crv{~*=<3pt>{.}(0,3.3)};
(-10,-10)*-{};(4.0,-10)*-{}
**\crv{~*=<4pt>{.}(0,-6.7)};
<7mm,0mm>*{};<4mm,-10mm>*{}**@{-},
<-7mm,0mm>*{};<-10mm,-10mm>*{}**@{--},
\endxy\Ea}
\\
&+&
\underbrace{
\Ba{c}
\begin{xy}
 <0mm,-0.5mm>*{\blacktriangledown};
 <0mm,0mm>*{};<0mm,5mm>*{}**@{~},
 <0mm,0mm>*{};<-2mm,-4.5mm>*{}**@{--},
 %
<-3mm,-5.5mm>*{\bu};
 <-3mm,-5.5mm>*{};<-5mm,-10.5mm>*{}**@{-},
 <-3mm,-5.5mm>*{};<-1mm,-10.5mm>*{}**@{-},
<-5mm,-12.5mm>*{_1};
<-0.5mm,-12.5mm>*{_2};
 \end{xy}
\Ea}_{\Ba{c}\xy (0,0)*\ellipse(3,1){-};(0,-5)*\ellipse(3,1){.};
<-3mm,0mm>*{};<-3mm,-10mm>*{}**@{--},
<3mm,0mm>*{};<3mm,-10mm>*{}**@{--},
\endxy\Ea}
+
\underbrace{
\Ba{c}
\begin{xy}
 <0mm,-0.5mm>*{\blacktriangledown};
 <0mm,0mm>*{};<0mm,5mm>*{}**@{~},
 <0mm,0mm>*{};<-2mm,-4.5mm>*{}**@{--},
<0mm,0mm>*{};<-6mm,-5.5mm>*{}**@{--},
<-2mm,-5.5mm>*{\bu};
 <-2mm,-5.5mm>*{};<-2mm,-10.5mm>*{}**@{-},
<-6mm,-5.5mm>*{\bu};
 <-6mm,-5.5mm>*{};<-6mm,-10.5mm>*{}**@{-},
<-6mm,-12.5mm>*{_1};
<-2mm,-12.5mm>*{_2};
 \end{xy}
\Ea}_{\Ba{c}\xy
(0,0)*\ellipse(2.9,1){-};
(-8,0)*-{};(8.0,-0)*-{}
**\crv{(0,3.1)}
**\crv{(0,-3.1)};
\endxy\Ea}
+
\underbrace{
\Ba{c}
\begin{xy}
 <0mm,-0.5mm>*{\blacktriangledown};
 <0mm,0mm>*{};<0mm,5mm>*{}**@{~},
 <0mm,0mm>*{};<-2mm,-4.5mm>*{}**@{--},
<0mm,0mm>*{};<3mm,-5.5mm>*{}**@{~},
<-2mm,-5.5mm>*{\bu};
 <-2mm,-5.5mm>*{};<-2mm,-10.5mm>*{}**@{-},
<3mm,-5.5mm>*{\blacklozenge};
 <3mm,-5.5mm>*{};<1mm,-10.5mm>*{}**@{-},
<-2mm,-12.5mm>*{_1};
<1mm,-12.5mm>*{_2};
 \end{xy}
\Ea}_{\Ba{c}\xy
(-7,0)*-{};(7.0,-0)*-{}
**\crv{(0,-3.3)};
(-10,-10)*-{};(4,-10)*-{}
**\crv{(0,-13.3)};
<7mm,0mm>*{};<4mm,-10mm>*{}**@{-},
<-7mm,0mm>*{};<-10mm,-10mm>*{}**@{-},
\endxy\Ea}
+
\underbrace{
\Ba{c}
\begin{xy}
 <0mm,-0.5mm>*{\blacktriangledown};
 <0mm,0mm>*{};<0mm,5mm>*{}**@{~},
 <0mm,0mm>*{};<-2mm,-4.5mm>*{}**@{--},
<0mm,0mm>*{};<3mm,-5.5mm>*{}**@{~},
<-2mm,-5.5mm>*{\bu};
 <-2mm,-5.5mm>*{};<-2mm,-10.5mm>*{}**@{-},
<3mm,-5.5mm>*{\blacklozenge};
 <3mm,-5.5mm>*{};<1mm,-10.5mm>*{}**@{-},
<-2mm,-12.5mm>*{_2};
<1mm,-12.5mm>*{_1};
 \end{xy}
\Ea}_{\Ba{c}\xy
(-7,0)*-{};(7.0,-0)*-{}
**\crv{(0,3.3)};
(10,-10)*-{};(-4.0,-10)*-{}
**\crv{~*=<4pt>{.}(0,-6.7)};
<7mm,0mm>*{};<10mm,-10mm>*{}**@{-},
<-7mm,0mm>*{};<-4mm,-10mm>*{}**@{--},
\endxy\Ea}
+
\underbrace{
\Ba{c}
\begin{xy}
 <0mm,-0.5mm>*{\blacktriangledown};
 <0mm,0mm>*{};<0mm,5mm>*{}**@{~},
 <0mm,0mm>*{};<2mm,-5.5mm>*{}**@{~},
<0mm,0mm>*{};<6mm,-5.5mm>*{}**@{~},
<2mm,-5.5mm>*{\blacklozenge};
 <2mm,-5.5mm>*{};<1mm,-10.5mm>*{}**@{-},
<6mm,-5.5mm>*{\blacklozenge};
 <6mm,-5.5mm>*{};<5mm,-10.5mm>*{}**@{-},
<1mm,-12.5mm>*{_1};
<5mm,-12.5mm>*{_2};
 \end{xy}
\Ea}_{\Ba{c}\xy
<0mm,0mm>*{};<-2.5mm,-6mm>*{}**@{-},
<-2.5mm,-6mm>*{};<0mm,-10mm>*{}**@{-},
<0mm,0mm>*{};<2.5mm,-4mm>*{}**@{--},
<2.5mm,-4mm>*{};<0mm,-10mm>*{}**@{--},
\endxy\Ea}
+
\underbrace{
\Ba{c}
\begin{xy}
 <0mm,-0.5mm>*{\blacktriangledown};
 <0mm,0mm>*{};<0mm,5mm>*{}**@{~},
 <0mm,0mm>*{};<2mm,-5.5mm>*{}**@{~},
<0mm,0mm>*{};<6mm,-5.5mm>*{}**@{~},
<2mm,-5.5mm>*{\blacklozenge};
 <2mm,-5.5mm>*{};<1mm,-10.5mm>*{}**@{-},
<6mm,-5.5mm>*{\blacklozenge};
 <6mm,-5.5mm>*{};<5mm,-10.5mm>*{}**@{-},
<1mm,-12.5mm>*{_2};
<5mm,-12.5mm>*{_1};
 \end{xy}
\Ea}_{\Ba{c}\xy
<0mm,0mm>*{};<-2.5mm,-6mm>*{}**@{-},
<-2.5mm,-6mm>*{};<0mm,-10mm>*{}**@{-},
<0mm,0mm>*{};<2.5mm,-4mm>*{}**@{-},
<2.5mm,-4mm>*{};<0mm,-10mm>*{}**@{-},
\endxy\Ea}
\Eeqrn

On the other hand the compactification formula (\ref{4: first compctfn fC_n,m(H)}) takes
in this case the form,
$$
\Ba{ccccccccc}
\fC_{2,0}(\bbH)& \rar & \hspace{-1mm} \widetilde{C}^{st}_{2,0}(\bbH)\times [0,+\infty]
\hspace{-3mm} &\times &
 \hspace{-3mm}  \widetilde{C}^{st}_{1,0}(\bbH)\times[0,+\infty] \hspace{-3mm}  &\times & \hspace{-3mm}
  \widetilde{C}^{st}_{1,0}(\bbH)\times[0,+\infty]
&\times & \widetilde{C}^{st}_{2}(\C)\times[0,+\infty]
\\
p=(z_1,z_2) \hspace{-3mm}   &      & \hspace{-3mm} (\frac{p-x_c(p)}{||p||}, ||p||)\hspace{-3mm} &&
\hspace{-3mm}
(\ii ,  y_1)\hspace{-3mm} &&\hspace{-3mm} (\ii, y_2)\hspace{-3mm}
&& \hspace{-3mm} (\frac{p-z_c(p)}{|p-z_c(p)|},||p||_0= |p-z_c(p)| )
\Ea
$$
so that each codimension 1 boundary stratum can be described  explicitly as follows:

(i) The boundary stratum $\Ba{c}
\begin{xy}
 <0mm,-0.5mm>*{\blacklozenge};
 <0mm,0mm>*{};<0mm,5mm>*{}**@{~},
 <0mm,0mm>*{};<1mm,-5mm>*{}**@{.},
<1mm,-5.5mm>*{\blacktriangledown};
 <1mm,-5.5mm>*{};<-1mm,-10.5mm>*{}**@{-},
<1mm,-5.5mm>*{};<-4mm,-10.5mm>*{}**@{-},
<-1mm,-12.5mm>*{_2};
<-4mm,-12.5mm>*{_1};
 \end{xy}
\Ea$ is given by the limit configurations, $p$, with $||p||=0$;
 a generic point in this boundary stratum can be obtained as the $\la\rar 0$ limit of a configuration
 $(x_0+\la z_1,x_0+\la z_2)$, $x_0\in \R$, $z_1,z_2\in \bbH$.

(ii) The boundary stratum
$\Ba{c}
\begin{xy}
 <0mm,-0.5mm>*{\blacklozenge};
 <0mm,0mm>*{};<0mm,5mm>*{}**@{~},
 <0mm,0mm>*{};<-1mm,-5mm>*{}**@{-},
<-1mm,-5.3mm>*{\circ};
 <-1mm,-5.5mm>*{};<-3mm,-10.5mm>*{}**@{-},
<-1mm,-5.5mm>*{};<1mm,-10.5mm>*{}**@{-},
<-3mm,-12.5mm>*{_1};
<1mm,-12.5mm>*{_2};
 \end{xy}
\Ea$
is given by the limit configurations, $p$, with $||p||_0=0$, $||p||$, $ y_1$,  $ y_2$
finite and non-zero; their  image under the projection
$\widehat{\fC}_{2,0}(\bbH)\rar \widetilde{C}^{st}_{2,0}(\bbH)$ consists  of a single point in $\bbH$;
 a generic point in this boundary stratum can be obtained as the $\la\rar 0$ limit of a configuration
 $(z_0+\la z_1,z_0+\la z_2)$ with $z_0, z_1,z_2\in \bbH$.

(iii) The boundary stratum
$\Ba{c}
\begin{xy}
 <0mm,-0.5mm>*{\blacklozenge};
 <0mm,0mm>*{};<0mm,5mm>*{}**@{~},
 <0mm,0mm>*{};<-3mm,-5.5mm>*{}**@{-},
<0mm,0mm>*{};<4mm,-5.5mm>*{}**@{.},
<4mm,-5.5mm>*{\blacktriangledown};
 <4mm,-5mm>*{};<1mm,-10mm>*{}**@{-},
<-3.5mm,-6.7mm>*{_1};
<0.5mm,-11.7mm>*{_2};
 \end{xy}
\Ea$
is given by the limit configurations, $p$, with $||p||$, $ y_1$ and $||p||_0$
finite and non-zero while  $ y_2=0$;
 a generic point in this boundary stratum can be obtained as the $\la\rar 0$ limit of a configuration
 $(z_1,x_2+ \ii \la y_2)$ with $z_1,z_2=x_2+\ii y_2\in \bbH$.

(iv) The boundary stratum $
\Ba{c}
\begin{xy}
 <0mm,-0.5mm>*{\blacktriangledown};
 <0mm,0mm>*{};<0mm,5mm>*{}**@{~},
 <0mm,0mm>*{};<-2mm,-4.5mm>*{}**@{--},
 %
<-3mm,-5.5mm>*{\bu};
 <-3mm,-5.5mm>*{};<-5mm,-10.5mm>*{}**@{-},
 <-3mm,-5.5mm>*{};<-1mm,-10.5mm>*{}**@{-},
<-5mm,-12.5mm>*{_1};
<-0.5mm,-12.5mm>*{_2};
 \end{xy}
\Ea$
is given by the limit configurations, $p$, with $||p||=+\infty$ and
$||p||_0$ a finite number; their  image under the projection
$\widehat{\fC}_{2,0}\rar \widetilde{C}^{st}_{2,0}(\bbH)$ consists  of a single point in $\bbH$;
 a generic point in this boundary stratum can be obtained as the $\la\rar +\infty$ limit of a configuration
 $(\la z_0 +  z_1, \la z_0 +  z_2)$, $z_0,z_1,z_2\in \bbH$.

(v) The boundary stratum $\Ba{c}
\begin{xy}
 <0mm,-0.5mm>*{\blacktriangledown};
 <0mm,0mm>*{};<0mm,5mm>*{}**@{~},
 <0mm,0mm>*{};<-2mm,-4.5mm>*{}**@{--},
<0mm,0mm>*{};<-6mm,-5.5mm>*{}**@{--},
<-2mm,-5.5mm>*{\bu};
 <-2mm,-5.5mm>*{};<-2mm,-10.5mm>*{}**@{-},
<-6mm,-5.5mm>*{\bu};
 <-6mm,-5.5mm>*{};<-6mm,-10.5mm>*{}**@{-},
<-6mm,-12.5mm>*{_1};
<-2mm,-12.5mm>*{_2};
 \end{xy}
\Ea$ is given by the limit configurations, $p$, with $||p||_0=+\infty$ and such that
their  image under the projection
$\widehat{\fC}_{2,0}\rar \widetilde{C}^{st}_{2,0}(\bbH)$ consists  of two different points in $\bbH$;
 a generic point in this boundary stratum can be obtained as the $\la\rar +\infty$ limit of a configuration
 $(\la z_1 , \la z_2)$,  $z_1,z_2\in \bbH$.

(vi)  The boundary stratum $\Ba{c}
\begin{xy}
 <0mm,-0.5mm>*{\blacktriangledown};
 <0mm,0mm>*{};<0mm,5mm>*{}**@{~},
 <0mm,0mm>*{};<-2mm,-4.5mm>*{}**@{--},
<0mm,0mm>*{};<3mm,-5.5mm>*{}**@{~},
<-2mm,-5.5mm>*{\bu};
 <-2mm,-5.5mm>*{};<-2mm,-10.5mm>*{}**@{-},
<3mm,-5.5mm>*{\blacklozenge};
 <3mm,-5.5mm>*{};<1mm,-10.5mm>*{}**@{-},
<-2mm,-12.5mm>*{_1};
<1mm,-12.5mm>*{_2};
 \end{xy}
\Ea$ is given by the limit configurations, $p$, with
 $y_1=+\infty$, $y_2$ finite; their image under the projection
$\widehat{\fC}_{2,0}\rar \widetilde{C}^{st}_{2,0}(\bbH)$ consists  of two different points in
$\overline{\bbH}$ the first of which lies in $\bbH$ and the second in $\R$;
 a generic point in this boundary stratum can be obtained as the $\la\rar +\infty$ limit of a configuration
 $(\la z_1,  z_2)$, $z_1, z_2\in \bbH$.

(vii)  The boundary stratum $\Ba{c}
\begin{xy}
 <0mm,-0.5mm>*{\blacktriangledown};
 <0mm,0mm>*{};<0mm,5mm>*{}**@{~},
 <0mm,0mm>*{};<2mm,-5.5mm>*{}**@{~},
<0mm,0mm>*{};<6mm,-5.5mm>*{}**@{~},
<2mm,-5.5mm>*{\blacklozenge};
 <2mm,-5.5mm>*{};<1mm,-10.5mm>*{}**@{-},
<6mm,-5.5mm>*{\blacklozenge};
 <6mm,-5.5mm>*{};<5mm,-10.5mm>*{}**@{-},
<1mm,-12.5mm>*{_1};
<5mm,-12.5mm>*{_2};
 \end{xy}
\Ea$
is given by the limit configurations, $p$, with
 $||p||=+\infty$,  $y_1$ and
$ y_2$ finite numbers; their image  under the projection
$\widehat{\fC}_{2,0}\rar \widetilde{C}^{st}_{2,0}(\bbH)$ consists  of two different points on the real line;
 a generic point in this boundary stratum can be obtained as the $\la\rar +\infty$ limit of a configuration
 $(-\la x+ \ii  y_1, \la x + y_2)$, $x\in \R$, $y_1,y_2\in \R^+$.

\sip

A straightforward but tedious inspection of higher codimension strata in
$\widehat{\fC}_{2,0}(\bbH)$
tells us that the above classified  codimension 1 strata (i)-(vii) are glued together into the following $3$-dimensional
compact manifold
with corners,
\Beq\label{5: picture for fC(H)_2,0}
\widehat{\fC}_{2,0}(\bbH)=\Ba{c}
\xy
 <5mm,15mm>*{};<5mm,0mm>*{}**@{.},
<-5mm,15mm>*{};<-5mm,0mm>*{}**@{.},
<5mm,30mm>*{};<5mm,15mm>*{}**@{--},
<-5mm,30mm>*{};<-5mm,15mm>*{}**@{--},
%
<-17mm,0mm>*{};<-12mm,18mm>*{}**@{--},
<-17mm,0mm>*{};<-22mm,12mm>*{}**@{-},
<-17mm,30mm>*{};<-22mm,12mm>*{}**@{-},
<-17mm,30mm>*{};<-12mm,18mm>*{}**@{--},
<17mm,0mm>*{};<12mm,12mm>*{}**@{-},
<17mm,0mm>*{};<22mm,18mm>*{}**@{-},
<17mm,30mm>*{};<22mm,18mm>*{}**@{-},
<17mm,30mm>*{};<12mm,12mm>*{}**@{-},
(0,15)*\ellipse(5,1){-};
(0,7.5)*\ellipse(5,1){.};
(0,0)*\ellipse(5,1){.};
(-17,0)*-{};(17.0,-0)*-{}
**\crv{~*=<4pt>{.}(0,6)}
**\crv{(0,-6)};
(-17,30)*-{};(17.0,30)*-{}
**\crv{(0,36)}
**\crv{(0,24)};
(-22,12)*-{};(12.0,12)*-{}
**\crv{(0,6)};
(-12,18)*-{};(22.0,18)*+{}
**\crv{~*=<4pt>{.}(0,22)}
\endxy
\Ea
\Eeq

\subsection{Semialgebraic structure on $\widehat{\fC}_{n,m}$} The right hand side of the embedding
(\ref{4: first compctfn fC_n,m(H)}) is a product of compact semialgebraic sets. Therefore $\widehat{\fC}_{n,m}$
comes naturally equipped with the structure of a compact semialgebraic set so that we can employ, if necessary,
the de Rham algebra
of PA differential forms on $\widehat{\fC}_{n,m}$. In fact, $\widehat{\fC}_{n,m}$ is a compact semialgebraic manifold:
its smooth semialgebraic atlas can be given by metric trees in a full analogy to  \S{\ref{3: Smooth atlas on Konts config spaces}}
and \S{\ref{4: Smooth atals on Mor(Lie_infty)}}.
\subsection{Higher dimensional version}
The operad of compactified configuration space $\widehat{\fC}\left(\bbH\right)$ has an obvious higher dimensional version,
$\widehat{\fC}\left(\bbH^d\right)$, $d\geq 3$, which describes the 4-coloured operad of morphisms of
open-closed homotopy Lie algebras.

\bip

{\large
\section{\bf Operads of Feynman graphs and their representations}\label{7: section}
}
\mip

\subsection{An operad of Feynman graphs $\fG$}\label{7: An operad of Feynman graphs fG}

 For a finite set $I$ we denote by ${G}_{I}$ a set of graphs\footnote{ A {\em graph}\, $\Ga$ is, by definition,  a 1-dimensional $CW$-complex whose $0$-cells are called {\em vertices} and $1$-dimensional cells are called {\em edges}. Its automorphism group as a $CW$-complex is denoted by $Aut(\Ga)$.}, $\{\Ga\}$, with $\# I$ vertices such that
\Bi
\item the edges of $\Ga$ are directed, beginning and ending at {\em different}\, vertices;
\item the vertices of $\Ga$ are labelled by elements of $I$, i.e.\ a bijection  $V(\Ga)\rar I$ is fixed;
\item the set of edges, $E(\Ga)$, is totally ordered.
\Ei
We identify two total orderings on the set $E(\Ga)$ (that is, isomorphisms $E(\Ga)\simeq [\# E(\Ga)]$),
if they differ by an even permutation of $[\# E(\Ga)]$. Thus there are precisely
two possible orderings\footnote{It is useful sometimes  to identify an orientation of $\Ga\in  {\cG}_{I}$
with a vector $\displaystyle e_\Ga:=\wedge_{e\in E(\Ga)}e$ in the real one dimensional
vector space $\wedge^{l}\R[E(\Ga)]$,
 where $\R[E(\Ga)]$ is the $l$-dimensional vector space spanned over
 $\R$ by the set $E(\Ga)$.\label{5: footnote on sign}}
on the set $E(\Ga)$
 and the group $\Z_2$ acts freely  on ${G}_{I}$ by ordering changes;  its orbit
is denoted by $\{\Ga, \Ga_{opp}\}$. If $I=[n]$, we write $G_n$ for
$G_I$. The subset of $G_n$ consisting of graphs with precisely $l$ edges is denoted
by $G_{n,l}$.
\sip

For any  fixed integer $d\in \Z$ we first set $\K\langle G_{n,l}\rangle$ to be the vector space spanned
by isomorphism classes, $[\Ga]$ of
graphs $\Ga\in G_{n,l}$ modulo the relation $[\Ga_{opp}]=(-1)^{d-1} [\Ga]$, and then define
a $\Z$-graded $\bS_n$-module,
$$
\fG(n):=\bigoplus_{l=0}^\infty \K\langle G_{n,l}\rangle[(1-d)l].
$$
Note that if $d$ is even, then any graph $\Ga\in G_{n,l} $  which has a pair of vertices  connected by two
or more edges with the same orientation  vanishes in $\fG(n)$; for example, $\xy (0,0)*{\bu}="a"; (10,0)*{\bu}="b";
{\ar@/^/"a";"b"}; {\ar@/_/"a";"b"}  \endxy =0$ in $\fG(2)$ for  $d$ even. If $\Ga$ does not vanish, then
its degree in  $\fG(n)$ is equal to $(d-1)\# E(\Ga)$, i.e.\ every edge contributes $d-1$  to the total degree.

\sip

The resulting $\bS$-module,
$$
\fG=\{ \fG(n)\}_{n\geq 1},
$$
has a natural operad structure defined as follows: let $[n]\rar [p]$ be an arbitrary surjection, and let $[n]=I_1\sqcup I_2\sqcup \ldots\sqcup I_p$ be the associated partition of $[n]$, then the map
\Beq\label{5: operadic composition in fG}
\Ba{rccc}
\circ: & \fG(p)\ot  \fG(I_1)\ot\ldots\ot
 \fG(I_p) & \lon & \fG(n)\\
        & (\Ga_0, \Ga_1, \ldots, \Ga_p)                                             & \lon &  \Ga_0\circ (\Ga_1\ot \ldots \ot \Ga_p)
\Ea
\Eeq
is given, by definition, by
\Beq\label{5: composition in Feynman operad}
\Ga_0\circ (\Ga_1\ot \ldots \ot \Ga_p) =
\sum_{\Ga\in  G_{I_1,\ldots,I_p}} (-1)^{\sigma_\Ga}\Ga,
\Eeq
where the summation runs over  a subset
$G_{I_1,\ldots,I_p}\subset   G_n$ consisting of all
graphs $\Ga$
satisfying the following condition: the subgraphs\footnote{ For any subset $A\subset [n]$
and any graph $\Ga$ in $G_{n}$ there is an associated  subgraph $\Ga_A$ of
$\Ga$ whose vertices
are, by definition, those vertices of $\Ga$ which are labelled by elements of $A$, and whose edges
are {\em all}\, the edges of $\Ga$ which connect these $A$-labelled vertices. If we compress all the
$A$-labelled vertices of $\Ga$ (together with all the edges connecting these $A$-labelled vertices) into a single vertex,
then we obtain from $\Ga$ a new graph which we denote by $\Ga/\Ga_A$. Analogously one defines
the quotient graph $\Ga/\{\Ga_{I_1},\ldots,\Ga_{I_p}\}$ for any partition $V(\Ga)=I_1\sqcup\ldots\sqcup I_p$.}
 $\Ga_{I_1}, \Ga_{I_2}, \ldots, \Ga_{I_p}$, of $\Ga$ spanned by the vertices labelled, respectively,  by the subsets
  ${I_1}, {I_2}, \ldots, {I_p}\subset [n]$
 are isomorphic, respectively, to $\Ga_1, \Ga_2,\ldots, \Ga_p$, and the quotient graph,
  $\Ga/\{\Ga_{I_1}, \Ga_{I_2}, \ldots, \Ga_{I_p}\}$ is isomorphic to $\Ga_0$.
The sign is determined by the equality
$e_{\Ga/\{\Ga_{I_1}, \Ga_{I_2}, \ldots, \Ga_{I_p}\}}e_{\Ga_{I_1}}e_{\Ga_{I_2}}\cdots e_{\Ga_{I_p}} =(-1)^{\sigma_\Ga}e_\Ga$
(see footnote \ref{5: footnote on sign}),
i.e.\ by comparing the orderings on the sets of edges. The unique element in $G_{1,0}$
serves as a unit in this operad. For example,
$$
\xy
   {\ar@/^0.6pc/(-5,0)*{\bu};(5,0)*{\bu}};
 {\ar@/^0.6pc/(5,0)*{\bu};(-5,0)*{\bu}};
\endxy \circ\left(
\xy
 {\ar@{->}(0,7)*{\bu};(0,0)*{\bu}};
\endxy
\ot
\bu
\right)=
\xy
 {\ar@{->}(-5,7)*{\bu};(-5,0)*{\bu}};
   {\ar@/^0.6pc/(-5,0)*{\bu};(5,0)*{\bu}};
 {\ar@/^0.6pc/(5,0)*{\bu};(-5,0)*{\bu}};
\endxy
+
\xy
 {\ar@{->}(-5,0)*{\bu};(-5,-7)*{\bu}};
   {\ar@/^0.6pc/(-5,0)*{\bu};(5,0)*{\bu}};
 {\ar@/^0.6pc/(5,0)*{\bu};(-5,0)*{\bu}};
\endxy
+
\xy
 {\ar@{->}(-5,4)*{\bu};(-5,-4)*{\bu}};
   {\ar@{->}(-5,4)*{\bu};(5,0)*{\bu}};
 {\ar@{->}(5,0)*{\bu};(-5,-4)*{\bu}};
\endxy
+
\xy
 {\ar@{->}(-5,4)*{\bu};(-5,-4)*{\bu}};
   {\ar@{<-}(-5,4)*{\bu};(5,0)*{\bu}};
 {\ar@{<-}(5,0)*{\bu};(-5,-4)*{\bu}};
\endxy
$$
where orientations of graphs on the r.h.s.\ are  chosen implicitly in a such a way that they contribute
with coefficient $+1$. Thus the operadic composition, $\Ga_0\circ (\Ga_1\ot \ldots \ot \Ga_p)$  is given simply by substitutions of the graphs $\Ga_1, \ldots, \Ga_p$ into the vertices of the graph $\Ga_0$ and redistributing the edges in all possible ways.

\sip

The number $d-1$ is called the {\em propagator degree}\, of the operad $\fG$. Perhaps  we should have
imprinted $d$ somehow into the notation, say use the symbol $\fG[d]$ to indicate its dependence on $d$, but we do not
bother doing it as in applications the propagator degree will be always clear from the context.


\subsubsection{\bf Coloured variants of  $\fG$} Let a $\Z$-graded $\bS_n$-module
$\bigoplus_{n=n_1+n_2\atop 2n_1 + n_2\geq 2} \sG^{\uparrow\downarrow}(n_1, n_2)$  be defined as $\fG(n)$ above except
that the vertices of graphs $\Ga$ from $\sG^{\uparrow\downarrow}(n_1, n_2)$  are coloured in one of two
possible colours, say white and black, i.e. $V(\Ga)$ comes equipped with a splitting into a
disjoint union $V(\Ga)=V_w(\Ga)\sqcup V_b(\Ga)$, $\# V_w(\Ga)=n_1$ and  $\# V_b(\Ga)=n_2$, satisfying the condition $2n_1+n_2 \geq 2$. It is useful to visualize such a graph as drawn on the closed
upper  half plane $\overline{\bbH}$ with white vertices placed in $\bbH$ and black vertices on the boundary $\R$
in the order which is consistent with their numbering, e.g.
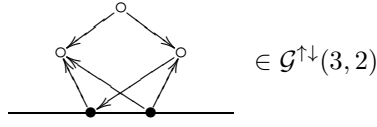
$$
\Ba{c}
\xy
(-15,0)*{}="a",
(15,0)*{}="b",
(0,14)*{\circ}="c",
(-8,8)*{\circ}="c1",
(8,8)*{\circ}="c2",
(-4,0)*{\bu}="d1",
(4,0)*{\bu}="d2",
\ar @{-} "a";"b" <0pt>
\ar @{->} "c";"c1" <0pt>
\ar @{->} "c";"c2" <0pt>
\ar @{<-} "c1";"d1" <0pt>
\ar @{<-} "c1";"d2" <0pt>
\ar @{->} "c2";"d1" <0pt>
\ar @{<-} "c2";"d2" <0pt>
\endxy
\Ea
\in \cG^{\uparrow\downarrow}(3,2)
$$

The $\bS$-bimodule
$$
\fG^{\uparrow\downarrow}=\{\fG(n)\oplus \bigoplus_{n=n_1+n_2} \cG^{\uparrow\downarrow} (n_1,n_2)\}_{n\geq 2, 2n_1+ n_2\geq 2}
$$
is naturally a 2-coloured operad with respect to substitutions of
\Bi
\item[(i)]
graphs from $\fG(n)$ into the vertices of
$\fG(m)$ and into the white vertices of  $\cG^{\uparrow\downarrow} (m_1,m_2)$,
 \item[(ii)]
graphs from $\cG^{\uparrow\downarrow} (n_1,n_2)$ into the black vertices of  $\cG^{\uparrow\downarrow} (m_1,m_2)$.
\Ei

\sip
One can  consider versions,  $\cG^{\downarrow}$ and $\cG^{\uparrow}$, of the 2-coloured operad
$\cG^{\uparrow\downarrow}$ spanned by graphs $\Ga$ such that all edges incident to any black vertex
are oriented towards (respectively, outwards) that black vertex. For example
$$
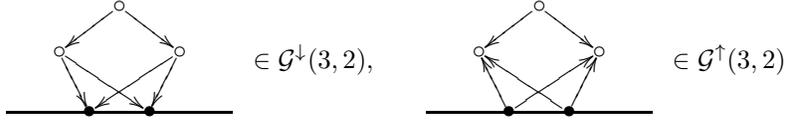
\Ba{c}
\xy
(-15,0)*{}="a",
(15,0)*{}="b",
(0,14)*{\circ}="c",
(-8,8)*{\circ}="c1",
(8,8)*{\circ}="c2",
(-4,0)*{\bu}="d1",
(4,0)*{\bu}="d2",
\ar @{-} "a";"b" <0pt>
\ar @{->} "c";"c1" <0pt>
\ar @{->} "c";"c2" <0pt>
\ar @{->} "c1";"d1" <0pt>
\ar @{->} "c1";"d2" <0pt>
\ar @{->} "c2";"d1" <0pt>
\ar @{->} "c2";"d2" <0pt>
\endxy
\Ea
\in \cG^{\downarrow}(3,2), \ \ \ \
\Ba{c}
\xy
(-15,0)*{}="a",
(15,0)*{}="b",
(0,14)*{\circ}="c",
(-8,8)*{\circ}="c1",
(8,8)*{\circ}="c2",
(-4,0)*{\bu}="d1",
(4,0)*{\bu}="d2",
\ar @{-} "a";"b" <0pt>
\ar @{->} "c";"c1" <0pt>
\ar @{->} "c";"c2" <0pt>
\ar @{<-} "c1";"d1" <0pt>
\ar @{<-} "c1";"d2" <0pt>
\ar @{<-} "c2";"d1" <0pt>
\ar @{<-} "c2";"d2" <0pt>
\endxy
\Ea
\in \cG^{\uparrow}(3,2)
$$

 \sip

More generally,
with any compactified configuration space, $\overline{C}=\{\overline{C}_n\}$, considered in \S 2-4 of this paper
one can associate  a coloured operad of Feynman graphs,
$\fG_{\overline{C}}=\{\fG_{\overline{C}}(n)\}$,
whose elements are, by definition, linear combinations of graphs from $\fG{(n)}$ (or from an appropriate subset of $\fG(n)$)
whose vertices are identified
with configuration of points in $\overline{C}_n$ and hence are coloured correspondingly.
For example, the above operad $\fG$ with propagator degree $d-1$  can be re-denoted as
$\fG_{\overline{C}(\R^d)}$ while any of the operads
$\fG^{\uparrow\downarrow}$, $\fG^\downarrow$ and $\fG^\uparrow$ can be associated with the configuration spaces
$\overline{C}(\bbH^d)$ and hence can be re-denoted by
$\fG^{\uparrow\downarrow}_{\overline{C}(\bbH^d)}$ , $\fG^\downarrow_{\overline{C}(\bbH^d)}$ and
 $\fG^\uparrow_{\overline{C}(\bbH^d)}$ respectively. It is the operad $\fG^\downarrow_{\overline{C}(\bbH^d)}$
 which Kontsevich used in his paper \cite{Ko} on the formality theorem; we shall consider below other two operads
 as well.

\subsubsection{\bf Dual cooperads of Feynman graphs} As the vector space  $\K\langle G_{n,l}\rangle$
is finite-dimensional for each $n$ and $l$, the dual $\bS$-modules,
$$
{\fG}^*:=\left\{\Hom({\fG}_k(n),\K)\right\}_{n\geq 1}, \ \
({\fG}^{\uparrow\downarrow})^*:=\left\{\Hom({\fG}^{{\uparrow\downarrow}}(n),\K)\right\}_{n\geq 1}, \ \ etc.
$$
have induced structures of  (coloured) cooperads in the category of graded vector spaces. these spaces come equipped
canonically  with distinguished bases which we can identify with graphs.
We denote by $\check{\fG}$, $\check{\fG}^{\uparrow\downarrow}$, etc.\
 their sub-cooperads spanned by {\em finite}\, linear combinations of such graphs.

For example, the dualization of the operadic composition (\ref{5: operadic composition in fG}) in $\fG$ gives the following formula for the  co-composition
in the cooperad $\check{\fG}$,
\Beq\label{5: co-composition in G_d}
\Ba{rccc}
\Delta: &  \check{\fG}(n) &\lon &
\displaystyle \bigoplus_{[n]=I_1\sqcup\ldots\sqcup I_p\atop 1\leq p \leq  n}
\check{\fG}(p) \ot \check{\fG}(I_1)\ot \check{\fG}(I_2) \ot\ldots
\ot \check{\fG}(I_p)
\mip\\
&   \Ga & \lon & \Delta(\Ga):= \displaystyle\sum_{p=0}^{n-1}\sum_{[n]=I_1\sqcup\ldots\sqcup I_p}
(-1)^{\var}\Ga/\{\Ga_{I_1},\ldots,\Ga_{I_p}\}\ot
\Ga_{I_1}\ot \Ga_{I_2}\ot\ldots\ot \Ga_{I_p}.
\Ea
\Eeq
where the summation runs over all possible partitions of $[n]$. As $\fG$ is unital,
we can equivalently describe cooperad structure in $\check{\fG}_d$ in terms
of the reduced co-composition, $\forall i\in [n]$,
\Beq\label{5: reduced co-composition in G_d}
\Ba{rccc}
\Delta_i^{red}: &  \check{\fG}(n) &\lon &
\displaystyle \bigoplus_{i\in A \subset[n]}
\check{\fG}(n-\# A +1) \ot \check{\fG}(A)
\mip\\
&   \Ga & \lon & \Delta_i^{red}(\Ga):= \displaystyle \sum_{i\in A \subset[n]}
(-1)^{\var}\Ga/\Ga_{A}\ot
\Ga_{A}.
\Ea
\Eeq
Setting  $\check{\fG}(1)=0$, we make the data $(\{ \check{\fG}_d(n)\}_{n\geq 2}$ into a non-unital pseudo-cooperad which we denote
by the same symbol $\check{\fG}$ but call the  {\em non-unital pseudo-cooperad
of Feynman graphs}. Analogously one defines non-unital pseudo-cooperads
 $\check{\fG}^{\uparrow\downarrow}$ etc.

\subsection{A class of representations of $\fG$} Note that operads of Feynman graphs
 depend on the choice
 of a  parameter $d$ which we do not show in the notations. Any representation
of $\fG$ in a graded vector space $X$,
$$
\rho: \fG \lon \cE nd_X
$$
is uniquely determined by its values on the generators, $\Ga\in G_{n,l}$, i.e. by
a collection of operators,
$$
\left\{\Phi_\Ga:=\rho(\Ga)\in \Hom_{(d-1)\# E(\Ga)}(X^{\ot\# V(\Ga)}, X)\right\}
$$
which, in accordance with (\ref{5: composition in Feynman operad}),  satisfy the equations
\Beq\label{5: Phi_Ga equations}
(-1)^K\Phi_{\Ga_0}\left(\Phi_{\Ga_1}(x_1,\ldots, x_{n_1}), \Phi_{\Ga_2}(x_{n_1+1}, \ldots, x_{n_1+n_2}),\ldots, \Phi_{\Ga_p}(x_{n_1+...+n_{p-1}+1},\ldots,x_{n_1+...+n_{p}})\right)=\hspace{10mm}
\Eeq
$$
\hspace{40mm}
\sum_{\Ga\in  \fG_d(n_1,\ldots,n_p)}
 (-1)^{\sigma(\Ga)}\Phi_\Ga(x_1,x_2, \ldots, x_{n_1+...+n_p})
$$
for any $\Ga_i\in G_{n_i,l_i}$, $i=0,1,\ldots,p$, and any $x_1,x_2,\ldots, x_{n_1+\ldots+n_p}\in A$.  Here $(-1)^K$ stands for the usual Koszul sign
arising under a composition of homogeneous maps and $\fG(n_1,\ldots,n_p)\subset
\fG(n_1+\ldots+n_p)$ is a subset consisting of all graphs $\Ga$
whose subgraphs,
 $\Ga_{I_1}, \Ga_{I_2}, \ldots, \Ga_{I_p}$, spanned by vertices labelled, respectively,  by
 $$
 I_1:=\{1, \ldots, n_1\}, \ I_2:=\{n_1+1, \ldots, n_1+n_2\},\ \ldots,\ I_p:=\{n_1+\ldots+n_{p-1}+1, \ldots, n_1+\ldots+n_p\}
 $$
are isomorphic to $\Ga_1, \Ga_2,\ldots,\Ga_p$, and the quotient graph, $\Ga/\{\Ga_1, \Ga_2,\ldots,\Ga_p\}$, is isomorphic to $\Ga_0$.

\sip

 Let $W$ be a $\Z$-graded vector
space and $\tau\in W^*\ot W^*$ a non zero element of degree $d-1$. Any element $\om\in W^*$ defines a map
$W\rar \K$ which can be uniquely extended to a derivation of the symmetric tensor algebra $\bigodot^\bu W$. Therefore,
the element $\tau$ gives naturally rise to a biderivation on $\bigodot^\bu W \otimes \bigodot^\bu W$
which we denote by $\Delta^\tau$. Moreover, for any $n\geq 2$ and any ordered
pair of different integers $i,j\in [n]$, $\tau$ gives rise to an automorphism,
$$
\Delta_{ij}^\tau: (\odot^\bu W)^{\ot n} \lon  (\odot^\bu W)^{\otimes n}
$$
which acts as $\Delta^\tau$ on $i$-th and $j$-th tensor factors and as the identity
on all other tensor factors. Next, for any graph $\Ga\in G_{n,l}$, we define
$$
\Phi_\Ga:  (\odot^\bu W)^{\ot n} \lon  \odot^\bu W
$$
as the composition,
\Beq\label{5: Def of Phi_Ga}
\Phi_\Ga:=\mu^{(n)}\circ \prod_{e\in E(V)} \Delta_{In(e),Out(e)}^\tau,
\Eeq
where, for an edge $e$ connecting a vertex labelled by $i\in [n]$ to a vertex labelled by $j$, we set
 $\Delta^\tau_{In(e),Out(e)}:= \Delta_{i,j}^\tau$, and $\mu^{(n)}$ stands for the natural multiplication,
 $$
 \Ba{rccc}
\mu^{(n)}: &   (\odot^\bu W)^{\ot n} & \lon &  \odot^\bu W \\
     &   (f_1,f_2,\ldots, f_n) & \lon & f_1f_2\cdots f_n.
\Ea
 $$
\subsubsection{\bf Proposition}\label{5: Proposition on tau and repr of Feynman operad}
{\em For any graded vector space $W$ and any element $\tau\in W^*\ot W^*$ of degree $k$
the operators (\ref{5: Def of Phi_Ga}) define a representation
of the operad of Feynman graphs in the vector space $X:=\odot^\bu W$.}

\begin{proof} One has only to check equations (\ref{5: Phi_Ga equations})
but this is straightforward due to the Leibniz rule.
\end{proof}

\subsubsection{\bf Poisson-Schouten algebras}\label{5: subsect  Poisson-Schouten algebras}
 Let $V$ be an arbitrary  $\Z$-graded vector space $V$ and let
$$
W_d:=V^*[2-d] \oplus V[-1].
$$
Then $W_d^*\ot W_d^*$ contains a distinguished element $\tau:=s^{1-d}\Id_V$ such that
$$
\Delta^{\tau}=\sum_\al \frac{\p}{\p \psi_\al} \ot \frac{\p}{\p x^\al}
$$
where, for a basis $\{e_\al\}$ in $V$ and the associated dual basis, $\{e^\al\}$, in $V^*$, we set
$$
x^\al:= s^{2-d}e^\al, \ \ \ \ \psi_\al:= s^{-1}e_\al.
$$
This operator makes the completed graded commutative algebra $\widehat{\odot^\bu} W_d\simeq \K[[x^\al,\psi_\al]]$
into a $d$-{\em algebra}\, (in the terminology of \cite{Ko2}) with the degree $1-d$ Lie brackets
given by,
\Beq\label{5: Poisson-Schouten bracket}
\{f\bu g\}_{1-d}=\sum_\al 
\frac{ f\overleftarrow{\p}}{\p \psi_\al} \frac{\overrightarrow{\p} g}{\p x^\al}\ +  (-1)^{|f||g|+ (d-1)(f+g)+d}
\frac{g \overleftarrow{\p} }{\p \psi_\al} \frac{\overrightarrow{\p} f}{\p x^\al}.
\Eeq
We call this particular class of $d$-algebras,
$$
\fg_d(V):=(\widehat{\odot^\bu} W_d, \{\ ,\ \}_{1-d}),
$$
the {\em Poisson-Schouten algebras}. Note that
$$
\fg_2(V)=\cT_{poly}(V)
$$
can be identified with the Schouten algebra of polyvector fields on $V$ viewed as an affine or
 formal manifold.

\sip

By Proposition~{\ref{5: Proposition on tau and repr of Feynman operad}}, every such an algebra comes equipped
with a representation of the operad $\fG$ (which, as we shall see below, can be quantized).
Let us have a brief look at the set of Maurer-Cartan elements of $\fg_d(V)$,
$$
\cM\cC(V):= \left\{\ga \in  \fg_d(V)\ | \{\ga\bu \ga\}_{1-d}=0 \ \mbox{and}\ |\ga|=d \right\},
$$
for the case $V=\R^d$ and $d\geq 2$:
\Bi
\item[$d=2:$] the set  $\cM\cC_2(V)$  consists of formal Poisson structures on $V$, that is,
 $\cM\cC_2(V)$ is the set of formal bi-vector fields
$\ga \in \wedge^2\cT_V$ satisfying the equation  $\{\ga\bu \ga\}_{-1}=0$; the 2-algebra $\fg_2(\R^d)$
is precisely the Schouten algebra of formal polyvector fields on $\R^d$;
\item[$d=3:$]  a generic element of degree $3$ has the form,
$$
\ga= \sum_{\al,\be,\ga} \underbrace{C_{\al\be\ga}}_{\wedge^3 V\rar K}  x^\al x^\be x^\ga +
\sum_{\al,\be,\ga} \underbrace{C_{\al\be}^{\ga}}_{\wedge^2 V\rar V}  x^\al x^\be \psi_\ga +
\sum_{\al,\be,\ga} \underbrace{C^{\al\be}_{\ga}}_{V\rar \wedge^2 V}  \psi_\al \psi_\be x^\ga +
\sum_{\al,\be,\ga} \underbrace{C^{\al\be\ga}}_{\K\rar \wedge^3 V}  \psi_\al \psi_\be \psi_\ga.
$$
It is easy to check that such an element satisfies the equation $\{\ga\bu \ga\}_{-2}=0$
if and only if the associated set of degree $0$ maps,
$$
{\K \rar \wedge^3 V,\ \  V \rar \wedge^2 V,\ \ \wedge^2 V\rar V, \ \ \wedge^3 V\rar \K},
$$
make $V$ into a Lie quasi-bialgebra. This notion was introduced by Drinfeld in order ro describe
infinitesimal breaking of the (co)associativity of basic  bialgebra operations due to associators.

\sip

\item[$d=4$:] the set $\cM\cC_4(V)$ describes triples, $([\ ,\ ], g, \Phi)$, consisting of a Lie algebra
structure on $V^*$, and  Lie invariant elements  $g\in \odot^2V^*$ and $\Phi\in \wedge^4 V$,
$$
\ga= \sum_{\al,\be} \underbrace{C_{\al\be}}_{g:\odot^2 V\rar \K}  x^\al x^\be  +
\sum_{\al,\be,\ga} \underbrace{C^{\al\be}_{\ga}}_{V\rar \wedge^2 V}  \psi_\al \psi_\be x^\ga +
\sum_{\al,\be,\ga,\var} \underbrace{C^{\al\be\ga\var}}_{\K\rar  \wedge^4 V}  \psi_\al \psi_\be \psi_\ga
\psi_\var.
$$

\sip

\item[$d\geq 5$:]  the set $\cM\cC_d(V)$ describes pairs, $([\ ,\ ], \Phi)$, consisting of a Lie algebra
structure on $V^*$ and a Lie invariant element $\Phi\in \wedge^{d} V$.
\Ei

\subsection{ A class of representations of the 2-coloured operads  $\fG^{\uparrow\downarrow}$, $\fG^\uparrow$ and
$\fG^\downarrow$}\label{7: representations of fG-arrows} Formulae (\ref{5: Def of Phi_Ga}) for  $\tau:=s^{1-d}\Id_V$ give a representation
of the 2-coloured operad
$$
\rho: \fG^{\uparrow\downarrow} \lon \cE nd_{\{X_1,X_2\}}
$$
in $X_1=X_2=\fg_d(V)$. We shall quantize this representation in \S
{\ref{9:  quantization of associative algebras of polyvector fields}} using the standard homogeneous volume form
on the sphere $S^{d-1}$ as a propagator.

\sip

The Poisson-Schouten algebra $\fg_{d}(V)$ contains two graded commutative subalgebras,
$\odot^\bu (V^*[2-d])$ and $\odot^\bu (V[-1])$, which are also Abelian Lie subalgebras. There are natural projections,
$$
\pi_{\downarrow}: \fg_{d}(V)\rar \odot^\bu (V^*[2-d]), \ \ \mbox{and}\ \
\pi_{\uparrow}: \fg_{d}(V)\rar \odot^\bu (V[-1]).
$$
If $\Ga\in  \cG^{\downarrow}$ has $n$ white vertices and $m$ black vertices, then  we set,
$$
\Ba{rccc}
\Phi_\Ga: & \left(\odot^\bu (V^*[2-d] \oplus V[-1])^{\otimes n}\right)\bigotimes
(\odot^\bu(V^*[2-d])^{\otimes m} &\lon &  \odot^\bu (V^*[2-d])\\
& \ga_1\ot\ldots\ot \ga_n\ot f_1\ot\ldots\ot f_m & \lon & \Phi_\Ga(\ga_1,\ldots,\ga_n, f_1,\ldots,f_m)
\Ea
$$
where
\Beq\label{5: Phi_Ga 2 colour}
 \Phi_\Ga(\ga_1,\ldots,\ga_n, f_1,\ldots,f_m):= \pi_{\downarrow}\circ \mu^{(n+m)}\circ \prod_{e\in E(V)}
 \Delta_{In(e),Out(e)}^\tau(\ga_1\ot \cdots\ot \ga_n\ot f_1 \ot \cdots \ot f_m).
\Eeq

It is easy to check that
for any graded vector space $V$
the operators (\ref{5: Def of Phi_Ga}) and (\ref{5: Phi_Ga 2 colour}) define a representation
of the 2-coloured operad of Feynman graphs $\fG^{\downarrow}$ in the vector spaces
$X_c:=\odot^\bu (V^*[2-d]\oplus V[-1])$ and $X_o= \odot^\bu (V^*[2-d])$. Quantization of this representation
in the case $d=2$ with the help of the Kontsevich propagator gives us his formality map (see below).

\sip

Replacing in the above construction $\odot^\bu(V^*[2-d])$ by  $\odot^\bu(V[-1])$ and $\pi_{\downarrow}$ by
$\pi_{\uparrow}$ one obtains a canonical representation of  $\fG^{\uparrow}$ in the vector spaces
$X_c:=\odot^\bu (V^*[2-d]\oplus V[-1])$ and $X_o= \odot^\bu (V[-1])$.

\bip

{\large
\section{\bf
De Rham field theories on configuration spaces and quantization}\label{8 Section: De Rham}
}

\mip

\subsection{Completed tensor product of de Rham algebras} Let $\cM an$ be the category of smooth manifolds
with corners (or semialgebraic manifolds) and let $\Omega_{\cM an}$ be the associated category\footnote{One might prefer using the language of functors and natural transformations in this subsection, but, to save the space and the time,  we choose to consider $\Omega_{\cM an}$ not as a functor but as a subcategory of the category of dg algebras.} of de Rham
algebras of  (PA) differential forms. The category  $\cM an$ is symmetric monoidal with respect to the ordinary
Cartesian product, $\times$, of manifolds. We shall define
a {\em completed}\, tensor product of de Rham algebras as follows,
$$
\Omega_{X_1}\widehat{\ot}\, \Omega_{X_2}:= \Omega_{X_1\times X_2},
$$
where $X_1$ and $X_2$ are arbitrary objects in $\cM an$. This completed tensor product makes
$\Omega_{\cM an}$ into a symmetric monoidal category so that we can define operads and cooperads in $\Omega_{\cM an}$.
Moreover, in this case one can associate with an arbitrary operad $\cP=\{\cP(n)\}$
in the category $\cM an$  a co-operad, $\Omega_\cP:=\{\Omega_{\cP(n)}\}$, of de Rham algebras in the category
$\Omega_{\cM an}$.

\subsection{De Rham field theory on $\overline{C}(\R^d)$}
 For any proper subset $A\subset [n]$ of cardinality
at least two there is a uniquely associated operadic
composition\footnote{See Section {\ref{Appendix: Third definition}} in Appendix
 for explanation of the notation $\circ_{\bu}^{([n]-A)\sqcup \bu, A}$.},
\Beq\label{6: circ_A composition}
\circ_A:= \circ_{\bu}^{([n]-A)\sqcup \bu, A}: \overline{C}_{([n]- A)\sqcup\bu}(\R^{d})\times
\overline{C}_{ A}(\R^{d})
 \lon \overline{C}_n(\R^{k+1}),
\Eeq
which sends the Cartesian product into a  corresponding  boundary component in $\p\overline{C}_n(\R^{d})$. Moreover,
 $$
\p \overline{C}_n(\R^d)= \bigcup_{A\varsubsetneq [n], \# A\geq 2} \circ_A
\left(\overline{C}_{([n]-A)\sqcup \bu}(\R^d)\times  \overline{C}_A(\R^d)\right)
$$

Let $\Omega_{\overline{C}_n(\R^d)}=\oplus_{p\geq 0} \Omega^p_{\overline{C}_n(\R^d)}$ stand for the de
Rham algebra of ($PA$) smooth complex valued differential forms on the $\overline{C}_n(\R^d)$.
The associated dg $\bS$-module,
$$
\Omega_{\overline{C}(\R^d)}:=\left\{\Omega_{\overline{C}_n(\R^d)}, d_{DR}\right\}_{n\geq 2}
$$
is naturally  a non-unital dg cooperad in the category $\Omega_{\cM an}$.

\subsubsection{\bf Definition} A {\em de Rham field theory}\,  on the  operad
 $\{\overline{C}_n(\R^d)\}_{n\geq 1}$ is a morphism\footnote{Strictly speaking, the objects on both sides of the arrow
 belong to different symmetric monoidal categories as the tensor product on the r.h.s.\ is completed;
 however, it is straightforward to adopt axioms behind the ordinary notion of a
  morphism of operads in the category of vector spaces to this particular case.},
$$
\Omega: (\check{\fG}, 0) \lon (\Omega_{\overline{C}(\R^d)}, d_{DR}),
$$
of dg cooperads.

\sip

Let us translate this definition into a system of explicit formulae.

\subsubsection{\bf Proposition}{\em A de Rham field theory on
 $\{\overline{C}_n(\R^d)\}_{n\geq 1}$ is equivalent to a family of maps
$$
\left\{
\Ba{rccc}
\Omega: & {G}_{n, l}& \lon & \Omega^{l(d-1)}_{\overline{C}_n(\R^d)}\\
& \Ga & \lon & \Omega_\Ga
\Ea
\right\}_{n\geq 1, l\geq 0},
$$
such that $d_{DR}\Omega_\Ga=0$,  $\Omega_{\Ga_{opp}}=-(-1)^{d-1}\Omega_\Ga$,
and, for any
boundary stratum in $\overline{C}_n(\R^d)$ given by the image of an operadic composition
(\ref{6: circ_A composition}),  one has
\Beq\label{6: de Rham on C_n R^d}
\circ_A^*(\Omega_\Ga) = (-1)^{\var}\Omega_{\Ga/\Ga_A}\wedge \Omega_{\Ga_A}
\Eeq
where the sign is determined by the equality $e_{\Ga/\Ga_A} e_{\Ga_{A}} =(-1)^{\var}e_\Ga$.}
\begin{proof} As differential in $\check{\fG}$ is trivial, the image of $\Omega$
belongs to the subspace of closed differential forms.
As the co-composition in $\fG$ is given by (\ref{5: co-composition in G_d})
while the co-composition
in the cooperad $\Omega_{\overline{C}(\R^d)}$,
$$
\Ba{rccc}
\Delta: &  \Omega_{\overline{C}(\R^d)} &\lon & \displaystyle \bigoplus_{A\subsetneq [n],\ \# A\geq 2}
\Omega_{\overline{C}_{([n]-A)\sqcup\bu}(\R^d)} \hat{\ot}   \Omega_{\overline{C}_{A}(\R^d)} \mip\\
&   \om & \lon & \displaystyle \Delta(\om):= \bigoplus_{A\subsetneq [n],\ \# A\geq 2}
 \circ_A^*(\om)|_{\overline{C}_{([n]-A)\sqcup\bu}(\R^d)\times  \overline{C}_{A}(\R^d)}
\Ea
$$
is just a direct sum of  restrictions of the differential form $\om$  to all the boundary components
$Im(\circ_A)$, we conclude that the map $\Omega$ is a morphism of cooperads if and only if
equations (\ref{6: de Rham on C_n R^d}) hold for any $n$ and any $A\varsubsetneq [n]$ with $\# A\geq 2$.
\end{proof}

\subsubsection{\bf Theorem}\label{6: Theorem on Lie_infty(k) algebras} {\em Let
$$
\Ba{rccc}
\rho: & \fG & \lon  & \cE nd_X\\
      &    \Ga                  &\lon &   \Phi_\Ga
\Ea
$$
be a representation of the operad of
Feynman graphs in a graded vector space $X$.
Then any de Rham field theory on   $\overline{C}(\R^d)$  makes  $X$ into a representation,
$$
\Ba{rccc}
\rho_Q: & (\cF \cC hains(\overline{C}(\R^d)),\p) & \lon  & \cE nd_X\\
      &    \overline{C}_n(\R^d)                  &\lon &  \mu_n\in \Hom(X^{\ot n},X)
\Ea
$$
of the operad of fundamental chains of $\overline{C}(\R^d)$, that is, makes $X$ into
a $\caL_\infty\{d-1\}$-algebra with operations, $\{\mu_n: \ot^n X\rar X\}_{n\geq 2}$ given  by
\Beq\label{Ch4: induced_Leib_infty}
\mu_n:=
\left\{
\Ba{cl}
 \sum_{\Ga\in G_{n,\frac{nd-2}{d-1}-1}} c_{\Ga}\Phi_\Ga
 & \mbox{if}\ d-1|nd-2  \\
 0 & \mbox{otherwise}.\\
 \Ea
 \right.
\Eeq
where\footnote{$c_\Ga$ and $\Phi_\Ga$ depend on the choice of the
linear ordering of the set $E(\Ga)$; however their product
$c_\Ga\Phi_\Ga$ is independent of that choice.}
\Beq\label{Ch4: c_gamma_on_C_n}
c_\Ga:= \int_{\overline{C}_n(\R^d)} \Omega_\Ga.
\Eeq
}
\begin{proof}
First we explain the condition $d-1|nd-2$. If $\Ga\in G_{n,l}$, then $c_\Ga\neq 0$ if and only if
$\Omega_\Ga$ is a top degree differential form on $\overline{C}_n(\R^d)$, i.e.\ if and only
if $\deg \Omega_\Ga= (d-1)l$ is equal to $\dim \overline{C}_n(\R^d)= dn-d-1=dn-2-(d-1)$.

\sip

Assume 
now that $n$ is such that $(d-1)|nd-3$. Then, for any $\Ga\in G_{n, \frac{nd-3}{d-1}-1}$, one has,
by the Stokes theorem,
$$
0=\int_{\overline{C}_n(\R^d)}d\Omega_\Ga =\int_{\p\overline{C}_n(\R^d)}\Omega_\Ga=
\sum_{A\varsubsetneq [n]\atop \# A\geq 2}(-1)^{\sigma_A}
\int_{\overline{C}_{(n- A)\sqcup \bu}(\R^d)}\Omega_{\Ga/\Ga_A}
\int_{\overline{C}_{ A}(\R^d)}\Omega_{\Ga_A}  =
\sum_{A\subset V(\Ga)\atop
A\, \mathit i \mathit s\, \mathit a \mathit d \mathit m \mathit i \mathit s \mathit s \mathit i \mathit b \mathit l \mathit e}(-1)^{\sigma(A)}c_{\Ga_A}c_{\Ga/\Ga_A}.
$$
where a subset $A\subset V(\Ga)$ is called {\em admissible}\, if   $d-1|d\# A -2$ and $\Ga_A\in G_{\# A, \frac{d\# A -2}{d-1}-1}$.
For admissible subsets $A\subset [n]$ both differential forms $\Omega_{\Ga_A}$ and $\Omega_{\Ga/\Ga_A}$ are
top degree forms
on  $\overline{C}_{\# A}(\R^d)$ and, respectively, $\overline{C}_{n-\# A +1}(\R^d)$.

For a pair of natural number $m$ and $d$
such that $d-1|md-2$ let us denote
$$
\Ga_{[m;d]}:= G_{m, \frac{md-2}{d-1}-1}.
$$
 Then, using (\ref{5: Phi_Ga equations}),
we obtain,
\Beqrn
\sum_{A\subsetneqq [n]\atop
\# A\geq 2}(-1)^{\sigma(A)}\hspace{-5mm} \underbrace{\mu_{([n]-A)\sqcup\bu}\circ_\bu \mu_{A}}_{operadic\ composition\ in\ \cE nd_{X}}&=&
\sum_{A\subsetneqq [n]\ \# A\geq 2\atop \stackrel{d-1|d\# A -2}{d-1| d(n-\#A+1)-2}}   \sum_{\Ga_1\in G_{[n-\# A+1,d]}}
\sum_{\Ga_2\in G_{[\#A,d]}}  (-1)^{\sigma(A)} c_{\Ga_1}c_{\Ga_2}                \Phi_{([n]-A)\sqcup\bu}\circ_\bu \Phi_{A}\\
&=&
\sum_{\Ga\in G_{n, \frac{nd-3}{d-1}-1}}
\left(
\sum_{A\subset Vert(\Ga)\atop
A\, \mathit i \mathit s\, \mathit a \mathit d \mathit m \mathit i \mathit s \mathit s \mathit i \mathit b \mathit l \mathit e}(-1)^{\sigma_A}c_{\Ga_A}c_{\Ga/\Ga_A}\right)\Phi_{\Ga}\\
  &=& 0
\Eeqrn
which proves the claim.
\end{proof}

\subsubsection{\bf Remark}\label{6: Remark on weights} The above proof is elementary but notationally looks unduly complicated ---
we have to take care about divisibility conditions of the type $d-1|nd-2$ to ensure  that
the integrals $\int_{\overline{C}_n(\R^d)} \Omega_\Ga$ make sense. If, however, we set formally
$\int_{\overline{C}_n(\R^d)} \Omega_\Ga=0$ when $\deg \Omega_\Ga\neq \dim \overline{C}_n(\R^d)$, then
the presentation gets  simplified. We assume this convention from now on.

\subsubsection{\bf Example} For any $i,j\in [1,\ldots,n]$
there is a natural map
$$
\Ba{rccc}
\pi_{ij}: & C_n(\R^{d}) & \lon & C_{2}(\R^{d})\simeq S^{d-1}\\
& (x_1,\ldots,x_n) & \lon & \frac{x_i-x_j}{|x_i-x_j|}
\Ea
$$
which forgets all points in the configuration except those labelled by
$i$ and $j$. This map extends to a map of the compactifications,
$
\bar{\pi}_{ij}: \overline{ C}_n(\R^{d})\lon \overline{ C}_n(\R^{d}).$

\begin{subtheorem}\label{6: prop on propagator on C(R^d)}  { Let $\om$ be a volume form on $S^{d-1}$ normalized so that $\int_{S^{d-1}}\om=1$.
Then the maps, $n\geq 2$,
\Beq\label{6: DeRFT on C(R^d)}
\Ba{rccc}
\Omega: & \fG(n) & \lon & \Omega_{\overline{C}_n(\R^d)}\\
        &   \Ga   & \lon &   \displaystyle   \Omega_\Ga:=\bigwedge_{e\in E(\Ga)}\om_e
\Ea
\Eeq
define a de Rham field theory on $\overline{C}(\R^d)$. Here, for an edge $e$
beginning at a vertex labelled by $i\in [n]$ and ending at a vertex labelled by $j\in [n]$,
we set
$
\om_e:= \bar{\pi}_{ij}^*\left(\om\right).$
}
\end{subtheorem}

\begin{proof} Consider the boundary stratum in $\overline{C}(\R^d)$ corresponding to a graph,
$$
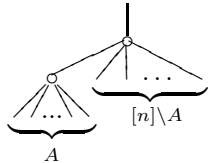
\Ba{c}
\begin{xy}
<10mm,0mm>*{\circ},
<10mm,0.8mm>*{};<10mm,5mm>*{}**@{-},
<0mm,-10mm>*{...},
<14mm,-5mm>*{\ldots},
<13mm,-7mm>*{\underbrace{\ \ \ \ \ \ \ \ \ \ \ \ \  }},
<14mm,-10mm>*{_{[n]\setminus A}};
<10.3mm,0.1mm>*{};<20mm,-5mm>*{}**@{-},
<9.7mm,-0.5mm>*{};<6mm,-5mm>*{}**@{-},
<9.9mm,-0.5mm>*{};<10mm,-5mm>*{}**@{-},
<9.6mm,0.1mm>*{};<0mm,-4.4mm>*{}**@{-},
<0mm,-5mm>*{\circ};
<-5mm,-10mm>*{}**@{-},
<-2.7mm,-10mm>*{}**@{-},
<2.7mm,-10mm>*{}**@{-},
<5mm,-10mm>*{}**@{-},
<0mm,-12mm>*{\underbrace{\ \ \ \ \ \ \ \ \ \ }},
<0mm,-15mm>*{_{A}},
\end{xy}
\Ea
$$
It is equal to the image of the map (\ref{6: circ_A composition}). Introduce in the neighbourhood of $\Img(\circ_A)$ a
coordinate chart,
$$
\cU_A=[0,\delta)\times  C_A^{st}(\R^d)\times C^{st}_{[n]-A+\bu}(\R^d),
$$
 corresponding to the metric version of the above graph
(see \S {\ref{3.3 higher dim version of Konts spaces}}). The boundary stratum $\Img (\circ_A)$ is given in this chart by the
equation $s=0$, where $s$ is the standard coordinate on the semi-open interval $[0,\delta)$ for some small $\delta\in \R^+$.
Thus, to prove the Proposition, we have to check that
$$
 \Omega_\Ga|_{s=0}= (-1)^{\var}p_1^*\left(\bigwedge_{e'\in E(\Ga/\Ga_A)}\om_{e'}\right)\wedge p_2^*\left(
 \bigwedge_{e''\in E(\Ga_A)}\om_{e''}\right)
$$
where $p_1: C_A^{st}(\R^d)\times C^{st}_{[n]-A+\bu}(\R^d)\rar C^{st}_{[n]-A+\bu}(\R^d)$ and
$p_2:  C_A^{st}(\R^d)\times C^{st}_{[n]-A+\bu}(\R^d)\rar C^{st}_{A}(\R^d)$ are natural projections, and  the sign
$ (-1)^{\var}$ is given just by regrouping of the factors $\om_e$ in  $\Omega_\Ga$, i.e.\ by the equality
$$
 \Omega_\Ga=(-1)^{\var}\left(\bigwedge_{e'\in E(\Ga)\setminus E(\Ga_A)}\om_{e'}\right)\wedge \left(
 \bigwedge_{e''\in E(\Ga_A)}\om_{e''}\right)
$$
All points parameterized by $A$ collapse in the limit $s\rar 0$  into a single point $x_0\in \R^d$;
we can represent this limit configuration as $s\rar 0$ limit of a configuration
 $\{x_0 + s x_i\}_{i\in A}\sqcup \{x_j\}_{j\in [n]\setminus A}$. Using now invariance of the forgetful map $\pi_{ij}$
 under the transformation $(x_i\rar x_0+ sx_i, x_j\rar x_0+ s x_j)$,
   we immediately conclude  that
$$
\bigwedge_{e'\in E(\Ga)\setminus E(\Ga_A)}\om_{e'}|_{s=0}=p_1^*\left(\bigwedge_{e'\in E(\Ga/\Ga_A)} \om_{e'}\right)\ \ \
\mbox{and}\ \ \
 \bigwedge_{e''\in E(\Ga_A)}\om_{e'}|_{s=0}=p_2^*\left(\bigwedge_{e''\in E(\Ga_A)} \om_{e''}\right)
$$
as required.
\end{proof}

For any vector space $V$ the associated vector space $\fg_d(V)$ is a representation of the operad $\fG$. Hence
we can apply Theorem {\ref{6: Theorem on Lie_infty(k) algebras}} and obtain the following

\begin{subcorollary}\label{6: Corollary on de Rham of Schouten-Bracket}
The $\caL_\infty$-structure, $\{\mu_n=\sum_{\Ga\in \fG_d(n)}c_\Ga\Phi_\Ga\}_{n\geq 2}$,
 induced on $\fg_d(V)$ by the de Rham field theory
(\ref{6: DeRFT on C(R^d)}) is homotopy non-trivial.
For $\om=Vol(S^{d-1})$, the standard homogeneous volume form on the
sphere,
the induced $\caL_\infty$-structure is precisely
the Poisson-Schouten structure (\ref{5: Poisson-Schouten bracket}).
\end{subcorollary}

\begin{proof} Let us first consider the case $\om=Vol(S^{d-1})$. Proof of the last sentence in the Corollary
is based on two Kontsevich's vanishing lemmas. It was shown  in
\cite{Ko} (for the case $d=2$) and
in \cite{Ko-fd} (for
$d\geq 3$)
that, for any integer $n\geq 3$ and for any graph $\Ga\in \cG_d(n)$,
one has
$$
\int_{\overline{C}_n(\R^{d})}\Omega_\Ga=0
$$
Hence, for $n\geq 3$,
$$
\mu_n=\sum_{\Ga\in \fG_d(n)} c_\Ga\Phi_\Ga=0
$$
and
\Beqrn
\mu_2 &=& \sum_{\Ga\in \fG_d((2))} c_\Ga\Phi_\Ga\\
      &=&c_{\Ga_1} \Phi_{\Ga_1} + c_{\Ga_2}\Phi_{\Ga_2}\\
      &=& \{\ \bu\ \}_{-1},
\Eeqrn
where we denoted
$
\Ga_1= \xy
(0,2)*{^{1}},
(8,2)*{^{2}},
(0,0)*{\bullet}="a",
(8,0)*{\bullet}="b",
\ar @{->} "a";"b" <0pt>
\endxy
$ and $
 \Ga_2= \xy
(0,2)*{^{2}},
(8,2)*{^{1}},
(0,0)*{\bullet}="a",
(8,0)*{\bullet}="b",
\ar @{->} "a";"b" <0pt>
\endxy
$ and used the fact that, by the normalization assumption on $\om$,  $c_{\Ga_1}=c_{\Ga_2}=1$.

\sip

A generic cohomologically non-trivial and normalized $(d-1)$-form $\om$ on $S^{d-1}$ is given by
$$
\om= Vol(S^{d-1}) + df
$$
for some semialgebraic function $f$ on $S^{d-1}$. In general, Kontsevich's vanishing lemmas
are not true for non-homogeneous propagators $\om$, and one obtains a non-trivial (but homotopy trivial)
$\caL_\infty$-deformation of the Poisson-Schouten bracket (see \cite{Me-Auto} for an explicit example).
This fact implies that the resulting $L_\infty$-structure is never homotopy equivalent to the trivial
(i.e.\ vanishing) $\caL_\infty$-structure on $\fg_d(V)$.
\end{proof}

\subsection{De Rham field theories on a a class of free topological operads} The above results for
$\overline{C}(\R^d)$  have obvious analogues for all operads of configuration spaces
considered in this paper. In fact, they hold true for any (coloured)
operad $\overline{C}=\{\overline{C}_n\}_{n\geq 1}$ in the category of smooth manifolds with
corners (or semialgebraic manifolds)
which, as an operad in the category of sets, is free, $\overline{C}=\cF ree \langle C\rangle$, and satisfies the following conditions:
\Bi
\item[(i)] each  $\overline{C}_n$ is a compact oriented (semialgebraic) manifold;
\item[(ii)] the $\bS$-space, $C=\{C_p\}_{p\geq 1}$, of generators is given by
$$
C_p:=\overline{C}_p\setminus \p\overline{C}_p,
$$
where  $\p\overline{C}_p$ is the boundary of $\overline{C}_n$;
\item[(iii)] as
$
\overline{C}=\cF ree \langle C \rangle$,  each manifold $\overline{C}_n$ is canonically stratified,
$
\overline{C}_n=\coprod_{T\in {\cT_n}} C_T$,
into a disjoint union of cartesian products,
$
C_T=\prod_{v\in V(T)} {C}_{\# In(v)}$. It is assumed that the set theoretic inclusion,
$C_T\hook \overline{C}_n$, is smooth (or semialgebraic) for any $T$.
\Ei
It follows from these assumptions that the operadic composition in $\overline{C}$ corresponding to a tree $T\in \cT_n$,
\Beq\label{6: operadic composition in overline C}
\mu_T: T\langle \overline{C} \rangle \lon \overline{C}_n,
\Eeq
is a smooth map which sends the Cartesian product $T\langle \overline{C}\rangle\simeq\prod_{v\in V(T)}
\overline{C}_{\# In(v)}$ into a boundary stratum in $\overline{C}_n$. We define
$(-1)^{\mu_T}$ to be $+1$ if this maps preserves orientations and $(-1)$ otherwise.
We also have,
$$
\p \overline{C}_n=\coprod_{T\in \cT_n^{boundary}} \mu_T\left( T\langle \overline{C} \rangle\right)
$$
for a suitable subfamily $\cT_n^{boundary}\subset \cT_n$. For concreteness and simplicity, we assume from now on
that  $\overline{C}=\{\overline{C}_n\}_{n\geq 1}$ is one of the operads of compactified configuration
spaces considered in \S2-4 so that
each connected component of the generator  $C_n$ of $\oC$ is a quotient of a configuration space, $\Conf_n(\V)$,
of pairwise distinct $[n]$-labeled points in a (subspace of a) vector space $\V$ modulo an action of an appropriate
 subgroup of $Aff(\V)$.

\sip

Let $\overline{\cC}_n$ be the vector space spanned by all possible  pairs,
$(\overline{C}_n^o, or)$,
consisting of a connected component, $\overline{C}_n^o$,  of $\overline{C}_n$ and a choice of
 orientation, $or$, on $\overline{C}_n^o$ modulo an equivalence relation, $(\overline{C}_n^o, -or)=
 -(\overline{C}_n^o, or)$. The standard orientation on $\overline{C}_n^o$ is denoted by $or^0$.
 The collection $\langle\overline{\cC}\rangle:=\{\langle\overline{\cC}_n\rangle\}_{n\geq 1}$ is naturally an
$\bS$-module. The {\em operad of fundamental chains}\, or {\em the face complex}\, of $\overline{C}$
is\footnote{The most natural way to define this notion is to use the theory of semialgebraic chains
developed in \cite{KoSo, HLTV}.}, by definition, a dg free operad, $\cF\cC hains(\overline{C}):=(\cF ree \langle  \overline{\cC}\rangle, \p)$
whose differential is read off from the above formula
for $\p\overline{C}_n$,
$$
\p (\overline{C}_n^o,or^o) = \sum_{T\in {\cT_n}^{boundary}} (-1)^{\mu_T} T\langle
(\overline{C}^o_\bu, or^o)\rangle.
$$
(see explicit formulae given in \S 2-4 for particular  examples).

\subsubsection{\bf Definition-Theorem}\label{6: Main Theorem-defintion}
{\em  Let $\fG^\circlearrowright$ be an arbitrary (coloured) operad of Feynman graphs. A {\em de Rham field theory}\,
(of type $\fG^\circlearrowright$)  on a semialgebraic (coloured) operad
 $\overline{C}=\{\overline{C}_n\}_{n\geq 1}$ is a morphism,
$
\Omega: (\check{\fG}^\circlearrowright, 0) \lon (\Omega_{\overline{C}}, d_{DR})$,
of dg cooperads.
Let
$$
\Ba{rccc}
\rho: & \fG^\circlearrowright & \lon  & \cE nd_X\\
      &    \Ga                  &\lon &   \Phi_\Ga
\Ea
$$
be  a representation of the operad of Feynman graphs in a graded vector space $X$
(or in a family of vector spaces
parameterized by the set of colours). Then a de Rham field theory, $\Omega$,
 on  $\overline{C}=\{\overline{C}_n\}$  makes the graded vector  space $X$ into a representation,
$$
\Ba{rccc}
\rho_Q: & \cF\cC hains(\overline{C}) & \lon  & \cE nd_X\\
      &    (\overline{C}_n^o,or^o)                &\lon &  \mu_n\in \Hom(X^{\ot n},X)
\Ea
$$
of the operad  of fundamental chains on $\overline{C}$ given explicitly by
\Beq\label{6: mu_n for generic top operad}
\mu_n:=\sum_{\Ga\in \fG(n)} \left(\int_{\overline{C}_n^o} \Omega(\Ga)\right)\Phi_\Ga.
\Eeq
The representation $\rho_Q$ is called a {\em quantization}\, of the representation $\rho$.}

\begin{proof} The natural pairing,
$$
\Ba{ccccc}
\cC hains(\overline{C}) & \times &   \Omega(\overline{C}) & \lon & \C\\
     \sigma             &        &    \omega  &\lon&     \int_\sigma \om
     \Ea
$$
gives rise to a morphisms of operads,
$$
\Ba{rccc}
\Omega^*:& \cC hains(\overline{C}) & \lon&   \fG^\circlearrowright \\
     & \sigma             &  \lon      &   \sum_{\Ga} \left(\int_\sigma \Omega_\Ga\right) \Ga
     \Ea
$$
The required representation  $\rho_Q$ is given finally by the composition,
$$
 \cF\cC hains(\overline{C}) \hook \cC hains(\overline{C}) \stackrel{\Omega^*}{\lon}  \fG^\circlearrowright
 \stackrel{\rho}{\lon}    \cE nd_X
$$
\end{proof}

\subsubsection{\bf Important remark}\label{6: Remark on de Rham}
In down to earth terms a de Rham field theory on $\oC$ is a family of maps
$$
\left\{
\Ba{rccc}
\Omega: & {G}_n& \lon & \Omega^{\bu}_{closed}(\overline{C}_n)\\
& \Ga & \lon & \Omega_\Ga
\Ea
\right\}_{n\geq 1, l\geq 0},
$$
such that, for any
boundary stratum in $\overline{C}_n(\R^d)$ given by the image of a reduced operadic composition
or, respectively, a ``full" composition one has
\Beq\label{6: 1 reduced de Rham  factorizatio on C_n}
\circ_A^*(\Omega_\Ga) = (-1)^{\var}\Omega_{\Ga/\Ga_A}\wedge \Omega_{\Ga_A}
\Eeq
or, respectively,
\Beq\label{6: 2 reduced de Rham  factorizatio on C_n}
\circ_{I_1\sqcup\ldots\sqcup I_p}^*(\Omega_\Ga) = (-1)^{\var}\Omega_{\Ga/\{\Ga_{I_1},\ldots,\Ga_{I_p}\}}\wedge
\Omega_{\Ga_{I_1}}\wedge \Omega_{\Ga_{I_2}}\wedge\ldots\wedge\Omega_{\Ga_{I_p}}.
\Eeq
where the signs are determined, as usually, by comparison of orientations on the sets of edges.
Once the above conditions hold true, formulae (\ref{6: mu_n for generic top operad}) define a quantization
of an arbitrary representation of an operad of Feynman diagrams. Here $\Omega^{\bu}_{closed}(\overline{C}_n)\subset
\Omega^{\bu}(\overline{C}_n)$
stands for the subspace of {\em closed}\, differential forms.

\sip

We are interested in this paper only in deformation quantization of representations of the  suboperad of
$\cC hains(\overline{C})$
spanned by the fundamental chains. As it stands,  Theorem~{\ref{6: Main Theorem-defintion}}
holds true for representations of the full chain operad as well; for  this Theorem to work only
 at the level of fundamental chains only one can weaken the notion of a de Rham field theory
 by requiring that factorization properties
 (\ref{6: 1 reduced de Rham  factorizatio on C_n})-({\ref{6: 2 reduced de Rham  factorizatio on C_n}}) hold only
  for {\em top degree}\, differential forms on $\overline{C}_\bu$. From now on we
 understand by a de Rham field theory on $\overline{C}$ such a weakened version as well.

\bip

\bip

{\large
\section
{\bf Examples of  quantized  representations of operads of Feynman diagrams}\label{9: section}
}
\mip

\subsection{Propagators on  $\overline{C}({\bbH}^{d})$}\label{9.1: subsection} Let us consider the operad
(see \S {\ref{3.3 higher dim version of Konts spaces}})
 $$
 \overline{C}(\bbH^{d})=\overline{C}_\bu(\R^d)\coprod \overline{C}_{\bu,\bu}(\bbH^{d})
 $$
 for $d\geq 2$. The space $\overline{C}_{2,0}(\bbH^{d})$ is a compact  $d$-dimensional
 semialgebraic space of the form
$$
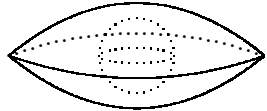
\Ba{c}
\xy
(0,0)*\ellipse(5,5){.};
(0,0)*\ellipse(5,1){.};
(-17,0)*-{};(17.0,-0)*-{}
**\crv{~*=<4pt>{.}(0,6)}
**\crv{(0,-6)};
(-17,0)*-{};(17.0,0)*-{}
**\crv{(0,-14)}
**\crv{(0,14)};
\endxy
\Ea
$$
Its boundary is a union $S_{in}^{d-1}\cup S_+^{d-1}\cup S_{-}^{d-1}$,
where $S_{in}^{d-1}$ is the sphere corresponding to two points in the upper half
space $\bbH^d$ collapsing to a point in  $\bbH^d$,  $S_+^{d-1}$ (respectively,  $S_+^{d-1}$)
 is the half-sphere corresponding to the limit configurations where the point labelled $1$ (respectively, $2$)
 approaches the boundary, $\p\overline{\bbH^d}$, of the closed upper-half space.

 \sip

The space $\overline{C}_{2,0}(\bbH^d)$ is homotopy equivalent to $S^{d-1}$. Any cohomologically
non-trivial  $PA$-form $\om$ on  $\overline{C}_{2,0}(\bbH^d)$ normalized so that
$$
\int_{S^{d-1}_{in}}\om^{in}=1, \ \ \ \mbox{where}\ \ \om^{in}:=\om|_{S_{in}^{d-1}},
$$
defines a de Rham field theory of type $\fG^{\uparrow\downarrow}$ on the
2-coloured operad $\overline{C}(\bbH^d)$ by the following two sets of maps
(see Remark {\ref{6: Remark on de Rham}})
\Beq\label{9: in/theory on C(H)}
\left\{
\Ba{rccc}
\Omega_{in}: & {\fG}(n)& \lon & \Omega^{(d-1)\# E(\Ga)}(\overline{C}_n(\R^d))\\
& \Ga & \lon & \Omega_{in}(\Ga):=\bigwedge_{e\in E(\Ga)} p_e^*(\om^{in})
\Ea
\right\},
\Eeq
and
\Beq\label{9: full theory on C(H)}
\left\{
\Ba{rccc}
\Omega: & {\sG}^{\uparrow\downarrow}(n+m)& \lon & \Omega^{(d-1)\# E(\Ga)}(\overline{C}_{n,m}(\bbH^d))\\
& \Ga & \lon & \Omega(\Ga):=\bigwedge_{e\in E(\Ga)} \pi_e^*(\om)
\Ea
\right\}
\Eeq
 Here
$$
p_e: \overline{C}_n(\R^d) \lon \overline{C}_2(\R^d), \ \ \ \mbox{and}\ \ \ \pi_e: C_{n,m}(\bbH^d) \lon \overline{C}_{2,0}(\R^d)
$$
are the natural forgetful maps. To prove
this claim we have to check factorization property (\ref{6: 1 reduced de Rham  factorizatio on C_n})
for the boundary stratum in $\overline{C}_{n,m}$ corresponding to graphs of the form,
$$
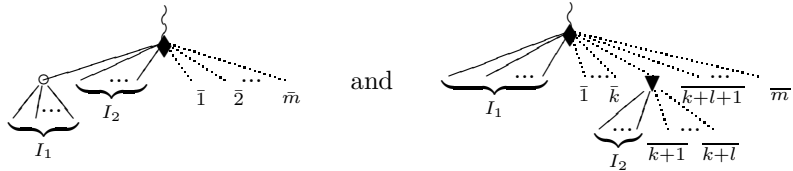
\Ba{c}
\begin{xy}
 <0mm,-0.5mm>*{\blacklozenge};
 <0mm,0mm>*{};<0mm,5mm>*{}**@{~},
 <0mm,0mm>*{};<-16mm,-5mm>*{}**@{-},
 <0mm,0mm>*{};<-11mm,-5mm>*{}**@{-},
 <0mm,0mm>*{};<-3.5mm,-5mm>*{}**@{-},
 <0mm,0mm>*{};<-6mm,-5mm>*{...}**@{},
 <0mm,0mm>*{};<16mm,-5mm>*{}**@{.},
 <0mm,0mm>*{};<8mm,-5mm>*{}**@{.},
 <0mm,0mm>*{};<3.5mm,-5mm>*{}**@{.},
 <0mm,0mm>*{};<11.6mm,-5mm>*{...}**@{},
   <0mm,0mm>*{};<17mm,-8mm>*{^{{\bar{m}}}}**@{},
<0mm,0mm>*{};<10mm,-8mm>*{^{{\bar{2}}}}**@{},
   <0mm,0mm>*{};<5mm,-8mm>*{^{{\bar{1}}}}**@{},
<-16mm,-13mm>*{\underbrace{\ \ \ \ \ \ \ \ }_{I_1}},
<-7mm,-8mm>*{\underbrace{\ \ \ \ \ \ \ \ }_{I_2}},
 (-16,-5)*{\circ}="a",
(-20,-10)*{}="b_1",
(-17,-10)*{}="b_2",
(-15,-9)*{...}="b_3",
(-12,-10)*{}="b_4",
\ar @{-} "a";"b_2" <0pt>
\ar @{-} "a";"b_1" <0pt>
\ar @{-} "a";"b_4" <0pt>
 \end{xy}\Ea       \ \ \ \ \mbox{and}\ \ \ \
\Ba{c}
\begin{xy}
 <0mm,-0.5mm>*{\blacklozenge};
 <0mm,0mm>*{};<0mm,5mm>*{}**@{~~},
 <0mm,0mm>*{};<-16mm,-6mm>*{}**@{-},
 <0mm,0mm>*{};<-11mm,-6mm>*{}**@{-},
 <0mm,0mm>*{};<-3.5mm,-6mm>*{}**@{-},
 <0mm,0mm>*{};<-6mm,-6mm>*{...}**@{},
<0mm,0mm>*{};<2mm,-9mm>*{^{\bar{1}}}**@{},
<0mm,0mm>*{};<6mm,-9mm>*{^{\bar{k}}}**@{},
<0mm,0mm>*{};<19mm,-9mm>*{^{\overline{k+l+1}}}**@{},
<0mm,0mm>*{};<28mm,-9mm>*{^{\overline{m}}}**@{},
<0mm,0mm>*{};<13mm,-16.6mm>*{^{\overline{k+1}}}**@{},
<0mm,0mm>*{};<20mm,-16.6mm>*{^{\overline{k+l}}}**@{},
 <0mm,0mm>*{};<11mm,-6mm>*{}**@{.},
 <0mm,0mm>*{};<6mm,-6mm>*{}**@{.},
 <0mm,0mm>*{};<2mm,-6mm>*{}**@{.},
 <0mm,0mm>*{};<17mm,-6mm>*{}**@{.},
 <0mm,0mm>*{};<25mm,-6mm>*{}**@{.},
 <0mm,0mm>*{};<4mm,-6mm>*{...}**@{},
<0mm,0mm>*{};<20mm,-6mm>*{...}**@{},
<6.5mm,-16mm>*{\underbrace{\ \ \ \ \   }_{I_2}},
<-10mm,-9mm>*{\underbrace{\ \ \ \ \ \ \ \ \ \ \ \   }_{I_1}},
 (11,-7)*{\blacktriangledown}="a",
(4,-13)*{}="b_1",
(9,-13)*{}="b_2",
(16,-13)*{...},
(7,-13)*{...},
(13,-13)*{}="b_3",
(19,-13)*{}="b_4",
\ar @{-} "a";"b_2" <0pt>
\ar @{.} "a";"b_3" <0pt>
\ar @{-} "a";"b_1" <0pt>
\ar @{.} "a";"b_4" <0pt>
 \end{xy}
\Ea
$$
These strata are images of the following operadic compositions,
$$
\circ_{I_1}: \overline{C}_{\# I_2+1, m}(\bbH^d)\times  \overline{C}_{I_1}(\R^d) \rar
\overline{C}_{n, m}(\bbH^d) \ \ \ \mbox{and}\ \ \
\circ_{\overline{k+1}}: \overline{C}_{\# I_1, m-l+1}(\bbH^d)\times  \overline{C}_{\# I_1,l}(\bbH^d) \rar
\overline{C}_{n, m}(\bbH^d),
$$
so that the required factorization conditions are given, respectively, by
$$
\circ_{I_1}^*(\Omega(\Ga))=(-1)^\epsilon \Omega(\Ga/\Ga_{I_1}) \wedge \Omega_{in}(\Ga_{I_1}) \ \ \ \mbox{and}\ \ \
\circ_{\bu}^*(\Omega(\Ga))=(-1)^\var
 \Omega(\Ga/\Ga_{A}) \wedge \Omega(\Ga_{A})
$$
where we set $A:=I_2\sqcup\{\overline{k+1},... ,\overline{k+l}\}$. Both these conditions can be checked
in the associated coordinate charts,
$$
\cU_1=[0,\delta)\times C_{\# I_2+1,m}^{st}(\bbH^d) \times C_{\# I_1}(\R^d) \ \ \mbox{and}\ \ \
\cU_1=[0,\delta)\times C_{\# I_1+1,m-l}^{st}(\bbH^d) \times C_{\# I_1,l}(\bbH^d)
$$
 using the same
arguments as in
 in  the proof of Theorem {\ref{6: prop on propagator on C(R^d)}} but with two small subtleties:
\Bi
\item[(1)]
in the limit $s\rar 0$, $s\in [0,\delta)$, the differential form $\wedge_{e''\in E(\Ga_{I_1})}\om_e$ does
{\em not}\, stay  invariant but tends to  $\wedge_{e''\in E(\Ga_{I_1})}\om^{in}_e$ explaining thereby appearance of
$\Omega_{in}$ in the r.h.s.\ of the first factorization condition;
\item[(2)] contrary to the above case (1)  the
differential form $\wedge_{e''\in E(\Ga_{A})}\om_e$ in the limit $s\rar 0$  does stay invariant; the reason is that the
point into which $A$-labelled points collapse is located on the boundary of the upper-half space and,
with no loss of generality, can be placed at $0\in \overline{\bbH^d}$; then, by definition of the coordinate chart
$\cU_2$, the parameter
$s$ acts on the $A$-labelled  configuration as an ordinary dilation while each $\om_{e''}$
is both translation- (along $\R^{d-1}\subset \overline{\bbH^d}$) and dilation-invariant.
\Ei

Therefore, we have the following

\subsubsection{\bf Theorem}\label{6: Theorem on quant in C(H^d)} {\em Given a representation,
$\rho: \fG^{\uparrow\downarrow}\rar \cE nd_{V_c,V_o}$,
of the 2-coloured operad of Feynman diagrams in a pair of vector spaces $(V_c,V_o)$. Then, for any
 any homologically
non-trivial smooth (or  $PA$) differential $(d-1)$-form $\om$ on  $\overline{C}_{2,0}(\bbH^d)$, the formulae
$$
\nu_n:= \sum_{\Ga\in G_n} \left(\int_{\overline{C}_n(\R^d)} \Omega_{in}(\Ga)\right) \rho(\Ga)
$$
define a $L_\infty\{d-1\}$-structure on the vector space $V_c$, and, for any MC-element $\ga$
in $(V_c, \nu_\bu)$, the formulae,
\Beq\label{6: general formula for A_infty}
\mu^\ga_m(v_1,\ldots, v_m):=\sum_{n=0}^\infty \frac{\hbar^n}{n!}\sum_{\Ga\in G_{n+m}}
\left(\int_{\overline{C}_{n,m}(\bbH^d)} \Omega(\Ga)\right) \rho(\Ga)(\ga^{\ot n}\ot v_1\ot\ldots\ot v_m),\ \ v_i\in V_0,
\Eeq
define,
\Bi
\item[(i)]  in general, a non-flat $A_\infty$-algebra structure in $V_o[[\hbar]]$ for $d=2$,
\item[(ii)] in general, a non-flat $L_\infty\{d-1\}$-algebra structure in $V_o[[\hbar]]$ for $d\geq 3$.
\Ei
}

\subsection{ Examples of propagators}\label{6: examples of propagators}
\subsubsection{\bf Homogeneous volume form on a sphere}
 A differential form on $C_{2,0}(\bbH^d)$,
$$
\om_0:= p^*\left(Vol(S^{d-1})\right)
$$
where $p: C_{2,0}(\bbH^d)\rar C_2(\R^d)$ is the natural projection, extends to the compactification
 $\overline{C}_{2,0}(\bbH^d)$ and hence defines a non-trivial de Rham field theory on the operad
 $\overline{C}(\bbH^d)$.

\subsubsection{\bf Kontsevich (anti)propagators}  Let $[z_1,z_2]$ be an arbitrary configuration in
$C_{2,0}(\bbH^d)$ and $(z_1,z_2)$ be its arbitrary
representative in $\Conf_{2,0}(\bbH^d)$. Let
$$
ds^2:= \frac{dx_1^2+\ldots+ dx_d^2}{x_d^2}
$$
be the standard hyperbolic metric on $\bbH^d:=\{(x_1,\ldots, x_d)\in \R^d\, |\, x_d>0\}$.
Let $S_H^{d-1}(z_1)$ and $\om^0_K$ be the unit hyperbolic sphere centered at $z_1$ and, respectively,
 its induced normalized volume form  (with respect to the above metric).
Using  the unique hyperbolic geodesic  $g(z_1,z_2)$ from $z_1$ to $z_2$,
$$
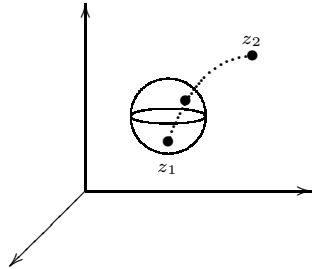
\Ba{c}
{\xy
(-4,6.5)*+{};(7.2,18)*-{}
**\crv{~*=<2pt>{.}(0,18)};
(-2,5)*\ellipse(5,5){};
(-2,5)*\ellipse(5,1){};
(-4,6.5)*{\bullet},
(-1.7,12)*{\bullet},
(-4,3)*{_{z_1}},
(7.2,18)*{\bullet}="f",
(7.2,20)*{_{z_2}},
(-15,0)*{}="a",
(15,0)*{}="b",
(-25,-10)*{}="c",
(-15,25)*{}="d",
\ar @{->} "a";"b" <0pt>
\ar @{->} "a";"d" <0pt>
\ar @{->} "a";"c" <0pt>
\endxy}
\Ea
$$
we define a smooth map
$$
\Ba{rccc}
p_H: & \Conf_{2,0}(\bbH^d) & \lon & S_H^{d-1}(z_1)\\
     & (z_1,z_2) &\lon & g(z_1,z_2)\cap S_H^{d-1}(z_1)
\Ea
$$
and set
$$
\om_K(z_1,z_2):=p_H^*(\om^0_K)
$$

This form is $\R^*\ltimes \R^d$-invariant and hence defines a closed homologically non-trivial
$(d-1)$-form on
$C_{2,0}(\bbH^d)$. Moreover, it extends to the compactification $\overline{C}_{2,0}(\bbH^d)$.
For $d=2$ this form is precisely the Kontsevich propagator used in the construction of his formality
map \cite{Ko}; hence the notation.
We set $\om_{\overline{K}}(z_1,z_2):=\om_K(z_2,z_1)$ and call it  {\em Kontsevich antipropagator}.

\subsection{Deformation quantization of associative algebras of polyvector fields}
\label{9:  quantization of associative algebras of polyvector fields}
For any graded vector space $V$,
the associated pair $X_c=X_o=\fg_d(V)$ carries a natural representation
of the Feynman operad $\fG^{\uparrow\downarrow}$. For propagators considered in  \S {\ref{6: examples of propagators}}
the induced $L_\infty\{d-1\}$ algebra structure on $X_c$ is, by Corollary \S
{\ref{6: Corollary on de Rham of Schouten-Bracket}},
precisely the Poisson-Schouten bracket (\ref{5: Poisson-Schouten bracket}). Hence any propagator from
\S {\ref{6: examples of propagators}} together with an MC element $\ga$
in the degree $1-d$ Lie  algebra $(\fg_d(V), \{\ \bu\ \})$ makes, by Theorem {\ref{6: Theorem on quant in C(H^d)}},
 $\fg_d(V)$ into
a (non-flat) $A_\infty$-algebra for $d=2$ or into a $L_\infty\{d-1\}$-algebra for $d\geq 3$.
The case $d\geq 3$ will be considered in more detail below while in this subsection we assume from now on
that $d=2$ so that  $(\fg_2(V), \{\ \bu\ \})$ is the Schouten algebra, $\cT_{poly}(V)$,
 of formal polyvector fields on $V$. Then an arbitrary Poisson structure, $\ga\in \cT_{poly}(V)$, and any
propagator from the set $\{\om_0, \om_K, \om_{\overline{K}}\}$ defines by formulae
(\ref{6: general formula for A_infty}) a non-flat $A_\infty$ structure
on $\cT_{poly}(V)$. For $\ga=0$ the only graph
contributing to (\ref{6: general formula for A_infty}) is the one consisting of two vertices on the real line
 with no edges,
\Beq\label{9: graph with 2 vertices no edges}
\xy
 (-10,0)*{}="a",
(10,0)*{}="b",
(+3,0)*{\bu},
(-3,0)*{\bu},
\ar @{->} "a";"b" <0pt>
\endxy
\Eeq
so that $\mu^{\ga=0}_m=0$ for $m\neq 2$ and $\mu_2^{\ga=0}$ is the ordinary product of polyvector
fields. Thus, for $\ga\neq 0$, formulae (\ref{6: general formula for A_infty}) describe a deformation
quantization of that ordinary product in $\cT_{poly}(V)$.
Using de Rham field theories discussed in \S {\ref{10: section}} below
using a de Rham field theory on $\widehat{\fC}(\bbH)$
one can show that the three non-flat
$A_\infty$-algebras structures on $\cT_{poly}(V)$ corresponding to propagators $\{\om_0, \om_K, \om_{\overline{K}}\}$
 are all homotopy equivalent to each other. The simplest of them is given by the propagator $\om_0$ as
in this case
$$
\int_{\overline{C}_{n,0}(\bbH} \om_0(\Ga)=0
$$
for any graph $\Ga\in \fG_{n,2n-2}$ with $n\geq 2$  (the reason is that the differential form
  $\om_0(\Ga)$ is invariant under vertical
translations and hence vanishes for degree reasons). Hence the only graph contributing to
$\mu_0^\ga$ is the following one
$$\xy
 (-10,0)*{}="a",
(10,0)*{}="b",
(0,8)*{\bu},
\ar @{->} "a";"b" <0pt>
\endxy
$$
so that $\mu_0^\ga=\ga$. The next term of the $A_\infty$-structure on $\cT_{poly}(V)$ is given by,
$$
\mu_1= \sum_{n\geq 1}\sum_{\Ga\in \fG_{n+1, 2n-1}} \int_{\overline{C}_{n,1}}\Omega(\Ga) \Phi_\Ga.
$$
The initial $n=1$  term in this sum is given by the graph
$$
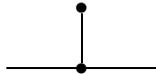
\xy
 (-10,0)*{}="a",
(10,0)*{}="b",
(0,8)*{\bu}="c",
(0,0)*{\bu}="d",
\ar @{-} "a";"b" <0pt>
\ar @{-} "c";"d" <0pt>
\endxy
$$
of weight $1$; the associated operator $\Phi_\Ga$ is $\{\ga
\bu\ldots \}$. Using the reflection $z \rar -\overline{z}$ one can easily check that
$\int_{\overline{C}_{n,1}}\Omega(\Ga)=0$ for $n$ even. For a graph $\Ga\in  \fG_{n+1, 2n-1}$
let us denote by $\Ga_i$ its drawing in $C_{n,1}$ with the vertex labelled by $i$ put at the point
in $\R$; for any $n\geq 2$ we have, by Kontsevich vanishing Lemma 6.4 in \cite{Ko},
$$
\sum_{i=1}^n \int_{\overline{C}_{n,1}(\bbH)}\Omega(\Ga_i)=\int_{\overline{C}_{n+1}(\C)}\Omega(\Ga)=0
$$
so that
 $$
 \mu_1^\ga(\mu_0^\ga)=\mu_1^\ga(\ga)=\{\ga\bu \ga\}=0
 $$
 which is in a full agreement with the claim that
formulae (\ref{6: general formula for A_infty}) for $\om=\om_0$ define a non-flat $A_\infty$-structure.
The graphs
$$
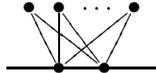
\xy
 (2,8)*{\ldots},
 (-10,0)*{}="a",
(10,0)*{}="b",
(-7,8)*{\bu}="c1",
(-3,8)*{\bu}="c2",
(7,8)*{\bu}="c3",
(-3,0)*{\bu}="d1",
(3,0)*{\bu}="d2",
\ar @{-} "a";"b" <0pt>
\ar @{-} "c1";"d1" <0pt>
\ar @{-} "c1";"d2" <0pt>
\ar @{-} "c2";"d1" <0pt>
\ar @{-} "c2";"d2" <0pt>
\ar @{-} "c3";"d1" <0pt>
\ar @{-} "c3";"d2" <0pt>
\endxy
$$
give, for example, non-zero contributions to $\mu_2^\ga$ while
$$
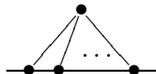
\xy
(2,2)*{\ldots},
 (-10,0)*{}="a",
(10,0)*{}="b",
(0,8)*{\bu}="c",
(-7,0)*{\bu}="d1",
(-3,0)*{\bu}="d2",
(7,0)*{\bu}="d3",
\ar @{-} "a";"b" <0pt>
\ar @{-} "c";"d1" <0pt>
\ar @{-} "c";"d2" <0pt>
\ar @{-} "c";"d3" <0pt>
\endxy
$$
and their products contribute to higher homotopies, $\mu^\ga_{m\geq 3}$. These graphs are not, of course,
the only non-trivial contributions;
for example, the weight of the following graph
$$
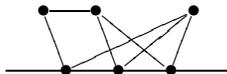
\xy
 (-15,0)*{}="a",
(15,0)*{}="b",
(10,8)*{\bu}="c",
(-10,8)*{\bu}="c1",
(-3,8)*{\bu}="c2",
(-7,0)*{\bu}="d1",
(0,0)*{\bu}="d2",
(7,0)*{\bu}="d3",
\ar @{-} "a";"b" <0pt>
\ar @{-} "c";"d1" <0pt>
\ar @{-} "c";"d2" <0pt>
\ar @{-} "c";"d3" <0pt>
\ar @{-} "c1";"d1" <0pt>
\ar @{-} "c1";"c2" <0pt>
\ar @{-} "c2";"d2" <0pt>
\ar @{-} "c2";"d3" <0pt>
\endxy
$$
is non-zero so that it contributes, in general,  to $\mu_3^\ga$. As we shall see below, this
non-flat $A_\infty$-structure on $\cT_{poly}(V)\simeq \odot^\bu V \ot \wedge^\bu V^*$
interpolates, in a sense,
two Koszul dual deformation quantizations on $ \odot^\bu V$ and  $\wedge^\bu V^*$ associated
with propagators $\om_K$ and, respectively, $\om_{\overline{K}}$. Note in this connection that
$\om_0=\frac{1}{2}(\om_K+\om_{\overline{K}})$.

\subsection{Kontsevich's formality maps} As we saw in the previous section,  any cohomologically non-trivial
 differential 1-form, $\om$, on $\overline{C}_2(\bbH)$
 gives us a non-trivial $A_\infty$
structure on $\cT_{poly}(V)$. The latter algebra contains two subalgebras,
$$
\odot^\bu(V^*)=:\f_V\ \ \ \mbox{and}\ \ \ \ \odot^\bu(V[-1])=:\f_{V^*[1]}
$$
which can be viewed as the rings of smooth formal function on affine manifolds $V$ and, respectively, $V^*[1]$.
Note that the pairs,
$$
\left(X_c=\cT_{poly}(V), X'_o=\f_V\right)\ \ \ \mbox{and}\ \ \ \left(X_c=\cT_{poly}(V), X''_o=\f_{V^*[1]}\right)
$$
carry naturally representations of Feynman operads $\fG^\downarrow$ and, respectively,  $\fG^\uparrow$ (see
\S {\ref{7: representations of fG-arrows}}). It is an elementary exercise to check that
{\em any cohomologically nontrivial normalized differential 1-form $\om$ on  $C_{2,0}(\bbH)$
satisfying the condition $\om|_{S_{+}^{1}=0}$ (respectively,  $\om|_{S_{-}^{1}}=0$)
defines by formulae (\ref{9: in/theory on C(H)}) and (\ref{9: full theory on C(H)}) a De Rham field theory
on $\overline{C}(\bbH)$ of type $\fG^\downarrow$ (respectively, $\fG^\uparrow$).
}

\sip

The Kontsevich propagator $\om_K$ satisfies the condition $\om_K|_{S_{+}^{1}=0}$ and hence defines
by formulae (\ref{6: general formula for A_infty}) a morphism of operads,
$$
\cF\cC hains(\overline{C}(\bbH^d) \lon \cE nd_{\{X_c=\cT_{poly}(V), X'_o=\f_V\}}
$$
which is the same as his famous formality map. Any MC element $\ga$ in the Schouten algebra
$\cT_{poly}(V)$ makes $\f_V[[\hbar]]$ into a (non-flat, in general) $\cA_\infty$-algebra. The same propagator $\om_K$
and the same MC element $\ga$ make also $\cT_{poly}(V)[[\hbar]]$ into a (non-flat) $\cA_\infty$-algebra.
It is clear that the natural inclusion
$$
\f_V[[\hbar]] \lon \cT_{poly}(V)[[\hbar]]
$$
is a morphism of these $\cA_\infty$-algebras which we denote by $F_K: (\f_V[[\hbar]], \om_K) \rar
(\cT_{poly}(V)[[\hbar]], \om_K)$, the symbol $\om_K$ indicates the origin of the induced $\cA_\infty$-structures.

\sip

Similarly, the Kontsevich antipropagator $\om_{\overline{K}}$
satisfies the condition $\om_K|_{S_{-}^{1}=0}$ and hence defines
 a morphism of operads,
$$
\cF\cC hains(\overline{C}(\bbH^d) \lon \cE nd_{\{X_c=\cT_{poly}(V), X''_o=\f_{V^*[1]}\}}
$$
which coincides again with Kontsevich's  formality map. Thus any MC element $\ga$ in
the Schouten algebra
$\cT_{poly}(V)$ makes $\f_{V^*[1]}[[\hbar]]$ into a (non-flat) $\cA_\infty$-algebra.
 The same antipropagator $\om_{\overline{K}}$
and the same MC element $\ga$ make also $\cT_{poly}(V)[[\hbar]]$ into a (non-flat) $\cA_\infty$-algebra.
The natural inclusion
$$
\f_{V^*[1]}[[\hbar]] \lon \cT_{poly}(V)[[\hbar]]
$$
is a morphism of  $\cA_\infty$-algebras which we denote by $F_{\overline{K}}:
(\f_{V^*[1]}[[\hbar]], \om_{\overline{K}}) \rar
(\cT_{poly}(V)[[\hbar]], \om_{\overline{K}})$.

\sip

Using the 4-coloured operad $\widehat{\fC}(\bbH)$
 one can construct  explicitly $\cA_\infty$-quasi-isomorphisms
(in any direction),
$$
(\cT_{poly}(V)[[\hbar]], \om_K) \rightleftarrows (\cT_{poly}(V)[[\hbar]], \om_0) \leftrightarrows
(\cT_{poly}(V)[[\hbar]], \om_{\overline{K}})
$$
and hence a diagram of canonical $\cA_\infty$-morphisms,
$$
(\f_V[[\hbar]], \om_K) \hook (\cT_{poly}(V)[[\hbar]], \om_K)
\rightleftarrows (\cT_{poly}(V)[[\hbar]], \om_0) \leftrightarrows
(\cT_{poly}(V)[[\hbar]], \om_{\overline{K}}) \hookleftarrow (\f_{V^*[1]}[[\hbar]], \om_{\overline{K}}).
$$
It would be interesting to see if this diagram can be used to define Koszul duality
of  generic {\em non-flat}\, $\cA_\infty$-structures on $\f_V$ and $\f_{V^*[1]}$. For a subclass of MC elements $\ga$ which make
$(\f_V[[\hbar]], \om_K)$ and $(\f_{V^*[1]}[[\hbar]], \om_{\overline{K}})$ into {\em flat}\,
$\cA_\infty$ algebras
the Koszul duality was already established in \cite{Sh3, CFFR} (in fact, not all $\cA_\infty$ algebras appearing in  \cite{CFFR} must be flat).

\subsection{Deformation quantization of the Schouten bracket} In this section the propagator degree, $d-1$, of all
operads under consideration is set by default to $2$, i.e.\ $d=3$.

\sip

For any vector space $V$, consider the following pair of Schouten-Poisson algebras,
$$
\fg_2(V)=\left(\odot^\bu(V^*\oplus V[-1]),\ \ \{\ \bu\ \}_{-1}\right)\simeq \cT_{poly}(V)
$$
and
$$
\fg_3(V^*[1]\oplus V)=\left( \odot^\bu(V^*\oplus V[-1] \oplus V[-2]\oplus V^*[-1]),
\{\ \ , \ \}_{-2}\right).
$$
If $\{x^\al\}$ is a basis in $V^*$, and $\{\psi_\al\}$, $\{\eta^\al\}$ and $\{y_\al\}$ the associated
bases in $V[-1]$,
$V^*[-1]$ and, respectively, $V[-2]$, then
$$
\fg_2(V)\simeq \R[[x^\al,\psi_\al]]\ \ \mbox{with}\ \
\{f\bullet g\}_{-1}:= \sum_\al
\frac{ f\overleftarrow{\p}}{\p \psi_\al} \frac{\overrightarrow{\p} g}{\p x^\al}\ +
(-1)^{|f||g|+ f+g}
\frac{g \overleftarrow{\p} }{\p \psi_\al} \frac{\overrightarrow{\p} f}{\p x^\al}
$$
and
$$
\fg_3(V[1]\oplus V^*)\simeq \R[[x^\al,\psi_\al, \eta^\al,y_\al]]
$$
with
$$
\{f, g\}_{-2}:=\sum_\al \left(
\frac{ f\overleftarrow{\p}}{\p y_\al} \frac{\overrightarrow{\p} g}{\p x^\al}
+ \frac{ f\overleftarrow{\p}}{\p \eta^\al} \frac{\overrightarrow{\p} g}{\p \psi_\al}
- (-1)^{|f||g|}\frac{ g\overleftarrow{\p}}{\p y_\al} \frac{\overrightarrow{\p} f}{\p x^\al}
- (-1)^{|f||g|}\frac{ g\overleftarrow{\p}}{\p \eta^\al} \frac{\overrightarrow{\p} f}{\p \psi_\al}
\right)
$$
where we use the fact that $|\psi_a|=-|x^a|+1$, $|\eta^a|=|x^a|+1$, $|y_a|=-|x^a|+2$.
Note that $\delta:= \sum_a \eta^a y_a$ is an MC element in $(\fg_3(V[1]\oplus V^*), \{\ ,\ \}_{-2})$
making the latter into a {\em dg}\, Lie algebra with the de Rham type differential $d:=\{\delta,\ \}_{-2}$;
its cohomology is equal to $\R$.

\sip

There is a natural projection
$$
\pi_{\downarrow}: \fg_3(V^*[1]\oplus V) \rar \fg_2(V)
$$
 so that
formulae completely analogous to (\ref{5: Def of Phi_Ga}) and (\ref{5: Phi_Ga 2 colour}) make the pair
\Beq\label{9: the pair}
X_1= \fg_3(V[1]\oplus V^*)\ \ \  \mbox{and}\ \ \  X_2=\cT_{poly}(V)
\Eeq
into a representation of the 2-coloured operad $\fG^{\downarrow}$, i.e.\ there exists a canonical morphism
of 2-coloured operads,
$$
\rho: \fG^\downarrow\lon \cE nd_{X_1,X_2}.
$$
On the other hand, the propagator $\om_K$ on $\overline{C}_{2,0}(\bbH^3)$ defines, by
Theorem~{\ref{6: Main Theorem-defintion}}, a morphism of 2-coloured operads,
$$
\Ba{ccc}
\f\cC \caL_\infty:=\cF\cC hains(\overline{C}(\bbH^3)) & \lon&   \fG^\downarrow \\
     \sigma             &  \lon      &   \sum_{\Ga} \left(\int_\sigma \Omega_\Ga\right) \Ga
     \Ea
$$
so that the composition of the above two morphisms makes the pair (\ref{9: the pair})
into an open-closed homotopy Lie 3-algebra, i.e.\ defines (see \S {\ref{3.3 higher dim version of Konts spaces}})
\Bi
\item[(i)] a $L_\infty\{2\}$-structure,
$\nu=\{\nu_n: \odot^n (X_1[3])\rar X_2[4]\}_{n\geq 1}$, on $X_1$, i.e.\ an ordinary
$L_\infty$-structure on $X_1[2]$; as restriction of the propagator  $\om_K$ to the inner 2-sphere in
$\overline{C}_{2,0}(\bbH^3)$ is the standard homogeneous volume form on $S^2$, we conclude by
Corollary~{\ref{6: Corollary on de Rham of Schouten-Bracket}} that this $L_\infty\{2\}$-structure
is precisely the Poisson-Schouten bracket $\{\ ,\ \}_{-2}$ in $X_1$;

\item[(ii)] a $L_\infty\{1\}$-structure,
$\mu=\{\mu_{0,n}: \odot^n (X_2[2])\rar X_1[3]\}_{n\geq 1}$, on $X_2$;
this structure is given by graphs $\Ga$ whose vertices lie in $\C=\p \overline{\bbH^3}$; as weights of all such
graphs w.r.t.\ $\om_K$ are equal to zero, we conclude that this $L_\infty\{1\}$-structure is trivial, i.e.\
all operations $\mu_{0,n}=0$ for $n\geq 2$;

\item[(iii)] a  $L_\infty$-morphism, $F$, of Lie algebras,
$$
F:  (X_1, \{\ ,\ \}_{-2}) \lon  (\mbox{Coder}(\odot^\bu (X_2[2])), [\ ,\ ])
$$
given by sums over graphs $\Ga\in \fG^\downarrow$. Note that as $\deg \om_K=2$,  only those
graphs contribute to $F$ which satisfy the condition $3n+2m-3=2l$, where $n$ is the number of vertices
of $\Ga$ lying in $\bbH^3$, $m$ is the number of vertices of $\Ga$ lying in the boundary plane
$\C=\p \overline{\bbH^3}$ and $l$ is the number of edges of $\Ga$; put another way, $l=3k+m$
and $n=2k+1$ for some $k\in \N$.
\Ei

At the first glance it might seem  that,  contrary to the case of $\overline{C}(\bbH)$ where the graph (\ref{9: graph with 2 vertices no edges})
encodes a natural graded commutative structure in $X_2=\f_V$, there is no graph with non-zero weight
in the case of
$\overline{C}(\bbH^3)$-theory which would encode the canonical Lie$\{1\}$ structure on $X_2=\cT_{poly}(V)$.
However, this is not quite so: if $\ga$ is an arbitrary MC element of the Poisson-Schouten 2-algebra
$(X_1=\fg_3(V^*[1]\oplus V^), \{\ ,\ \}_{-2})$, then the aforementioned $L_\infty$-morphism $F$ sends $\ga$ into an
MC element, $F(\ga)$, of the Lie algebra  $(\mbox{Coder}(\odot^\bu (X_2[2])), [\ ,\ ])$ which is the same as
a $L_\infty\{1\}$-structure on $X_2=\cT_{poly}(V)$. If we take $\ga=\delta$, then the only graph contributing
to $F(\delta)$ is the following one
$$
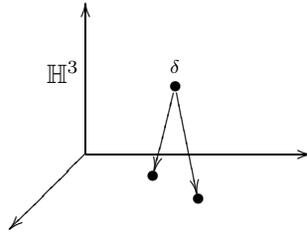
\Ba{c}
{\xy
(-3,9)*{\bullet}="0",
(-6,-3)*{\bullet}="1",
(0,-6)*{\bullet}="2",
(-3,11)*{^{\delta}},
(-18,11)*{\bbH^3},
(-15,0)*{}="a",
(15,0)*{}="b",
(-25,-10)*{}="c",
(-15,20)*{}="d",
\ar @{->} "a";"b" <0pt>
\ar @{->} "a";"d" <0pt>
\ar @{->} "a";"c" <0pt>
\ar @{->} "0";"1" <0pt>
\ar @{->} "0";"2" <0pt>
\endxy}
\Ea
$$
and has weight $1$. Therefore, $F(\delta)$ is nothing but the ordinary Schouten structure, $\{\ \bu\ \}_{-1}$,
 in $\cT_{poly}(V)$!

 \sip

The conclusion is that any MC element, $\ga$, in the {\em dg}\, Lie algebra
$$
\left(\fg_3(V^*[1]\oplus V), \{\ ,\ \}_{-2},
d=\{\delta,\ \}_{-2}\right)
$$
 gives a deformation, $F(\ga)$, of the Schouten bracket in $\cT_{poly}(V)$. As the former dg
Lie algebra is cohomologically trivial, any such a deformation of $\{\ \bu\ \}_{-1}$ must be homotopy trivial.
Still homotopy trivial does not mean trivial, and we
 can get
 some funny deformations of  $(\cT_{poly}(V), \{\ \bu\ \}_{-1})$ in this way.

\subsubsection{\bf Proposition}{\em To any Lie coalgebra structure,
$$
\Delta: V \lon V\wedge V
$$
in a vector space $V=\R^d$ there corresponds an MC element in the dg  Lie algebra
\Beq\label{9: dg -2 algebra}
(\fg_3(V^*[1]\oplus V)\simeq  \R[[x^\al,\psi_\al, \eta^\al,y_\al]], \{\ ,\ \}_{-2},
d=\{\delta,\ \}_{-2}),
\Eeq
given explicitly by the following formal power series
$$
\ga^\Delta:=\frac{1}{2}\sum_{\al,\be,delta} C_{\al\be}^\delta \eta^\al \eta^\be y_\delta + \sum_{\al,\be}\left(\sum_{n=1}^\infty
\frac{B_n}{n!} \sum_{\xi_i,\zeta_i}C_{\al \xi_1}^{\zeta_1}C_{\zeta_1\xi_2}^{\zeta_2}\cdots C_{\zeta_{n-1}\xi_n}^\be
x^{\xi_1}x^{\xi_2}\cdots x^{\xi_n}
\right)\eta^\al y_\be
$$
where $C_{\al\be}^\delta\in \R$ are the structure constants of $\Delta$ in a basis $x^\al$,
$$
\Delta(x^\delta)=\sum_{\al\be} C_{\al\be}^\delta x^\al x^\be
$$
and $B_n$ are the Bernoulli numbers, $B_1=-\frac{1}{2}$, $B_2=\frac{1}{6}$, $B_3=0$, $B_4=-\frac{1}{30}$, etc.
}

\begin{proof} Set
$$
\gamma:= - \frac{1}{2}\sum_{\al,\be,\delta} C_{\al\be}^\delta \eta^\al \eta^\be \psi_\delta +  \sum_{\al,\be}C_\al^\be(x) \eta^\al y_\be
$$
for some $C_\al^\be(x)\in \R[[x^\al]]$. This degree 3 element satisfies the MC equation in (\ref{9: dg -2 algebra}),
$$
d\ga + \frac{1}{2}\{\ga, \ga\}_{-2}=0,
$$
if and only if
$$
C_{\al \xi}^{\delta}C^\xi_{\be\ga} + C_{\ga \xi}^{\delta}C^\xi_{\al\be} + C_{\be \xi}^{\delta}C^\xi_{\ga\al}=0
$$
and the series $\hat{C}_\al^\be(x):= \delta_\al^\be + C_\be^\al(x)$ satisfies the equations,
\Beq\label{9: Second eqs for hatC_b^a}
C_{\al\be}^\ga \hat{C}_\ga^\delta(x) = \hat{C}_\al^\ga(x)\p_\ga \hat{C}_\be^\delta(x) - \hat{C}_\be^\ga(x)\p_\ga \hat{C}_\al^\delta(x).
\Eeq
The first equations are the co-Jacobi identities for $\Delta$. The second equations have a
nice geometric meaning. Let $\fg:=V^*$ be the Lie algebra dual to the coalgebra $V$; in a basis $\{e_\al\}$ of $\fg$
dual to $\{x^\al\}$ the Lie algebra structure is given by,
$$
[e_\al,e_\be]=\sum_\ga C_{\al\be}^\ga e_\ga.
$$
Equations (\ref{9: Second eqs for hatC_b^a}) are equivalent to saying that the map,
$$
\Ba{rccc}
f: & \fg & \lon & \cT_\fg\\
&  e_\al & \lon & \sum_\be\hat{C}_\al^\be(x)\frac{\p}{\p x^\be}
\Ea
$$
is a map of Lie algebras. Here  $\cT_\fg$ is the Lie algebra of formal vector fields on $\fg$, i.e.\ derivations
of the completed symmetric tensor algebra $\widehat{\odot^\bu} V$. It was proven in
\cite{Me-Bern} (see Corollary 4.1.2 there) that for any
Lie algebra $\fg$ such a canonical map exists and is given by the formulae
$$
\hat{C}_\al^\be(x)=\delta_\al^\be + \sum_{n=1}^\infty
\frac{B_n}{n!} \sum_{\xi_i,\zeta_i}C_{\al \xi_1}^{\zeta_1}C_{\zeta_1\xi_2}^{\zeta_2}\cdots C_{\zeta_{n-1}\xi_n}^\be
x^{\xi_1}x^{\xi_2}\cdots x^{\xi_n}
$$
This fact completes the proof.
\end{proof}
\sip

The Schouten brackets and the wedge product of polyvector fields make
$\cT_{poly}(\R^d)\simeq \R[[x^\al,\psi_\al]]$ into a Gerstenhaber algebra so that the Schouten bracket,
$\{\ \bu\ \}_{-1}$, is
uniquely determined by its values on the generators $x^\al$ and $\psi_\al$,
\Beqrn
\{x^\al\bu x^\be\}_{-1} &=& 0,\\
\{\psi_\al\bu x^\be\}_{-1} &=& \delta_\be^\al,\\
\{\psi_\al \bu \psi_\be\}_{-1} &=& 0.
\Eeqrn

\begin{corollary}\label{9: corollary on extension of Schouten}
For any Lie coalgebra structure on $V=\R^d$ the following formulae,
\Beqrn
\{x^\al\bu x^\be\}_{-1} &=& 0,\\
\{\psi_\al\bu x^\be\}_{-1} &=& \delta_\be^\al + \sum_{n=1}^\infty
\frac{B_n \hbar^n}{n!}\sum_{\xi_i,\zeta_i}C_{\al \xi_1}^{\zeta_1}C_{\zeta_1\xi_2}^{\zeta_2}\cdots C_{\zeta_{n-1}\xi_n}^\be
x^{\xi_1}x^{\xi_2}\cdots x^{\xi_n}
,\\
\{\psi_\al \bu \psi_\be\}_{-1} &=& \hbar\sum_\ga C_{\al\be}^\ga\psi_\ga
\Eeqrn
give a 1-parameter deformation of the standard Gerstenhaber algebra structure on $\cT_{poly}(V)$.

\end{corollary}

\bip

{\large
\section{\bf Towards the theory of open-closed morphisms of deformation quantizations}\label{10: section}
}

\mip

\subsection{De Rham field theory on the first model of $\widehat{\fC}_\bu(\C)$}\label{Subsect 10.1}
In \S \ref{4': Section on Mor(L_infty)} we introduced
two different configuration space models for the operad $\cM or(L_\infty)$, $\widehat{\fC}(\C)$ and $\widehat{C}(\bbH)$, the first being
a compactification of the quotient space $\Conf_n(\C)/\{z\rar z + \C\}$ and the second being that of the space
$\Conf_n(\bbH)/\{z\rar \R^+ z + \R\}$;  moreover, each model was equipped with two non-isomorphic semialgebraic structures.
In \cite{Me-Auto} we studied de Rham field theories on $\widehat{C}(\bbH)$ and obtained some exotic $L_\infty$-automorphisms of the
Lie algebra $\cT_{poly}(\R^d)$ for any $d$. In this section we outline an analogous theory for the compactification
$\widehat{\fC}(\C)$.

\sip

Consider
the 2-coloured operad of semialgebraic manifolds
$$
\widehat{\fC}(\C):=\overline{C}_\bu(\C)\sqcup \widehat{\fC}_{\bu}(\C)\sqcup
\overline{C}_\bu(\C)
$$
equipped with a semialgebraic structure given by the embedding (\ref{4': first compactifn of fC(C)}),
and let
$$
\fG_{\widehat{\fC}(\C)}=\fG_{\overline{C}_\bu(\C)}\bigoplus
\fG_{\widehat{\fC}_{\bu}(\C)}\bigoplus \fG_{\overline{C}_\bu(\C)}
$$
be an associated 2-coloured operad operad of Feynman diagrams. Each summand in the latter direct sum is
spanned by (equivalence classes of) directed graphs satisfying all the conditions given in
\S {\ref{7: An operad of Feynman graphs fG}}; the only important difference from the definition of $\fG$
in \S {\ref{7: An operad of Feynman graphs fG}} is that graphs in each
summand should be understood as drawn on the corresponding summand of the operad $\widehat{\fC}(\C)$
and equipped thereby with the corresponding colour. We leave to the reader as an exercise to write
down explicitly composition rules in $\fG_{\widehat{\fC}(\C)}$ with the help of the basic 1-coloured formula
(\ref{5: composition in Feynman operad}).

\sip

A de Rham field theory on $\widehat{\fC}(\C)$, that is a morphism of dg cooperads,
$$
\Omega: \left(\check{\fG}_{\widehat{\fC}(\bbH)},0\right) \lon \left(\Omega_{\widehat{\fC}(\C)}, d_{DR}\right),
$$
is the same as
\Bi
\item[(i)] a pair of de Rham field theories on $\overline{C}_\bu(\C)$, i.e.\ a pair of maps
$$
 \left\{
\Ba{rccc}
\Omega^{in}: & {G}_{n, l}& \lon & \Omega^{l}_{closed}(\overline{C}_{n}(\C)
)\\
& \Ga & \lon & \Omega^{in}_\Ga
\Ea
\right\}_{n\geq 2} \ \ \ \mbox{and}\ \ \ \left\{
\Ba{rccc}
\Omega^{out}: & {G}_{n, l}& \lon & \Omega^{l}_{closed}(\overline{C}_{n}(\C)
)\\
& \Ga & \lon & \Omega^{out}_\Ga
\Ea
\right\}_{n\geq 2}
$$
satisfying on the boundary strata of $\overline{C}_{n}(\C)$ the factorization property (\ref{6: de Rham on C_n R^d}), and
\item[(ii)] a map
$$
 \left\{
\Ba{rccc}
\Xi: & {G}_{n, l}& \lon & \Omega^{l}_{closed}(\widehat{\fC}_{n}(\C))\\
& \Ga & \lon & \Xi_\Ga
\Ea
\right\}_{n\geq 2}
$$
such that  $\Xi_{\Ga_{opp}}=-\Xi_\Ga$ and, for any $\Ga\in  \cG_{n, 2n-3}$,
and any boundary embeddings
$$
j_A: \widehat{\fC}_{n-\# A +1}(\C)\times  \overline{C}_{\# A}(\C) \hook \widehat{\fC}_{n}(\C),
\ \ \
j_{A_1,...,A_k}: \overline{C}_{k}(\C)\times \widehat{\fC}_{n-\# A_1 +1}(\C)\times\ldots\times
\widehat{\fC}_{n-\# A_k +1}(\C) \hook  \widehat{\fC}_{n}(\C)
$$
 one has
\Beq\label{10: boundary in}
j_A^*(\Xi_\Ga) \simeq (-1)^{\sigma_A}\Xi_{\Ga/\Ga_A}  \wedge \Omega^{in}_{\Ga_A},
\Eeq
\Beq\label{10: boundary out}
j_{A_1,...,A_k}^*(\Xi_\Ga)\simeq (-1)^{\sigma_{A,,...,A_k}}
\Omega^{out}_{\Ga/\{\Ga_{A_1},...,\Ga_{A_k}\}}\wedge \Xi_{\Ga_{A_1}}\wedge\ldots
\wedge \Xi_{\Ga_{A_k}},
\Eeq
where the sign $(-1)^{\sigma_{A_1...A_k}}$ is defined by the equality
$$
\f_\Ga=(-1)^{\sigma_{A_1...A_k}} \f_{\Ga/\{\Ga_{A_1},...,\Ga_{A_k}\}}\wedge \f_{\Ga_{A_1}}
\wedge\ldots\wedge  \f_{\Ga_{A_k}},
$$
i.e.\ it is given just by a rearrangement of the wedge product of edges of $\Ga$.
\Ei

\sip

The space $\widehat{\fC}_{2}(C)$ is the cylinder $S^1\times [0,+\infty]$, see (\ref{5: fC_2(C) cylinder}), whose boundaries
are circles $S^1_{in}$ and, respectively, $S^1_{out}$, corresponding  to two points moving too close to each other
and, respectively, to two points moving too far away from each.
For any pair of integers $i,j\in [n]$, $i\neq j$,
there is an associated forgetting map,
\Beq\label{Ch3: p-forgetting map}
\Ba{rccc}
p_{ij}: & \fC_{n}(\C)  &\lon & \fC_{2}(\C) \\
        & (z_1, \ldots, z_n) & \lon & (z_i, z_j),
\Ea
\Eeq
which extends to a smooth map of their compactifications,
$$
{p}_{ij}:\widehat{\fC}_{n}(\C)\rar
\widehat{\fC}_{2}(\C).
$$
 Hence, for any closed differential 1-form  $\om$ on $\widehat{C}_{2}$ the pull-back ${p}_{ij}^*(\om)$ is a well-defined
 one-form on  $\widehat{\fC}_{n}(\C)$. 
 In particular, for any graph $\Ga\in G_{n,l}$  and any edge $e\in E(\Ga)$ there is an associated
 differential form $p_e^*(\om)\in \Omega^1( \widehat{\fC}_n(\C))$, where $p_e:=p_{ij}$ if the edge $e$ begins at the vertex labelled by
 $i$ and ends at the vertex labelled by $j$.
Similar forgetful maps $\pi_{ij}$ and $\pi_e$ can be defined for the configuration spaces $\overline{C}_\bu(\C)$.

\mip

Let $\om$ be an arbitrary closed differential 1-form on $\widehat{\fC}_{2}(\C)$ such that the restrictions
\Beq\label{10: restrictions of propagator to both boundary S^1}
\om_{in}:= \om|_{S^1_{in}} \ \ \ \mbox{and}\ \ \  \om_{out}:= \om|_{S^1_{out}}
\Eeq
define cohomologically non-trivial 1-forms on the circle normalized so that
$$
\int_{S^1} \om_{in} =\int_{S^1} \om_{out}=2\pi.
$$
We call such a differential form a {\em propagator}\, on $\widehat{\fC}_2(\C)$.
Define a series of maps,
\Beq\label{Ch4: Omega-in maps}
\Ba{rccc}
\Omega^{in}: & G_{n,l}\hspace{-2mm} & \rar &\hspace{-2mm} \Omega^l(\widehat{\fC}_n(\C))\vspace{2mm} \\
& \Ga \hspace{-2mm}& \rar & \hspace{-2mm}\Omega^{in}_\Ga:=\hspace{-3mm}\displaystyle
\bigwedge_{e\in E(\Ga)}\hspace{-2mm}
\frac{{\pi}^*_e\left(\om_{in}\right)}{2\pi}
\Ea
\ \ \ \ \ \ \ \ \ \ \
\Ba{rccc}
\Omega^{out}: & G_{n,l}\hspace{-2mm} & \rar &\hspace{-2mm} \Omega^l(\widehat{\fC}_n(\C))\vspace{2mm}  \\
& \Ga \hspace{-2mm}& \rar & \hspace{-2mm}\Omega^{in}_\Ga:=\hspace{-3mm}\displaystyle
\bigwedge_{e\in E(\Ga)}\hspace{-2mm}
\frac{{\pi}^*_e\left(\om_{out}\right)}{2\pi}
\Ea
\Eeq
and
\Beq\label{Ch4: Xi maps}
\Ba{rccc}
\Xi: &  G_{n,l} & \lon & \Omega^l(\widehat{C}_{n,0})\vspace{2mm} \\
& \Ga & \lon & \Xi_\Ga:=\displaystyle\hspace{-3mm} \bigwedge_{e\in E(\Ga)}\hspace{-1mm}
\frac{{p}^*_e\left(\om\right)}{2\pi}.
\Ea
\Eeq

\mip
\subsubsection{\bf Theorem}\label{10: Theorem on de Rham field th on first model fC} {\em For any propagator on $\widehat{\fC}_{2}(\C)$ the
associated  data (\ref{Ch4: Omega-in maps})-(\ref{Ch4: Xi maps})
define a de Rham field theory on the semialgebraic operad $\widehat{\fC}(\C)$.
}

\begin{proof}
Equation (\ref{10: boundary in}) is equivalent to the following one,
$$
\int_{ \widehat{\fC}_{(n-\# A)\sqcup  \bu}(\C)\times  \overline{C}_{\# A}(\C)}
j_A^*(\Xi_\Ga)= (-1)^{\sigma_A}\int_{{\widehat{\fC}_{(n-\# A)\sqcup \bu}(\C)}} \Xi_{\Ga/\Ga_A} \int_{\overline{C}_{\# A}(\C)}  \Omega^{in}_{\Ga_A}.
$$
Studying the embedding $\widehat{\fC}_{(n-\# A)\sqcup \bu}(\C)\times  \overline{C}_{\# A}(\C) \hook \widehat{\fC}_{n}(\C)$
in local coordinates defined by metric graphs,
 one easily sees that both sides of
the above equation are zero unless $\Ga_A$ has $2\# A-3$ edges
(so that  $\Omega^{in}_{\Ga_A}$ is a top degree form on $ \overline{C}_{\# A}(\C)$) in which case the
equality is almost obvious (cf.\ the proof of Proposition {\ref{6: prop on propagator on C(R^d)}}).

\sip

Consider next, for a partition $[n]=A_1\sqcup\ldots \sqcup A_k$,  the associated boundary stratum of the form,
$$
j_{A_1,...,A_k}: {C}_{k}(\C)\times {\fC}_{A_1}(\C)\times\ldots\times
{\fC}_{A_k}(\C) \hook  \widehat{\fC}_{n}(\C).
$$
By definition, this is a subset of $\widehat{\fC}_{n}(\C)$ obtained in the limit
 $\tau\rar\infty$ from a class of configurations in $Conf_{n}(\C)$ determined by the data:
\Bi
\item[(a)]  a configuration $\tau\cdot p=(\tau z_1, \ldots, \tau z_k)\in \Conf_n(\C)$ obtained from a
standard configuration
$p=(z_1, \ldots, z_k)\in C^{st}_n(\C)\simeq C_k(\C)$ by $\tau$-dilation,
\item[(b)]  a collection of configurations, $p_1\in  {\fC}^{st}_{\# A_1}(\C)$,
$\ldots$, $p_k\in  {\fC}^{st}_{\# A_k}(\C)$,
$1\leq i\leq k$, placed, respectively, at the positions $\tau z_1, \ldots, \tau z_k$ in $\bbH$, that is,
a configuration $(\tau z_1 + p_1,\ldots,  \tau z_k + p_k)\in \Conf_n(\C)$.
\Ei
We have, therefore,
$$
 j_{A_1,...,A_k}^*(\Xi_\Ga)=(-1)^{\sigma_{A,,...,A_k}}
\bigwedge_{e\in E(\Ga/\{\Ga_{A_1}\ldots \Ga_{A_k}\} )}\hspace{-1mm}
\frac{\underset{\tau\rar +\infty}{\lim}{p}^*_e\left(\om\right)}{2\pi}\ \
\prod_{i=1}^k
\bigwedge_{e\in E(\Ga_{A_i})}\hspace{-1mm}
\frac{\underset{\tau\rar +\infty}{\lim}{p}^*_e\left(\om\right)}{2\pi}.
$$
As $p_e^*(\om)$ is invariant under translations we have for each $i\in [k]$
$$
\underset{\tau\rar +\infty}{\lim} \bigwedge_{e\in E(\Ga_{A_i})}\hspace{-1mm}
\frac{{p}^*_e\left(\om\right)}{2\pi}=\bigwedge_{e\in E(\Ga_{A_i})}\hspace{-1mm}
\frac{{p}^*_e\left(\om\right)}{2\pi}.
$$
On the other hand, for any $e\in E(\Ga/\{\Ga_{A_1}\ldots \Ga_{A_k}\} )$,
$$
\underset{\tau\rar +\infty}{\lim}{p}^*_e\left(\om\right)=\pi_e^*(\om_{out})
$$
so that
$$
\bigwedge_{e\in E(\Ga/\{\Ga_{A_1}\ldots \Ga_{A_k}\} )}\hspace{-1mm}
\frac{\underset{\tau\rar +\infty}{\lim}{\fp}^*_e\left(\om\right)}{2\pi}=
\Omega^{out}_{\Ga/\{\Ga_{A_1},...,\Ga_{A_k}\}}.
$$
These two facts prove
equality (\ref{10: boundary out}) and hence complete the proof of the Theorem.
\end{proof}

\subsubsection{\bf Corollary}\label{10: Corollary on Lie_infty auto}
 {\em For any propagator  $\om$ on $\widehat{\fC}_{2}(\C)$ and any graded vector space $V$ there are
associated
\Bi
\item[(i)] two $\caL_\infty$-structures on  $\cT_{poly}(V)$,
$$
\mu^{in}=\{\mu^{in}_n: \odot^n\cT_{poly}(V) \rar \cT_{poly}(V)[3-2n]\}_{n\geq 2} \ \ \mbox{and}\ \ \
\mu^{out}=\{\mu^{out}_n: \odot^n\cT_{poly}(\R^d) \rar  \cT_{poly}(\R^d)[3-2n]\}_{n\geq 2},
$$
given by formulae (\ref{Ch4: induced_Leib_infty})-(\ref{Ch4: c_gamma_on_C_n})
for $\Omega=\Omega^{in}$ and, respectively, $\Omega=\Omega^{out}$; $\mu^{in}_2$ and $\mu^{out}_2$ coincide with the Schouten bracket;
\item[(ii)] a $\caL_\infty$ morphism,
$$
F=
\left\{
\Ba{rccc}
F^{Leib}_n: & \odot^n\cT_{poly}(V) &\lon &  \cT_{V}(\R^d)[2-2n]\\
& \ga_1\ot \ldots\ot \ga_n     & \lon & F_n(\ga_1, \ldots, \ga_n)
\Ea
\right\}_{n\geq 1},
$$
from $\mu^{in}$-structure to  $\mu^{out}$-structure
given by the formulae,
\Beq\label{Ch4: induced_Aut_Leib_infty}
F_n(\ga_1, \ldots, \ga_n):=
\left\{
\Ba{cl}
\Id & \mbox{for}\ n=1,\\
 \sum_{\Ga\in \fG_{n,2n-2}} C_{\Ga}\Phi_\Ga(\ga_1, \ldots, \ga_n)
 & \mbox{for}\ n\geq 2
 \Ea
 \right.
\Eeq
with
\Beq\label{Ch4: weight C_Ga}
C_\Ga:= \int_{\widehat{\fC}_{n}(\C)} \displaystyle \bigwedge_{e\in E(\Ga)}\hspace{-1mm}
\frac{{p}^*_e\left(\om\right)}{2\pi}.
\Eeq
\Ei
}

\subsubsection{\bf Remark} If $\om$ is a propagator on $\widehat{\fC}_2(\C)$ such that its restrictions
(\ref{10: restrictions of propagator to both boundary S^1}) to both boundary circles
coincide with the standard homogeneous volume form $d Arg(z_1-z_2)$ on $S^1$, then
formulae (\ref{Ch4: induced_Aut_Leib_infty})
and (\ref{Ch4: weight C_Ga}) define a universal $\caL_\infty$ automorphism of the Schouten algebra of
 polyvector fields. The propagator
$$
\om(z_1,z_2)=d Arg(z_1-z_2)
$$
is well-defined on $\widehat{\fC}_2(\C)$ and satisfies the aforementioned boundary conditions. However, all the weights
(\ref{Ch4: weight C_Ga}) with $n\geq 2$ vanish in this case so that the associated automorphism $F$ is
just the identity map.
Any other smooth propagator on $\widehat{\fC}_2(\C)$ is of the form
\Beq\label{10: om=dArg + df}
\om(z_1,z_2)=d Arg(z_1-z_2) + df(z_1,z_2)
\Eeq
for some smooth (or semialgebraic) function $f$ on  $\widehat{\fC}_2(\C)$. It was proven in \cite{Me-Auto} that any such
 propagator defines a $\caL_\infty$-automorphism of $\cT_{poly}(V)$ which is homotopy equivalent to the trivial one.
Thus the class of smooth propagators on  $\widehat{\fC}_2(\C)$ can not give us an exotic (i.e.\ homotopy non-trivial)
universal automorphism of the Schouten algebra. Hence one should try using {\em singular}\, propagators for that purpose (cf.\ \cite{Ko2}) which have at most simple polar singularity at the collapsing stratum $S^1\subset \widehat{\fC}_2(\C)$. We conjecture that
$$
\om(z_1, z_2):= \frac{1}{\ii}d\log\frac{z_1-z_2}{1+ |z_1-z_2|}
$$
is an example of such suitable propagator on  $\widehat{\fC}_2(\C)$ which gives us an exotic universal $L_\infty$-automorphism  of the Schouten algebra via Corollary {\ref{10: Corollary on Lie_infty auto}}. We hope to discuss elsewhere our motivation for that conjecture, and its relation to the Deligne-Drinfeld conjecture on the structure of the Grothendieck-Teichm\"uller algebra $grt$.

De Rham field theories on the upper half space model of $\cM or(L_\infty)$
have been studied in \cite{Me-Auto}.

\subsection{\bf On (auto)morphisms  of deformation quantizations} Any de Rham field theory, $\Omega$,
on the  4-coloured operad,
$$
\widehat{\fC}(\bbH):=
\underbrace{\overline{C}_\bu(\C)\bigsqcup \overline{C}_{\bu,\bu}(\bbH)}_{in}
\bigsqcup \widehat{\fC}_{\bu,\bu}(\bbH) \bigsqcup \widehat{\fC}_{\bu}(\C)\bigsqcup \underbrace{
\overline{C}_\bu(\C)\bigsqcup \overline{C}_{\bu,\bu}(\bbH)}_{out},
$$
gives us an open-closed homotopy morphism between the two deformation quantizations corresponding to the
 ``in"  and``out" colours, respectively.

\sip
It follows from Tamarkin's proof of Kontsevich formality theorem that the Grothendieck-Teichm\"uller group,
$GRT$,
acts (up to homotopy)  on deformation quantizations.
Following the general philosophy one can expect that this action can be explicitly presented as an
open-closed homotopy morphism determined by a propagator
 $\om\in \Omega^1_{\widehat{\fC}_{2,0}(\bbH)}$ which vanishes on all
boundary strata of $\widehat{\fC}_{2,0}(\bbH)$ except the inner cylinders, both Kontsevich eyes and the spaces
 $B_1$ and $B_2$ shown in the following picture
$$
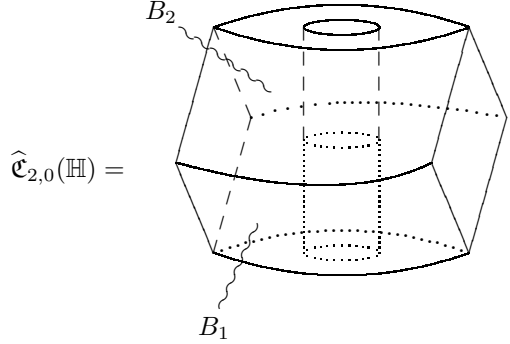
\widehat{\fC}_{2,0}(\bbH)=\Ba{c}
\xy
 <-17mm,-10mm>*{B_1};<-11mm,5mm>*{}**@{~},
<-24mm,32mm>*{B_2};<-9mm,21mm>*{}**@{~},
 <5mm,15mm>*{};<5mm,0mm>*{}**@{.},
<-5mm,15mm>*{};<-5mm,0mm>*{}**@{.},
<5mm,30mm>*{};<5mm,15mm>*{}**@{--},
<-5mm,30mm>*{};<-5mm,15mm>*{}**@{--},
%
<-17mm,0mm>*{};<-12mm,18mm>*{}**@{--},
<-17mm,0mm>*{};<-22mm,12mm>*{}**@{-},
<-17mm,30mm>*{};<-22mm,12mm>*{}**@{-},
<-17mm,30mm>*{};<-12mm,18mm>*{}**@{--},
<17mm,0mm>*{};<12mm,12mm>*{}**@{-},
<17mm,0mm>*{};<22mm,18mm>*{}**@{-},
<17mm,30mm>*{};<22mm,18mm>*{}**@{-},
<17mm,30mm>*{};<12mm,12mm>*{}**@{-},
(0,15)*\ellipse(5,1){-};
(0,7.5)*\ellipse(5,1){.};
(0,0)*\ellipse(5,1){.};
(-17,0)*-{};(17.0,-0)*-{}
**\crv{~*=<4pt>{.}(0,6)}
**\crv{(0,-6)};
(-17,30)*-{};(17.0,30)*-{}
**\crv{(0,36)}
**\crv{(0,24)};
(-22,12)*+{};(12.0,12)*-{}
**\crv{(0,6)};
(-12,18)*-{};(22.0,18)*+{}
**\crv{~*=<4pt>{.}(0,22)}
\endxy
\Ea
$$
It is not hard to construct such a de Rham field theory on $\widehat{\fC}(\bbH)$ out of a smooth propagator
on $\widehat{\fC}_{2,0}(\bbH)$ along the lines explained in the previous section; however, as $\widehat{\fC}_{2,0}(\bbH)$
is contractible to the topological circle, any such a theory gives us a {\em homotopy trivial}\ open-closed
transformation. Therefore again only {\em singular}\, propagators can, in principle, give us explicit formulae for the action of GRT
on deformation quantizations. One such  singular propagator on $\overline{C}_{\bu,\bu}(\bbH)$
was introduced by Kontsevich in \cite{Ko2}
but a rigorous proof of his claim that this propagator works indeed is not yet available in the literature
(to the best knowledge of the author).

\bip
\appendix
\renewcommand{\thesubsection}{{\bf A.\arabic{subsection}}}
\renewcommand{\thesubsubsection}{{\bf A.\arabic{subsection}.\arabic{subsubsection}}}
\section{ Operads and coloured operads \cite{BM, GJ,GK, LV} }
\subsection{ Trees}
Let $\cT$ be the set of all possible  connected genus $0$ graphs
constructed
 from the following 1-vertex  directed graphs called
$n$-{\em corollas},
\Beq\label{corolla}
\begin{xy}
 <0mm,0mm>*{\bu};
 <0mm,0mm>*{};<0mm,4.5mm>*{}**@{-},
<-0.4mm,-0.2mm>*{};<-8mm,-3mm>*{}**@{-},
 <-0.5mm,-0.3mm>*{};<-4.5mm,-3mm>*{}**@{-},
 <0mm,0mm>*{};<0mm,-2.6mm>*{\ldots}**@{},
 <0.5mm,-0.3mm>*{};<4.5mm,-3mm>*{}**@{-},
 <0.4mm,-0.2mm>*{};<8mm,-3mm>*{}**@{-};
<0mm,-5mm>*{\underbrace{\ \ \ \ \ \ \ \ \ \ \ \ \ \ }};
<0mm,7mm>*{^{the\ output\ leg}};
<0mm,-7mm>*{_{n\ \ input\ legs}};
 \end{xy}, \ \ \ n\geq 0,
\Eeq
by taking their disjoint unions then and gluing some output legs  to the same number of input legs. The resulting graph is called a {\em tree}.
The glued legs are called the {\em edges} of the tree, and all the rest legs are called the  {\em legs} of the tree. Each tree $T$ has, by construction, a unique output leg. The set of edges of $T$ is denoted by $E(T)$, the set of vertices by $V(T)$, and the set
of input legs by $L(T)$. If $\# L(T)=n$, then $T$ is called an $n$-tree. The subset of $\cT$
consisting of $n$-trees is denoted by $\cT_n$.
Note that every  edge as well as every leg of a tree is naturally directed;
we tacitly assume in all our pictures that the direction flow runs from the bottom to the top.
For a vertex $v$, we denote by $In_v$ the set of its input legs.
\sip

Let $I$ be a finite set. An $I$-tree is an $\# I$-tree equipped with a bijection
$I\rar L(T)$. The set of $I$-trees is denoted by $\cT_I$.

\subsection{$\cS$-modules} Let $\cC$ be a symmetric monoidal category
with the tensor product denoted by $\ot$ and the unit object denoted by $\id$. We assume  that the category $\cC$ has small limits and colimits, and that for any object $O$ the functor $O\ot$
preserves colimits.

\sip

Let $\cS$ be the groupoid of finite sets. A functor $\f:\cS\rar \cC$ is called an $\cS$-{\em module}.
The subcategory, $\bS$, of $\cS$ whose objects are the sets
$[n]$, $n\in \N$, and morphisms are the permutation groups $\bS_n$ is the full sceleton of $\cS$. The restriction of $\f$ to $\bS$ is called an $\bS$-module. One can reconstruct $\f$ from its restriction to $\bS$
by setting $\f(I)$ to be the colimit
$$
\f(I):=\left( \bigoplus_{[\# I]\rar I} \f([\# I]) \right)_{\bS_{\# I}}.
$$

Given an $\cS$-module $\f$ and a tree $T$, one constructs a {\em decorated tree} $T\langle\f\rangle$ as the colimit,
$$
T\langle\f\rangle:= \left( \bigoplus_{\sigma: [\# V(T)]\rar V(T)} \f(In_{\sigma(1)})\ot
 \f(In_{\sigma(2)})\ot\ldots\ot  \f(In_{\sigma(\# V(T))})\right)_{\bS_{\# V(T)}},
$$
and then defines an $\cS$-module, $\cT\langle\f\rangle: \cS \rar \cC$, which is given on a finite set $I$
as the following colimit,
$$
\cT\langle\f\rangle(I):= \bigoplus_{T\in \cT_I} T\langle\f\rangle.
$$
The association $\cT: \f \rightsquigarrow \cT\langle\f \rangle$ is an endofunctor in the category
of $\cS$-modules which comes canonically equipped with a natural transformation
$t:\cT\circ \cT\rar \cT$ as, for any finite set $I$, there is a natural in $\f$ ``tautological" map
$$
\cT\langle\cT\langle\f\rangle\rangle(I) \lon   \bigoplus_{T\in \cT_I} T\langle\f\rangle=\cT\langle\f\rangle
$$
which sends a tree $T'\in\cT\langle\cT\langle\f\rangle\rangle$ whose vertices, $v$, are decorated by some $\f$-decorated trees, $T''_v\in \cT\langle\f\rangle$, into the $\f$-decorated tree $T$ obtained from $T'$ by replacing each $v$ with $T''_v$.

\subsection{Definitions of an operad}
\subsubsection{\bf First definition} A {\em non-unital  operad}\, is an $\cS$-module, $\f$, together with a natural transformation,
$$
\mu:\cT\langle\f\rangle \lon  \f
$$
such that the  diagrams,
$$
\xymatrix{
\cT\langle\cT\langle\f\rangle\rangle \ar[r]^{\cT(\mu)} \ar[d]_{t} &   \cT\langle\f\rangle\ar[d]_\mu\\
 \cT\langle\f\rangle \ar[r]_{\mu} & \f
}\ \ \ \mbox{and} \ \ \ \
\xymatrix{
\f \ar[r]^{\nu} \ar[dr]_{Id} &   \cT\langle\f\rangle\ar[d]_\mu\\
  & \f
}
$$
commute. Here $\nu: \f\rar \langle\cT\rangle$ stands for the trasnformation which identifies
$\f(I)$ with the decorated $I$-corolla.

\sip

We omit the definition of a {\em morphism}\ of (non-unital) operads as it is obvious.

\subsubsection{\bf Example} For any $\cS$-module $\cE$ the $\cS$ module $\cT\langle\cE \rangle$
has a natural structure of a non-unital operad called the {\em free operad generated by $\cE$}.

\subsubsection{\bf Example} For any vector space $V$ the $\bS$-module
 $\cE nd_V:=\{\Hom(V^{\ot n},V)\}_{n\geq 0}$
has a natural structure of an operad called the {\em endomorphism operad}\, of $V$.

\sip

Let $\uparrow$ denote the exceptional tree without vertices, and set
$$
\cT^+:=\cT \sqcup \uparrow
$$
be the enlarged family of trees. For any $\cS$-module $\f$ set the decorated graph
$\uparrow\hspace{-1mm}\langle\f\rangle$ to be $\id$, the unit in the category $\cC$.
\sip

 An {\em  operad with unit}\, is defined by replacing in the as bove definition
 of a non-unital operad the symbol $\cT$ by the symbol $\cT^+$.

\subsubsection{\bf Second definition}  A {\em non-unital  operad}\, is an $\cS$-module, $\f$, together with a family of natural transformations,
$$
\{\mu_T: T\langle\f\rangle \lon  \f\}_{T\in \cT}
$$
parameterized by all possible tress from the family $\cT$ such that
\Beq\label{graph-associativity}
\mu_T=\mu_{T/T'}\circ \tilde{\mu}_{T'}
\Eeq
 for any subtree $T'\subset T$. Here $T/T'$ stands for the tree obtained from $T$ by shrinking
 the whole subtree $T'$ into a single $L(T')$-corolla, and
  $\tilde{\mu}_{T'}: T\langle \f \rangle \rar (T/T')\langle \f\rangle$ stands for the natural
  transformation
which equals $\mu_{T'}$ on the decorated vertices lying in $T$ and which is identity on all other vertices of $T$.
Enlarging the family of trees from $\cT$ to $\cT^+$ as above, one obtains similarly
a definition of an operad with unit.

\subsubsection{\bf Third definition}\label{Appendix: Third definition}
 A {\em non-unital  operad}\, is an $\cS$-module, $\f$, together with a family
of natural transformations,
$$
\{\mu_T: T\langle\f\rangle \lon  \f\}_{T\in \cT^{red}}
$$
parameterized by the subfamily $\cT^{red}\subset \cT$ of 2-vertex trees such that
for any three vertex tree $\bar{T}$ the diagram
$$
\xymatrix{
\bar{T}\langle\f\rangle \ar[r]^{\tilde{\mu}_{T'}} \ar[d]_{\tilde{\mu}_{T''}} &   \bar{T}/T'\langle\f\rangle\ar[d]^{\mu_{\bar{T}/T'}}\\
 \bar{T}/T''\langle\f\rangle \ar[r]_{\mu_{\bar{T}/T''}} & \f}
$$
commutes. Here $T'$ and $T''$ stand for the two only possible 2-vertex subtrees of $\bar{T}$.
\sip

For arbitrary finite sets $I$ and $J$ let $T_I$ be the $I$-corolla, $T_J$ be the $J$-corolla,
and, for any $i\in I$, let $T_{(I-i)\sqcup J}$ be the 2-vertex tree obtained by gluing the output vertex
of $T_J$ into the $i$-labeled input leg of $T_I$. The associated
composition
$$
\mu_{T_{(I-i)\sqcup J}}: \f(I)\ot \f(J)\lon \f\left((I-i)\sqcup J\right)
$$
is often denoted in the literature by $\circ_i^{I,J}$.

An {\em  operad with unit}\, is, by definition, a non-unital operad equipped with a morphism
$\id_\bu: \id\rar \f(\bu)$ for any one point set $\bu$ such that the compositions
$$
\f(I)\rar \f(I)\ot \id \stackrel{\Id\ot\id_\bu}{\lon} \f(I)\ot \f(\bu)
\stackrel{\circ_i^{I,\bu}}{\lon}\f(I)
$$
and
$$
\f(I)\rar  \id\ot \f(I)\stackrel{\id_\bu\ot\Id}{\lon} \f(\bu)\ot \f(I)
\stackrel{\circ_\bu^{\bu,I}}{\lon}\f(I)
$$
are the identities for any finite set $I$ and any $i\in I$.

\subsubsection{\bf Fourth definition}
 A {\em non-unital  operad}\, is an $\cS$-module, $\f$,
together with a family of natural transformations,
$$
\left\{\circ_{f}: \f(I)\ot \bigotimes_{i\in I}\f(f^{-1}(i)) \lon \f(J)\right\}_{f:J\rar I}
$$
parameterized by a family, $\{f:J\twoheadrightarrow I\}$,  of surjections of finite sets, such that,
for any triple $K \stackrel{g}{\twoheadrightarrow} J\stackrel {f}{\twoheadrightarrow} I$ the diagram
$$
\xymatrix{
 \left[\f(I)\ot \bigotimes_{i\in I}\f(f^{-1}(i))\right]\bigotimes_{j\in J}\f(g^{-1}(j))
  \ar[r]^{\hspace{13mm}\circ_f\ot \Id} \ar[d]_{\simeq} &   \f(J)\ot \bigotimes_{j\in J}\f(g^{-1}(j))\ar[dd]_\mu\\
\f(I)\ot \left[\bigotimes_{i\in I}\f(f^{-1}(i))\bigotimes_{j_i\in (f)^{-1}(i)}\f(g^{-1}(j_i))\right] \ar[d]_{\ot\circ_{g_i}} &\\
 \f(I)\ot \bigotimes_{i\in I}\f((fg)^{-1}(i))\ar[r]_{\circ_{fg}} & \f(K)
}
$$
commutes.
An {\em  operad with unit}\, is, by definition, a non-unital operad equipped with a morphism
$\id_\bu: \id\rar \f(\bu)$ for any one point set $\bu$ such that the compositions
$$
\f(I)\rar \f(I)\ot \id^{\ot \# I} \stackrel{\Id\ot\id_\bu^{\ot \# I}}{\lon} \f(I)\ot \f(\bu)^{\ot \# I}
\stackrel{\circ_{Id}}{\lon}\f(I)
$$
and
$$
\f(I)\rar  \id\ot \f(I)\stackrel{\circ_{I\leftarrowtail \bu}}{\lon} \f(I)
$$
are identities.

\subsubsection{\bf (Non)equivalences of definitions} All the four definitions
of operads with unit are equivalent to each other. The first three definitions of non-unital
operads are equivalent to each other, but not to the fourth definition. Every non-unital operad
in the sense of the first three definitions is a non-unital operad in the sense of the fourth definition,
 but, obviously, not vice versa. A free non-unital operad in the sense of the fourth
  definition uses {\em leveled}\, trees rather than the ordinary ones.

\sip

In this paper we always understand a non-unital operad in the sense of any of the first three definitions.
However, when we work with non-unital {\em coloured}\, operads we can in principle have a mixture of both approaches, one approach for one set of colours and another inequivalent approach for another set of colours. Such a mixture of two non-equivalent approaches does indeed happen in the geometric
models for various operads of homotopy morphisms between homotopy algebras. We give a rigorous definition of that mixture below under the name of {\em a non-unital coloured operad of transformation type}.


\subsection{Coloured operads} Let $\Phi$ be a set which we refer to as the set of {\em colours}.
An $n$-{\em corolla},
$$
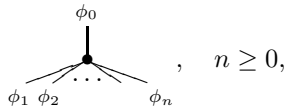
\begin{xy}
 <0mm,0mm>*{\bu};
 <0mm,0mm>*{};<0mm,4.5mm>*{}**@{-},
<-0.4mm,-0.2mm>*{};<-8mm,-3mm>*{}**@{-},
 <-0.5mm,-0.3mm>*{};<-4.5mm,-3mm>*{}**@{-},
 <0mm,0mm>*{};<0mm,-2.6mm>*{\ldots}**@{},
 <0.5mm,-0.3mm>*{};<4.5mm,-3mm>*{}**@{-},
 <0.4mm,-0.2mm>*{};<8mm,-3mm>*{}**@{-};
<0mm,6mm>*{^{\phi_0}};
<-9mm,-5mm>*{_{\phi_1}};
<-5mm,-5mm>*{_{\phi_2}};
<10mm,-5mm>*{_{\phi_n}};
 \end{xy}, \ \ \ n\geq 0,
$$
whose all legs are decorated with some (not-necessarily distinct) elements $\phi_0,\phi_1,\ldots,\phi_n\in \Phi$ is a called an $\Phi$-coloured $n$-corolla.
If the set $\Phi$ consists just of a few elements, then we often make legs dashed or wiggy to indicate their colours,
for example
$$
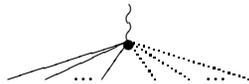
\Ba{c}
\xy
 <0mm,-0.5mm>*{\bu};
 <0mm,0mm>*{};<0mm,5mm>*{}**@{~},
 <0mm,0mm>*{};<-16mm,-5mm>*{}**@{-},
 <0mm,0mm>*{};<-11mm,-5mm>*{}**@{-},
 <0mm,0mm>*{};<-3.5mm,-5mm>*{}**@{-},
 <0mm,0mm>*{};<-6mm,-5mm>*{...}**@{},
 <0mm,0mm>*{};<16mm,-5mm>*{}**@{.},
 <0mm,0mm>*{};<8mm,-5mm>*{}**@{.},
 <0mm,0mm>*{};<3.5mm,-5mm>*{}**@{.},
 <0mm,0mm>*{};<11.6mm,-5mm>*{...}**@{},
 \endxy
\Ea
$$

Let $\cT^\Phi$ be the set of all possible  connected genus $0$ graphs
constructed
 from $\Phi$-coloured corollas
by taking their disjoint unions  and then gluing some output legs  with  input legs of the same colour.
The resulting graph is called a  {\em $\Phi$-coloured tree}.

Now repeating all the  first three definitions above with the symbol $\cT$ replaced by $\cT^\Phi$
we obtain three equivalent definitions of a (non-unital) $\Phi$-{\em coloured operad}\, in a
 symmetric monoidal category $\cC$.

\subsection{Coloured operads of transformation type}
Many important examples of coloured operads come from ordinary operads and their {\em modules}.

\sip

Let $\f_{in}$  and $\f_{out}$ be ordinary non-unital operads. An $\cS$-module $\cM$ is said to be a {\em bimodule of transformation type}\, over operads $\f_{in}$ and $\f_{out}$ if
\Bi
\item[(i)] $\cM$ is a right module over $\f_{out}$ in the sense of the first definitions of a non-unital operad, i.e.\ for any finite sets $I$ and $J$ and any $i\in I$ there is a morphism
$$
\circ_i^{I,J}: \cM(I)\ot \f_{in}(J)\lon \cM((I-i)\sqcup J)
$$
which is natural in $i$, $I$ and $J$ and satisfies obvious associativity conditions;
\item[(ii)] $\cM$ is a right pseudo-module over $\f_{in}$ in the sense of the fourth definition of a non-unital operad, i.e.\ for any surjection $f: J\twoheadrightarrow I$ there is a morphism,
$$
\circ_{f}: \f(I)\ot \bigotimes_{i\in I}\cM(f^{-1}(i)) \lon \cM(J)
$$
such that, for any triple $K \stackrel{g}{\twoheadrightarrow} J\stackrel {f}{\twoheadrightarrow} I$ the diagram
$$
\xymatrix{
 \left[\f(I)\ot \bigotimes_{i\in I}\f(f^{-1}(i))\right]\bigotimes_{j\in J}\cM(g^{-1}(j))
  \ar[r]^{\hspace{13mm}\circ_f\ot \Id} \ar[d]_{\simeq} &   \f(J)\ot \bigotimes_{j\in J}\cM(g^{-1}(j))\ar[dd]_\mu\\
\f(I)\ot \left[\bigotimes_{i\in I}\f(f^{-1}(i))\bigotimes_{j_i\in (f)^{-1}(i)}\cM(g^{-1}(j_i))\right] \ar[d]_{\ot\circ_{g_i}} &\\
 \f(I)\ot \bigotimes_{i\in I}\cM((fg)^{-1}(i))\ar[r]_{\circ_{fg}} &  \cM(K)
}
$$
commutes.
\Ei

The colimit $\f_{in}\oplus \cM\oplus \f_{out}$ has then a natural structure of a non-unital two-coloured operad  of mixed type which we call a {\em non-unital coloured operad of transformation type}.
Such operads often occur when, for example, one is interested in universal morphisms from $\f_{in}$-algebras
to $\f_{out}$-algebras.
Propositions {\ref{2: Propos on face complex of Mor(A_infty)}} and
{\ref{4': Propos on the face complex of Mor(Lie_infty)}} describe typical examples of such 2-coloured operads.

\sip

The above notion can be straightforwardly generalized to the case when  $\f_{in}$  and $\f_{out}$ are themselves non-unital coloured operads. The associated colimits $\f_{in}\oplus \cM\oplus \f_{out}$
are also called {\em non-unital coloured operad of transformation type}. Theorem
{\ref{4: Theorem on Mor(CO_infty) operad}} describes an example.

\bip

\bip


{\em Acknowledgement}. {\small It is a pleasure to thank Johan Alm,
Johan Gran\aa ker  and Carlo Rossi for valuable discussions. I am very grateful to the referee for a spotting  a mistake in the original definition of the operad $\cG^{\uparrow\downarrow}$ in \S 7.1.1 and for numerous
very useful suggestions and comments.

\newpage
\def\cprime{$'$}

\end{document}